\documentclass[a4paper]{amsart}
\usepackage{cmbright}
\usepackage{mathrsfs}
\usepackage[utf8]{inputenc}
\usepackage[T1]{fontenc}
\usepackage[english]{babel}
\usepackage{amsfonts}
\usepackage{amsthm}
\usepackage{amssymb}
\usepackage{mathcomp}
\usepackage{mathrsfs}
\usepackage[bb=boondox]{mathalfa}
\usepackage{graphicx}
\usepackage{epstopdf}
\usepackage[centertags]{amsmath}
\usepackage[pdftex]{hyperref}
\usepackage{vmargin}
\usepackage{tikz}
\usepackage{tikz-cd}
\usepackage{dsfont}
\usepackage{cases}
\usepackage{mathtools}
\usepackage{shuffle}
\usepackage{adjustbox}

\newcommand{\mycomment}[1]{}

\hypersetup{
    colorlinks,
    linkcolor=[rgb]{0.733, 0.149, 0.286},
    citecolor=[rgb]{0.733, 0.149, 0.286},
    urlcolor=[rgb]{0.733, 0.149, 0.286}
}

\title{Homotopical operadic calculus in positive characteristic}
\author{Brice Le Grignou \quad Victor Roca i Lucio}

\address{Brice Le Grignou,}
\email{\href{mailto:bricelegrignou@gmail.com}{bricelegrignou@gmail.com}}

\address{Victor Roca i Lucio, Ecole Polytechnique Fédérale de Lausanne, EPFL,
CH-1015 Lausanne, Switzerland}
\email{\href{mailto:victor.rocalucio@epfl.ch}{victor.rocalucio@epfl.ch}}

\date{\today}

\subjclass[2020]{18M70,18N40,18N55,18N60}

\keywords{Homotopical operadic calculus, algebraic operads, Koszul duality, bar-cobar adjunctions, positive characteristic.}

\newtheorem{theoremintro}{Theorem}

\newtheorem*{definitionintro}{Definition}

\begin{document}

\theoremstyle{plain}
\newtheorem{theorem}{Theorem}
\newtheorem*{theorem*}{Theorem}
\newtheorem{lemma}{Lemma}
\newtheorem{proposition}{Proposition}
\newtheorem{assumption}{Assumption}
\newtheorem{corollary}{Corollary}

\theoremstyle{definition}
\newtheorem{definition}{Definition}
\newtheorem{hypothesis}{Hypothesis}

\theoremstyle{remark}
\newtheorem{remark}{\sc Remark}
\newtheorem{example}{\sc Example}
\newtheorem*{notation}{\sc Notation}


\newcommand{\qi}{\xrightarrow{ \,\smash{\raisebox{-0.65ex}{\ensuremath{\scriptstyle\sim}}}\,}}
\newcommand{\lqi}{\xleftarrow{ \,\smash{\raisebox{-0.65ex}{\ensuremath{\scriptstyle\sim}}}\,}}

\newcommand{\draftnote}[1]{\marginpar{\raggedright\textsf{\hspace{0pt} \tiny #1}}}
\newcommand{\ac}{{\scriptstyle \text{\rm !`}}}

\newcommand{\Ch}{\categ{Ch}}
\newcommand{\Catesmall}{\categ{Cat}_{\categ E, \mathrm{small}}}
\newcommand{\Operade}{\Operad_{\categ E}}
\newcommand{\Operadecprime}{\Operad_{\categ E}}
\newcommand{\Operadesmall}{\Operad_{\categ E, \mathrm{small}}}
\newcommand{\eII}{\mathcal{I}}
\newcommand{\ecateg}[1]{\mathcal{#1}}
\newcommand{\cmonlax}{\categ{CMon}_\lax}
\newcommand{\cmonoplax}{\categ{CMon}_\oplax}
\newcommand{\cmonstrong}{\categ{CMon}_\strong}
\newcommand{\cmonstrict}{\categ{CMon}_\strict}
\newcommand{\cat}{\mathrm{cat}}
\newcommand{\lax}{\mathrm{lax}}
\newcommand{\oplax}{\mathrm{oplax}}
\newcommand{\strong}{\mathrm{strong}}
\newcommand{\strict}{\mathrm{strict}}
\newcommand{\pl}{\mathrm{pl}}
\newcommand{\qp}{\mathrm{qp}}
\newcommand{\gr}{\mathrm{gr}}

\newcommand{\Cats}{\mathsf{Cats}}
\newcommand{\Functors}{\mathsf{Functors}}

\newcommand{\Ob}{\mathrm{Ob}}
\newcommand{\tr}{\mathrm{tr}}
\newcommand{\catch}{\mathsf{Ch}}
\newcommand{\categ}[1]{\mathsf{#1}}
\newcommand{\set}[1]{\mathrm{#1}}
\newcommand{\catoperad}[1]{\mathsf{#1}}
\newcommand{\operad}[1]{\mathcal{#1}}
\newcommand{\algebra}[1]{\mathrm{#1}}
\newcommand{\coalgebra}[1]{\mathrm{#1}}
\newcommand{\cooperad}[1]{\mathcal{#1}}
\newcommand{\ocooperad}[1]{\overline{\mathcal{#1}}}
\newcommand{\catofmod}[1]{{#1}\mathrm{-}\mathsf{mod}}
\newcommand{\catofcog}[1]{#1\mathrm{-}\mathsf{cog}}
\newcommand{\catcog}[1]{\Cog\left(#1\right)}

\newcommand{\catdgmod}[1]{\categ{dg}~#1\text{-}\categ{mod}}
\newcommand{\catpdgmod}[1]{\categ{pdg}~#1\text{-}\categ{mod}}
\newcommand{\catgrmod}[1]{\categ{gr}~#1\text{-}\categ{mod}}

\newcommand{\catdgalg}[1]{\categ{dg}~#1\text{-}\categ{alg}}
\newcommand{\catpdgalg}[1]{\categ{pdg}~#1\text{-}\categ{alg}}
\newcommand{\catgralg}[1]{\categ{gr}~#1\text{-}\categ{alg}}
\newcommand{\catcurvalg}[1]{\categ{curv}~#1\text{-}\categ{alg}}

\newcommand{\catdgcompalg}[1]{\categ{dg}~#1\text{-}\categ{alg}^{\mathsf{qp}\text{-}\mathsf{comp}}}
\newcommand{\catpdgcompalg}[1]{\categ{pdg}~#1\text{-}\categ{alg}^{\mathsf{qp}\text{-}\mathsf{comp}}}
\newcommand{\catgrcompalg}[1]{\categ{gr}~#1\text{-}\categ{alg}^{\mathsf{qp}\text{-}\mathsf{comp}}}
\newcommand{\catcurvcompalg}[1]{\categ{curv}~#1\text{-}\categ{alg}^{\mathsf{qp}\text{-}\mathsf{comp}}}

\newcommand{\catdgcog}[1]{\categ{dg}~#1\text{-}\categ{cog}}
\newcommand{\catpdgcog}[1]{\categ{pdg}~#1\text{-}\categ{cog}}
\newcommand{\catgrcog}[1]{\categ{gr}~#1\text{-}\categ{cog}}
\newcommand{\catcurvcog}[1]{\categ{curv}~#1\text{-}\categ{cog}}

\newcommand{\dgoperads}{\categ{dg}~\categ{Operads}}
\newcommand{\pdgoperads}{\categ{pdg}~\categ{Operads}}
\newcommand{\groperads}{\categ{gr}~\categ{Operads}}
\newcommand{\curvcooperads}{\categ{curv}~\categ{Cooperads}}
\newcommand{\grcooperads}{\categ{gr}~\categ{Cooperads}}

\newcommand{\pdgcooperads}{\categ{pdg}~\categ{Cooperads}}
\newcommand{\dgcooperads}{\categ{dg}~\categ{Cooperads}}
\newcommand{\conil}{\mathsf{conil}}

\newcommand{\catalg}[1]{\Alg\left(#1\right)}

\newcommand{\catofcolcomod}[1]{\mathsf{Col}\mathrm{-}{#1}\mathrm{-}\mathsf{comod}}
\newcommand{\catofcoalgebra}[1]{{#1}\mathrm{-}\mathsf{cog}}
\newcommand{\catofalg}[1]{\operad{#1}\mathrm{-}\mathsf{alg}}
\newcommand{\catofalgebra}[1]{{#1}\mathrm{-}\mathsf{alg}}
\newcommand{\mbs}{\mathsf{S}}
\newcommand{\catocol}[1]{\mathsf{O}_{\set{#1}}}
\newcommand{\catoftrees}{\mathsf{Trees}}
\newcommand{\cattcol}[1]{\catoftrees_{\set{#1}}}
\newcommand{\catcorcol}[1]{\mathsf{Corol}_{\set{#1}}}
\newcommand{\Einfty}{\mathcal{E}_{\infty}}
\newcommand{\nuEinfty}{\mathcal{nuE}_{\infty}}
\newcommand{\Fun}[3]{\mathrm{Fun}^{#1}\left(#2,#3\right)}
\newcommand{\III}{\operad{I}}
\newcommand{\treeoperad}{\mathbb{T}}
\newcommand{\treemodule}{\mathbb{T}}
\newcommand{\core}{\mathrm{Core}}
\newcommand{\forget}{\mathrm{U}}
\newcommand{\treemonad}{\mathbb{O}}
\newcommand{\cogcomonad}[1]{\mathbb{L}^{#1}}
\newcommand{\cofreecog}[1]{\mathrm{L}^{#1}}

\newcommand{\barfunctor}[1]{\mathrm{B}_{#1}}
\newcommand{\baradjoint}[1]{\mathrm{B}^\dag_{#1}}
\newcommand{\cobarfunctor}[1]{\mathrm{C}_{#1}}
\newcommand{\cobaradjoint}[1]{\mathrm{C}^\dag_{#1}}
\newcommand{\Operad}{\mathsf{Operad}}
\newcommand{\coOperad}{\mathsf{coOperad}}

\newcommand{\Aut}[1]{\mathrm{Aut}(#1)}

\newcommand{\verte}[1]{\mathrm{vert}(#1)}
\newcommand{\edge}[1]{\mathrm{edge}(#1)}
\newcommand{\leaves}[1]{\mathrm{leaves}(#1)}
\newcommand{\inner}[1]{\mathrm{inner}(#1)}
\newcommand{\inp}[1]{\mathrm{input}(#1)}

\newcommand{\field}{\mathbb{K}}
\newcommand{\mbk}{\mathbb{K}}
\newcommand{\mbn}{\mathbb{N}}

\newcommand{\id}{\mathrm{Id}}
\newcommand{\ii}{\mathrm{id}}
\newcommand{\unit}{\mathds{1}}

\newcommand{\Lin}{Lin}

\newcommand{\BijC}{\mathsf{Bij}_{C}}

\newcommand{\kk}{\Bbbk}
\newcommand{\PP}{\mathcal{P}}
\newcommand{\C}{\mathcal{C}}
\newcommand{\Sy}{\mathbb{S}}
\newcommand{\Tree}{\mathsf{Tree}}
\newcommand{\treemod}{\mathbb{T}}
\newcommand{\Dend}{\Omega}
\newcommand{\aDend}{\Omega^{\mathsf{act}}}
\newcommand{\cDend}{\Omega^{\mathsf{core}}}
\newcommand{\cDendpart}{\cDend_{\mathsf{part}}}

\newcommand{\build}{\mathrm{Build}}
\newcommand{\col}{\mathrm{col}}

\newcommand{\HOM}{\mathrm{HOM}}
\newcommand{\Hom}[3]{\mathrm{hom}_{#1}\left(#2 , #3 \right)}
\newcommand{\ov}{\overline}
\newcommand{\otimeshadamard}{\otimes_{\mathbb{H}}}

\newcommand{\I}{\mathcal{I}}
\newcommand{\Aa}{\mathcal{A}}
\newcommand{\BB}{\mathcal{B}}
\newcommand{\CC}{\mathcal{C}}
\newcommand{\DD}{\mathcal{D}}
\newcommand{\EE}{\mathcal{E}}
\newcommand{\FF}{\mathcal{F}}
\newcommand{\II}{\mathbb{1}}
\newcommand{\RR}{\mathcal{R}}
\newcommand{\UU}{\mathcal{U}}
\newcommand{\VV}{\mathcal{V}}
\newcommand{\WW}{\mathcal{W}}
\newcommand{\AAA}{\mathscr{A}}
\newcommand{\BBB}{\mathscr{B}}
\newcommand{\CCC}{\mathscr{C}}
\newcommand{\DDD}{\mathscr{D}}
\newcommand{\EEE}{\mathscr{E}}
\newcommand{\FFF}{\mathscr{F}}

\newcommand{\PPP}{\mathscr{P}}
\newcommand{\QQQ}{\mathscr{Q}}

\newcommand{\QQ}{\mathcal{Q}}

\newcommand{\KKK}{\mathscr{K}}
\newcommand{\KK}{\mathcal{K}}

\newcommand{\ra}{\rightarrow}

\newcommand{\Ai}{\mathcal{A}_{\infty}}
\newcommand{\uAi}{u\mathcal{A}_{\infty}}
\newcommand{\uEinfty}{u\mathcal{E}_{\infty}}
\newcommand{\uAW}{u\mathcal{AW}}

\newcommand{\uAlg}{\mathsf{Alg}}
\newcommand{\nuAlg}{\mathsf{nuAlg}}
\newcommand{\cAlg}{\mathsf{cAlg}}

\newcommand{\ucAlg}{\mathsf{ucAlg}}
\newcommand{\Cog}{\mathsf{Cog}}
\newcommand{\nuCog}{\mathsf{nuCog}}
\newcommand{\uAWcog}{u\mathcal{AW}-\mathsf{cog}}

\newcommand{\uCog}{\mathsf{uCog}}
\newcommand{\cCog}{\mathsf{cCog}}
\newcommand{\ucCog}{\mathsf{ucCog}}
\newcommand{\cNilCog}{\mathsf{cNilCog}}
\newcommand{\ucNilCog}{\mathsf{ucNilCog}}
\newcommand{\NilCog}{\mathsf{NilCog}}

\newcommand{\Cocom}{\mathsf{Cocom}}
\newcommand{\uCocom}{\mathsf{uCocom}}
\newcommand{\NilCocom}{\mathsf{NilCocom}}
\newcommand{\uNilCocom}{\mathsf{uNilCocom}}
\newcommand{\Liealg}{\mathsf{Lie}-\mathsf{alg}}
\newcommand{\cLiealg}{\mathsf{cLie}-\mathsf{alg}}
\newcommand{\Alg}{\mathsf{Alg}}
\newcommand{\Linfty}{\mathcal{L}_{\infty}}
\newcommand{\CMC}{\mathfrak{CMC}}
\newcommand{\Tfree}{\mathbb{T}}

\newcommand{\Hinich}{\mathsf{Hinich} -\mathsf{cog}}

\newcommand{\Ccomod}{\mathscr C -\mathsf{comod}}
\newcommand{\Pmod}{\mathscr P -\mathsf{mod}}

\newcommand{\cCoop}{\mathsf{cCoop}}

\newcommand{\Set}{\mathsf{Set}}
\newcommand{\sSet}{\mathsf{sSet}}
\newcommand{\dgMod}{\mathsf{dgMod}}
\newcommand{\gMod}{\mathsf{gMod}}
\newcommand{\catOrd}{\mathsf{Ord}}
\newcommand{\catBij}{\mathsf{Bij}}
\newcommand{\catSmod}{\mbs\mathsf{mod}}
\newcommand{\EEtw}{\mathcal{E}\text{-}\mathsf{Tw}}
\newcommand{\OpBim}{\mathsf{Op}\text{-}\mathsf{Bim}}

\newcommand{\Palg}{\mathcal{P}-\mathsf{alg}}
\newcommand{\Qalg}{\mathcal{Q}-\mathsf{alg}}
\newcommand{\Pcog}{\mathcal{P}-\mathsf{cog}}
\newcommand{\Qcog}{\mathcal{Q}-\mathsf{cog}}
\newcommand{\Ccog}{\mathcal{C}-\mathsf{cog}}
\newcommand{\Dcog}{\mathcal{D}-\mathsf{cog}}
\newcommand{\uCoCog}{\mathsf{uCoCog}}

\newcommand{\Artinalg}{\mathsf{Artin}-\mathsf{alg}}

\newcommand{\colim}[1]{\underset{#1}{\mathrm{colim}}}
\newcommand{\Map}{\mathrm{Map}}
\newcommand{\Def}{\mathrm{Def}}
\newcommand{\Bij}{\mathrm{Bij}}
\newcommand{\op}{\mathrm{op}}

\newcommand{\undern}{\underline{n}}
\newcommand{\dginterval}{{N{[1]}}}
\newcommand{\dgsimplex}[1]{{N{[#1]}}}

\newcommand{\cofree}{ T^c}
\newcommand{\Tw}{ Tw}
\newcommand{\End}{\mathcal{E}\mathrm{nd}}
\newcommand{\catEnd}{\mathsf{End}}
\newcommand{\coEnd}{\mathrm{co}\End}
\newcommand{\Mult}{\mathrm{Mult}}
\newcommand{\coMult}{\mathrm{coMult}}

\newcommand{\Lie}{\mathcal{L}\mathr{ie}}
\newcommand{\As}{\mathcal{A}\mathrm{s}}
\newcommand{\uAs}{\mathrm{u}\As}
\newcommand{\coAs}{\mathrm{co}\As}
\newcommand{\Com}{\mathcal{C}\mathrm{om}}
\newcommand{\uCom}{\mathrm{u}\Com}
\newcommand{\Perm}{\catoperad{Perm}}
\newcommand{\uBE}{\mathrm{u}\mathcal{BE}}
\newcommand{\uBEs}{{\uBE}^{\mathrm s}}

\newcommand{\comp}{\circ}
\newcommand{\restrictionextension}{\mathrm{RE}}
\newcommand{\extension}{\mathrm{E}}
\newcommand{\extensionone}{\mathrm{E}_1}
\newcommand{\extensiontwo}{\mathrm{E}_2}
\newcommand{\restrictionone}{\mathrm{R}_1}
\newcommand{\restrictiontwo}{\mathrm{R}_2}

\newcommand{\itemt}{\item[$\triangleright$]}

\newcommand{\poubelle}[1]{}

\newcommand{\D}{\operad D}

\newcommand{\Victor}[1]{\textcolor{blue}{#1}}
\newcommand{\Brice}[1]{\textcolor{red}{#1}}

\maketitle

\begin{abstract}
Algebraic operads provide a powerful tool to understand the homotopy theory of the types of (co)algebras they encode. So far, the principal results and methods that this theory provides were only available in characteristic zero. The reason is that operads carry an action of all the symmetric groups, whose representation theory becomes much more involved in positive characteristic. The goal of this paper is to extend these results and methods to a positive characteristic setting. We solve the main problems that appear in this new setting by using the notion of a quasi-planar cooperad as the building block of the theory.  
\end{abstract}

\setcounter{tocdepth}{1}
\tableofcontents
\setcounter{tocdepth}{1}
\tableofcontents

\section*{Introduction}
Operads are algebraic objects which encode other types of algebraic structures: Lie algebras, associative algebras, Batalin--Vilkovisky algebras, Gerstenhaber algebras, $\mathcal{A}_\infty$-algebras, $\mathcal{L}_\infty$-algebras, and many more. The theory of algebraic operads provides us with a large set of methods which allow us to study the homotopy theory of differential graded (dg) $\mathcal{P}$-algebras if they are encoded by a dg operad $\mathcal{P}$. We will call this set of methods \textit{homotopical operadic calculus}. We refer to \cite{LodayVallette} for a detailed account of the theory of algebraic operads.

\medskip

Let us explain these methods. Over a characteristic zero field, the category of dg $\mathcal{P}$-algebras always admits a model category structure where weak-equivalences are given by quasi-isomorphisms and fibrations are given by degree-wise epimorphisms. This model structure can be obtained by transfer along the free-forgetful adjunction and presents the homotopy theory of dg $\mathcal{P}$-algebras. Cofibrant objects in this category do not admit an easy description, and the homotopy category is in general quite hard to understand. One way to get an easier description of this homotopy category is via bar-cobar adjunctions. To any operad $\mathcal{P}$ one can always associate a cooperad $\mathcal{C}$ which is Koszul dual in some sense. There are in fact two canonical choices for $\mathcal{C}$: in certain cases, a minimal choice given by classical Koszul duality, or in general the one given by the operadic bar construction on $\mathcal{P}$. For any choice of Koszul dual cooperad $\C$, this duality is instantiated in the data of a twisting morphism $\alpha: \C \longrightarrow \mathcal{P}$. And for any such twisting morphism $\alpha$, there exists a bar-cobar adjunction 

\[
\begin{tikzcd}[column sep=5pc,row sep=3pc]
          \mathsf{dg}~\C\text{-}\mathsf{cog} \arrow[r, shift left=1.1ex, "\Omega_{\alpha}"{name=F}] & \mathsf{dg}~ \mathcal{P}\text{-}\mathsf{alg}, \arrow[l, shift left=.75ex, "\mathrm{B}_{\alpha}"{name=U}]
            \arrow[phantom, from=F, to=U, , "\dashv" rotate=-90]
\end{tikzcd}
\]

between the category of dg $\mathcal{P}$-algebras and the category of dg $\C$-coalgebras. Notice that these constructions are generalizations of the classical bar construction of \cite{Bar} and of the classical cobar of \cite{Cobar} for associative (co)algebras to other types of algebras. This general construction also recovers the adjunction between dg Lie algebras and conilpotent dg cocommutative coalgebras of \cite{Quillen69}. Let us mention that having a bar-cobar adjunction is essential for Koszul duality in the sense of \cite{Priddy70}, which can be generalized to many types of algebras, see \cite{Joan2,Najib,Berglund1}. Furthermore, it is crucial in order to generalize André--Quillen cohomology \cite{andre,quillen} to other types of algebras, see \cite{Joan1}.

\medskip

The natural model category structure on dg $\mathcal{P}$-algebras can always be transferred along the bar-cobar adjunction, and if $\alpha$ satisfies a "Koszulness" property, this adjunction becomes a Quillen equivalence. This type of arguments was first developed for dg Lie algebras by V. Hinich in \cite{Hinich} using the adjunction of \cite{Quillen69}. Note that these are the constructions that bridge between Quillen and Sullivan's approaches of \cite{Quillen69,Sullivan77}. Then K. Lefèvre-Hasegawa developed these methods in \cite{LefevreHasegawa03} for non-unital dg associative algebras and conilpotent dg coassociative coalgebras. The analogue for unital dg associative algebras was developed by L. Positselski in \cite{twokinds}, where one needs to consider \textit{curved} conilpotent dg coalgebras on the other side. The general non-unital case was proven by B. Vallette in \cite{Brunohomotopy}. Finally, we refer to \cite{unitalalgebras} for the general constructions in the unital case. 

\medskip

A great advantage of looking at the transferred model category structure on dg $\C$-coalgebras is that fibrant-cofibrant objects admit a very easy description: they are given by dg $\C$-coalgebras whose underlying graded $\C$-coalgebra structure is cofree (usually called quasi-free objects). This naturally leads to the notion of an $\infty$-morphism between two dg $\mathcal{P}$-algebras: it is given by a morphism of dg $\C$-coalgebras between their respective bar constructions. This relaxed notion of morphisms has a great advantage: any $\infty$-quasi-isomorphism admits an homotopy inverse, and it can be shown that two dg $\mathcal{P}$-algebras are linked by a \textit{zig-zag} of quasi-isomorphisms if and only if there exists an $\infty$-quasi-isomorphism between them. For instance, this fact was crucially exploited by M. Kontsevich in order to prove the deformation-quantization of Poisson manifolds in \cite{Kontsevich03}. 

\medskip

A cofibrant replacement $\mathcal{P}_\infty$ of an operad $\mathcal{P}$ encodes a notion of a $\mathcal{P}$-algebra \textit{up to homotopy}. Any quasi-isomorphism of operads induces a Quillen equivalence between their respective categories of algebras, therefore considering algebraic structures up to homotopy does not change the homotopy category. A key feature of algebraic structures up to homotopy is the homotopy transfer theorem: given a contraction of dg modules 

\vspace{1pc}
\hspace{10.5pc}
\begin{tikzcd}[column sep=3.5pc,row sep=3pc]
V \arrow[r, shift left=1.1ex, "p"{name=F}] \arrow[loop left, distance=3em, start anchor={[yshift=-1.5ex]west}, end anchor={[yshift=1.5ex]west}]{}{h}
&H~, \arrow[l, shift left=.75ex, "i"{name=U}]
\end{tikzcd}
\vspace{1pc}

and a dg $\mathcal{P}_\infty$-algebra structure on $V$, there exists a dg $\mathcal{P}_\infty$-algebra structure on $H$ such that $p$ and $i$ can also be extended into "inverse" $\infty$-quasi-isomorphisms. This is particularly useful when one chooses a contraction between a dg module  and its homology. Since the transferred structure onto the homology does not lose any homotopical information, one can always replace a dg $\mathcal{P}_\infty$-algebra by  the transferred dg $\mathcal{P}_\infty$-algebra structure on its homology. We refer to \cite{LodayVallette} for an account of these methods and their history.

\medskip

Let briefly mention some of the applications of these methods: in rational homotopy theory, one can use them to study formality questions \cite{Berglund2,GeoffroyGabriel} or to construct rational models \cite{Liemodels,BrunoDaniel}. In symplectic geometry, $\infty$-morphisms and the homotopy transfer theorem play a key role in \cite{Fukaya}, and more general operadic methods can be applied in Morse theory, see \cite{mazuir1,mazuir2}. See \cite{BrunoSurvey} and \cite{LodayVallette} for more detailed instances of applications.

\medskip

Dual results were recently proven for coalgebras by D. Lejay and the first author in \cite{linearcoalgebras}. Coalgebras over cooperads only encode a certain type of coalgebras which are called \textit{conilpotent}. If one wants to encode all coalgebras (say, coassociative, cocommutative, etc) without a conilpotency condition, one cannot use cooperads. We use here (reversed) operads. In this case, given a twisting morphism $\alpha: \C \longrightarrow \mathcal{P}$, they constructed a \textit{complete} bar-cobar adjunction

\[
\begin{tikzcd}[column sep=5pc,row sep=3pc]
          \mathsf{dg}~\mathcal{P}\text{-}\mathsf{cog} \arrow[r, shift left=1.1ex, "\widehat{\Omega}_{\alpha}"{name=F}] & \mathsf{dg}~ \mathcal{C}\text{-}\mathsf{alg}^{\mathsf{comp}}, \arrow[l, shift left=.75ex, "\widehat{\mathrm{B}}_{\alpha}"{name=U}]
            \arrow[phantom, from=F, to=U, , "\dashv" rotate=-90]
\end{tikzcd}
\]

between dg $\mathcal{P}$-coalgebras and complete dg $\C$-algebras. The notion of an algebra over a cooperad gives a new type of algebraic structures called \textit{absolute algebras}. These are algebraic structures endowed with meaningful notion of infinite sums of operations without presupposing an underlying topology. Here \textit{complete} means that the canonical topology induced by the absolute structure is separated. Most of the classical algebraic structures encoded by an operad (Lie algebras, associative algebras, $\mathcal{L}_\infty$-algebras, etc.) have an absolute analogue. Once again, one can transfer (when it exists) the model category structure of dg $\mathcal{P}$-coalgebras to dg $\C$-algebras using the complete bar-cobar adjunction. And if the operad $\mathcal{P}$ considered is given by the cobar construction $\Omega \C$ of $\C$, then the complete bar-cobar adjunction becomes a Quillen equivalence. The aforementioned methods, as well as the key notion of an absolute algebra were used by the second author in \cite{integration} to develop the integration theory of curved absolute $\mathcal{L}_\infty$-algebras, extending the work of \cite{Getzler09} and \cite{robertnicoud2020higher}.

\medskip

Fibrant-cofibrant dg $\C$-algebras are also given by quasi-free objects: one can then define $\infty$-morphisms and $\infty$-quasi-isomorphisms for coalgebras. They were shown to be "invertible" as well, which makes them a useful tool to study the homotopy theory of $\mathcal{P}$-coalgebras. Finally, an analogue version of the homotopy transfer theorem also holds for dg $\Omega\C$-coalgebras. Note that coalgebras over an operad can also be interpreted from a properadic perspective, and $\infty$-morphisms coincide with those introduced by E. Hoffbeck, J. Leray and B. Vallette in \cite{BrunoJohanEric}.

\medskip

Linear duality is a bridge between these two bar-cobar adjunctions, and thus between the homotopy theory of algebras and coalgebras. Indeed, the linear dual of a dg $\mathcal{P}$-coalgebra is always a dg $\mathcal{P}$-algebra, and the linear dual of a dg $\C$-coalgebra is a dg $\C$-algebra. These functors admit adjoints, and one obtains a \textit{duality square} made of four commuting adjunctions: 
\[
\begin{tikzcd}[column sep=5pc,row sep=5pc]
\catdgalg{\operad P}^{\mathsf{op}} \arrow[r,"\mathrm{B}_{\alpha}^{\mathsf{op}}"{name=B},shift left=1.1ex] \arrow[d,"(-)^\circ "{name=SD},shift left=1.1ex ]
&\catdgcog{\operad C}^{\mathsf{op}} \arrow[d,"(-)^*"{name=LDC},shift left=1.1ex ] \arrow[l,"\Omega_{\alpha}^{\mathsf{op}}"{name=C},,shift left=1.1ex]  \\
\catdgcog{\operad P}\arrow[r,"\widehat{\Omega}_{\alpha} "{name=CC},shift left=1.1ex]  \arrow[u,"(-)^*"{name=LD},shift left=1.1ex ]
&\mathsf{dg}~ \mathcal{C}\text{-}\mathsf{alg}^{\mathsf{comp}}~, \arrow[l,"\widehat{\mathrm{B}}_{\alpha}"{name=CB},shift left=1.1ex] \arrow[u,"(-)^\vee"{name=TD},shift left=1.1ex] \arrow[phantom, from=SD, to=LD, , "\dashv" rotate=0] \arrow[phantom, from=C, to=B, , "\dashv" rotate=-90]\arrow[phantom, from=TD, to=LDC, , "\dashv" rotate=0] \arrow[phantom, from=CC, to=CB, , "\dashv" rotate=-90]
\end{tikzcd}
\] 

This square can be promoted (under some conditions) into a square of Quillen adjunctions, using the model category structures constructed so far. This duality square allowed the second author to show that the homotopy theory of algebras and coalgebras are equivalent on objects with finite dimensional homology, see \cite{absolutealgebras} for more details on these constructions. Finally, let us mention that all these results also extend to non-augmented operads, which encode counital types of coalgebras, where on the other side one needs to consider curved algebras over a curved cooperad.

\medskip

\textbf{An overview of the positive characteristic case.} A (dg) operad is a collection $\{\mathcal{P}(n)\}_{n \in \mathbb{N}}$ of (dg) $\kk[\mathbb{S}_n]$-modules, for all $n \geq 0$, together with an additional structure that allows one to compose elements. The category of $\kk[\mathbb{S}_n]$-modules is fairly well-behaved when $\kk$ is a characteristic zero field: it is semi-simple, there is a combinatorial description of simple objects, and much more. In particular, the homotopy theory of dg $\kk[\mathbb{S}_n]$-modules behaves in a similar fashion to the homotopy theory of dg $\kk$-modules. 

\medskip

Things change drastically over a positive characteristic field. Understanding the representation theory of the symmetric groups is an active subject of research, with many open questions, see for instance \cite{ICMSymmetric}. The homotopy category of dg $\kk[\mathbb{S}_n]$-modules becomes very rich and hard to understand. A concrete consequence of these facts is that many statements in fail or at least, require extra hypothesis, over a positive characteristic field. 

\medskip

The first obstruction is given by the fact that, in general, dg $\mathcal{P}$-algebras do not admit a model category structure transferred from dg modules. Under the assumption that $\mathcal{P}(n)$ is a projective dg $\kk[\mathbb{S}_n]$-module for all $n \geq 0$, M. Spitzweck constructed in \cite{Spitzweck} a \textit{semi-model} category structure on dg $\mathcal{P}$-algebras, which is a weaker notion of a model category structure. B. Fresse then showed in \cite{Fresse} that the category of dg operad themselves admits a semi-model category structure and that any quasi-isomorphism between $\mathbb{S}$-projective dg operads induces an Quillen equivalence of semi-model categories. Parallel to these results, there is also the axiomatic point of view developed by C. Berger and I. Moerdijk in \cite{BergerMoerdijk}. They consider \textit{admissible} operads, those that admit a transferred model structure from dg modules, and give conditions under which an operad is admissible (for instance, if it is cofibrant). The same authors also constructed general resolutions for $\mathbb{S}$-projective dg operads in \cite{Wconstruction}. These results already have breakthrough applications. As an example, understanding the homotopy theory of $\mathcal{E}_\infty$-algebras in positive characteristic allowed M. Mandell to construct algebraic models for $p$-adic homotopy types, see \cite{Mandell1,Mandell2}. 

\medskip

The homotopy theory of operads themselves plays a crucial role in homotopical operadic calculus. In this direction, M. Dehling and B. Vallette developed in \cite{BrunoMalte} a general method that gives cofibrant resolutions of the operadic structure \textit{and} of the action of the symmetric groups \textit{at the same time}. Furthermore, they show that their cofibrant resolution, applied to an operad $\mathcal{P}$, is in fact isomorphic to $\Omega \mathrm{B}(\mathcal{E} \otimes \mathcal{P})$: here $\Omega,\mathrm{B}$ are respectively the cobar and bar constructions at the operadic level, and $\mathcal{E}$ is the Barratt-Eccles of \cite{BergerFresse}. This gives a clear indication that the Barratt-Eccles operad provides a universal way to construct resolutions that also take into account the action of the symmetric groups. Related to these ideas is the work of E. Hoffbeck, who constructed an analogue of André--Quillen cohomology for algebras over an dg operad in positive characteristic in \cite{eric1}, and developed the obstruction theory associated to it in \cite{eric2}. The later work crucially uses $\mathcal{E}$ and an universal filtration on $\mathrm{B}(\mathcal{E} \otimes \mathcal{P})$. 

\medskip

In a different direction, L. Brantner, R. Campos and J. Nuiten develop in \cite{pdalgebras} a theory of divided power operads, in order to encode divided power algebras. In particular, they endow these categories with semi-model category structures. This allows them to give point-set models of partition Lie algebras as defined in \cite{brantnermathew}, where they were shown to encode formal moduli problems over a positive characteristic field. Finally, let us mention several results in the direction of classical Koszul duality which are valid over any ring \cite{FressePartitionPoset,ChapotonVallette,FresseEn}. 

\medskip

\textbf{Main results.} The purpose of this paper is to generalize all the aforementioned methods of homotopical operadic calculus over any field. Let us mention first that the main motivation for developing these methods in positive characteristic is the article \cite{deuxiemepapier}. In \textit{op.cit}, we provide a new point of view on formal moduli problems via the duality square constructed here together with the main results of \cite{presentationalgebras}. Putting these two results together allows us to give a new proof of many of the main results concerning formal moduli problems, while generalizing the Koszul duality inherent to these types of results over a field of any characteristic. Another application is the paper by the second author about the integration theory of partition $\mathcal{L}_\infty$-algebras \cite{lietheoryp}, which also extensively uses the results of this paper. We expect that they will find many more applications, as their analogues in characteristic zero have so far.

\medskip

From now on, let us assume we work over a field $\kk$. The main idea to bypass all the major difficulties that arise over a positive characteristic field is to introduce the notion of a quasi-planar dg cooperad. Informally speaking, it amounts to a dg cooperad whose underlying graded cooperad is non-symmetric/planar, and even if its differential can interact in non-trivial ways with the underlying action of the symmetric groups, this interaction is still "controlled" by a suitable filtration on the cooperad. Note that even if our results are valid in the more general unital/curved setting, we will state them in the augmented/dg setting for simplicity. 

\begin{definitionintro}[Quasi-planar dg cooperad]
Let $(\C,d_\C)$ be a conilpotent dg cooperad. It is \textit{quasi-planar} if there exists a ladder of conilpotent dg cooperads

\[
\C^{(0)} \longrightarrow \C^{(1)} \longrightarrow \cdots \longrightarrow \C^{(i)} \longrightarrow \cdots 
\]
\vspace{0.1pc}

such that $\C$ is the colimit of the diagram and such that the following conditions are satisfied.

\medskip
\begin{enumerate}
\item For all $i$, the underlying conilpotent graded cooperad $\C^{(i)}_{\mathrm{gr}}$ is planar, meaning that there exists a conilpotent graded non-symmetric cooperad $\C_\pl^{(i)}$ and a given isomorphism

\[
\C_\pl^{(i)} \otimes \mathbb{S} \cong \C^{(i)}_{\mathrm{gr}}~,
\]
\vspace{0.1pc}

where $(- \otimes \mathbb{S})$ denotes the left adjoint to the forgetful functor from conilpotent graded cooperads to conilpotent graded non-symmetric cooperads. Moreover, the filtration preserves this planar structure in the sense that the graded map $\C^{(i)} \to \C^{(j)}$ (for $i <j$) is the image through the functor $-\otimes \mathbb S$ of the a map $\C^{(i)}_\pl \to \C^{(j)}_\pl$.

\medskip

\item For all $i$, the restriction of the coderivation of $d_{\C(i+1)}$ to $\C_\pl(i+1) \otimes 1$ factors through

        \[
        \C_\pl^{(i+1)} \otimes 1 \longrightarrow \left( \C_\pl^{(i+1)} \otimes 1 \right) + \C^{(i)} \hookrightarrow \C^{(i+1)};
        \]
        \vspace{0.1pc}
in other words, the differential of $\gr_{i+1} \C = (\C_\pl^{(i+1)}/\C_\pl^{(i)}) \otimes \mathbb S$ has the form $d_\pl \otimes \id_{\mathbb S}$.
\end{enumerate}
\end{definitionintro}

A first version of this notion is already present in \cite{linearcoalgebras}. Let us point out that the ladder can be indexed by any small ordinal $\alpha$, not only by $\omega$. Nevertheless, using Theorem \ref{thm A intro}, we will show that any quasi-planar cooperad $\C$ admits a \textit{canonical} quasi-planar ladder indexed by $\omega$. 

\medskip

Two key observations. Firstly, for any dg operad $\mathcal{P}$, the dg cooperad $\mathrm{B}(\mathcal{E} \otimes \mathcal{P})$ is a quasi-planar cooperad. Secondly, any operad $\mathcal{P}$ can be replaced by $\Omega\mathrm{B}(\mathcal{E} \otimes \mathcal{P})$ up to quasi-isomorphism. If $\mathcal{P}$ is $\mathbb{S}$-projective, this replacement does not change the underlying homotopy category. And if it is not, there was no meaningful homotopy category for dg $\mathcal{P}$-algebras to begin with. So, for all intents and purposes, one can restrict to operads of the form $\Omega \C$, where $\C$ is quasi-planar conilpotent dg cooperad.

\medskip

Therefore, let us fix a quasi-planar dg cooperad $\C$. We show that its cobar construction $\Omega \C$ is a cofibrant object in the semi-model category of dg operad of \cite{Fresse}. Not only that, we also show that any dg operad of the form $\Omega \C$, with $\C$ quasi-planar, admits a canonical $\operad E$-comonoid structure, given by explicit formulas.

\begin{theoremintro}[Theorem \ref{thm:becomodule}]\label{thm A intro}
Let $\C$ be a quasi-planar dg cooperad. There is a canonical morphism of dg operads 
\[
\begin{tikzcd}[column sep=2pc,row sep=0pc]
\Delta_{\operad E, \C}: \Omega\C  \arrow[r]
&\mathcal{E} \otimes \Omega\C \\
\end{tikzcd}
\]

which endows the dg operad $\Omega \operad C$ with a left $\operad E$-comodule structure.
\end{theoremintro}

This theorem can be seen as the positive characteristic analogue of the following fact: any dg operad $\mathcal{P}$ admits a canonical $\mathrm{u}\mathcal{C}\mathrm{om}$-comonoid structure. Notice that is is precisely this $\mathrm{u}\mathcal{C}\mathrm{om}$-comonoid structure that gives the canonical convolution curved $\mathcal{L}_\infty$-algebra structure on the hom-graded modules between types of coalgebras and types of algebras which are linked by Koszul duality. Here, it will produce a canonical convolution (curved absolute) \textit{partition} $\mathcal{L}_\infty$-algebra structures on the hom-graded modules between types of coalgebras and types of algebras, in the sense of \cite{pdalgebras}, see also \cite{deuxiemepapier}. Exploring these convolution structure will be the subject of future work; we refer to \cite{mappingcoalgebrascharp}.

\medskip

By standard axiomatic arguments in the same spirit as in \cite{BergerMoerdijk}, the dg operad $\Omega \C$ is both admissible and coadmissible, meaning both dg $\Omega \C$-algebras and dg $\Omega \C$-coalgebras admit a transferred structures from dg modules. Then we consider the bar-cobar adjunction relative to the universal twisting morphism $\iota: \C \longrightarrow \Omega \C$

\[
\begin{tikzcd}[column sep=5pc,row sep=3pc]
    \mathsf{dg}~\C\text{-}\mathsf{cog} \arrow[r, shift left=1.1ex, "\Omega_{\iota}"{name=F}] & \mathsf{dg}~ \Omega\C \text{-}\mathsf{alg}~. \arrow[l, shift left=.75ex, "\mathrm{B}_{\iota}"{name=U}]
            \arrow[phantom, from=F, to=U, , "\dashv" rotate=-90]
\end{tikzcd}
\]

\begin{theoremintro}[Theorems \ref{thm: existence de structure de modèles} and \ref{thm: the bar-cobar is a Quillen equivalence}]
The exists a combinatorial model category structure on the category of dg $\C$-coalgebras given by the following sets of maps:

\medskip

\begin{enumerate}
\item the set of weak-equivalences is given by morphisms $f$ such that $\Omega_\iota(f)$ is a quasi-isomorphism,

\medskip

\item the set of cofibrations is given by degree-wise monomorphisms,

\medskip 

\item the set of fibrations is given by morphisms with the right lifting property with respect to acyclic cofibrations.

\medskip
\end{enumerate}

Furthermore, considering this model structure promotes the bar-cobar adjunction into Quillen equivalence. 
\end{theoremintro}

This means that the homotopy theory of dg $\Omega \C$-algebras can be read using their Koszul dual dg $\C$-coalgebras. Here again, a great advantage is that fibrant-cofibrant dg $\C$-coalgebras are exactly quasi-free dg $\C$-coalgebras which are also images of dg $\Omega \C$-algebras. Notice that, in principle, the notion of a coalgebra over a cooperad in positive characteristic encodes types of \textit{divided power} conilpotent coalgebras. Nevertheless, since $\C$ is quasi-planar, there are no divided power operations that appear at the level of algebraic structures.

\medskip

We define $\infty$-morphisms and $\infty$-quasi-isomorphisms of dg $\Omega \C$-algebras: we show that they are "invertible", therefore two dg $\Omega \C$-algebras are linked by a zig-zag of quasi-isomorphisms if and only if there exists an $\infty$-quasi-isomorphism between them. We also prove a version of the homotopy transfer theorem: given a dg $\Omega \C$-algebra $A$, we construct an $\infty$-quasi-isomorphic dg $\Omega \C$-algebra structure on the homology of $A$. This means that $A$ can be replaced by its homology without losing any homotopical data. These methods should lead to applications in the study of formality questions over positive characteristic field, in the same spirit as in \cite{GeoffroyGabriel}. 

\medskip

Then we turn to dg $\Omega\C$-coalgebras, and generalize the results of \cite{linearcoalgebras} in this case. Consider the complete bar-cobar adjunction relative to $\iota$

\[
\begin{tikzcd}[column sep=5pc,row sep=3pc]
          \mathsf{dg}~\Omega \C\text{-}\mathsf{cog} \arrow[r, shift left=1.1ex, "\widehat{\Omega}_{\iota}"{name=F}] & \mathsf{dg}~\C\text{-}\mathsf{alg}^{\mathsf{qp}\text{-}\mathsf{comp}}~, \arrow[l, shift left=.75ex, "\widehat{\mathrm{B}}_{\iota}"{name=U}]
            \arrow[phantom, from=F, to=U, , "\dashv" rotate=-90]
\end{tikzcd}
\]
between the category of dg $\Omega \C$-coalgebras and the category of \textit{qp-complete} dg $\C$-algebras. Here qp-complete algebra refer to algebras which are complete for the \textit{canonical} quasi-planar ladder of the quasi-planar cooperad $\C$. 

\begin{theoremintro}[Theorems \ref{thm: existence model structure for absolute} et \ref{thm: complete bar-cobar is a Quillen equivalence}]
The exists a combinatorial model category structure on the category of qp-complete dg $\C$-algebras given by the following sets of maps:

\medskip

\begin{enumerate}
\item the set of weak-equivalences is given by morphisms $f$ such that $\widehat{\mathrm{B}}_{\iota}(f)$ is a quasi-isomorphism,

\medskip

\item the set of fibrations is given by degree-wise epimorphisms,

\medskip 

\item the set of cofibrations is given by morphisms with the left lifting property with respect to acyclic fibrations.

\medskip
\end{enumerate}

Furthermore, considering this model structure promotes the bar-cobar adjunction into Quillen equivalence. 
\end{theoremintro}

Again, this means that one can read the homotopy theory of dg $\Omega \C$-coalgebras by using dg $\C$-algebras. Here again, fibrant-cofibrant dg $\C$-algebras are much simpler as they are exactly given by quasi-free dg $\C$-algebras which are essentially the images of the dg $\Omega \C$-coalgebras under the complete cobar functor. The notion of a dg $\C$-algebra corresponds to the absolute version of divided powers dg $\C^*$-algebras, and is also naturally endowed with divided power operations. 

\medskip

We define $\infty$-morphisms and $\infty$-quasi-isomorphisms of dg $\Omega \C$-coalgebras too: we show that they are "invertible", therefore two dg $\Omega \C$-coalgebras are linked by a zig-zag of quasi-isomorphisms if and only if there exists an $\infty$-quasi-isomorphism between them. We also prove a similar version of the homotopy transfer theorem: given a dg $\Omega \C$-algebra $V$, we construct a $\infty$-quasi-isomorphic dg $\Omega \C$-coalgebra structure on the homology of $V$, which allows one to pass from a coalgebra $V$ to its homology without loosing homotopical data. 

\medskip

Note that once the theory is established for operads of the form $\Omega \C$, with $\C$ a quasi-planar dg cooperad, then we can easily extend all the above results to any cofibrant dg operad $\operad P$. Indeed, given a dg operad $\operad P$, we construct what we call the \textit{quasi-planar} bar-cobar adjunctions 

\[
\begin{tikzcd}[column sep=5pc,row sep=3pc]
          \mathsf{dg}~\operad P\text{-}\mathsf{alg} \arrow[r, shift left=1.1ex, "\mathrm{B}^{\mathrm{q.p}}_{\pi}"{name=A}]
         &\mathsf{curv}~\mathrm{B}(\operad E \otimes \operad P) \text{-}\mathsf{cog}~, \arrow[l, shift left=.75ex, "\Omega^{\mathrm{q.p}}_{\pi}"{name=B}] \arrow[phantom, from=A, to=B, , "\dashv" rotate=90]
\end{tikzcd}
\begin{tikzcd}[column sep=5pc,row sep=3pc]
          \mathsf{dg}~\operad P\text{-}\mathsf{cog} \arrow[r, shift left=1.1ex, "\widehat{\Omega}^{\mathrm{q.p}}_{\pi}"{name=A}]
         &\mathsf{curv}~\mathrm{B}(\operad E \otimes \operad P) \text{-}\mathsf{alg}^{\mathsf{qp}\text{-}\mathsf{comp}}~, \arrow[l, shift left=.75ex, "\widehat{\mathrm{B}}^{\mathrm{q.p}}_{\pi}"{name=B}] \arrow[phantom, from=A, to=B, , "\dashv" rotate=-90]
\end{tikzcd}
\]

and in the case when $\operad P$ is cofibrant, all the above results can be translated \textit{mutatis mutandis} to this setting. Furthermore, let us mention that if one replicates the constructions of \cite{Joan1} using the quasi-planar bar-cobar adjunction, one recovers the cohomology theories constructed in \cite{eric1,eric2}. 

\medskip

Finally, we extend the \textit{duality square} of \cite{absolutealgebras} to this new setting. Recall that it intertwines the bar-cobar adjunction with the complete bar-cobar adjunction via a pair of duality adjunctions. Using the model category structures constructed so far, we prove that all the functors in this duality square are Quillen adjunctions. As in the characteristic zero case, this allows us to show that the homotopy category of dg $\operad P$-algebras with finite dimensional homology is equivalent to the homotopy category of dg $\operad P$-coalgebras with finite dimensional homology, where $\operad P$ is a cofibrant dg operad.

\medskip

These results also open the door to a classical theory of Koszul duality in the positive characteristic setting, both at the operadic level and at the algebra/coalgebra level. Since quasi-planar dg cooperads seem to be the right notion, it would be interesting to understand, given an operad $\mathcal{P}$, what is (and when does it exist) the \textit{minimal} quasi-planar dg cooperad $\C$ such that there exists a quasi-isomorphism $\Omega\C \qi \mathcal{P}$. In this case, one could think of $\C$ as the Koszul dual of $\mathcal{P}$, in the original sense of \cite{GinzburgKapranov,GetzlerJones94}. Very similar ideas were already explained in \cite[Section 2.7]{BrunoMalte}, where the authors also suggest a similar approach using their theory of \textit{higher cooperads}. Finally, at the algebra level, one should be able to ask similar questions using the bar-cobar constructions of this paper, knowing that they have a reasonable homotopical behaviour. These questions are all beyond the reach of this paper. Nevertheless, let us mention that comparing the approach carried out in this paper with \cite{BrunoMalte}, and specially comparing quasi-planar cooperads with higher cooperads, shall be the subject of future work in \cite{mappingcoalgebrascharp}.

\medskip

\subsection*{Acknowledgements}
We would like to thank Guille Carrión, Aras Ergus, Eric Hoffbeck, Geoffroy Horel, Joost Nuiten and Bruno Vallette for interesting discussions. We would also like to thank Bruno Vallette for useful comments on the draft version. 

\subsection*{Notations and conventions}\label{subsection: notations and conventions}

\begin{enumerate}
    \item \textit{Universe.} We fix a universe $\mathscr{U}$. A small set is an element of this universe. A large set is a subset of it.
    
\medskip

    \item \textit{Ordinal products.} Let $\omega$ be the first infinite ordinal which is the poset of natural numbers. Let $\omega \cdot \omega$ be the ordinal whose underlying poset is the two times product of $\omega$ equipped with the lexicographic order. In general, if $\alpha$ and $\beta$ are two small ordinals, we will denote $\alpha \cdot \beta$ the small ordinal whose underlying poset is the product poset equipped with the lexicographic order.
    
    \medskip
    
    Given a small ordinal $\alpha$ (and more generally a poset), one can view it as a category. Objects are indexed by the ordinal $\alpha$, and there is only one non-trivial arrow between two objects $i$ and $j$ in $\alpha$, if and only if $i < j$. 
    
    \medskip
    
    \item \textit{Ladders.} Let $\mathsf{C}$ be a cocomplete category and let $\alpha$ be a small ordinal. The data of an $\alpha$-ladder amounts to the data of a functor cocontinuous functor
    
    \[
    c: 1+ \alpha \longrightarrow \mathsf{C}~.
    \]
    
    This corresponds to a "ladder" diagram of objects in $\mathsf{C}$
    
    \[
    0 \longrightarrow c(0) \longrightarrow c(1) \longrightarrow \cdots \longrightarrow c(i) \longrightarrow \cdots
    \]
    
    indexed by $\alpha$, such that for every limit ordinal $k \in \alpha$, we have 
    \[
   	c(k) \cong \colim{i < k}~ c(i)~. 
    \]
    The colimit of the ladder diagram will be denoted by 
    \[
    c(\alpha) \coloneqq \colim{i \in \alpha}~ c(i)~.
    \]
   	Usually and depending on the category $\mathsf{C}$, additional requirements will be made on the transition maps $c(i) \longrightarrow c(i+1)$. For instance, they will be required to be monomorphisms or some kind of cofibrations.
   	
   	\medskip
   	
   	We considering ladder diagrams in categories of operads and cooperads, the notation $c(i)$ for the image of $i \in \alpha$ will be replaces by $c^{(i)}$ in order to avoid confusions with the arity. 
    
    \medskip
    
    \item \textit{Coladders.} Let $\mathsf{D}$ be a complete category and let $\alpha$ be a small ordinal. The data of an $\alpha$-coladder amounts to the data of a functor continuous functor
    
    \[
    d: (1+\alpha)^{\op} \longrightarrow \mathsf{D}~.
    \]
    
    This corresponds to a "coladder" diagram of objects in $\mathsf{D}$
    
    \[
    0 \longleftarrow d(0) \longleftarrow d(1) \longleftarrow \cdots \longleftarrow d(i) \longleftarrow \cdots
    \]
    indexed by $\alpha$, such that for every limit ordinal $k \in \alpha$, we have 
    \[
   	d(k) \cong \lim_{i < k} d(i)~. 
    \]
    The limit of the coladder diagram will be denoted by 
    \[
    d(\alpha) \coloneqq \lim_{i \in \alpha} d(i)~.
    \]
   	Usually and depending on the category $\mathsf{D}$, additional requirements will be made on the transition maps $d(i+1) \longrightarrow d(i)$. For instance, they will be required to be epimorphisms or some kind of fibrations.
   	
   	\medskip
    
    \item \textit{Permutations and shuffles.} The symmetric group on $n$ elements, given by the set of bijections between $\{1,\cdots, n\}$ and itself, will be denoted by $\mathbb S_n$. Elements in $\mathbb S_n$ are called permutations. A permutation $\sigma$ in $\mathbb S_n$ is determined by its values $\sigma(1),\cdots,\sigma(n)$.  
    
    \medskip
    
    Let $\sigma$ be a permutation in $\mathbb S_{a+b}$. It is a $(a,b)$-shuffle if it satisfies that
   	\[
   	\sigma(1) < \cdots < \sigma(a) \quad \text{and} \quad \sigma(a+1) < \cdots < \sigma(a+ b)~. 
   	\]
    The set of all $(a,b)$-shuffles in $\mathbb S_{a+b}$ will be denoted by $\mathrm{Sh}(a,b)$. 
    
\end{enumerate}


\newpage

\section{Differential graded modules, $\mathbb{S}$-modules, operads and cooperads}

\vspace{2pc}

In this section, we describe categorical properties of the ground categories of graded, pre-differential and differential modules over a field $\kk$ of any characteristic. We recall some results about the homotopy theory of dg $\kk[G]$-modules, where $G$ is a finite group. These results will be useful when considering dg $\mathbb S$-modules. Finally, we recall the definitions of planar and symmetric (co)operads. 

\subsection{Graded modules, differential graded modules and pre-differential graded modules}

\begin{definition}[Graded modules]
    Let  $\catgrmod{\kk}$ be the category of graded $\kk$-modules. Objects are given by families $X = \{X_n\}_{n \in \mathbb Z}$ of $\kk$-modules indexed by the set of integers $\mathbb Z$, and morphisms by collections of linear maps indexed by $\mathbb Z$ which respect the grading.
\end{definition}

    It forms a closed symmetric monoidal category endowed with the tensor product $X \otimes Y$ given by 
    $$
    (X\otimes Y)_n \coloneqq \bigoplus_{i+j=n} X_i \otimes Y_j~,
    $$
    where the internal hom $[X,Y]$ is given by 
    $$
    [X, Y]_n \coloneqq \prod_k [X_k , Y_{n+k}],
    $$
    for every two graded $\kk$-modules $X,Y$ and every integer $n$ in $\mathbb Z$. An element $x$ of $X_n$ will be called an homogeneous element of degree $n$ of $X$. We denote the degree of a homogeneous element by $|x|$. A degree $n$ map from $X$ to $Y$ is just a homogeneous element of degree $n$ of $[X,Y]$.

    \medskip

    Moreover, let $f \in [X,X']_n$ and $g \in [Y, Y']_m$. We denote $f \otimes g$ the element of $[X \otimes Y ,X' \otimes Y']_{n+m}$ defined as
    $$
    (f \otimes g)(x \otimes y) \coloneqq (-1)^{|g||x|} f(x) \otimes g(y)
    $$
    and $[f, g]$ the element of $[[X', Y], [X, Y']]_{n+m}$ defined as
    $$
    [f , g](h) \coloneqq (-1)^{|f||h|} g \circ h \circ f.
    $$

The category of graded $\kk$-modules enjoys several categorical properties.

\medskip
\begin{enumerate}
    \item The (monadic) forgetful functor towards graded sets commutes with sifted colimits (it commutes with filtered colimits and reflexive coequalisers).

\medskip

    \item The tensor product commutes with colimits and with finite limits.

\medskip

    \item Subsequently the functor $X \mapsto X^{\otimes n}$ commutes with sifted colimits and with \textit{finite} cosifted limits. 

\medskip
    \item Filtered colimits commute with finite limits (since both are computed in graded sets).

\medskip
    \item Coproducts commute with finite limits.

\medskip
    \item Products commute with finite colimits.
\end{enumerate}

\begin{example}
    However, cofiltered limits do not commute in general with finite colimits. Indeed, in the context where $\kk = \mathbb R$, let $X$ be the sub $\mathbb R$-module of $\mathbb R^{\mathbb N}$ spanned by sequence $(x_0, x_1, \ldots)$ so that $\sum_n |x_n| < +\infty$ and let $s: X \longrightarrow X$ be the shift endomorphism
    $$
    s(x_0, x_1, x_2,\ldots) = (0, x_0, x_1, x_2,\ldots).
    $$
    Moreover, let $Y\subseteq X$ the sub $\mathbb R$-module spanned by sequence whose sum is zero. It is stable through $s$. Moreover the following sequence is exact
    $$
    0 \longrightarrow Y \hookrightarrow X \xrightarrow{\Sigma} \mathbb R \longrightarrow 0
    $$
    where $\Sigma$ denotes the sum.
    The following diagram is commutative 
    $$
    \begin{tikzcd}
        \cdots 
        \ar[r, "s"]
        &Y
        \ar[r,"s"] \ar[d, hook]
        & \cdots
        \ar[r, "s"]
        & Y
        \ar[r, "s"] \ar[d, hook]
        & Y \ar[d, hook]
        \\
        \cdots \ar[r, "s"']
        &X \ar[r,"s"'] \ar[d, "\Sigma"]
        & \cdots \ar[r, "s"']
        & X \ar[r,"s"'] \ar[d, "\Sigma"]
        & X \ar[d, "\Sigma"]
        \\
        \cdots \ar[r, equal]
        &\mathbb R \ar[r,equal]
        & \cdots \ar[r,equal]
        & \mathbb R \ar[r,equal]
        & \mathbb R  .    
    \end{tikzcd}
    $$
    The limit of the two first row are $0$ while the limit of the last row is $\mathbb R$. Thus, it cannot be the quotient of the second limit by the first one.
\end{example}

\begin{definition}[Pre-differential graded modules]
    We denote $\catpdgmod{\kk}$ the category of pre-differential graded (pdg) $\kk$-modules. Objects are given by graded $\kk$-modules $\{X_n\}_{n \in \mathbb Z}$ equipped with a degree $-1$ endomorphism $d: X_{(-)} \longrightarrow X_{(-)- 1}$, and morphisms by morphisms of graded modules which commute with the pre-differentials.
\end{definition}

The closed symmetric monoidal category structure on graded $\kk$-modules induces a closed symmetric monoidal category structure on pdg $\kk$-modules. For any two pdg $\kk$-modules $X,Y$, the graded tensor product $X \otimes Y$ is equipped with the pre-differential
    $$
    d_{X \otimes Y} \coloneqq d_x \otimes \id_Y  + d_X \otimes \id_Y~,
    $$
    and the internal hom $[X, Y]$ with the pre-differential
    $$
    d_{[X, Y]} \coloneqq [\id, d_Y] - [d_x, \id].
    $$

\begin{definition}[Differential graded modules]
    We denote $\catdgmod{\kk}$ the category of differential graded (dg) $\kk$-modules. It is the full subcategory of pdg $\kk$-modules $(X, d)$  that $d^2=d \circ d=0$. 
\end{definition}

\begin{remark}
A differential graded module is also referred as a chain complex.
\end{remark}

It inherits from pdg $\kk$-modules the structure of a closed symmetric monoidal category. One can check that if $X$ and $Y$ are dg modules, their tensor product $X \otimes Y$ and their internal hom $[X, Y]$ is again a dg module.

\begin{definition}[Spheres and disks]
    Let $n \in \mathbb Z$ be an integer. The $n$-th disk, denoted by $D^n$, is the dg module given by
    $$
    D^n_m
    \coloneqq
    \begin{cases}
        \kk \text{ if } m = n-1, n~,
        \\
        0  \text{ otherwise}~,
    \end{cases}
    $$
    where the differential $D^n_n \longrightarrow D^n_{n-1}$ is the identity of map of $\kk$.

    \medskip
    
    The $n$-th sphere, denoted by $S^n$, is the dg module given by
    $$
    S^n_m
    \coloneqq
    \begin{cases}
        \kk \text{ if } m = n~,
        \\
        0  \text{ otherwise}.
    \end{cases}
    $$
\end{definition}

\begin{definition}[Suspensions]
    Let $k \in \mathbb Z$ be an integer and let $X$ is a graded (resp. pdg, dg) $\kk$-module. The tensor product $S^k \otimes X$ is referred as the $k$-th suspension of $X$ and will be denoted by $s^k X$.
\end{definition}

The category of dg $\kk$-modules can be endowed with a combinatorial model category structure, determined by the following sets of maps:

\medskip

    \begin{enumerate}
        \item the set of weak equivalences is given by quasi-isomorphisms;

        \medskip
        
        \item the set of fibrations is given by degree-wise epimorphisms;

        \medskip
        
        \item the set of cofibrations is given by degree-wise injections;

        \medskip
        
        \item the set of generating cofibrations is given by the canonical inclusion maps $S^n \longrightarrow D^{n+1}$ for all $n$ in $\mathbb Z$;

        \medskip
        
        \item the set of generating acyclic cofibrations is given by the inclusion maps $0 \longrightarrow D^{n+1}$ for all $n$ in $\mathbb Z$.
    \end{enumerate}

\begin{remark}
Since $\kk$ is a field, this model category structure is both the \textit{injective} model structure and the \textit{projective} model structure.
\end{remark}
    
This model category structure can be left-transferred along the forgetful-truncation adjunction 
\[
\begin{tikzcd}[column sep=5pc,row sep=3pc]
          \catdgmod{\kk}_{\geq 0} \arrow[r, shift left=1.1ex, "\mathrm{U}"{name=F}] & \catdgmod{\kk}, \arrow[l, shift left=.75ex, "\tau_{\geq 0}"{name=U}]
            \arrow[phantom, from=F, to=U, , "\dashv" rotate=-90]
\end{tikzcd}
\]

to dg $\kk$-modules in non-negative degrees. It can also be right-transferred via the truncation-forgetful adjunction 
\[
\begin{tikzcd}[column sep=5pc,row sep=3pc]
          \catdgmod{\kk} \arrow[r, shift left=1.1ex, "\tau_{\leq 0}"{name=F}] & \catdgmod{\kk}_{\leq 0}, \arrow[l, shift left=.75ex, "\mathrm{U}"{name=U}]
            \arrow[phantom, from=F, to=U, , "\dashv" rotate=-90]
\end{tikzcd}
\]

to dg $\kk$-modules in non-positive degrees. 

\subsection{Finite group action on dg modules}
\label{sectionfinite group}
Let $G$ be a finite group. The data of an action of $G$ on a dg $\kk$-module is equivalent to the data of a dg $\kk[G]$-module, where $\kk[G]$ denotes the group algebra of $G$. 

\medskip

The functor
\[
\begin{tikzcd}[column sep=4.5pc,row sep=0.5pc]
\mathsf{dg}~\kk\text{-}\mathsf{mod} \arrow[r]
&\mathsf{dg}~\kk[G]\text{-}\mathsf{mod} \\
X \arrow[r,mapsto]
& \kk[G] \otimes X
\end{tikzcd}
\]
is both left and right adjoint to the forgetful functor $U_G$, and fits in the following adjunction diagram
\[
\begin{tikzcd}[column sep=6pc,row sep=2pc]
\mathsf{dg}~\kk[G]\text{-}\mathsf{mod}
\arrow[r,"U_G"{name=B, below}]
&\mathsf{dg}~\kk\text{-}\mathsf{mod}~.
\arrow[l,bend right=30,"{\kk[G]~ \otimes-}",swap,""{name=A,above}] \arrow[l,bend left=30,"{\kk[G]~ \otimes-}",""{name=C,above}] \arrow[phantom, from=B, to=A, "\dashv" rotate=-90] \arrow[phantom, from=C, to=B, "\dashv" rotate=-90]
\end{tikzcd}
\]

There is a bialgebra structure on $\kk[G]$, where the coproduct $\Delta$ is induced by the set-theoretical diagonal of $G$. It allows us to endow the category of dg $\kk[G]$-modules with with a monoidal category structure. The tensor product of two dg $\kk[G]$-modules $X, Y$ is given by the underlying dg module $X \otimes Y$ endowed with the following dg $\kk[G]$-module structure:
\[
\begin{tikzcd}[column sep=2.4pc,row sep=0.5pc]
\kk[G] \otimes (X \otimes Y) \arrow[r,"\Delta ~\otimes~ \mathrm{id}"]
&\kk[G \times G] \otimes (X \otimes Y) \arrow[r,"\cong"]
&(\kk[G] \otimes X) \otimes (\kk[G] \otimes Y) \arrow[r,"\gamma_X~\otimes~\gamma_Y"]
&X \otimes Y~,
\end{tikzcd}
\]

where $\gamma_X, \gamma_Y$ denote, respectively, the dg $\kk[G]$-module structures of $X$ and $Y$.

\medskip

One can either right-transfer or left-transfer the model category structure on dg modules to dg $\kk[G]$-modules. This gives \textit{two different combinatorial model structures}. 
\medskip
\begin{enumerate}
    \item The right-transferred structure is called the \textit{projective model structure}. Fibrations are given by degree-wise epimorphisms and weak-equivalences by quasi-isomorphisms.
\medskip
    \item The left-transferred structure is called the \textit{injective model structure}.  Cofibrations are given by degree-wise monomorphisms and weak-equivalences by quasi-isomorphisms.
\end{enumerate}

\begin{definition}[Projective and injective modules]
Let $X$ be a dg $\kk[G]$-module. 

\medskip

\begin{enumerate}
    \item $X$ is said to be \textit{projective} if it is a fibrant-cofibrant object in the projective model structure. 

\medskip

    \item $X$ is said to be \textit{injective} if it is a fibrant-cofibrant object in the injective model structure.
\end{enumerate}
\end{definition}

\begin{lemma}
Let $X$ be dg $\kk[G]$-module. 

\medskip

\begin{enumerate}
\item If $X$ is injective, then its linear dual $X^*$ is projective.

\medskip

\item If $X$ is projective, then its linear dual $X^*$ is injective.
\end{enumerate}
\end{lemma}

\begin{proof}
    Linear duality defines a Quillen adjunction 

    \[
\begin{tikzcd}[column sep=5pc,row sep=3pc]
          \mathsf{dg}~\kk[G]\text{-}\mathsf{mod}  \arrow[r, shift left=1.1ex, "(-)^*"{name=F}] & \mathsf{dg}~\kk[G]\text{-}\mathsf{mod}^{\mathsf{op}} , \arrow[l, shift left=.75ex, "(-)^*"{name=U}]
            \arrow[phantom, from=F, to=U, , "\dashv" rotate=-90]
\end{tikzcd}
\]

when the category on the left is endowed with the projective model structure and the category on the right with the opposite of the injective model structure.
\end{proof}

The two adjunctions $\kk[G] \otimes - \dashv U_G \dashv \kk[G] \otimes -$ restrict to dg $\kk$-modules and dg $\kk[G]$-modules which are either in non-negative degrees or in non-positive degrees.

\medskip

\begin{enumerate}
    \item The category of dg $\kk[G]$-modules in non-negative degrees can be endowed with the \textit{projective model structure.} Fibrations are given by degree-wise epimorphisms and weak-equivalences by quasi-isomorphisms. Moreover, cofibrations are degree-wise injections whose cokernel is degre-ewise projective.

    \medskip
    
    \item The category of dg $\kk[G]$-modules in non-positive degrees can be endowed with the \textit{injective model structure.} Cofibrations are given by degree-wise injections and weak-equivalences by quasi-isomorphisms. Moreover, fibrations are degree-wise epimorphisms whose kernel is degree-wise injective.
\end{enumerate}

\medskip

\begin{definition}[Quasi-free modules]
    A dg $\kk[G]$-module is \textit{quasi-free} if its underlying graded $\kk[G]$-module is free, that is, in the essential image of the functor $\kk[G] \otimes -$ from graded $\kk$-module to graded $\kk[G]$-modules.
\end{definition}

A quasi-free dg $\kk[G]$-module $X$ is degree-wise projective and degree-wise injective. Nevertheless, it might not be cofibrant in the projective model structure nor fibrant in the injective model structure. 

\medskip

\begin{enumerate}
    \item If $X$ is \textit{bounded below}, it is cofibrant in the projective model structure since, up to some a finite degree translation, it is the image of a cofibrant object of $\catdgmod{\kk[G]}_{\geq 0}$.

\medskip

    \item If $X$ is \textit{bounded above}, it is fibrant in the injective model structure since, up to some a finite degree translation, it is the image of a fibrant object of $\catdgmod{\kk[G]}_{\leq 0}$.
\end{enumerate}

\begin{remark}
To the best of our knowledge, these results are particular instances of results in \cite{DgalgebraHalperin}. As we were unable to find it, we refer to \cite{Abbasirad2014HomotopyTO} for more details.
\end{remark}

\begin{proposition}\label{prop: tensor with a quasi-free}
Quasi-free dg $\kk[G]$-modules are stable through tensor product in the following sense: $X \otimes Y$ is quasi-free whenever $X$ or $Y$ is quasi-free. 
\end{proposition}

\begin{proof}
Let $X$ and $Y$ be dg $\kk[G]$-modules, and assume that
$Y \cong \kk[G] \otimes Z$ as graded $\kk[G]$-modules. The canonical morphism of dg $\kk[G]$-modules
$$
\kk[G] \otimes (U_G(X) \otimes Z) \longrightarrow X \otimes Y
$$
is an isomorphism, with inverse
$$
x \otimes g \otimes z \mapsto g\otimes ( g^{-1}(x) \otimes z).
$$
\end{proof}

\begin{remark}
The above proposition holds if one replaces $\kk[G]$ by any Hopf algebra $H$.
\end{remark}

\begin{proposition}\label{prop: iso avec la norme}
Let $X$ be a quasi-free dg $\kk[G]$-module. The canonical norm map 
\[
\mathbb{N}_X: X_G \longrightarrow X^G
\]
from the coinvariants to the invariants is an isomorphism of dg modules.
\end{proposition}

\begin{proof}
Follows from the fact that $X$ is quasi-free and that limits and colimits are computed degree-wise.
\end{proof}

Finally, the projective and injective model category structures are compatible in the following sense.

\begin{proposition}\label{lemmatensored}
The category of dg $\kk[G]$-modules together with the projective model structure is homotopically enriched-tensored-cotensored over the category dg $\kk[G]$-module together with the injective model structure. For every injective cofibration (i.e: a degree-wise injection) $f:A \rightarrowtail B$ and every projective cofibration $g: X \rightarrowtail Y$, the morphism

\[
f \diamond g :(A \otimes Y) \coprod_{A~ \otimes ~X} (B \otimes X) \rightarrowtail B \otimes Y
\]

is a projective cofibration. Furthermore, it is acyclic whenever $f$ or $g$ is.
\end{proposition}

\begin{proof}
Let us suppose that $g$ is a generating projective cofibration, given by the inclusion $S^{k} \otimes \kk[G] \longrightarrow D^{k+1} \otimes  \kk[G]$. In that case, one has a canonical isomorphism of diagrams between the commutative square
    $$
    \begin{tikzcd}[column sep=2.5pc,row sep=2.5pc]
        A \otimes (S^{k} \otimes \kk[G])
        \ar[r] \ar[d]
        & A \otimes (D^{k+1} \otimes \kk[G])
        \ar[d]
        \\
        B \otimes (S^{k} \otimes \kk[G])
        \ar[r]
        & B \otimes (D^{k+1} \otimes \kk[G])
    \end{tikzcd}
    $$
and the commutative square
    $$
    \begin{tikzcd}[column sep=2.5pc,row sep=2.5pc]
        (U_{G}(A) \otimes S^{k}) \otimes \kk[G]
        \ar[r] \ar[d]
        & (U_{G}(A) \otimes D^{k+1}) \otimes \kk[G]
        \ar[d]
        \\
        (U_{G}(B) \otimes S^{k}) \otimes \kk[G]
        \ar[r]
        & (U_{G}(B) \otimes D^{k+1}) \otimes \kk[G]~.
    \end{tikzcd}
    $$
Thus, the map $f \diamond g$ is isomorphic to the image through the functor $- \otimes \kk[G]$ of the map
\[
(U_{G}(A) \otimes D^{k+1}) \coprod_{(U_{G}(A)~ \otimes ~ S^{k})} ((U_{G}(B) \otimes S^{k})) \to (U_{G}(B) \otimes D^{k+1})~,
\]
which is a cofibration of dg $\kk$-modules. Hence $f \diamond g$ is a projective cofibration. Now the set of morphisms $g$ so that $f \diamond g$ is a projective cofibration is stable through pushouts, transfinite composition and retracts. Hence it contains all the projective cofibrations.

\medskip

Finally, if $f$ is acyclic or $g$ is acyclic, then $f \diamond g$ is acyclic too, as it is an acylic cofibration of dg $\kk$-modules and since the model category on dg $\kk$-modules is a monoidal model category.
\end{proof}

\subsection{$\mathbb N$-modules, planar operads and planar cooperads}
We consider $\mathbb N$ as a category where objects are natural integers and where there are only identity morphisms. 

\begin{definition}[dg $\mathbb N$-modules]
    A dg $\mathbb N$-modules amounts to the data of a functor 
    \[
    X: \mathbb N \longrightarrow \catdgmod{\kk}.
    \]    
    The object $X(n)$ is called the arity $n$ part of $X$. We denote $\catdgmod{\mathbb N}$
    the category of dg $\mathbb N$-modules. 
\end{definition}

\begin{remark}
The category of dg $\mathbb N$-modules admits a canonical combinatorial model category structure, determined by the following sets of maps:

\medskip

\begin{enumerate}
        \item the set of weak equivalences $f:X \qi Y$ is given by arity-wise quasi-isomorphisms $f(n): X(n) \qi Y(n)$ for all $n \geq 0$;

        \medskip
        
        \item the set of fibrations $f:X \twoheadrightarrow Y$ is given by  arity-wise degree-wise epimorphisms $f(n): X(n) \twoheadrightarrow Y(n)$ for all $n \geq 0$;

        \medskip
        
        \item the set of cofibrations $f:X \rightarrowtail Y$ is given by  arity-wise degree-wise injections $f(n): X(n) \rightarrowtail Y(n)$ for all $n \geq 0$.
    \end{enumerate}
\end{remark}

The planar horizontal product on dg $\mathbb N$-modules $X,Y$ is given by the Day convolution product

\[
(X \circledast_\pl Y)(n) \coloneqq \bigoplus_{k+l = n} X(k) \otimes Y(l)~.
\]
\vspace{0.1pc}

This endows dg $\mathbb N$ with a symmetric monoidal category structure, where the unit is given by $\kk$ concentrated in arity $0$.

\medskip

There is another monoidal structure given by the \textit{planar composition product} 

\[
(X \comp_\pl Y)(n) \coloneqq \bigoplus_{k \geq 0} X(k) \otimes Y^{\circledast_\pl k} (n)~.
\]
\vspace{0.1pc}

The unit for the composition is given by $\operad I$, defined as follows 
$$
\II (n) \coloneqq
\begin{cases}
    0 \text{ if }n \neq 1~,
    \\
    \kk \text{ if }n = 1.
\end{cases}
$$

\begin{definition}[Planar dg operad]
A \textit{planar dg operad} $\operad {P}$ amounts to the data of a monoid $(\mathcal{P},\gamma,\eta)$ in the category of dg $\mathbb N$-modules with respect to the composition product. 
\end{definition}

\begin{definition}[Augmented planar dg operad]
An \textit{augmented planar dg operad} $\operad P$ amounts to the data of a planar dg operad $(\mathcal{P},\gamma,\eta)$ equipped with a morphism of planar dg operads $\nu: \operad P \longrightarrow \operad I$ such that $\nu \circ \eta = \mathrm{id}.$
\end{definition}

Given an augmented planar dg operad $\operad P$, we will denote by $\overline{\operad P}$ the kernel of the augmentation map. 

\begin{definition}[Planar dg cooperad]
A \textit{planar dg cooperad} $\operad C$ amounts to the data of a comonoid $(\C, \Delta, \epsilon)$ in the category of dg $\mathbb N$-modules with respect to the composition product. 
\end{definition}

Given a planar dg cooperad $\C$, we will denote by $\overline{\operad C}$ the kernel of the counit map. 

\begin{definition}[Coaugmented planar dg cooperad]
A \textit{coaugmented planar dg cooperad} $\operad C$ amounts to the data of a planar dg cooperad $(\C, \Delta, \epsilon)$ equipped with a morphism of planar dg cooperads $\mu: \operad I \longrightarrow \operad C$ such that $\epsilon \circ \mu = \mathrm{id}$. 
\end{definition}

\begin{remark}
All the definitions of this subsection make sense in the graded or the pre-differential setting. 
\end{remark}

\subsection{$\mathbb S$-modules, operads and cooperads}
\label{sectioncomositionproduct}
In this subsection, we deal with dg (resp. graded or pdg) $\mathbb S$-modules.

\begin{definition}[dg $\mathbb S$-module]
    Let $\mathbb S$ be the groupoid whose objects are natural integers and whose morphisms are given by 
    $$
    \hom_{\mathbb S}(n,m) = 
    \begin{cases}
        \emptyset \text{ if }n \neq m~,
        \\
        \mathbb S_n \text{ if }n = m~.
    \end{cases}
    $$
    A \textit{dg} $\mathbb S$\textit{-module} $M$ amounts to the data of a functor 
   	\[
   	M: \mathbb S^\op \longrightarrow \catdgmod{\kk}
   	\]
    from $\mathbb S^\op$ to dg modules. It corresponds to collection of dg modules $\{M(n)\}$ for $n \geq 0$, where each $M(n)$ is endowed with a (right) action of $\mathbb{S}_n$. We denote by $\mathsf{dg}~\mathbb{S}\text{-}\mathsf{mod}$ the category of dg $\mathbb{S}$-modules. 
\end{definition}

\begin{remark}
We define analogously the categories of graded or pdg $\mathbb{S}$-modules.
\end{remark}

There is a diagram of adjunctions 
\[
\begin{tikzcd}[column sep=6pc,row sep=2pc]
\mathsf{dg}~\mathbb{S}\text{-}\mathsf{mod}
\arrow[r,"U"{name=B, below}]
&\mathsf{dg}~\mathbb{N}\text{-}\mathsf{mod}~,
\arrow[l,bend right=30,"-\otimes~\mathbb{S}",swap,""{name=A,above}] \arrow[l,bend left=30,"-\otimes~\mathbb{S}",""{name=C,above}] \arrow[phantom, from=B, to=A, "\dashv" rotate=-90] \arrow[phantom, from=C, to=B, "\dashv" rotate=-90]
\end{tikzcd}
\]

between the categories of dg $\mathbb{N}$-modules and of dg $\mathbb{S}$-modules. The functor $- \otimes \mathbb S$ is given by 
\[
(X \otimes \mathbb S)(n) \coloneqq X(n) \otimes \kk[\mathbb S_n]~,
\]
for all $n \geq 0$. We will also denote by $- \otimes \mathbb S$ the endofunctor of dg $\mathbb N$-modules that is given by the (co)free dg $\mathbb S$-module functor composed with the forgetful functor.

\medskip

One can either right-transfer or left-transfer the model category structure on dg $\mathbb{N}$-modules to dg $\mathbb{S}$-modules. This gives \textit{two different combinatorial model structures}. 
\medskip
\begin{enumerate}
    \item The right-transferred structure is called the \textit{projective model structure}. Fibrations are given by degree-wise arity-wise epimorphisms and weak-equivalences by arity-wise quasi-isomorphisms.
\medskip
    \item The left-transferred structure is called the \textit{injective model structure}. Cofibrations are given by degree-wise arity-wise monomorphisms and weak-equivalences by arity-wise quasi-isomorphisms.
\end{enumerate}

\begin{remark}
All the results of Subsection \ref{sectionfinite group} can be translated to dg $\mathbb{S}$-modules as the hold for any finite group $G$ and since the homotopy theory of dg $\mathbb{S}$-modules is determined arity-wise.
\end{remark}

The \textit{composition product} $\circ$ of two dg $\mathbb S$-modules $M,N$ is defined as follows

\[
M \comp N(n) \coloneqq \bigoplus_{k\geq 0} M(k) \otimes_{\mathbb{S}_k} \left( \bigoplus_{i_1 + \cdots + i_k = n} \mathrm{Ind}_{\mathbb{S}_{i_1} \times \cdots \times \mathbb{S}_{i_k}}^{\mathbb{S}_n} (N(i_1) \otimes \cdots \otimes N(i_k))\right)~.
\]
\vspace{0.1pc}

The unit for the composition is given by $\operad I$, defined as follows 
$$
\II (n) \coloneqq
\begin{cases}
    0 \text{ if }n \neq 1~,
    \\
    \kk \text{ if }n = 1.
\end{cases}
$$

They endow the category of dg $\mathbb S$-modules with a monoidal category structure.

\begin{definition}[dg operad]
A \textit{dg operad} $\operad {P}$ amounts to the data of a monoid $(\mathcal{P},\gamma,\eta)$ in the category of dg $\mathbb S$-modules with respect to the composition product. 
\end{definition}

\begin{definition}[augmented dg operad]
An \textit{augmented dg operad} $\operad P$ amounts to the data of a dg operad $(\mathcal{P},\gamma,\eta)$ equipped with a morphism of dg operads $\nu: \operad P \longrightarrow \operad I$ such that $\nu \circ \eta = \mathrm{id}.$
\end{definition}

Given an augmented planar dg operad $\operad P$, we will denote by $\overline{\operad P}$ the kernel of the augmentation map. 

\begin{definition}[$\mathbb{S}$-something dg operad]
Let $\operad P$ be a dg operad.

\begin{enumerate}
\item It is called $\mathbb S$-\textit{projective} if its underlying dg $\mathbb S$-module is cofibrant for the projective model structure.

\medskip

\item It is called $\mathbb S$-\textit{injective} if its underlying dg $\mathbb S$-module is fibrant for the injective model structure. 

\medskip

\item It is called $\mathbb S$-\textit{quasi-free} if its underlying dg $\mathbb S$-module is quasi-free.
\end{enumerate}
\end{definition}

\begin{remark}
An $\mathbb S$-projective dg operad is usually called an $\mathbb S$-cofibrant dg operad in the literature. We adopt this non-standard terminology in order to be able to differentiate between $\mathbb S$-projective and $\mathbb{S}$-injective dg operads, which both have cofibrant underlying dg $\mathbb{S}$-modules, but in different model category structures.
\end{remark}

\begin{definition}[dg cooperad]
A \textit{dg cooperad} $\operad C$ amounts to the data of a comonoid $(\C, \Delta, \epsilon)$ in the category of dg $\mathbb S$-modules with respect to the composition product. 
\end{definition}

Given a dg cooperad $\C$, we will denote by $\overline{\operad C}$ the kernel of the counit map. 

\begin{definition}[coaugmented dg cooperad]
A \textit{coaugmented dg cooperad} $\operad C$ amounts to the data of a dg cooperad $(\C, \Delta, \epsilon)$ equipped together with a morphism of planar dg cooperads $\mu: \operad I \longrightarrow \operad C$ such that $\epsilon \circ \mu = \mathrm{id}$. 
\end{definition}

There is a strong monoidal structure on the functor $-\otimes \mathbb S$ which yields two adjunctions 

\[
\begin{tikzcd}[column sep=5pc,row sep=3pc]
          \dgoperads_\pl  \arrow[r, shift left=1.1ex, "-\otimes \mathbb S"{name=F}] &\dgoperads , \arrow[l, shift left=.75ex, "\mathrm{U}_{\mathbb S}"{name=U}]
            \arrow[phantom, from=F, to=U, , "\dashv" rotate=-90]
\end{tikzcd}
\]
\[
\begin{tikzcd}[column sep=5pc,row sep=3pc]
          \dgcooperads_\pl  \arrow[r, shift left=1.1ex, "-\otimes \mathbb S"{name=F}] &\dgcooperads, \arrow[l, shift left=.75ex, "\mathrm{U}_{\mathbb S}"{name=U}]
            \arrow[phantom, from=F, to=U, , "\dashv" rotate=90]
\end{tikzcd}
\]

that lift, respectively, the adjunction $-\otimes \mathbb S \dashv U_{\mathbb S}$ and the adjunction $U_{\mathbb S} \dashv -\otimes \mathbb S$ that relate dg $\mathbb S$-modules to dg $\mathbb N$-modules.

\medskip

\begin{remark}
The adjunction $- \otimes \mathbb S \dashv U_{\mathbb S}$ relating dg operads to planar dg operads is monadic since its right adjoint preserves coreflexive equalisers and is conservative. However, the other adjunction $U_{\mathbb S} \dashv - \otimes \mathbb S$ is not a priori comonadic.
\end{remark}

\begin{remark}
All the definitions of this subsection make sense in the graded or the pre-differential setting. 
\end{remark}

\subsection{Tree modules and conilpotent cooperads}
In this subsection, we briefly recall how operads are algebras over the tree monad and how conilpotent cooperads are exactly coalgebras over the tree comonad, both in the planar and in the symmetric case. These constructions can all be found in \cite{LodayVallette}. For a more detailed discussion about the point of view adopted here, see the forthcoming note \cite{notesym}.

\medskip

\textbf{Planar tree endofunctor.} For every dg $\mathbb N$-module $X$, one can define the \textit{planar tree module} $\treemod_\pl(X)$ of $X$, which is the dg $\mathbb N$-module given, for $m \geq 0$, by 
\[
\treemod_\pl(X)(m)  = \bigoplus_t t(X)~,
\]
where the sum is taken over the isomorphism classes of planar trees with $m$ leaves.

\medskip

Let $n$ be in $\mathbb N$. We define the following sub-functors of the planar tree module:

\medskip

\begin{itemize}
	\item The \textit{reduced planar tree endofunctor} $\overline{\treemod}_\pl(X)$, given by the sum over all non-trivial planar trees;
	
\medskip

    \item The $n$\textit{-levelled planar tree endofunctor} $\treemod_{\pl, \leq n}(X)$, given by the sum over planar trees whose height is equal or lower than $n$ (recall that the height of the trivial tree with no node is $0$);
    
\medskip

    \item The $n$-\textit{weight planar tree endofunctor} $\treemod_{\pl}^{(\leq n)}(X)$, given by the sum over planar trees with $n$ nodes or less.
 
\end{itemize}

\medskip

All these constructions are natural in $X$ and define endofunctors of the category of dg $\mathbb N$-modules. 

\medskip

\textbf{Planar operads.} The planar tree module endofunctor $\treemod_\pl$ admits a monad structure induced by the grafting of planar trees.

\begin{proposition}
The category of algebras over the monad $\treemod_\pl$ is canonically isomorphic to the category of planar dg operads.
\end{proposition}

\textbf{Conilpotent planar cooperads.} The \textit{reduced} planar tree module endofunctor $\overline{\treemod}_\pl$ admits a comonad structure induced by partitioning planar trees. Furthermore, there is a fully faithful functor
\[
\mathrm{Conil}: \mathsf{dg}~\overline{\treemod}_\pl\text{-}\mathsf{cog} \longrightarrow (\dgcooperads_\pl)_{\operad I/}
\]
from dg $\overline{\treemod}_\pl$-coalgebras to coaugmented planar dg cooperads.

\begin{definition}[Conilpotent planar dg cooperad]
Let $\C$ be a coaugmented planar dg cooperad. It is \textit{conilpotent} if it belongs to the essential image of the functor $\mathrm{Conil}$ from dg $\overline{\treemod}_\pl$-coalgebras to planar dg cooperads. We denote $\dgcooperads^{\categ{conil}}_\pl$ the full sub-category of coaugmented planar dg cooperads spanned by conilpotent ones.
\end{definition}

\begin{remark}
The idea behind this definition is the following: a (non-counital) cooperad can also be described in terms of partial decomposition maps
\[
\Delta_i: \C(n+k-1) \longrightarrow \C(n) \otimes \C(k)~,
\]
and it is conilpotent if and only if any iteration of these partial decompositions is eventually trivial. If this is the case, then the data of all the possible iterations is exactly encoded by the $\overline{\treemod}_\pl$-coalgebra structure. 
\end{remark}

For every natural integer $n \geq 1$, the comonad structure on $\overline{\treemod}_\pl$ restricts to $\overline{\treemod}_\pl^{(\leq n)}$. The inclusion of comonads $\overline{\treemod}_\pl^{(\leq n)} \rightarrowtail \overline{\treemod}_\pl$ induces an endofunctor $\mathrm{F}^{\mathrm{rad}}_n$ in the category of conilpotent planar dg cooperads.

\begin{definition}[Coradical filtration]
Let $\C$ be a conilpotent planar dg cooperad. Its $n$-\textit{coradical filtration} is given by the conilpotent planar dg cooperad $\mathrm{F}^{\mathrm{rad}}_n \C$. It induces a ladder diagram 
\[
\mathrm{F}^{\mathrm{rad}}_0 \C \rightarrowtail \mathrm{F}^{\mathrm{rad}}_1 \C \rightarrowtail \cdots \mathrm{F}^{\mathrm{rad}}_n \C \rightarrowtail \cdots~,
\]
indexed by $\mathbb{N}$, where all the arrows are monomorphisms.
\end{definition}

For any conilpotent planar dg cooperad $\C$, there is a canonical isomorphism between $\C$ and the colimit of the following ladder diagram
\[
\mathrm{F}^{\mathrm{rad}}_0 \C \rightarrowtail \mathrm{F}^{\mathrm{rad}}_1 \C \rightarrowtail \cdots \mathrm{F}^{\mathrm{rad}}_n \C \rightarrowtail \cdots
\]
in the category of conilpotent planar dg cooperads. 

\begin{proposition}
Let $n \geq 0$ and $\C$ be a conilpotent planar dg cooperad. Then $\mathrm{F}^{\mathrm{rad}}_n \C$ fits in the following pullback 
\[
        \begin{tikzcd}[column sep=2.5pc,row sep=2.5pc]
           \mathrm{F}^{\mathrm{rad}}_n \C \arrow[dr, phantom, "\lrcorner", very near start]
            \ar[r] \ar[d]
            & \overline{\treemod}_{\pl}^{(\leq n)} \C
            \ar[d,rightarrowtail]
            \\
            \C
            \ar[r,"\delta_\C"]
            & \overline{\treemod}_{\pl} \C.
        \end{tikzcd}
\]
in the category of dg $\mathbb{N}$-modules, where $\delta_\C$ denotes the dg $\overline{\treemod}_{\pl}$-coalgebra structure of $\C$.
\end{proposition}

\begin{proof}
This is a direct application of Proposition \ref{propadjointliftingepi}.
\end{proof}

\textbf{Tree endofunctor.}
Let $M$ be a dg $\mathbb{S}$-module, we can define the \textit{tree endofunctor} $\treemod(M)$ as the following reflexive coequalizer 
\[
\begin{tikzcd}[column sep=3pc,row sep=4pc]
\mathrm{Coeq}\Bigg(\displaystyle (\treemod_\pl((\mathrm{U}(M) \otimes \mathbb S) \otimes \mathbb S \arrow[r,"",shift right=1.1ex,swap]  \arrow[r,""{name=SD},shift left=1.1ex ]
&\treemod_\pl(\mathrm{U}(M)) \otimes \mathbb S\Bigg) \arrow[r,dashed]
&\treemod (M)~,
\end{tikzcd}
\]
where one of the maps is build from the dg $\mathbb{S}$-module structure of $M$ and the other using the monad structures $- \otimes \mathbb S$. Note that this definition is equivalent to the more standard one in \cite{LodayVallette}.

\begin{notation}
Let $n$ be a natural integer. We define analogously variants of the tree endofunctor by replacing the planar tree endofunctor $\treemod_\pl$ in the above coequalizer. 

\medskip

\begin{itemize}
\item The \textit{reduced tree} endofunctor $\overline{\treemod}$ is given by replacing $\treemod_\pl$ with $\overline{\treemod}_\pl$.

\medskip

\item We denote by $\treemod_{\leq n}$ the endofunctor obtained by replacing $\treemod_\pl$ with $\overline{\treemod}_{\pl,\leq n}$.
    
\medskip

\item We denoted by $\treemod^{(n)}$ the endofunctor obtained by replacing $\treemod_\pl$ with $\overline{\treemod}_{\pl}^{(n)}$.
    
\medskip

\item We denoted by $\treemod^{(\leq n)}$ the endofunctor obtained by replacing $\treemod_\pl$ with $\overline{\treemod}_{\pl}^{(\leq n)}$.
\end{itemize}
\end{notation}

\begin{proposition}
Let $X$ be a dg $\mathbb{N}$-module. The canonical map 
\[
\nu_X: \treemod (X \otimes \mathbb S) \longrightarrow \treemod_\pl (X) \otimes \mathbb S
\]
is a isomorphism of dg $\mathbb{S}$-modules, natural in $X$. 
\end{proposition}

\medskip

\textbf{Operads.} There is a monad structure on the tree endofunctor $\treemod$ which can be constructed using the monad structure on the planar tree endofunctor. We refer to \cite{notesym} for more details.

\begin{proposition}
The category of algebras over the monad $\treemod$ is canonically isomorphic to the category of dg operads.
\end{proposition}

\textbf{Conilpotent cooperads.} There is a comonad structure on the reduced tree endofunctor $\overline{\treemod}$, which again can be constructed from the comonad structure on the reduced planar tree endofunctor $\overline{\treemod}_\pl$. We refer to \cite{notesym} for more details. There is a fully faithful functor 
\[
\mathrm{Conil}: \mathsf{dg}~\overline{\treemod}\text{-}\mathsf{cog} \longrightarrow (\dgcooperads)_{\operad I/}
\]
from dg $\overline{\treemod}$-coalgebras to coaugmented dg cooperads.

\begin{definition}[Conilpotent dg cooperad]
Let $\C$ be a coaugmented dg cooperad. It is \textit{conilpotent} if it belongs to the essential image of the functor $\mathrm{Conil}$ from dg $\overline{\treemod}$-coalgebras to dg cooperads. We denote $\dgcooperads^{\categ{conil}}$ the full sub-category of coaugmented dg cooperads spanned by conilpotent ones.
\end{definition}

The adjunction 
\[
\begin{tikzcd}[column sep=5pc,row sep=3pc]
          \dgcooperads_\pl  \arrow[r, shift left=1.1ex, "-\otimes \mathbb S"{name=F}] &\dgcooperads, \arrow[l, shift left=.75ex, "\mathrm{U}_{\mathbb S}"{name=U}]
            \arrow[phantom, from=F, to=U, , "\dashv" rotate=90]
\end{tikzcd}
\]
restricts to an adjunction 
\[
\begin{tikzcd}[column sep=5pc,row sep=3pc]
          \dgcooperads_\pl^{\categ{conil}}  \arrow[r, shift left=1.1ex, "-\otimes \mathbb S"{name=F}] &\dgcooperads^{\categ{conil}}, \arrow[l, shift left=.75ex, "\mathrm{U}_{\mathbb S}"{name=U}]
            \arrow[phantom, from=F, to=U, , "\dashv" rotate=90]
\end{tikzcd}
\]
between conilpotent planar dg cooperads and conilpotent dg cooperads.

\medskip

For every natural integer $n \geq 1$, the comonad structure on $\overline{\treemod}$ restricts to $\overline{\treemod}^{(\leq n)}$. The inclusion of comonads $\overline{\treemod}^{(\leq n)} \rightarrowtail \overline{\treemod}$ induces an endofunctor $\mathrm{F}^{\mathrm{rad}}_n$ in the category of conilpotent dg cooperads.

\begin{definition}[Coradical filtration]
Let $\C$ be a conilpotent dg cooperad. Its $n$-\textit{coradical filtration} is given by the conilpotent dg cooperad $\mathrm{F}^{\mathrm{rad}}_n \C$. It induces a ladder diagram 
\[
\mathrm{F}^{\mathrm{rad}}_0 \C \rightarrowtail \mathrm{F}^{\mathrm{rad}}_1 \C \rightarrowtail \cdots \mathrm{F}^{\mathrm{rad}}_n \C \rightarrowtail \cdots~,
\]
indexed by $\mathbb{N}$, where all the arrows are monomorphisms.
\end{definition}

For any  conilpotent dg cooperad $\C$, there is a canonical isomorphism between $\C$ and the colimit of the following ladder diagram
\[
\mathrm{F}^{\mathrm{rad}}_0 \C \rightarrowtail \mathrm{F}^{\mathrm{rad}}_1 \C \rightarrowtail \cdots \mathrm{F}^{\mathrm{rad}}_n \C \rightarrowtail \cdots
\]
in the category of conilpotent dg cooperads.

\medskip

All the construction we have performed in this section commutes with the forgetful functors 
$$
\catdgmod{\kk} \longrightarrow \catpdgmod{\kk} \longrightarrow \catgrmod{\kk}~.
$$
This implies that if a dg (or pdg) conilpotent cooperad $\C$ proceeds, as a graded conilpotent cooperad, from a planar one $\C_\pl$ (in the sense that there is an isomorphism $\C \cong \C_\pl \otimes \mathbb S$ of graded cooperads), then the conilpotent dg cooperad $\overline{\mathrm{F}}_n \C$ fits in the following pullback diagram
\[
        \begin{tikzcd}[column sep=2.5pc,row sep=2.5pc]
           \mathrm{F}^{\mathrm{rad}}_n \C \arrow[dr, phantom, "\lrcorner", very near start]
            \ar[r] \ar[d]
            & \overline{\treemod}^{(\leq n)} \C
            \ar[d,rightarrowtail]
            \\
            \C
            \ar[r,"\delta_\C"]
            & \overline{\treemod} \C~,
        \end{tikzcd}
\]
in the category of dg $\mathbb{S}$-modules, where $\delta_\C$ denotes the dg $\overline{\treemod}$-coalgebra structure of $\C$.

\begin{remark}
If the characteristic of the base field $\kk$ is zero, both the tree module $\treemod (-)$ and the composition product $\circ$ preserve {\bf{finite}} cosifted limits. Therefore the forgetful functors from (conilpotent) dg cooperads to dg $\mathbb S$-modules preserve these limits. Thus, as in the planar case, given a dg cooperad $\C$, the conilpotent dg cooperad $\overline{\mathrm{F}}_n \C$ fits in a pullback diagram as above.
\end{remark}


\newpage

\section{Homotopy theory of operads and quasi-planar cooperads}

\vspace{2pc}

The main goal of this section is to introduce the notion of a \textit{quasi-planar cooperad}. We will describe in a future work its relation to the notion of a higher cooperad introduced in \cite{BrunoMalte}. First, we recall the bar-cobar constructions at the operadic level (the unital/curved version of \cite{unitalalgebras}). Then, we recall the semi-model category structure of \cite{Fresse}, which encodes the homotopy theory of dg operads. Finally, we introduce the notion of a quasi-planar cooperad, construct basic examples, and show their cobar constructions provide us with cofibrant resolutions which are particularly well-behaved when working over a field of any characteristic. Furthermore, for any quasi-planar conilpotent curved cooperad $\C$, where construct a canonical $\mathcal{E}$-comonoid structure over $\Omega \C$, where $\mathcal{E}$ is the Barratt-Eccles operad. This allows us to define a \textit{canonical} quasi-planar ladder for any quasi-planar conilpotent curved cooperad $\C$, which can be viewed as the positive characteristic analogue of the coradical filtration.

\subsection{Conilpotent curved cooperads and the operadic bar-cobar adjunction}\label{subsection: operadic bar-cobar}
The Koszul dual notion of a non-necessarily augmented dg operad is a conilpotent curved cooperad. We recall the bar-cobar adjunction at the operadic level that links these two notions constructed in \cite{unitalalgebras}.

\begin{definition}[Curved cooperad]\label{def curved cooperad}
A \textit{curved cooperad} $\operad C$ amounts to the data of a pdg cooperad $(\C,\Delta,\epsilon,d)$ equipped with a degree $-2$ map $\theta: \operad C \longrightarrow \operad I$ called the \textit{curvature} such that $\theta \circ d = 0$, and such that the following diagram commutes: 
\[
\begin{tikzcd}[column sep=8.5pc,row sep=2.5pc]
\operad C
\arrow[r,"\Delta_{(1)}"] \arrow[rd,"d^2"']
& \treemod^{(2)} \overline{\operad C} \arrow[d,"(\id ~\otimes~  \theta)~-~(\theta~ \otimes~ \id)~"] 
\\
&\operad C~.
\end{tikzcd}
\]
Here $\Delta_{(1)}$ is given by the sum over all possible decompostions into pairs. 
\end{definition}

\begin{remark}
A curved cooperad with zero curvature is a dg cooperad.
\end{remark}

\begin{remark}
A curved cooperad is said to be coaugmented (resp. conilpotent) if its underlying pdg cooperad is. 
\end{remark}

\textbf{The operadic bar construction.} Given a dg operad $\operad P$, one can build a conilpotent curved cooperad $\mathrm{B} \operad P$ whose underlying conilpotent graded cooperad is given by $\treemod (s\operad P \oplus s^2\operad I)$. We endow it with the unique coderivation whose projection onto the generators is the sum of the following maps:

\medskip

\begin{enumerate}
    \item the map
    \begin{align*}
\treemod (s\operad P \oplus s^2\operad I)
\twoheadrightarrow \treemod^{(2)} (s\operad P) 
& \longrightarrow s \operad P
\\
(sp \otimes_i sp') \otimes \{\sigma\}
& \mapsto (-1)^{|p|} s (p \circ_i p')^\sigma
\end{align*}
where $sp \otimes_i sp'$ labels a 2 nodes tree whose second node is plugged at the $i^{th}$ leaf of the root node,

\medskip

\item the map
\begin{align*}
\treemod (s\operad P \oplus s^2\operad I)
\twoheadrightarrow s\operad P
& \longrightarrow s \operad P
\\
sp 
& \mapsto -s d(p)~,
\end{align*}
\item the map
\begin{align*}
\treemod (s\operad P \oplus s^2\operad I)
\twoheadrightarrow s^2\operad I
& \longrightarrow s \operad P
\\
s^2 1
& \mapsto -s 1_{\operad P}.
\end{align*}
\end{enumerate}

\medskip

We denote by $d_\gamma$, $d_{\operad P}$ and $d_u$ the respective coderivations induced by these maps. The curvature is given by the following map
\begin{align*}
\Theta: \treemod (s\operad P \oplus s^2\operad I)
\twoheadrightarrow s^2\operad I
& \longrightarrow \operad I
\\
s^2 1
& \mapsto 1.
\end{align*}
One can check that $(\treemod (s\operad P \oplus s^2\operad I), d_\gamma + d_{\PP} + d_u,\Theta)$ forms a conilpotent curved cooperad.

\medskip

\textbf{The operadic cobar construction.} Given a coaugmented curved cooperad $\operad C$, one can build a dg operad  $\Omega \operad C$ whose underlying graded operad is $\treemod (s^{-1}\overline{\operad C})$. We endow it with the unique derivation whose restriction to the generators is the sum of the following maps:

\medskip

\begin{enumerate}
\item the map
\begin{align*}
s^{-1} \overline{\operad C}
& \longrightarrow \treemod^{(2)} (s^{-1} \overline{\operad C})
\hookrightarrow \treemod (s^{-1} \overline{\operad C})
\\
s^{-1} c
& \mapsto - \sum (-1)^{|c_{(1)}|}s c_{(1)} \otimes  s c_{(2)} \otimes \{\sigma\}
\end{align*}
where $\Delta_{(1)}(c) = \sum c_{(1)} \otimes  c_{(2)} \otimes \{\sigma\}$ denotes the sum of all possible decompositions of $c$ into a pairs,

\medskip
    
\item the map
    \begin{align*}
s^{-1} \overline{\operad C} & \longrightarrow s^{-1} \overline{\operad C} \hookrightarrow \treemod (s^{-1} \overline{\operad C})
\\
s^{-1} c
& \mapsto - s^{-1} d(c)
\end{align*}

\medskip

\item the map
\begin{align*}
s^{-1} \overline{\operad C} &\longrightarrow \operad I
\hookrightarrow \treemod (s^{-1} \overline{\operad C})
\\
s^{-1} c
& \mapsto \theta(c).
\end{align*}
\end{enumerate}

We denote by $d_\Delta, d_\C$ and $d_\theta$ the respective derivations induced by these maps. One can check $(\treemod (s^{-1}\overline{\operad C}), d_\Delta + d_\C + d_\theta)$ forms a dg operad.

\medskip

\textbf{The operadic bar-cobar adjunction.} This two functors form an adjunction 

\[
\begin{tikzcd}[column sep=5pc,row sep=3pc]
          \dgoperads  \arrow[r, shift left=1.1ex, "\mathrm{B}"{name=F}] & \curvcooperads^{\mathsf{conil}} , \arrow[l, shift left=.75ex, "\Omega"{name=U}]
            \arrow[phantom, from=F, to=U, , "\dashv" rotate=90]
\end{tikzcd}
\]

between dg operads and conilpotent curved cooperads. 

\begin{remark}
If a dg operad $\mathcal{P}$ is augmented, then its bar construction is in fact a conilpotent dg cooperad. Up to natural weak-equivalences, this adjunction coincides with the bar-cobar adjunction of \cite{LodayVallette} between augmented dg operads and conilpotent dg cooperads.
\end{remark}


\subsection{The Barratt--Eccles operad and the Hadamard tensor product} 
Let us review the construction of the Barratt--Eccles operad of C. Berger and B. Fresse in \cite{BergerFresse}.  

\begin{definition}[Barratt--Eccles dg operad]
    The \textit{unital barrat-Eccles dg operad}, denoted by $\operad E$, is defined as follows. For $n \geq 2$, the arity $n$ component $\operad E(n)$ is given, in degree $m$, by the free $\kk$-module generated by the sequences of distinct permutations
    $$
    (\sigma_0, \sigma_1, \ldots, \sigma_m) \in \mathbb S_n^{m+1}
    $$
    where $\sigma_i \neq \sigma_{i+1}$ for $0\leq i \leq m-1$. The right $\mathbb S_n$-action is given by
    $$
(\sigma_0, \sigma_1, \ldots, \sigma_m)^\sigma = (\sigma_0\sigma, \sigma_1 \sigma, \ldots, \sigma_m \sigma).
    $$
    The differential of $\mathcal{E}(n)$ is given as follows:
    $$
    d((\sigma_0, \sigma_1, \ldots, \sigma_m)) = (\sigma_1, \ldots, \sigma_m)
    - (\sigma_0, \sigma_2, \ldots, \sigma_m) + \cdots + (-1)^m (\sigma_0, \sigma_1, \ldots, \sigma_{m-1})~.
    $$
    For $n=0,1$, $\operad E(0) = \operad E(1) = \kk$ endowed with the trivial action and the zero differential. We refer to \cite[Section 1.13]{BergerFresse} for the specific formulae of the operadic compositions.
\end{definition}

\begin{remark}
Notice that, for all $n \geq 0$, the dg $\kk[S_n]$-module $\mathcal{E}(n)$ is quasi-free and concentrated in positive degrees, therefore it is also projective. Therefore $\operad E$ is both an $\mathbb S$-quasi-free and an $\mathbb S$-projective dg operad.
\end{remark}

The canonical morphism of operads $\operad E \longrightarrow \mathrm{u}\operad{C}\mathrm{om}$ is arity-wise a quasi-isomorphism. To see this, it suffices to notice it admits a section and that the degree $1$ endomorphism of $\operad E$
$$
h: (\sigma_0, \sigma_1, \ldots, \sigma_m) \mapsto (1,\sigma_0, \sigma_1, \ldots, \sigma_m)
$$
satisfies $\partial(h) = \pi_{\mathrm{u}\operad C \mathrm{om}}$, where $\pi_{\mathrm{u}\operad C \mathrm{om}}$ is the projection onto the image of this section.

\begin{definition}
    We define $\operad E_\pl$ as the sub-graded $\mathbb N$-module of $\operad E$ given, in arity $n \geq 2$ and degree $m$, by the free $\kk$-module generated by the sequences 
    $$
    (\sigma_0, \sigma_1, \ldots, \sigma_m) \in \mathbb S_n^{m+1}
    $$
    where $\sigma_i \neq \sigma_{i+1}$ for $0\leq i \leq m-1$ and such that $\sigma_0 =1$. In arities $n = 0,1$, we have $\operad E_\pl(0)= \operad E_\pl(1)= \kk$. 
\end{definition}

The canonical morphism of graded $\mathbb S$-modules
    $$
    \operad E_\pl \otimes \mathbb S \longrightarrow \operad E
    $$
is an isomorphism.

\begin{definition}[The Hadamard tensor product]
    Let $\operad P$ be a dg operad. We denote by $\operad E \otimes \operad P$ the \textit{Hadamard tensor} product of $\operad P$ with $\operad E$. It is the dg operad whose underlying dg $\mathbb S$-module is
    $$
    (\operad E \otimes \operad P)(n) = \operad E(n) \otimes \operad P(n)
    $$
    equipped with the diagonal action of $\mathbb S_n$. The operad structure is given by the map
    $$
    (\operad E \otimes \operad P) \comp (\operad E \otimes \operad P)
    \longrightarrow  (\operad E \otimes \operad E) \otimes (\operad P \comp \operad P)
    \longrightarrow \operad E \otimes \operad P~.
    $$
\end{definition}

\begin{proposition}
Let $\operad P$ be a dg operad. The dg $\mathbb S$-module of the Hadamard tensor product $\operad E \otimes \operad P$ is quasi-free with generators $\operad E_\pl \otimes U_{\mathbb S}(\operad P)$. 
\end{proposition}

\begin{proof}
Follows from Proposition \ref{prop: tensor with a quasi-free}.
\end{proof}

Notice that the canonical map $\operad E \otimes \operad P \qi \operad P$ is also an arity-wise quasi-isomorphims.

\begin{remark}
We choose, as a convention, to systematically consider $\operad E$ on the left hand side of the tensor products $\operad E \otimes \operad P$ for consistency reasons that will become more apparent later. One could have chosen the opposite convention, all the results also hold. 
\end{remark}


\subsection{Another presentation of the Barratt--Eccles operad}
In this subsection, we give a new presentation of the Barratt-Eccles operad and its composition. The computations introduced here will be used in Subsections \ref{subsection: E-comodule} and \ref{subsection: quasi-planar canonical filtration}.

\medskip

Let us describe the composition in the Barratt-Eccles operad in terms of shuffle permutations. Throughout this subsection and the next one, given two permutations $\mu \in \mathbb S_p, \psi \in \mathbb S_q$ and $1 \leq i \leq p$, $\mu \comp_i \psi$ refers to the composition of $\mu$ with $\psi$ at the $i$-th leaf. This composition is given by the partial compositions of the associative operad in the category of sets.

\medskip

\begin{definition}[Admissible permutations]
    Let $p, q$ be two natural integers with $p \geq 1$, let $1 \leq i \leq p$ and let $n = p+q -1$. We say that
    a permutation $\sigma \in \mathbb S_n$ is
    \begin{enumerate}
    \medskip
    
        \item $(p,q,i)$-\textit{bottom admissible} if there exists a permutation $\mu \in \mathbb S_p$ such that
        $$
        \sigma = \mu \comp_{\mu^{-1}(i)} \id_q~;
        $$
        \item $(p,q,i)$-\textit{top admissible} if there exists a permutation $\psi \in \mathbb S_q$ such that 
        \[
        \sigma = \id_p \comp_i \psi~;
        \]
        \item $(p,q,i)$-\textit{admissible} if it is $(p,q,i)$-top admissible or $(p,q,i)$-bottom admissible;
       	\medskip
        \item $(p,q,i)$-\textit{non admissible} if it is not $(p,q,i)$-admissible.
    \end{enumerate}
\end{definition}

\begin{lemma}\label{lemma: twisted composition of permutation injective}
Given $p,q \geq 1$ and $1 \leq i \leq p$, the function 
\begin{align*}
    \mathbb S_p \times \mathbb S_q &\longrightarrow \mathbb S_n
    \\
    (\mu, \psi) &\mapsto \mu \comp_{\mu^{-1}(i)} \psi
\end{align*}
is injective. In particular, the intersection of $(p,q,i)$-bottom admissible permutations with $(p,q,i)$-top admissible permutations only contains the trivial permutation.
\end{lemma}

\begin{proof}
A permutation $\sigma$ in $\mathbb S_n$ can be written as a composition $\mu \circ_{\mu^{-1}(i)} \psi$, with $\mu \in \mathbb S_p$ and $\psi \in \mathbb S_q$ if and only if $\sigma^{-1}$ sends the segment $\{i,\cdots, i + q-1\}$ to a segment. Then the permutation $\psi$ is determined by the induced function between these two segments and $\mu$ is determined by the function induced by collapsing these each of these segments into one element. 
\end{proof}

Notice that for every $\sigma_1, \sigma_2 \in \mathbb S_n$ and $\mu_1, \mu_2 \in \mathbb S_m$ and $1 \leq 1 \leq n$ we have 
$$
(\sigma_2 \sigma_1) \comp_i (\mu_2 \mu_1)
=
(\sigma_2 \comp_{\sigma_1(i)} \mu_2) (\sigma_1 \comp_{i} \mu_1)~.
$$

\begin{definition}
Let $k,p\geq 1, q \geq 0$, $1 \leq i \leq p$ and $n = p+q-1$ be natural integers and let us consider two sequences of permutations $\underline\mu \in \mathbb S_p^k$ and $\underline\psi \in \mathbb S_q^k$. We define $\underline\mu \ltimes_i \underline \psi$ as the sequence $\underline \sigma = (\sigma_1, \ldots, \sigma_k) \in \mathbb S_n^k$ given by  

\[
\left\{
\begin{tikzcd}[column sep=0pc,row sep=0.25pc]
&\sigma_1 \coloneqq \mu_1 \comp_{\mu_1^{-1}(i)} \psi_1;\\
&\vdots \\
&\sigma_j \coloneqq \mu_j \comp_{\mu_{j}^{-1} \cdots \mu_1^{-1}(i)} \psi_j \quad \text{for}  \quad 1 < j < k;\\
&\vdots \\
&\sigma_k \coloneqq \mu_k \comp_{\mu_{k}^{-1} \cdots \mu_1^{-1}(i)} \psi_k.
\end{tikzcd}
\right.
\]
\end{definition}

One can notice that 
    $$
    (\mu_1, \ldots, \mu_{k-1}, \mu_{k}) \ltimes_i (\psi_1, \ldots, \psi_{k-1}, \psi_{k}) = (\mu_1 \comp_{\mu_1^{-1}(i)} \psi_1) \sqcup
    ((\mu_2, \ldots, \mu_{k}) \ltimes_{\mu_1^{-1}(i)} (\psi_2, \ldots, \psi_{k} ))~,
    $$
where $\sqcup$ stands for the concatenation of sequences of permutations.

\begin{lemma}\label{lemma: twisted sequence composition of permutation injective}
    If $q \geq 1$, the function $- \ltimes_i - : \mathbb S_p^k \times \mathbb S_q^k \longrightarrow \mathbb S_n^k$ is injective.
\end{lemma}

\begin{proof}
    This follows from a straightforward induction on $k$ using Lemma \ref{lemma: twisted composition of permutation injective}.
\end{proof}

\begin{definition}[Admissible sequence]
Let $p, q, k$ be three natural integers such that $p,k \geq 1$, let $1 \leq i \leq p$ and let $n = p+q -1$. A sequence of non-trivial permutations $\underline\sigma \in \mathbb S_n^k$ is
    \begin{enumerate}
    \medskip
    
        \item $(p,q,i)$\textit{-admissible} if there exists two sequences of permutations $\underline\mu \in \mathbb S_p^k$ and $\underline\psi \in \mathbb S_q^k$ such that $\underline\sigma = \underline\mu \ltimes_i \underline\psi$, and if for every $1 \leq j \leq k$ at least (and necessarily at most) one of the two permuations $\mu_j, \psi_j$ is trivial;
        
    \medskip
    
        \item $(p,q,i)$\textit{-non admissible} if it is not $(p,q,i)$-admissible.
    \end{enumerate}
\end{definition}

\begin{definition}
    Let $\underline\mu \in \mathbb S_p^a$ and $\underline\psi \in \mathbb S_q^b$
    be two sequences of permutations with $p,a,b \geq 1$ and let $\phi$ be a
    $(a,b)$-shuffle. We denote $\underline\mu \comp_{i,\phi} \underline\psi$
    the sequence of permutation in $\mathbb S_n^k$ (where $k=a + b$ and $n=p+q-1$) given by
    $$
    \underline\mu \comp_{i,\phi} \underline\psi \coloneqq
    \phi(\underline\mu \sqcup \id_p^b) \ltimes_i \phi(\id_q^a \sqcup \underline\psi)~,
    $$
    where
    \begin{enumerate}
    \medskip
    
        \item $\phi(\underline\mu \sqcup \id_p^b) \in \mathbb S_p^k$ is the sequence of permutations $(\mu'_1, \ldots, \mu'_k)$ such that $\mu'_j = \mu_{\phi^{-1} (j)}$ if $\phi^{-1} (j) \leq a$ and $\mu'_j = \id_p$ otherwise;
        
    \medskip
    
        \item $\phi(\id_q^a \sqcup \underline\psi ) \in \mathbb S_q^k$ is the sequence of permutations $(\psi'_1, \ldots, \psi'_k)$ such that $\psi'_j = \psi_{\phi^{-1} (j)-a}$ if $\phi^{-1} (j) \geq a +1$ and $\psi'_j = \id_q$ otherwise.
    \end{enumerate}
\end{definition}

\begin{lemma}
    The function
\[
\begin{tikzcd}[column sep=2pc,row sep=0.5pc]
\displaystyle \coprod_{a+b =k} (\mathbb S_p-\{\id_p\})^a \times (\mathbb S_q-\{\id_p\})^b \times \mathrm{Sh}(a,b) \arrow[r]
&(\mathbb S_n-\{\id_p\})^k \\
(\underline\mu, \underline\psi, \phi) \arrow[r,mapsto]
&\underline\mu \comp_{i, \phi} \underline\psi
\end{tikzcd}
\]
    is injective, and its image is given by $(p,q,i)$-admissible sequences.
\end{lemma}

\begin{proof}
    This is a direct consequence of the definition of a $(p,q,i)$-admissible sequence combined with Lemma \ref{lemma: twisted sequence composition of permutation injective}.
\end{proof}

Let again $\operad E$ denote the Barratt-Eccles dg operad of \cite{BergerFresse}.

\begin{definition}
Let $\underline \sigma \coloneqq (\sigma_1, \ldots, \sigma_k) \in \mathbb S_n^k$ be a sequence of permutations, we define $\rho(\underline \sigma)$, an element of $\operad E(n)_k$, as follows
    as follows
    \begin{enumerate}
    \medskip
        \item if $k=0$, then $\rho(\underline\sigma) \coloneqq \rho(*) = (\id_n)$,
        
    \medskip
    
        \item if $k \geq 1$ and $\underline\sigma = (\sigma_1, \ldots, \sigma_k)$, 
        $$
        \rho (\underline\sigma) \coloneqq (\id_n, \sigma_k , \sigma_{k-1}\sigma_k, \ldots, \sigma_1 \cdots \sigma_k)~.
        $$
    \end{enumerate}    
\end{definition}

\begin{remark}
A direct computation shows that $\{ \rho(\underline\sigma)\}_{\underline\sigma}$ for $\underline \sigma \coloneqq (\sigma_1, \ldots, \sigma_k) \in \mathbb S_n^k$ forms a basis of $(\operad E(n)_\pl)_k$ for all $n,k \geq 0$ as a graded $\mathbb{N}$-module. Thus it is a	 basis of $\operad E(n)_k$ as a graded $\mathbb{S}$-module.
\end{remark}

\medskip

\textbf{Partial compositions.} Let us recall the partial compositions of the Barratt-Eccles dg operad, constructed in \cite[Section 1.1.3]{BergerFresse}. For two elements $x = (\mu_0, \ldots,\mu_a) \in \operad E(p)_a$ and $y = (\psi_0, \ldots,\psi_b) \in \operad E(p)_b$, given $1 \leq i \leq a$ and an $(a,b)$-shuffle $\phi$, let us consider 
$$
x \comp_{i, \phi}^{\operad E} y \in \operad E(p+q-1)_{a+b}~.
$$
It is the sequence of permutations $(\sigma_0, \ldots, \sigma_{a+b})$ defined by 
$$
\sigma_j = \mu_{j_d} \comp_i \psi_{j_u}~,
$$
where 
\begin{align*}
    j_d &= \#\{k \in \mathbb N| 1 \leq k \leq a,  \ \phi(k) \leq j\}~,
    \\
    j_u &= \#\{k \in \mathbb N| a+1 \leq k \leq a+b,  \ \phi(k) \leq j\}.
\end{align*}

The partial compositions in $\mathcal{E}$ are given by 
$$
x \comp_i^{\operad E} y = \sum_{\phi \in \mathrm{Sh}(a,b)} (-1)^{\epsilon(\phi)}x \comp_{i, \phi}^{\operad E} y~,
$$
where the sum is taken over all $(a,b)$-shuffles and where $\epsilon(\phi)$ is the sign of the permutation $\phi$.

\begin{notation}
    For every $(\sigma_1, \ldots, \sigma_n) \in \operad E(k)_{n+1}$ we denote
    $$
    \mathrm{inv}((\sigma_1, \ldots, \sigma_n)) = (\sigma_n, \ldots, \sigma_1) \in \operad E(k)_{n+1}.
    $$
\end{notation}

\medskip

\begin{lemma}\label{lemma: composition Barratt eccles}
    For every $\underline\mu \in \mathbb S_p^a, \underline\psi \in \mathbb S_q^b$ and every $(a,b)$-shuffle $\phi$, we have
    $$
    (\mathrm{inv}\rho(\underline\mu)) \comp_{\underline\mu^{-1}(i), \phi}^{\operad E} (\mathrm{inv}\rho(\underline \psi))
    =
    \mathrm{inv}\rho(\underline\mu \comp_{i, \phi} \underline \psi)~,
    $$
    where $\underline\mu^{-1}(i) = \mu_a^{-1}\cdots \mu_1^{-1}(i)$. Subsequently,
    $$
    (\mathrm{inv}\rho(\underline\mu)) \comp_{\underline\mu^{-1}(i)} (\mathrm{inv}\rho(\underline \psi))
    =
    \sum_{\phi \in \mathrm{Sh}(a,b)} (-1)^{\epsilon({\phi})} \mathrm{inv}\rho(\underline\mu \comp_{i, \phi} \underline \psi)~,
    $$
    where the sum is taken over the $(a,b)$-shuffles and $\epsilon({\phi})$ is the sign of the shuffle ${\phi}$.
\end{lemma}

\begin{proof}
    For $1 \leq j \leq a+b$ we denote $\mu'_j = \mu_{\phi^{-1}(j)}$ if $\phi^{-1}(j) \leq a$ and $\mu'_j = \id_p$ otherwise; similarly $\psi'_j = \psi_{\phi^{-1}(j)-a}$ if $\phi^{-1}(j) \geq a+1$ and $\mu'_j = \id_q$ otherwise.
    One can check that both 
    \[
    (\mathrm{inv}\rho(\underline\mu)) \comp_{\underline\mu^{-1}(i), \phi}^{\operad E} (\mathrm{inv}\rho(\underline \psi)) \quad \text{and} \quad 
    \mathrm{inv}\rho(\underline\mu \comp_{i, \phi} \underline \psi)
    \]
   	are equal to 
    
    \[
    (\id_{a+b}, \mu'_{a+b} \comp_{\underline\mu^{-1}(i)} \psi'_{a+b},
    \cdots, (\mu'_1 \cdots \mu'_{a+b} )\comp_{\underline\mu^{-1}(i)} (\psi'_1 \cdots\psi'_{a+b}))~.
    \]
\end{proof}

\begin{definition}[Opposite shuffle]
Let $\phi$ be a $(a,b)$-shuffle. Its \textit{opposite} $(a,b)$\textit{-shuffle} $\overline{\phi}$ is given by 

$$
\overline{\phi}(j) = 
\begin{cases}
    a+b+1 - \phi (a+1 - j) \text{ if }j \leq a~;
    \\
    a+b+1 - \phi (2a+b 1 - j) \text{ if }j \geq a+1~.
\end{cases}
    $$
\end{definition}

\medskip

In other words, the opposite $(a,b)$-shuffle $\overline{\phi}$ is given by inverting indepedently the segments $\{1, \ldots, a\}$ and $\{a+1, \ldots, a+b\}$, applying $\phi$ to them, and finally inverting the whole segment $\{1, \ldots, a+b\}$. By inverting, we mean applying the unique permutation that yields the maximum number of inversions. Notice that the signatures of a $(a,b)$-shuffle and its opposite are related by the following formula
    $$
    (-1)^{\epsilon(\overline{\phi})} = (-1)^{ab +\epsilon({\phi})}~.
    $$

\begin{lemma}\label{lemma: inversion Barratt eccles}
    For every $(\mu_0, \ldots, \mu_a) \in \operad E(p)_{a}, (\psi_0, \ldots, \psi_b) \in \operad E(q)_{b}$, $1 \leq i \leq p$ and every $(a,b)$-shuffle $\phi$ we have
    $$
    \mathrm{inv}(\mu_0, \ldots, \mu_{a}) \comp_{i, \phi}^{\operad E} \mathrm{inv}(\psi_0, \ldots, \psi_b) = 
    \mathrm{inv}((\mu_0, \ldots, \mu_{a}) \comp_{i, \overline\phi}^{\operad E} (\psi_0, \ldots, \psi_b))~.
    $$
    
    Subsequently
    
    $$
    \mathrm{inv}(\mu_0, \ldots, \mu_{a}) \comp_i \mathrm{inv}(\psi_0, \ldots, \psi_b) = (-1)^{ab} 
    \mathrm{inv}((\mu_0, \ldots, \mu_{a}) \comp_i (\psi_0, \ldots, \psi_b))~.
    $$
\end{lemma}

\begin{proof}
Let us denote
$$
(\sigma_0, \ldots, \sigma_{a+b})
=
\mathrm{inv}(\mathrm{inv}(\mu_0, \ldots, \mu_{a}) \comp_{i, \phi}^{\operad E} \mathrm{inv}(\psi_0, \ldots, \psi_b))~.
$$
By definition, one has
$$
\sigma_{a+b-j} = \mu_{a-j_d} \comp_i \psi_{b-j_u}
$$
for every $0 \leq j \leq a+b$, where as above
\begin{align*}
    j_d &= \#\{k \in \mathbb N| 1 \leq k \leq a,  \ \phi(k) \leq j\}
    \\
    j_u &= \#\{k \in \mathbb N| a+1 \leq k \leq a+b,  \ \phi(k) \leq j\}.
\end{align*}
By denoting $\overline{j}= a+b-j$,
 $\overline{j}_d= a-j_d$ and $\overline{j}_u= b-j_u$, this rewrites as 
$$
\sigma_{\overline{j}} = \mu_{\overline{j}_d} \comp_i \psi_{\overline{j}_u}
$$
for every $0 \leq \overline j \leq a+b$. Combined with the fact that 
\begin{align*}
    \overline{j}_d = a-j_d &=\#\{k \in \mathbb N| 1 \leq k \leq a,  \ \overline\phi(k) > j\}
    \\
    &=\#\{k \in \mathbb N| 1 \leq k \leq a,  \ \overline\phi(k) \geq j+1\}
    \\
    &=\#\{k \in \mathbb N| 1 \leq k \leq a,  \ a+b+1 - \overline\phi(k) \leq a+b+1 -(j+1)\}
    \\
    &=\#\{k \in \mathbb N| 1 \leq k \leq a,  \ a+b+1 - \overline\phi(a+1-k) \leq \overline j\}
    \\
    &=\#\{k \in \mathbb N| 1 \leq k \leq a,  \ \overline\phi(k) \leq \overline{j}\}~,
\end{align*}
and similarly
$$
\overline{j}_u = \#\{k \in \mathbb N| a+1 \leq k \leq a+b,  \ \overline\phi(k) \leq \overline{j}\}~,
$$
we get that 
$$
(\sigma_0, \ldots, \sigma_{a+b}) = (\mu_0, \ldots, \mu_{a}) \comp_{i, \overline\phi}^{\operad E} (\psi_0, \ldots, \psi_b)~.
$$
\end{proof}

\begin{lemma}\label{lemma:inversecompinbe}
For every $\underline\mu \in \mathbb S_p^a, \underline\psi \in \mathbb S_q^b$, we have
    
    \[
    \rho(\underline\mu) \comp_{\underline\mu^{-1}(i)} \rho(\underline \psi) = (-1)^{ab}\sum_{\phi \in \mathrm{Sh}(a,b)} {\epsilon({\phi})} \rho(\underline\mu \comp_{i, \phi} \underline \psi)~,
    \]
    \vspace{0.1pc}
    
    where $\underline\mu^{-1}(i) = \mu_a^{-1}\cdots \mu_1^{-1}(i)$, where the sum is taken over the $(a,b)$-shuffles and where $\epsilon({\phi})$ is the sign of the $(a,b)$-shuffle ${\phi}$.
\end{lemma}

\begin{proof}
    This is a combination of Lemma \ref{lemma: inversion Barratt eccles} and Lemma \ref{lemma: composition Barratt eccles}.
\end{proof}

\textbf{Comonoid structure.} As described in \cite[Section 1.1]{BergerFresse}, the Barratt-Eccles dg operad $\operad E$ has the structure of a comonoid for the Hadamard tensor product. This structure is given by the map

\[
\begin{tikzcd}[column sep=2pc,row sep=0pc]
\Delta_{\operad E}: \operad E \arrow[r]
&\operad E \otimes \operad E \\
(\sigma_0, \ldots, \sigma_n)  \arrow[r,mapsto]
&\displaystyle \sum_{i=0}^k (\sigma_0, \ldots, \sigma_i) \otimes (\sigma_i, \ldots, \sigma_n)~.
\end{tikzcd}
\]

\begin{lemma}\label{lemma:monoidalpropertyoftherhomap}
For every $(\sigma_1, \ldots, \sigma_k) \in \mathbb S_n^k$, one has that
    $$
    \Delta_{\operad E} (\rho((\sigma_1, \ldots, \sigma_k)))
    = \sum_{i=1}^{k+1} \rho((\sigma_i, \ldots, \sigma_k)) \otimes \rho((\sigma_1, \ldots, \sigma_{i-1}))^{\sigma_i \cdots \sigma_k}~,
    $$
    where $\Delta_{\operad E}$ is the structure of a comonoid of the Barratt-Eccles dg operad $\operad E$ for the Hadamard tensor product.
\end{lemma}

\begin{proof}
This follows from direct inspection.    
\end{proof}

\subsection{Homotopy theory of operads}
There is a semi-model category structure on the category of dg operads constructed by B. Fresse in \cite{Fresse}.

\begin{theorem}[{\cite[Chapter 12]{Fresse}}]
The category of dg operads admits a \textit{semi-model structure}, determined by the following sets of maps:
\medskip
\begin{itemize}
    \item the set of fibrations is given by morphisms of dg operads which are arity-wise degree-wise epimorphisms,
    \medskip
    \item the set of weak-equivalences is given by morphisms of dg operads which are arity-wise quasi-isomorphism,
    \medskip
    \item the set of cofibrations is given by morphisms of dg operads which have the left lifting property with respect to 
    maps that are both fibrations and weak-equivalences.
    \medskip
\end{itemize}
\end{theorem}

Let us denote by $[1]$ the category $0 \longrightarrow 1$ with two objects and a single non-trivial arrow. A functor from this category is equivalent to the data of two objects and a morphism between them.

\begin{lemma}\label{lemma:operadelementarycof}
Let us consider a morphism of dg operad $\mathcal{P}^{(0)} \longrightarrow \mathcal{P}^{(1)}$. Suppose there exists a morphism of graded $\mathbb{N}$-modules $X^{(0)} \longrightarrow X^{(1)}$ such that the following diagram of functors
    \[
    \begin{tikzcd}[column sep=3.5pc,row sep=3.5pc]
        {[1]}
        \ar[rr, "\operad P"] \ar[d, "X",swap]
        && \dgoperads
        \ar[d,"\mathrm{U}"]
        \\
        \catgrmod{\mathbb{N}}
        \ar[r,"\treemod_\pl"]
        & \groperads_\pl
        \ar[r,"\mathbb S \otimes -"]
        & \groperads
    \end{tikzcd}
    \]
    commutes. Furthermore, suppose that
    \medskip
    \begin{enumerate}
        \item the map $X^{(0)} \longrightarrow X^{(1)}$ is a degree-wise injection;

        \medskip
        
        \item the restriction of the derivation of $\operad P^{(1)}$
        to the generators $X^{(1)}$ factors as
        $$
        X^{(1)} \longrightarrow \operad P^{(0)} + X^{(1)} \hookrightarrow \operad P^{(1)}~.
        $$
    \end{enumerate}
    Then the map $\operad P^{(0)} \longrightarrow \operad P^{(1)}$ is a cofibration.
\end{lemma}

\begin{proof}
    Let us decompose
    $X^{(1)} \cong X^{(0)} \oplus Y$ as graded $\mathbb{N}$-modules. The restriction to $Y$ of the derivation of $\operad P^{(1)}$ is
    a degree $-1$ map $(d_Y, \phi) :Y \longrightarrow Y \oplus  \operad P^{(0)}$. The fact that the whole derivation of $\operad P^{(1)}$ squares to zero amounts to the facts that
    \medskip
    \begin{enumerate}
        \item the map $d_Y$ squares to zero and thus $Y$ gets the structure of a dg module ;
    \medskip
        \item the map $\psi : S^{-1} \otimes Y \longrightarrow \operad P^{(0)}$, that sends $s \otimes y$ to $\phi(y)$ is a morphism of dg modules.
    \end{enumerate}

    \medskip
    
    Thus the following diagram
    \[
    \begin{tikzcd}[column sep=3.5pc,row sep=3.5pc]
        \treemod((S^{-1} \otimes Y) \otimes \mathbb S)
        \ar[r, "\psi"] \ar[d]  \arrow[dr, phantom, "\ulcorner", very near end]
        & \operad P^{(0)}
        \ar[d]
        \\
        \treemod((D^{0} \otimes Y) \otimes \mathbb S)
        \ar[r]
        & \operad P^{(1)} .
    \end{tikzcd}
    \]
    is a pushout square of dg operads. Since the left vertical map is cofibration, so is the right vertical one.
\end{proof}

Given a small ordinal $\alpha$, one can view it as a category. Objects are indexed by the ordinal $\alpha$, and there is only one non-trivial arrow between two objects $i$ and $j$ in $\alpha$, if and only if $i < j$. Cocontinuous functors from $\alpha$ to a category give what we call $\alpha$-indexed \textit{ladders}. We refer to Subsection \ref{subsection: notations and conventions} for more details.

\begin{proposition}\label{proposition:operadqpcofibrant}
Let $\alpha$ be a small ordinal and let 
\[
\operad P^{(0)} \longrightarrow \operad P^{(1)} \longrightarrow \operad P^{(2)} \longrightarrow \cdots \longrightarrow \operad P^{(i)} \longrightarrow \cdots~,
\]
be an $\alpha$-indexed ladder of dg operads. Suppose there exists an $\alpha$-indexed ladder 
\[
X^{(0)} \longrightarrow X^{(1)} \longrightarrow X^{(2)} \longrightarrow \cdots \longrightarrow X^{(i)} \longrightarrow \cdots~,
\]
of graded $\mathbb{N}$-modules such that the following diagram of functors
    \[
    \begin{tikzcd}[column sep=3.5pc,row sep=3.5pc]
        \alpha
        \ar[rr, "\operad P"] \ar[d, "X",swap]
        && \dgoperads
        \ar[d,"\mathrm{U}"]
        \\
        \catgrmod{\kk}
        \ar[r,"\treemod_\pl"]
        & \groperads_\pl
        \ar[r,"\mathbb S \otimes -"]
        & \groperads~.
    \end{tikzcd}
    \]
    commutes. Furthermore, suppose that
    \medskip
    \begin{enumerate}
        \item for every $i, i+1 \in \alpha$, the map $X^{(i)} \longrightarrow X^{(i+1)}$ is a degree-wise injection;

    \medskip
    
        \item for every $i \in\alpha$, the restriction of the derivation of $\operad P^{(i+1)}$
        to the generators $X^{(i+1)}$ factors as
        $$
        X^{(i+1)} \longrightarrow \operad P^{(i)} + X^{(i+1)} \hookrightarrow \operad P^{(i+1)}.
        $$
    \end{enumerate}
    Then the colimit dg operad 
    \[
    \operad P^{(\alpha)} \coloneqq \colim{i \in \alpha} ~\operad P^{(i)} 
    \]
    is cofibrant.
\end{proposition}

\begin{proof}
    Let $v$ be the subset of $\alpha +1$ of elements $i$ so that $\operad P^{(i)}$ is cofibrant. It satisfies the following properties:
    \medskip
    
    \begin{enumerate}
        \item it contains the first element $0 \in \alpha +1$ since $\operad P^{(0)} \cong \operad I$;

        \medskip
        
        \item for $i, i+1 \in \alpha+1$, if $i \in v$, then $i+1 \in v$ by Lemma \ref{lemma:operadelementarycof};

        \medskip
        
        \item if $i \in \alpha+1$ is a limit ordinal and if $v \sqcap \{j < i\} = \{j < i\}$, then $i \in v$ since cofibrations are stable through transfinite composition.
    \end{enumerate}
    
    \medskip
    
    Therefore $v = \alpha$, which implies that $P^{(\alpha)}$ is cofibrant.
\end{proof}


\subsection{Conilpotent cooperad ladders}

\begin{definition}[Cooperad ladder]
Let $\alpha$ be a small ordinal. An $\alpha$-\textit{cooperad ladder} amounts to the data of a functor 
\[
\operad C: \alpha \longrightarrow \curvcooperads^{\mathsf{conil}}~,
\]
which satisfies the following conditions:

\medskip

\begin{enumerate}
    \item for every limit ordinal $i \in \alpha$, the map
    \[
\colim{j < i}~\C^{(j)} \to \C^{(i)}
    \]
    is an isomorphism; equivalently, the extension of $\C^{(-)}$ to $1 + \alpha$
    that sends $0$ to $\operad I$ is cocontinuous;
    \item for every $i, i+1 \in \alpha$, the map $\operad C^{(i)} \longrightarrow \operad C^{(i+1)}$
is an arity-wise degree-wise injection;

\medskip

    \item the decomposition map of the underlying pdg conilpotent cooperad
\[
\begin{tikzcd}
\operad C^{(i+1)} \arrow[r]
&\treemod \operad C^{(i+1)} \arrow[r]
&\overline \treemod \operad C^{(i+1)}
\end{tikzcd}
\]
factors through $\overline \treemod \operad C^{(i)}$. 
\end{enumerate}

\medskip

We denote
\[
\operad C^{(\alpha)} \coloneqq \colim{i \in \alpha}~\operad C^{(i)}~,
\]
the colimit of the cooperad ladder.
\end{definition}

The \textit{associated graded} of this ladder is given by 
$$
\gr_{i} ~\operad C \coloneqq \operad C^{(i)}/ \colim{j < i}~\operad C^{(j)}~.
$$
In fact, we have that $\gr_{i} ~\operad C \cong \operad C^{(i)}/ \operad C^{(i-1)}$ and that $\gr_{0} ~\operad C = \operad C^{(0)}$. Furthermore, $\gr_i~ \operad C = 0$ when $i$ is a limit ordinal. 

\begin{remark}
    For all $i \in \alpha$, the associated graded object $\gr_i \operad C$ is a dg module. Indeed, the induced coderivation on it squares to zero because of the condition in Definition \ref{def curved cooperad}. 
\end{remark}

\begin{example}[Coradical ladder]
Let $\operad C$ be a conilpotent curved cooperad. The coradical filtration of $\operad C$
$$
\mathrm{F}^{\mathrm{rad}}_0~ \operad C \rightarrowtail \mathrm{F}^{\mathrm{rad}}_1~ \operad C
\rightarrowtail \cdots \rightarrowtail \mathrm{F}^{\mathrm{rad}}_n~\operad C \rightarrowtail \cdots \rightarrowtail \colim{i \in \omega}~ \mathrm{F}^{\mathrm{rad}}_i~\operad C \cong \operad C
$$
is a cooperad ladder called the \textit{coradical ladder}. Notice that it is also the coradical development of the constant ladder.
\end{example}

\subsection{Quasi-planar cooperad ladders and quasi-planar cooperads}

\begin{definition}[pseudo-planar cooperad]
   A \textit{pseudo-planar} conilpotent curved cooperad $\C$ amounts to the data $(\C_\pl, \C, \varphi_\C)$ of a
    graded planar conilpotent cooperad $\C_\pl$ and a conilpotent curved cooperad $\C$ together with an isomorphism of graded conilpotent cooperads
    \[
    \varphi_\C: \C \cong \C_\pl \otimes \mathbb S.
    \]
    A morphism  $f: (\C_\pl, \C, \varphi_\C) \longrightarrow (\operad D_\pl, \operad D,\varphi_\D)$ of pseudo-planar curved conilpotent cooperads is the data of a morphism of graded planar conilpotent cooperads $f_\pl: \operad C_\pl \to \operad D_\pl$ such that the composite map
    \[
    f:\C \cong \C_\pl \otimes \mathbb S
    \xrightarrow{f_\pl \otimes \mathbb S} \operad D_\pl \otimes \mathbb S \cong \operad D
    \]
    is a morphism of conilpotent curved cooperads. We denote the category of pseudo-planar conilpotent curved cooperads by $\curvcooperads^{\mathsf{conil}}_{\mathrm{pseudo-pl}}$.
\end{definition}

\begin{remark}
Notice that the category of seudo planar conilpotent curved cooperads fits into the following pseudo-commutative square diagram
    \[
\begin{tikzcd}[column sep=3.5pc,row sep=3.5pc]
        \curvcooperads^{\mathsf{conil}}_{\mathrm{pseudo-pl}}
        \ar[r] \ar[d]
        & \curvcooperads^{\mathsf{conil}}
        \ar[d,"\mathrm{U}"]
        \\
        \grcooperads^{\mathsf{conil}}_\pl
        \ar[r, "- \otimes \mathbb S"']
        & \grcooperads^{\mathsf{conil}}~,
    \end{tikzcd}
    \]
which is a homotopy pullback. 
\end{remark}

\begin{definition}[Quasi-planar cooperad ladder]
A \textit{quasi-planar cooperad ladder} amounts to the data
of a small ordinal $\alpha$ and a functor
\begin{align*}
     \alpha &\to \curvcooperads^{\mathsf{conil}}_{\mathrm{pseudo-pl}}
     \\
     i & \mapsto (\C_\pl^{(i)}, \C^{(i)}, \varphi_{\C^{(i)}})~,
\end{align*}

    that satisfies the following conditions:
    
    \medskip
    \begin{enumerate}
        \item the functor $\operad C$ is a cooperad ladder (thus $\operad C_\pl$ is also cocontinuous);

    \medskip
    
        \item for every $i, i+1 \in \alpha$, the restriction of the coderivation of $\operad C^{(i+1)}$ to $\operad C_\pl^{(i+1)} \otimes 1$ factors through
        
        \[
        \operad C_\pl^{(i+1)} \otimes 1 \longrightarrow \operad C_\pl^{(i+1)} \otimes 1 + \operad C^{(i)} \hookrightarrow \operad C^{(i+1)};
        \]
        \vspace{0.1pc}
        
        in other words, the differential of $\gr_{i+1} ~\operad C = \left(\operad C_\pl^{(i+1)}/ \operad C_\pl^{(i)}\right) \otimes \mathbb S$ has the form $d_\pl \otimes \id_{\mathbb S}$.
    \end{enumerate}
\end{definition}

\begin{definition}[Quasi-planar cooperad]
A \textit{quasi-planar} conilpotent curved cooperad $\C$ amounts to the data of a pseudo-planar curved conilpotent cooperad $(\C_\pl, \C, \varphi_\C)$ which is the colimit of a quasi-planar cooperad ladder.
\end{definition}

\begin{remark}
The main ideas for this definition where already in \cite[Definition 11.2]{linearcoalgebras}. 
\end{remark}

\begin{remark}
In a future work, we will show describe the relation between the notion of a quasi-planar cooperad and the notion of a \textit{higher cooperad} introduced in \cite{BrunoMalte}.
\end{remark}

\begin{proposition}\label{prop: quasi-planaire implique S-projectif dans le cas dg}
Let $\C$ be a quasi-planar conilpotent dg cooperad. The underlying dg $\mathbb{S}$-module of $\C$ is projective.
\end{proposition}

\begin{proof}
Let us prove this by induction on the quasi-planar ladder that defines $\C$. The dg $\mathbb{S}$-module $\C^{(0)}$ is free as a dg $\mathbb{S}$-module, therefore it is projective. 

\medskip

Let us write $\C^{(i)}$ as the direct sum $\C^{(i-1)} \oplus Y$ in the category of graded $\mathbb{S}$-modules. Then $\C^{(i)}$ is obtained by attaching $Y$ to $\C^{(i-1)}$ via a pushout of the form
\[
\begin{tikzcd}[column sep=3.5pc,row sep=3.5pc]
       (S^{-1} \otimes Y) \otimes \mathbb S
        \ar[r] \ar[d]  \arrow[dr, phantom, "\ulcorner", very near end]
        & \operad C^{(i-1)}
        \ar[d]
        \\
       (D^{0} \otimes Y) \otimes \mathbb S
        \ar[r]
        & \operad C^{(i)}~,
\end{tikzcd}
\]
which implies that $\C^{(i)}$ is again cofibrant in the category of dg $\mathbb{S}$-modules endowed with the projective model structure. We conclude by the fact that the colimit of cofibrant objects is again a cofibrant.
\end{proof}

\begin{remark}
When $\C$ is a quasi-planar conilpotent curved cooperad, the pre-differential of its underlying pdg $\mathbb{S}$-module can still be constructed by attaching cells at each step. 
\end{remark}

\medskip

Quasi-planar conilpotent curved cooperads are sent to cofibrant dg operads in the semi-model category structure of \cite{Fresse} by the operadic cobar construction of Subsection \ref{subsection: operadic bar-cobar}.

\begin{proposition}\label{proposition:cooperadqpcofibrant}
Let $\operad C$ be a quasi-planar conilpotent curved cooperad. Then the dg operad $\Omega\operad C$ is cofibrant.
\end{proposition}

\begin{proof}
We can notice that the composition of the following functors

\[
\begin{tikzcd}[column sep=2.5pc,row sep=0pc]
\alpha \arrow[r,"\operad C"]
&\curvcooperads^{\mathsf{conil}} \arrow[r,"\Omega"]
&\dgoperads
\end{tikzcd}
\]
\vspace{0.1pc}

satisfies the conditions of Proposition \ref{proposition:operadqpcofibrant}, which allows us to conclude.
\end{proof}


\subsection{The bar-Barratt-Eccles construction is quasi-planar}

Let $\operad P$ be a dg operad. Our goal is to prove that the curved conilpotent cooperad $\mathrm{B}(\operad E \otimes \operad P)$ has the canonical structure of a quasi-planar cooperad.

For all arities $m \geq 0$, remember that there is a canonical isomorphism of graded $\mathbb S$-modules

\[
\begin{tikzcd}[column sep=2pc,row sep=0pc]
\operad E(m) \otimes \operad P(m)  \cong (\operad E_\pl(m) \otimes \mathbb S_m) \otimes \operad P(m) \arrow[r]
& (\operad E_\pl(m) \otimes \operad P(m)) \otimes \mathbb S_m \\
 (e \otimes \{\sigma\}) \otimes p  \arrow[r,mapsto]
&\displaystyle (e \otimes p^{\sigma^{-1}}) \otimes \{\sigma\}~,
\end{tikzcd}
\]
\vspace{0.1pc}

    where the right action of $\mathbb S_m$ on $\operad E(m) \otimes \operad P(m)$ is diagonal, given by $(e \otimes p)^{\sigma} = e^\sigma \otimes p^\sigma$, and where $(\operad E_\pl(m) \otimes \operad P(m)) \otimes \mathbb S_m$ is the free right $\mathbb S_m$-module on $\operad E_\pl(m) \otimes \operad P(m)$. This isomorphism is described in Proposition \ref{prop: tensor with a quasi-free}. Notice that it commutes with $\id_{\operad E} \otimes d_{\operad P}$.

    \medskip

    Let us denote $\mathrm{B}(\operad E \otimes \operad P)_\pl$ the graded planar conilpotent cooperad
    \[
\mathrm{B}(\operad E \otimes \operad P)_\pl = \treemod_\pl (s \operad E_\pl \otimes \operad P \oplus s^2 \operad I).
    \]
    One has a canonical isomorphism of graded conilpotent cooperad 
    \[
\mathrm{B}(\operad E \otimes \operad P) = \treemod (s \operad E \otimes \operad P \oplus s^2 \operad I) \cong \treemod \left((s \operad E_\pl \otimes \operad P \oplus s^2 \operad I) \otimes \mathbb S \right) \cong
\left(\treemod_\pl(s \operad E_\pl \otimes \operad P \oplus s^2 \operad I) \otimes \mathbb S \right)
= \mathrm{B}(\operad E \otimes \operad P)_\pl \otimes \mathbb S.
    \]

\medskip

     One can to filter the graded $\mathbb N$-module $s(\operad E_\pl \otimes \operad P)$ by the degree of the elements of $\operad E_\pl$:
    \[
    (\tilde F_n s(\operad E_\pl \otimes \operad P))(m) \coloneqq
    s(\operad E_{\pl, \leq n-1}(m)\otimes \operad P(m))~, \quad n \in \mathbb N~,
    \]
    where $\operad E_{\pl,\leq n}(m)$ is the graded $\mathbb S$-module that consists in elements of $\operad E_\pl(m)$ whose degree is at most $n$
    $$
    \operad E_{\pl,\leq n}(m)_k = 
    \begin{cases}
        \operad E_\pl(m)_k \text{ if }k \leq n;
        \\
        0 \text{ otherwise.}
    \end{cases}
    $$
    One can notice that
    \[
    (\tilde F_0 s(\operad E_\pl \otimes \operad P))(m) = 
    0~,
    \]
    Then, one filters $s (\operad E_\pl \otimes \operad P) \oplus s^2\operad I$ as
    \begin{align*}
        &\tilde F_0 (s (\operad E_\pl \otimes \operad P) \oplus s^2\operad I) = 0
        \\
        &\tilde F_n (s (\operad E_\pl \otimes \operad P) \oplus s^2\operad I)
        = \tilde F_n (s (\operad E_\pl \otimes \operad P)
        \oplus s^2\operad I, \quad n \geq 1.
    \end{align*}
    Thus, one can also filter $\mathrm{B}(\operad E \otimes \operad P)_\pl = \treemod_\pl (s (\operad E_\pl \otimes \operad P) \oplus s^2\operad I)$ by
    $$
    \tilde F_n \mathrm{B}(\operad E \otimes \operad P)_\pl(m) = \bigoplus_t   \sum_{i_1+ \ldots + i_k = n}
    \bigotimes_{j=1}^k ~ \tilde F_{i_j} (s (\operad E_\pl \otimes \operad P)
        \oplus s^2\operad I) (l_j) ~,
    $$
    where the first sum is taken over the planar trees $t$ with $m$-leaves ; it has $k$ nodes whose arities are $l_1, \ldots, l_k$.

\medskip

    Finally, this induces a filtration $\tilde F_n \mathrm{B}(\operad E \otimes \operad P) \simeq (\tilde F_n \mathrm{B}(\operad E \otimes \operad P)_\pl) \otimes \mathbb S$ on the underlying graded conilpotent cooperad of $\mathrm{B}(\operad E \otimes \operad P)$.
    
\begin{remark}
Let us point out that the same filtration was used in \cite[Section 2]{eric2} in order to develop the obstruction theory of morphisms of algebras over a dg operad in positive characteristic.
\end{remark}

\begin{proposition}\label{prop: B(E otimes P) quasi-planar avec la bonne filtration}
For every integer $n \geq 0$, $\tilde F_n \mathrm{B}(\operad E \otimes \operad P)$ is a conilpotent curved sub-cooperad of $\mathrm{B}(\operad E \otimes \operad P)$. Moreover,
    the diagrams
    \[
    \tilde F_0 \mathrm{B}(\operad E \otimes \operad P)_\pl \longrightarrow \tilde F_1 \mathrm{B}(\operad E \otimes \operad P)_\pl \longrightarrow \cdots \longrightarrow \tilde F_n \mathrm{B}(\operad E \otimes \operad P)_\pl
    \longrightarrow \cdots
    \]
    \[
    \tilde F_0 \mathrm{B}(\operad E \otimes \operad P) \longrightarrow \tilde F_1 \mathrm{B}(\operad E \otimes \operad P) \longrightarrow \cdots \longrightarrow \tilde F_n \mathrm{B}(\operad E \otimes \operad P)
    \longrightarrow \cdots
    \]
    form a quasi-planar $\omega$-ladder whose colimit is given by the pseudo-planar conilpotent curved cooperad $(\mathrm{B}(\operad E \otimes \operad P)_\pl,\mathrm{B}(\operad E \otimes \operad P))$.
\end{proposition}

\begin{proof}
The only point that is not clear is the fact that the non-planar component of the pre-differential
of $\tilde F_{n+1} \mathrm{B}(\operad E \otimes \operad P)$ actually targets $\tilde F_n \mathrm{B}(\operad E \otimes \operad P)$. This follows from the fact that this non-planar component is given by applying the non-planar part of the differential of $\operad E$.
\end{proof}

\begin{remark}
A first version of this proposition was proven in \cite[Proposition 11.31]{linearcoalgebras}.
\end{remark}

\begin{example}
Let $f: \operad P \longrightarrow \operad Q$ be a morphism of dg operads. The morphism of conilpotent curved cooperads $\mathrm{B}(f \otimes \operad E): \mathrm{B}(\operad E \otimes \operad P) \longrightarrow \mathrm{B}(\operad E \otimes \operad Q)$ is pseudo-planar in the sense that it restricts to a map
\[
\mathrm{B}(f \otimes \operad E)_\pl: \mathrm{B}(\operad E \otimes \operad P)_\pl \longrightarrow \mathrm{B}(\operad E \otimes \operad Q)_\pl~,
\]
in a compatible way with the pseudo-planar isomorphisms.
\end{example}

Over a characteristic zero field, dg operads admit a model category where weak-equivalences are given by arity-wise quasi-isomorphisms and where fibrations are given by arity-wise degree-wise epimorphisms. One can then consider the transferred model category structure, along the operadic bar-cobar adjunction, on conilpotent curved cooperads constructed in \cite{grignou2021algebraic}. Every conilpotent curved cooperad is then shown to be weak-equivalent to a quasi-planar conilpotent curved cooperad.

\begin{proposition}[{\cite[Lemma 11.32]{linearcoalgebras}}]\label{prop: retour en caractéristique zero}
    Let $\kk$ be a field of characteristic zero and let $\operad C$ be a conilpotent curved cooperad. There exists a quasi-planar conilpotent curved cooperad $\operad C'$ and an acyclic cofibration $\operad C \qi \operad C'$. 
\end{proposition}

\begin{proof}
The canonical quasi-morphism $\mathcal{E} \otimes \Omega \C  \qi \Omega \C$ admits a section, since both dg operads are cofibrant. Therefore it suffices to consider the following composition
\[
\C \rightarrowtail \mathrm{B}\Omega \C \rightarrowtail \mathrm{B}(\mathcal{E} \otimes \Omega \C)~,
\]
which is an acylic cofibration since it is the composition of two acyclic cofibrations.
\end{proof}

\begin{remark}
Proposition \ref{prop: retour en caractéristique zero} will imply that when $\kk$ is a field of characteristic zero, one does recover the previously known homotopical operadic calculus of \cite{unitalalgebras} and \cite{linearcoalgebras}.
\end{remark}

\begin{corollary}\label{cor: cofibrant implies E-split}
    Let $\operad C$ be a conilpotent curved cooperad. The following assertions
    are equivalent
    \medskip
    
    \begin{enumerate}
        \item the operad $\Omega \operad C$ is cofibrant;

    \medskip
    
        \item the morphism of dg operads $\Omega \operad C \otimes \operad E \longrightarrow \Omega \operad C$
        admits a section.
    \end{enumerate}
\end{corollary}

\begin{proof}
    If $\Omega \operad C$ is cofibrant, it is clear that the map
    $\Omega \operad C \otimes \operad E \longrightarrow \Omega \operad C$
    admits a section. Conversely, let us suppose that this map admits a section. Then the projection $\Omega \mathrm{B}((\Omega \operad C) \otimes \operad E) \longrightarrow \Omega \operad C$ also admits a section given by
    $$
\Omega \operad C \longrightarrow \Omega \mathrm{B} \Omega \operad C \longrightarrow \Omega \mathrm{B} (\mathcal{E} \otimes \Omega \C)~.
    $$
    Therefore $\Omega \operad C$ is a retract of $\Omega \mathrm{B} (\mathcal{E} \otimes \Omega \C)$. Since $\Omega \mathrm{B} (\mathcal{E} \otimes \Omega \C)$ is cofibrant (by Proposition \ref{proposition:cooperadqpcofibrant} combined with Proposition \ref{prop: B(E otimes P) quasi-planar avec la bonne filtration}), then its retract $\Omega \operad C$ is also cofibrant.
\end{proof}

\begin{remark}
In fact, for any $\operad C$ be a quasi-planar conilpotent curved cooperad, the dg operad $\Omega \C$ is not only $\mathcal{E}$-split, but there exists explicit canonical map $\Omega\C \longrightarrow \mathcal{E} \otimes \Omega \C$ which endows $\Omega \C$ with a $\mathcal{E}$-comodule structure. This is done in the next Subsection \ref{subsection: E-comodule}.
\end{remark}


\subsection{The universal structure of a Barratt-Eccles comodule}\label{subsection: E-comodule}
In this subsection, we equip any dg operad of the form $\Omega\C$ with a $\mathcal{E}$-comodule structure, where $\C$ is again a quasi-planar conilpotent curved cooperad and where $\mathcal{E}$ denotes the Barratt-Eccles operad. This comonoid structure can be viewed as the positive characteristic analogue of the canonical $\mathrm{u}\mathcal{C}\mathrm{om}$-comonoid structure that exists for any dg operad. By that, we mean that it produces universal convolution (curved absolute) partition $\mathcal{L}_\infty$-algebra structures on the hom-graded modules between types of coalgebras and types of algebras. For more details, we refer to the forthcoming paper \cite{mappingcoalgebrascharp}.

\medskip

Recall that the restriction of the differential on $\Omega \C$ to the generators $s^{-1}\overline{\C}_\pl$ decomposes into three maps
\begin{itemize}
    \item $d_{s^{-1}\overline{\C}}: s^{-1}\overline{\C}_\pl \longrightarrow s^{-1}\overline{\C}_\pl \otimes \mathbb S \cong s^{-1}\overline{\C}$, which is given by $s^{-1}d_{\overline{\C}}$;
    \item $d_\Delta: s^{-1} \overline{\C}_\pl \longrightarrow \treemod_\pl^{(2)} s^{-1} \overline{\C}_\pl$, induced by the (planar) partial decompositions;
    \item $d_\theta: s^{-1}\overline{\C}_\pl \longrightarrow \operad I$, induced by the curvature of $\C$.
\end{itemize}

For $p \geq 1, q \geq 0, 1 \leq i \leq p$, let us denote 
\[
\Delta^{p,q}_i: s^{-1} \overline{\C}_\pl (p+q-1) \longrightarrow
 s^{-1} \overline{\C}_\pl (p) \otimes s^{-1} \overline{\C}_\pl (q)
\]

the $i$-th (planar) partial decomposition map of the cooperad $\C_\pl$ into arity $p$ and arity $q$ elements. It is the composition of the map $d_\Delta$ at arity $p+q-1$ with the projection onto planar trees $t$ with two nodes
\begin{itemize}
    \item the root node with $p$ leaves;
    \item the second node with $q$ leaves and that is plugged to the root node at its ith leaf.
\end{itemize}

Let $M$ be a pdg $\mathbb S$-module that is \textit{quasi-free}, that is, such that there exists a graded $\mathbb N$-module $N$ and an isomorphism of graded $\mathbb S$-module $M \cong N \otimes \mathbb S$. We are interested in how the pre-differential interacts with the action of the symmetric groups. Note for any given $n \geq 0$, there is an isomorphism

\[
M(n) \cong \bigoplus_{\sigma \in \mathbb{S}_n} N(n) \otimes \{\sigma\}~,
\]

of graded modules.

\begin{definition}[$\sigma$ pre-differential]
Let $M$ be a pdg $\mathbb S$-module that is \textit{quasi-free}, that is, such that there exists a graded $\mathbb N$-module $N$ and an isomorphism of graded $\mathbb S$-module $M \cong N \otimes \mathbb S$. Let $\sigma$ be in $\mathbb S_n$, the $\sigma$ \textit{pre-differential}, denoted by $d_\sigma$, is the degree $-1$ endomorphism of $N(n)$ given by the composite

\[
N(n) \simeq N(n) \otimes \{\id\} \hookrightarrow M(n)  \xrightarrow{d_M} M(n) \twoheadrightarrow N(n) \otimes \{\sigma \} \simeq N(n)~.
\]
\vspace{0.1pc}

We denote by $D_\sigma$ the $\mathbb S_n$-equivariant degree $-1$ endomorphism of $M(n)$ whose restriction to the generators $N$ is the composition

\[
N(n) \otimes \{\id\} \hookrightarrow M(n)  \xrightarrow{d_M} M(n) \twoheadrightarrow N(n) \otimes \{\sigma \} \hookrightarrow M(n)~.
\]
\vspace{0.1pc}

For any $x$ in $N(n)$, we have $D_\sigma(X \otimes \{\id\}) = d_\sigma (x) \otimes \{\sigma \}$.
\end{definition}

Similarly, for a sequence of permutations $\underline \sigma \coloneqq (\sigma_1, \ldots, \sigma_k) \in \mathbb S_n^k$,
we set

\[
d_{\underline{\sigma}} \coloneqq d_{\sigma_1} \cdots d_{\sigma_k} \quad \text{and} \quad D_{\underline{\sigma}} \coloneqq D_{\sigma_1} \cdots D_{\sigma_k}~.
\]
\vspace{0.1pc}

Notice that, for any $x$ in $N(n)$, we have
\[
D_\emptyset(x \otimes \{\id\}) = x \otimes \{\id\}~; \quad
D_{(\sigma_1, \ldots, \sigma_k)}(x \otimes \{\id\}) = 
d_{(\sigma_1, \ldots, \sigma_k)}(x) \otimes \{\sigma_1 \cdots\sigma_k\}~.
\]
\vspace{0.1pc}

The case of interest is when $M$ is the underlying pdg $\mathbb S$-module of the quasi-planar curved cooperad $\C$ and $N$ is the graded $\mathbb{N}$-module $\C_\pl$ of planar generators. 

\begin{lemma}\label{lemma: finitude des d sigma}
Let $n \geq 0$ and let $\C$ be a quasi-planar conilpotent curved cooperad. For every planar generator $x \in \C_\pl(n)$, there exists a natural integer $k_x$ such that for every sequence of permutations $(\sigma_1, \ldots, \sigma_k) \in \mathbb S_n^k$ of length $k \geq k_x$, we have 
    $$
    d_{(\sigma_1, \ldots, \sigma_k)}(x) = 0~.
    $$
The same result holds, \textit{mutatis mutandis}, for $s^{-1} \overline{\C}$ equipped with the pre-differential $d_{s^{-1}\overline{\C}} = s^{-1}d_{\overline{\C}}$.
\end{lemma}

\begin{proof}
Since the cooperad $\C$ is quasi-planar, it is the colimit of a quasi-planar ladder $\{\C^{(i)}\}_{i \in \alpha}$. Let us prove by an ordinal induction that $\C^{(i)}$ satisfies the property of the lemma for every $i<\alpha+1$.

\medskip

The conilpotent curved cooperad $\C^{(0)}$ is planar, therefore $d_\sigma$ is zero whenever the permutation $\sigma$ is non-trivial. So we only need to worry about the planar pre-differential $D_{\id}$. Now notice that by definition $\C^{(0)}$ admits no non-trivial decomposition, therefore its pre-differential squares to zero. This proves the property for $\C^{(0)}$.
    
\medskip

Now, let us suppose that $\C^{(j)}$ satisfies the property for every $j<i$. If $i$ is a limit ordinal, $\C^{(i)}$ also satisfies the property. Otherwise, $i$ has the form $i = l+1$. One can notice that on $\gr_i\C$, the maps $D_\sigma$ is zero whenever the permutation $\sigma$ is non-trivial and that $d^2 = D_{\id}^2=0$. So, for every $x \in \C^{(i)}(n)$, and for every two permutations $\sigma_1, \sigma_2$, $D_{(\sigma_1, \sigma_2)}(x)$ is in $\C^{(l)}(n)$. This proves the property for $\C^{(i)}$, and allows us to conclude. Finally, the result for $s^{-1} \overline{\C}$ is a direct consequence of the previous one.
\end{proof}

\medskip

\textbf{The comodule map.} Let us consider the morphism of graded $\mathbb{N}$-modules 

\[
\begin{tikzcd}[column sep=2pc,row sep=0pc]
\Delta_{\operad E, \C}: s^{-1}\overline{\C}_\pl  \arrow[r]
&\mathcal{E}_\pl \otimes s^{-1}\overline{\C} \\
s^{-1}x \otimes \{\id\}  \arrow[r,mapsto]
&\displaystyle \sum_{k \geq 0} \sum_{\underline\sigma \in \mathbb S_n^k} \rho(\underline\sigma) \otimes D_{\underline\sigma}(s^{-1}x \otimes \{\id\})~,
\end{tikzcd}
\]

where $s^{-1} x \in s^{-1} \overline{\C}_\pl (n)$ is a planar generator and where the sum is taken over all sequences of permutations $\underline\sigma = (\sigma_1, \ldots, \sigma_k) \in \mathbb S_n^k$ for all $k \geq 0$. The formula is well-defined by Lemma \ref{lemma: finitude des d sigma}. Moreover, notice that $\Delta_{\operad E, \C}$ is the restriction of the morphism of graded operads
$$
\Omega \C = \treemod(s^{-1}\overline{\C})
\xrightarrow{\Delta_{\operad E, \C}}
\treemod(\operad E \otimes s^{-1}\overline{\C})
\longrightarrow \treemod(\operad E) \otimes \treemod(s^{-1}\overline{\C})
\longrightarrow \operad E \otimes \Omega \C~,
$$
which we also denote by $\Delta_{\operad E, \C}$.

\begin{theorem}\label{thm:becomodule}
The map

\[
\begin{tikzcd}[column sep=2pc,row sep=0pc]
\Delta_{\operad E, \C}: s^{-1}\overline{\C}_\pl  \arrow[r]
&\mathcal{E}_\pl \otimes s^{-1}\overline{\C} \\
s^{-1}x \otimes \{\id\}  \arrow[r,mapsto]
&\displaystyle  \sum_{k \geq 0} \sum_{\underline\sigma \in \mathbb S_n^k}  \rho(\underline\sigma) \otimes D_{\underline\sigma}(s^{-1}x \otimes \{\id\})~,
\end{tikzcd}
\]

where $s^{-1} x \in s^{-1} \overline{\C}_\pl (n)$ is a planar generator and where the sum is taken over all sequences of permutations $\underline\sigma = (\sigma_1, \ldots, \sigma_k) \in \mathbb S_n^k$ for all $k \geq 0$, induces a morphism of dg operads 
\[
\Omega \operad C \longrightarrow \operad E \otimes \Omega \operad C
\]
which endows $\Omega \operad C$ with a left $\operad E$-comodule structure.
\end{theorem}

\begin{proof}
By definition, it induces a morphism of graded operads. By Lemma \ref{lemma:becomodulecommutesdiff0}, Lemma \ref{lemma:becomodulecommutesdiff1} and Lemma \ref{lemma:becomodulecommutesdiff2}, this induced morphism commutes with the differentials. Thus, it is a morphism of dg operads. By Lemma \ref{lemma becomodulecoass}, it defines a left $\operad E$-comodule structure.
\end{proof}

The rest of this subsection is devoted to proving Theorem \ref{thm:becomodule}.

\begin{lemma}\label{lemma:becomodulecommutesdiff0}
The diagram
    \[
\begin{tikzcd}[column sep=3pc,row sep=3pc]
        s^{-1}\overline{\operad C}_{\pl}
        \ar[r, "\Delta_{\operad E, \C}"] \ar[d, "d_\theta"']
        & \operad E\otimes s^{-1} \overline{\operad C}_{\pl}  \otimes \mathbb S
        \ar[d, "\id \otimes d_\theta"]
        \\
        \operad I
        \ar[r, "\Delta_{\operad E, \C}"']
        & \operad E \otimes \operad I \otimes \mathbb S
\end{tikzcd}
\]

is commutative, where $d_\theta$ is the term of the differential in $\Omega\C$ induced by the curvature $\theta$ of $\C$.
\end{lemma}

\begin{proof}
    The diagram trivially commutes at arity $n \neq 1$. In arity $1$, the commutation is a direct consequence of the fact that $\Delta_{\operad E, \C} (x \otimes \id) = (\id) \otimes x \otimes \{\id\}$ for any $x \in s^{-1} \overline{\C}_\pl(1)$.
\end{proof}

\begin{lemma}\label{lemma becooperadd}
For every natural integers, $n \geq 2, k \geq 1$, we have an equality of degree $-1$ maps from $s^{-1}\overline{\C}(n)$ to $\operad E (n)\otimes s^{-1}\overline{\C}(n)$:
    
    \[
    \sum_{\underline\sigma \in \mathbb S_n^k}
    d_{\operad E} \rho(\underline{\sigma}) \otimes D_{\underline\sigma}(-) =
    \sum_{\underline\sigma \in \mathbb S_n^k}
    \rho((\sigma_1, \ldots, \sigma_{k-1}))^{\sigma_k} \otimes D_{\underline\sigma}(-)
    + (-1)^k \rho((\sigma_2, \ldots, \sigma_{k})) \otimes D_{\underline\sigma}(-)~.
    \]
\end{lemma}

\medskip

\begin{proof}
    We have
    
    $$
    d_{\operad E} (\rho(\underline\sigma))
    = \rho((\sigma_1, \ldots, \sigma_{k-1}))^{\sigma_k}
    + \sum_{0 <i<k} (-1)^{k-i} \delta_i(\rho(\underline\sigma))
    + (-1)^k  \rho((\sigma_2, \ldots, \sigma_k))~,
    $$
    where
    $$
    \delta_i(\rho(\underline\sigma))= \rho(\sigma_1 \ldots, \sigma_{i}\sigma_{i+1}, \ldots, \sigma_k)~.
    $$
    This gives us
    \begin{align*}
        \sum_{\underline\sigma \in \mathbb S_n^k}
    d_{\operad E} \rho(\overline{\sigma}) \otimes D_{\underline\sigma}(-)
    =& \sum_{\underline\sigma \in \mathbb S_n^k}
    \Big( \rho((\sigma_1, \ldots, \sigma_{k-1}))^{\sigma_k} \otimes D_{\underline\sigma}(-)
    \\
    &+ \sum_{0<i<k} (-1)^{k-i}\sum_{\underline\sigma} \delta_i(\rho(\underline\sigma)) \otimes D_{\underline\sigma}(-)\Big)
    \\
    &+ (-1)^k \sum_{\underline\sigma \in \mathbb S_n^k}\rho((\sigma_2, \ldots, \sigma_{k})) \otimes D_{\underline\sigma}(-)~.
    \end{align*}

Since the underlying graded conilpotent cooperad of $\C$ is planar, the curvature equation tells us that $d_{s^{-1}\overline{\C}}^2$ is planar and that, for every permutation $\sigma \in \mathbb S_n$,
$$
\sum_{\mu \xi = \sigma} D_{\mu}D_\xi
=
\begin{cases}
    d_{s^{-1}\overline{\C}}^2 \text{ if }\sigma=\id ~,
    \\
    0 \text{ otherwise.}
\end{cases}
$$
Noticing that $\delta_i(\rho(\underline\sigma))$ depends only on the product $\sigma_i \sigma_{i+1}$ and not on the particular values of $\sigma_i$ and $\sigma_{i+1}$, we have
$$
\sum_{\underline\sigma \in \mathbb S_n^k} \delta_i(\rho(\underline\sigma)) \otimes D_{\underline\sigma}(-)=
\sum_{\underline\mu=(\mu_1, \ldots, \mu_{k-1}) \in \mathbb S_n^{k-1}} \sum_{\sigma\sigma' = \mu_i}
\rho(\underline\mu) \otimes D_{(\mu_1, \ldots, \sigma\sigma', \ldots, \mu_{k-1})}(-)~.
$$
If $\mu_i=\id$ then $\rho(\underline\mu)=0$ and if $\mu_i \neq \id$ then $D_{(\mu_1, \ldots, \sigma\sigma', \ldots, \mu_{k-1})}=0$. Thus 
$$
\sum_{\underline\sigma \in \mathbb S_n^k} \delta_i(\rho(\underline\sigma)) \otimes D_{\underline\sigma}(-)=0~,
$$
which concludes the proof.
\end{proof}

\begin{lemma}\label{lemma:becomodulecommutesdiff1}
    The following diagram

\[
\begin{tikzcd}[column sep=3pc,row sep=3pc]
        s^{-1}\overline{\operad C}_{\pl}
        \ar[r, "\Delta_{\operad E, \C}"] \ar[d, "d_{s^{-1}\overline{\C}}"']
        & \operad E \otimes s^{-1} \overline{\operad C}_{\pl}  \otimes \mathbb S
        \ar[d, "d_{\operad E} \otimes \id + \id \otimes d_{s^{-1}\overline{\C}}"]
        \\
        s^{-1}\overline{\operad C}_{\pl}\otimes \mathbb S
        \ar[r, "\Delta_{\operad E, \C}"']
        & \operad E \otimes s^{-1} \overline{\C}_\pl  \otimes \mathbb S
\end{tikzcd}
\]

is commutative.
\end{lemma}

\begin{proof}
In arity $0$ and $1$, the commutation is straightforward to check. In arity $n \geq 2$, we have $d_{s^{-1}\overline{\C}} = \sum_{\sigma \in \mathbb{S}_n} D_\sigma$. This gives the following equalities of maps on $s^{-1} \overline{\C}(n)$

\begin{align*}
    \Delta_{\operad E, \C} d_{s^{-1}\overline{\C}}(-)  &= \sum_{\underline\sigma \in \mathbb S_n^k,~ \sigma \in \mathbb S_n} \rho(\underline\sigma)^{\sigma} \otimes D_{\underline\sigma}D_\sigma(-)~,
    \\
    (\id \otimes d_{s^{-1}\overline{\C}}) \Delta_{\operad E, \C}(-) &= \sum_{k \geq 0} \sum_{\underline\sigma \in \mathbb S_n^k, ~ \sigma \in \mathbb S_n} (-1)^k \rho(\underline\sigma) \otimes D_{\sigma} D_{\underline \sigma}(-)~.
\end{align*}
It follows from Lemma \ref{lemma becooperadd} that 
$$
\Delta_{\operad E, \C} d_{s^{-1} \overline{\C}}(-) = (d_{\operad E} \otimes \id) \Delta_{\operad E, \C}(-) + (\id \otimes d_{s^{-1} \overline{\C}}) \Delta_{\operad E, \C}(-)~,
$$
and thus that the square also commutes in arity $n \geq 2$.
\end{proof}

\begin{lemma}\label{lemma betransferfirst}
    Let $p,q \geq 1$, $1 \leq i \leq p$,$n= p+q-1$ and let $\sigma \in \mathbb S_n$ be a non-trivial permutation. One has the following egalites between maps from $s^{-1}\overline{\C}_\pl)(n)$ to $(\treemod_\pl^{(2)} (s^{-1}\overline{\C}_\pl)) (n)\otimes \mathbb S_n$ depending on the $(p,q,i)$-admissibility of $\sigma$.
    \begin{enumerate}
        \item If $\sigma$ is $(p,q,i)$-top admissible, that is, if there exists a unique non-trivial permutation $\psi \in \mathbb S_q$ such that $\sigma = \id_p \comp_i \psi$, then
        $$
        \Delta_i^{p,q} D_\sigma =- (\id \otimes D_\psi) \Delta_i^{p,q}~.
        $$
        \item If $\sigma$ is $(p,q,i)$-bottom admissible, that is, if there exists a unique non-trivial permutation $\mu \in \mathbb S_p$ such that $\sigma = \mu \comp_{\mu^{-1}(i)} \id_q$, then
        $$
        \Delta_{i}^{p,q} D_\sigma = -  (D_\mu \otimes \id) \Delta_{\mu^{-1}(i)}^{p,q}~.
        $$
        \item If $\sigma$ is $(p,q,i)$-non admissible, then
        $$
        \Delta_i^{p,q} D_\sigma = 0~.
        $$
    \end{enumerate}
\end{lemma}

\begin{proof}
Since the derivation of $\Omega \C$ squares to zero, one has the following equation
\[
d_\Delta d_{s^{-1}\overline{\C}} + (d_{s^{-1}\overline{\C}} \otimes \id + \id \otimes d_{s^{-1}\overline{\C}}) d_\Delta = 0
\]
between maps from $s^{-1}\overline{\C}_\pl$ to $(\treemod^{(2)} (s^{-1}\overline{\C}))\simeq (\treemod_\pl^{(2)} (s^{-1}\overline{\C}_\pl)) \otimes \mathbb S$. In arity $n$, we denote
$\pi_{p,q,i,\sigma}$ the projection
$$
(\treemod_\pl^{(2)} (s^{-1}\overline{\C}_\pl)) (n)\otimes \mathbb S_n \twoheadrightarrow
\treemod^{(2)}_{\pl, p,q,i} (s^{-1} \overline{\C}_\pl) \otimes \{\sigma\}~,
$$
where $\treemod^{(2)}_{\pl, p,q,i} (s^{-1} \overline{\C}_\pl)$ denotes the planar trees labelled by $s^{-1} \overline{\C}_\pl$ that have two nodes: the root node with $p$ input and another node with $q$ inputs, where the top node is attached to the $i^{th}$ input of the root node.
One has
$$
\pi_{p,q,i,\sigma} d_\Delta d_{s^{-1}\overline{\C}} = \Delta_i^{p,q} D_\sigma.
$$
Moreover,
$$
\pi_{p,q,i,\sigma} (d_{s^{-1}\overline{\C}} \otimes \id)d_\Delta
=
\begin{cases}
    (D_\mu \otimes \id) \Delta_{\mu^{-1}(i)}^{p,q} \text{ if $\sigma$ is bottom admissible;}
    \\
    0 \text{ otherwise.}
\end{cases}
$$
and
$$
\pi_{p,q,i,\sigma} (\id \otimes d_{s^{-1}\overline{\C}}) d_\Delta
=
\begin{cases}
    (\id \otimes D_\psi) \Delta_{i}^{p,q} \text{ if $\sigma$ is top admissible;}
    \\
    0 \text{ otherwise.}
\end{cases}
$$
Then, the three equations of the proposition are just the three possible cases (in terms of $(p,q,i)$-admissibility of $\sigma$) of the equation
$$
\pi_{p,q,i,\sigma}(d_\Delta d_{s^{-1}\overline{\C}} + (d_{s^{-1}\overline{\C}} \otimes \id + \id \otimes d_{s^{-1}\overline{\C}}) d_\Delta )= 0~.
$$
\end{proof}

\begin{lemma}\label{lemma betransferzero}
Let us denote by $u$ the generator element of $u\operad A ss(0) = \operad E(0)$. 
    Let $p \geq 1$, let $\sigma \in \mathbb S_p$ be a non-trivial permutation and let $1 \leq i \leq p+1$. Then,
    
    $$
    \sum_{\mu~ \comp_{\mu^{-1}(i)} u = \sigma} (D_\mu \otimes \id) \Delta_{\mu^{-1}(i)}^{p, 0}
    = - \Delta_{i}^{p,0} D_\sigma ~, 
    $$
    where the sum ranges over all $\mu$ in $\mathbb S_{p+1}$ such that $\mu~ \comp_{\mu^{-1}(i)} u = \sigma$.
\end{lemma}

\begin{proof}
    It follows from the same arguments as those used in the proof of Lemma \ref{lemma betransferfirst}.
\end{proof}

\begin{lemma}\label{propositionbeecomposition}
Let $k \geq 0$, $p \geq 1$, $q \geq 0$, $1 \leq i \leq p$, $n= p+q-1$ and let $\underline\sigma \in \mathbb S_n^k$ be a sequence of non-trivial permutations. Then

\[
\Delta_{i}^{p,q} D_{\underline\sigma}(-) = \sum_{ \underline \mu \comp_{i, \phi} \underline \psi = \underline \sigma}
 (-1)^k\epsilon(\phi) (D_{\underline \mu} \otimes D_{\underline \psi})\Delta_{\underline{\mu}^{-1}(i)}^{p,q}(-)~,
\]
where the sum ranges over all $\underline \mu \in \mathbb S_p^a$ and all $\underline \psi \in \mathbb S_q^b$, with $a \geq 1$, $b \geq 1$ and $a+b =k$, such that there exists a $(a,b)$-shuffle $\phi$ such that $\underline \mu \comp_{i, \phi} \underline \psi = \underline \sigma$, and where $\underline{\mu}^{-1}(i) = \mu_a^{-1} \cdots \mu_1^{-1}(i)$.
\end{lemma}

\begin{proof}
First, in the case where $q \geq 1$, a straightforward induction using Lemma \ref{lemma betransferfirst} shows that:
\begin{enumerate}
\item if $\underline\sigma$ is $(p,q,i)$-admissible, that is, if there exist an unique shuffle permutation $\phi \in \mathbb S_k$ (which is an $(a,b)$-shuffle) and an unique pair of sequences of permutations $\underline \mu \in \mathbb S_p^a, \underline \psi \in \mathbb S_q^b$ such that 
\[
\underline \sigma = \underline \mu \comp_{i, \phi} \underline \psi~,
\]
then
\[
\Delta_{i}^{p,q} D_{\underline\sigma}(-) =
 (-1)^k\epsilon(\phi) (D_{\underline \mu} \otimes D_{\underline \psi})\Delta_{j}^{p,q}(-)~,
\]
where $j= \mu_a^{-1} \cdots \mu_1^{-1}(i)$.

\medskip

\item If $\underline\sigma$ is not $(p,q,i)$-admissible, then 
\[
\Delta_{i}^{p,q} D_{\underline\sigma}(-) = 0~.
\]
\end{enumerate}

In the case where $q = 0$, the result follows from a similar induction using Lemma \ref{lemma betransferzero}.
\end{proof}

\begin{lemma}\label{lemma:becomodulecommutesdiff2}
    The following diagram 

\[
\begin{tikzcd}[column sep=3pc,row sep=3pc]
        s^{-1}\overline{\C}_{\pl}
        \ar[rr, "\Delta_{\operad E, \C}"] \ar[d, "d_\Delta"']
        && \operad E \otimes s^{-1} \overline{\C}_\pl \otimes \mathbb S
        \ar[d, "\id \otimes d_\Delta"]
        \\
        \treemod^{(2)}_{\pl}(s^{-1}\overline{\C}_{\pl})
        \ar[r, "\treemod^{(2)}_{\pl}(\Delta_{\operad E, \C})"]
        & \treemod^{(2)}_{\pl}(\operad E  \otimes s^{-1}\overline{\C})
        \ar[r, "\gamma"]
        & \operad E \otimes \treemod^{(2)}_{\pl}(s^{-1}\overline{\C}_{\pl}) \otimes \mathbb S.
\end{tikzcd}
\]

is commutative, where $\gamma$ is the composition within the operad $\operad E \otimes \Omega \operad C$.
\end{lemma}

\begin{proof}
First let $p\geq 1,q \geq 0$, $1 \leq i \leq p$ and $n = p+q-1$. As in the proof of Lemma \ref{lemma betransferfirst}, we denote $\treemod^{(2)}_{\pl, p,q,i} (s^{-1} \overline{\C}_\pl)$ the planar trees labelled by $s^{-1} \overline{\C}_\pl$ that have two nodes: the root node with $p$ inputs and another node with $q$ inputs, which is attached to the $i^{th}$ input of the root node.
Let us prove that, in arity $n$, the composition of the two maps 
$$
s^{-1} \overline{\C}_\pl (n) \rightrightarrows
\operad E(n) \otimes \treemod^{(2)}_{\pl}(s^{-1}\C_{\pl})(n) \otimes \mathbb S_n
$$
with the projection
$$
\id \otimes \pi_{p,q,i} \otimes \id:
\operad E(n) \otimes \treemod^{(2)}_{\pl}(s^{-1}\C_{\pl})(n) \otimes \mathbb S_n
\twoheadrightarrow
\operad E(n) \otimes \treemod^{(2)}_{\pl, p,q,i} (s^{-1} \overline{\C}_\pl) \otimes \mathbb S_n
$$
are equal. This amounts to show that 
$$
\sum_{\underline\mu \in \mathbb S_p^a, \underline\psi \in \mathbb S_q^b} 
(-1)^{|\mu||\psi|}
(\rho(\underline(\mu)) \comp_{\underline\mu^{-1}(i)} \rho(\underline(\psi)))
\otimes
((D_{\underline{\mu}} \otimes D_{\underline{\psi}}) \Delta^{p,q}_{\underline{\mu}^{-1}(i)})
= 
(\id \otimes \Delta^{p,q}_{i}) \left(\sum_{\underline\sigma \in \mathbb S_n^k} \rho(\underline\sigma)\otimes 
 D_{\underline\sigma}\right)~,
$$
which follows from the sequence of equalities
\begin{align*}
&\sum_{\underline\mu \in \mathbb S_p^a, \underline\psi \in \mathbb S_q^b} 
(-1)^{|\mu||\psi|}
(\rho(\underline(\mu)) \comp_{\underline\mu^{-1}(i)} \rho(\underline(\psi)))
\otimes
((D_{\underline{\mu}} \otimes D_{\underline{\psi}}) \Delta^{p,q}_{\underline{\mu}^{-1}(i)})
\\
=& \sum_{\underline\mu \in \mathbb S_p^a, \underline\psi \in \mathbb S_q^b, \phi \in \mathrm{Sh}(a,b)} (-1)^{\epsilon(\phi)} \rho(\underline\mu \comp_{i, \phi} \underline\psi) \otimes
((D_{\underline{\mu}} \otimes D_{\underline{\psi}}) \Delta^{p,q}_{\underline{\mu}^{-1}(i)}) \quad \text{ by Lemma \ref{lemma:inversecompinbe}}
\\
=&
\sum_{\underline\sigma \in \mathbb S_n^k}
 \sum_{\underline\mu \comp_{i, \phi} \underline\psi= \underline\sigma} (-1)^{\epsilon(\phi)} \rho(\underline\sigma) \otimes
((D_{\underline{\mu}} \otimes D_{\underline{\psi}}) \Delta^{p,q}_{\underline{\mu}^{-1}(i)})
\\
=&
\sum_{\underline\sigma \in \mathbb S_n^k} \rho(\underline\sigma)\otimes \left(
 \sum_{\underline\mu \comp_{i, \phi} \underline\psi= \underline\sigma} (-1)^{\epsilon(\phi)}  
((D_{\underline{\mu}} \otimes D_{\underline{\psi}}) \Delta^{p,q}_{\underline{\mu}^{-1}(i)})\right)
\\
=&
\sum_{\underline\sigma \in \mathbb S_n^k} \rho(\underline\sigma)\otimes \left(
  (-1)^{|\sigma|} 
( \Delta^{p,q}_{i} D_{\underline\sigma})\right)
\quad\text{ by Lemma \ref{propositionbeecomposition}}
\\
=&
(\id \otimes \Delta^{p,q}_{i}) (\sum_{\underline\sigma} \rho(\underline\sigma)\otimes 
 D_{\underline\sigma})~.
\end{align*}
This proves the lemma.
\end{proof}

\begin{lemma}\label{lemma becomodulecoass}
    The following diagram commutes
    \[
\begin{tikzcd}[column sep=3pc,row sep=3pc]
        s^{-1}\overline{\operad C}
        \ar[r, "\Delta_{\operad E, \C}"] \ar[d, "\Delta_{\operad E, \C}"']
        & \mathcal{E} \otimes s^{-1}\overline{\operad C}
        \ar[d, "\Delta_{\mathcal{E}} \otimes \id"]
        \\
        \mathcal{E} \otimes s^{-1}\overline{\operad C}
        \ar[r, "\id \otimes \Delta_{\operad E, \C}"]
        & \mathcal{E} \otimes \mathcal{E} \otimes s^{-1}\overline{\operad C}~,
\end{tikzcd}
\]
\end{lemma}

\begin{proof}
This is direct consequence of Lemma \ref{lemma:monoidalpropertyoftherhomap}
\end{proof}

\begin{corollary}\label{corollary: adjoint map is pseudo-planar}
    The canonical morphism of conilpotent curved cooperads
    \[
    \C \to \mathrm{B}(\operad E \otimes \Omega \C)
    \]
    that is adjoint to the morphism of dg operads $\Delta_{\operad E, \C}: \Omega \C \to \operad E \otimes \Omega \C$ is pseudo-planar in the sense that it restricts to a morphism of graded planar cooperads
    \[
\C_\pl \to \mathrm{B}(\operad E \otimes \Omega \C)_\pl~,
    \]
    which is compatible with the pseudo-planar isomorphisms.
\end{corollary}

\begin{proof}
    This is a direct consequence of the fact that the restriction of the map
    $\Delta_{\operad E, \C}: \Omega \C \to \operad E \otimes \Omega \C$
    to $s^{-1}\C_\pl$ factors through $(\operad E \otimes \Omega \C)_\pl = \operad E_\pl \otimes \Omega \C$.
\end{proof}


\subsection{The comodule structure on a bar-Barratt-Eccles construction}\label{subsection: the comodule structure on a bar-Barratt-Eccles construction}
Let $\operad P$ be a dg operad. The goal of this subsection is to compute the map
$$
\Delta_{\operad E, \mathrm{B}(\operad E \otimes \operad P)}:
s^{-1} \overline{\mathrm{B}}(\operad E \otimes \operad P)
\longrightarrow \operad E \otimes s^{-1} \overline{\mathrm{B}}(\operad E \otimes \operad P)~.
$$

The structure of an Hadamard comonoid on the Barratt-Eccles operad $\operad E$
yields a map of graded $\mathbb S$-modules
\[
s(\operad E \otimes \operad P)
\xrightarrow{s(\Delta_{\operad E } \otimes \id)}
s(\operad E \otimes \operad E \otimes \operad P)
\cong \operad E \otimes s(\operad E \otimes \operad P)
\]
Together with the map
\[
s^2 \operad I \cong \operad I \otimes s^2 \operad I \longrightarrow \operad E \otimes s^2 \operad I
\]
it yields a map
\[
s(\operad E \otimes \operad P \oplus s\operad I) \longrightarrow \operad E \otimes  s(\operad E \otimes \operad P \oplus s\operad I).
\]
Then, we get a composite map:
\[
\begin{tikzcd}
s^{-1} \overline{\mathrm{B}}(\operad E \otimes \operad P)
\ar[d, equal]
\\
s^{-1} \overline{\treemod} (s(\operad E \otimes \operad P \oplus s\operad I))
\ar[d]
\\
s^{-1} \overline{\treemod} (\operad E \otimes s(\operad E \otimes \operad P \oplus s\operad I))
\ar[d]
\\
s^{-1} \overline{\treemod} (\operad E) \otimes \overline{\treemod} (s(\operad E \otimes \operad P \oplus s\operad I))
\ar[d]
\\
s^{-1} \operad E \otimes \overline{\treemod} (s(\operad E \otimes \operad P \oplus s\operad I))
\ar[d, equal]
\\
\operad E \otimes s^{-1}\overline{\mathrm{B}}(\operad E \otimes \operad P).
\end{tikzcd}
\]
One can notice that the restriction of this map to $s^{-1} \overline{\treemod}_\pl (s(\operad E_\pl \otimes \operad P \oplus s\operad I))$ factors through $s^{-1} \operad E_\pl \otimes \overline{\treemod} (s(\operad E \otimes \operad P \oplus s\operad I))$ since the restriction of $\Delta_{\operad E}: \operad E \longrightarrow \operad E \otimes \operad E$ to $\operad E_\pl$ factors through $\operad E_\pl \otimes \operad E$ and since $\operad E_\pl$ is stable through the operadic compososition
of $\operad E$.

\begin{proposition}\label{proposition: deltae on bar}
    The map $\Delta_{\operad E, \mathrm{B}(\operad E \otimes \operad P)}$ is equal to the above composition.
\end{proposition}

\begin{proof}
    Let us denote $f_1 = \Delta_{\operad E, \mathrm{B}(\operad E \otimes \operad P)}$ and let $f_2$ be the above composition. One can check that $f_2$ makes the diagram of Lemma \ref{lemma:diagramrecurrenceecomoduleb} commute as well as $f_1$ (by Lemma \ref{lemma:becomodulecommutesdiff2}). Moreover the restriction of $f_1$ and $f_2$ are equal on $s^{-1}s(\operad E \otimes \operad P \oplus s\operad I)$ by Lemma \ref{lemma:initrecurrenceecomoduleb}. We conclude by Lemma \ref{lemma:diagramrecurrenceecomoduleb}.
\end{proof}

\begin{lemma}
Let $p \in \operad P(m)_k$, let $(\sigma_1, \ldots, \sigma_n) \in \mathbb S_m^n$ and let $(\mu_i, \ldots, \mu_n) \in \mathbb S_m^{n-i+1}$ be non trivial permutations. Then
the element
\[
D_{(\mu_i, \ldots, \mu_n)}  \left( \rho((\sigma_1, \ldots, \sigma_n))  \otimes p \right) \in (\operad E \otimes \operad P)(m)_{k-n-1+i}~,
\]
is equal to
    \[
    \begin{cases}
         \left( \rho((\sigma_1, \ldots, \sigma_{i-1})) \otimes p^{(\sigma_i \cdots \sigma_n)^{-1}}
        \right)^{\sigma_i \cdots \sigma_n} \text{ if } (\mu_i, \ldots, \mu_n) = (\sigma_i, \ldots, \sigma_n)~,
        \\
        0 \text{ otherwise}.
    \end{cases}
    \]
\end{lemma}

\begin{proof}
    This follows from a direct inspection.
\end{proof}

\begin{lemma}\label{lemma:initrecurrenceecomoduleb}
    The following square diagram of graded $\mathbb S$-modules commutes
    \[
    \begin{tikzcd}
        s^{-1} s(\operad E_\pl \otimes \operad P) \otimes \mathbb S
        \ar[dd, "\Delta_{\operad E, B(\operad E \otimes \operad P)}"']
        \ar[r, "\simeq"]
        & s^{-1}s (\operad E \otimes \operad P)
        \ar[d, "s^{-1}(\Delta_{\operad E} \otimes \id)"]
        \\
        & s^{-1} s(\operad E \otimes \operad E \otimes \operad P)
        \ar[d, "\simeq"]
        \\
        \operad E_\pl \otimes s^{-1}s (\operad E \otimes \operad P) \otimes \mathbb S
        \ar[r, "\simeq"']
        & \operad E \otimes s^{-1}s (\operad E \otimes \operad P).
    \end{tikzcd}
    \]
\end{lemma}

\begin{proof}
    The two maps from $s^{-1} s(\operad E_\pl \otimes \operad P) \otimes \mathbb S$ to $\operad E \otimes s^{-1}s (\operad E \otimes \operad P)$ are morphisms of graded $\mathbb S$-modules that send
    and element $s^{-1}s\rho ((\sigma_1, \ldots, \sigma_n)) \otimes p$ to
    \[
\sum_i \rho((\sigma_i, \ldots, \sigma_n)) \otimes s^{-1}s \rho ((\sigma_1, \ldots, \sigma_{i-1}))^{\sigma_i \cdots \sigma_n} \otimes p.
    \]
\end{proof}

\begin{lemma}\label{lemma:diagramrecurrenceecomoduleb}
Let us consider two maps
\[
f_1,f_2 : s^{-1}\overline{\mathrm{B}}(\operad E \otimes \operad P)_{\pl} \longrightarrow \operad E_\pl 
\otimes s^{-1}\overline{\mathrm{B}}(\operad E \otimes \operad P)~.
\]
Let us suppose that the following diagram commutes for $i\in \{1,2\}$
\[
\begin{tikzcd}[column sep=3pc,row sep=3pc]
        s^{-1}\overline{\mathrm{B}}(\operad E \otimes \operad P)_{\pl}
        \ar[rr, "f_i"] \ar[d, "d_\Delta"']
        && \operad E \otimes s^{-1} \overline{\C}_\pl \otimes \mathbb S
        \ar[d, "\id \otimes d_\Delta"]
        \\
        \treemod^{(2)}_{\pl}(s^{-1}\overline{\mathrm{B}}(\operad E \otimes \operad P)_{\pl})
        \ar[r, "\treemod^{(2)}_{\pl}(f_i)"]
        & \treemod^{(2)}_{\pl}(\operad E  \otimes s^{-1}\overline{\mathrm{B}}(\operad E \otimes \operad P))
        \ar[r, "\gamma"]
        & \operad E \otimes \treemod^{(2)}_{\pl}(s^{-1}\overline{\mathrm{B}}(\operad E \otimes \operad P)_\pl) \otimes \mathbb S.
\end{tikzcd}
\]
Then the two assertions are equivalent
\begin{enumerate}
    \item $f_1 = f_2$;
    \item the restriction to $s^{-1}s(\operad E \otimes \operad P)_{\pl}$ of $f_1$ and
    $f_2$ are equal.
\end{enumerate}
\end{lemma}

\begin{proof}
It is clear that (2) is a consequence of (1). Let us assume (2).
    Noticing that the vertical map
    \[
\id \otimes d_\Delta : \operad E \otimes s^{-1} \overline{\C}_\pl \otimes \mathbb S
\longrightarrow
\operad E \otimes \treemod^{(2)}_{\pl}(s^{-1}\overline{\mathrm{B}}(\operad E \otimes \operad P)_\pl) \otimes \mathbb S
    \]
    is injective, the assertion (1) follows from a straightforward induction on the coradical filtration of $\mathrm{B}(\operad E \otimes \operad P)$, that is, the number of nodes of the trees that make up $\mathrm{B}(\operad E \otimes \operad P)$.
\end{proof}


\subsection{The quasi-planar filtration of quasi-planar cooperads}\label{subsection: quasi-planar canonical filtration}
Let $\C$ be a quasi-planar conilpotent curved cooperad that is the colimit of a quasi-planar cooperad ladder $(\C^{(i)})_{i \in \alpha}$. We built in the previous subsection \ref{subsection: E-comodule} a morphism of dg operads
$\Omega \C \longrightarrow \operad E \otimes \Omega \C $
which induces a morphism of conilpotent curved cooperads
$$
\C \longrightarrow \mathrm{B}(\operad E \otimes \Omega \C)
$$
that is pseudo-planar in the sense that its to $\C_\pl$ factors through $\mathrm{B}(\operad E \otimes \operad P)_\pl$, see Corollary \ref{corollary: adjoint map is pseudo-planar}.

\begin{lemma}\label{lemma: pullback pseudo-planar ccc}
    Let us consider a cospan diagram in the category of pseudo-planar curved conilpotent cooperads
    \[
\D' \xrightarrow{f} \D \xleftarrow{g} \D''
    \]
    such that both maps $f,g$ are arity-wise degree-wise injective.
Then this diagram has a pullback $(\D_\pl,\D''')$. Moreover the
commutative square diagram
   \[
\begin{tikzcd}[column sep=3pc,row sep=3pc]
    \D'''_\pl
    \ar[r] \ar[d] \arrow[dr, phantom, "\lrcorner", very near start]
    &\D' \ar[d]
    \\
    \D''_\pl
    \ar[r]
    & \D_\pl
\end{tikzcd}
\]
is a pullback in the category of graded conilpotent cooperads and in the category of graded $\mathbb N$-modules and the commutative square diagram
\[
\begin{tikzcd}[column sep=3pc,row sep=3pc]
    \D'''
    \ar[r] \ar[d] \arrow[dr, phantom, "\lrcorner", very near start]
    &\D' \ar[d]
    \\
    \D''
    \ar[r]
    & \D
\end{tikzcd}
\]
is a pullback both in the category of conilpotent curved cooperads and in the category of pdg $\mathbb S$-modules.
\end{lemma}

\begin{proof}
    Let $\D'''_\pl$ be the pullback intersection in the category of graded $\mathbb N$-modules of the cospan $\D'_\pl \longrightarrow \D_\pl \leftarrow \D''_\pl$. 
    As a direct consequence of the fact that the tensor product of graded $\kk$-modules preserves intersections, $\D'''_\pl$ has the canonical structure of a graded conilpotent cooperads that makes it the pullback in the category of graded conilpotent cooperads of the same cospan.
    
    \medskip

    Now, let $\D'''$ be the pullback in the category of conilpotent pdg cooperads of the cospan $\D' \longrightarrow \D \leftarrow \D''$. 
    Since both functors
    \[
    \grcooperads_\pl^{\mathrm{conil}} \xrightarrow{-\otimes \mathbb S}
    \grcooperads^{\mathrm{conil}} \leftarrow \pdgcooperads^{\mathrm{conil}}
\]
preserves limits, one has a canonical isomorphism of graded conilpotent cooperads $\D''' \simeq \D'''_\pl \otimes \mathbb S$. In particular, the map $\D''' \longrightarrow \D'$ is degree-wise injective. Subsequently $\D'''$ is curved as $\D'$ is and thus it is a pullback in the category of curved conilpotent cooperads. The pair $(\D'''_\pl, \D''')$ is the expected pullback in the category of pseudo-planar curved conilpotent cooperads.
\end{proof}

\begin{definition}[Quasi-planar filtration]
Let $\C$ be a quasi-planar conilpotent curved cooperad. For every integer $n \geq 0$, we define the $n$\textit{-th quasi-planar filtration} as the pair $(F_n^{\qp}\C_\pl, F_n^{\qp}\C)$ given by the following pullback in the category of pseudo-planar curved conilpotent cooperads
    \[
\begin{tikzcd}[column sep=3pc,row sep=3pc]
    F_n^{\qp} \C
    \ar[r] \ar[d] \arrow[dr, phantom, "\lrcorner", very near start]
    &\C \ar[d]
    \\
    \tilde F_n \mathrm{B}(\operad E \otimes \Omega \C )
    \ar[r]
    & \mathrm{B}(\operad E \otimes \Omega \C )~.
\end{tikzcd}
    \]
This gives a diagram of pseudo-planar conilpotent curved cooperads
\[
\operad I = F_0^{\qp} \C \longrightarrow F_1^{\qp} \C \longrightarrow \cdots \longrightarrow F_n^{\qp} \C \longrightarrow \cdots
\]
called the \textit{quasi-planar filtration} of $\C$.
\end{definition}

As described in Lemma \ref{lemma: pullback pseudo-planar ccc}, the diagram
   \[
\begin{tikzcd}[column sep=3pc,row sep=3pc]
    F_n^{\qp} \C_\pl
    \ar[r] \ar[d] \arrow[dr, phantom, "\lrcorner", very near start]
    &\C_\pl \ar[d]
    \\
    \tilde F_n \mathrm{B}(\operad E \otimes \Omega \C )_\pl
    \ar[r]
    & \mathrm{B}(\operad E \otimes \Omega \C )_\pl
\end{tikzcd}
\]
is a pullback  both in the category of graded planar conilpotent cooperads and in the category of graded $\mathbb N$-modules and the diagram 
\[
\begin{tikzcd}[column sep=3pc,row sep=3pc]
    F_n^{\qp} \C
    \ar[r] \ar[d] \arrow[dr, phantom, "\lrcorner", very near start]
    &\C \ar[d]
    \\
    \tilde F_n \mathrm{B}(\operad E \otimes \Omega \C )
    \ar[r]
    & \mathrm{B}(\operad E \otimes \Omega \C )
\end{tikzcd}
    \]
is a pullback both in the category of conilpotent curved cooperads and in the category of conilpotent pdg cooperads and in the category of pdg $\mathbb N$-modules. In particular, the map $F_n^{\qp} \C_\pl \rightarrowtail \C_\pl$ is degree-wise injective as well as the map $F_n^{\qp} \C \longrightarrow \C$.

\begin{proposition}
The quasi-planar filtration of the quasi-planar curved conilpotent cooperad $\C$ is a quasi-planar $\omega$-ladder whose colimit is $\C$.
\end{proposition}

\begin{proof}
 The fact that the diagram forms a quasi-planar ladder follows from the fact that the diagram
 \[
    \tilde F_0 \mathrm{B}(\operad E \otimes \operad P) \longrightarrow \tilde F_1 \mathrm{B}(\operad E \otimes \operad P) \longrightarrow \cdots \longrightarrow \tilde F_n \mathrm{B}(\operad E \otimes \operad P)
    \longrightarrow \cdots
 \]
 is a quasi-planar $\omega$-ladder. The fact that its colimit is $\C$ is a consequence of the fact that the colimt of the diagram just above is $\mathrm{B}(\operad E \otimes \operad P)$ and of the fact that pullbacks in graded $\mathbb N$-modules commute with directed colimits.
\end{proof}

\begin{proposition}
    For every dg operad $\operad P$, the quasi-planar filtration cooperad ladder on $\mathrm{B}(\operad E \otimes \operad P)$ is canonically isomorphic to the quasi-planar $\omega$-ladder $\tilde F_n \mathrm{B}(\operad E \otimes \operad P)$.
\end{proposition}

\begin{proof}
    Knowing that the map
    \[
    \Delta_{\operad E, \mathrm{B}(\operad E \otimes \operad P)}:
    s^{-1}\overline{\mathrm{B}} (\operad E \otimes \operad P)
    \longrightarrow \operad E \otimes s^{-1}\overline{\mathrm{B}} (\operad E \otimes \operad P)
    \]
    proceeds from the structure of an Hadamard comonoid on the operad $\operad E$ (see Proposition \ref{proposition: deltae on bar}), one can notice that the composite injective map
    \[
\tilde F_n \mathrm{B}(\operad E \otimes \operad P) \longrightarrow \mathrm{B}(\operad E \otimes \operad P)
\longrightarrow
\mathrm{B}(\operad E \otimes \Omega B(\operad E \otimes \operad P))
    \]
factors through $\tilde F_n \mathrm{B}(\operad E \otimes \Omega \mathrm{B}(\operad E \otimes \operad P))$. We thus get the following commutative diagram
    \[
    \begin{tikzcd}
    \tilde F_n \mathrm{B}(\operad E \otimes \operad P)
    \ar[rd] \ar[rrd, bend left] \ar[rdd, bend right]
    \\
        &F_n^\qp \mathrm{B}(\operad E \otimes \operad P)
        \ar[r] \ar[d]
        & \mathrm{B}(\operad E \otimes \operad P)
        \ar[d]
        \\
        &\tilde F_n \mathrm{B}(\operad E \otimes \Omega \mathrm{B}(\operad E \otimes \operad P))
        \ar[r] \ar[d]
        & \mathrm{B}(\operad E \otimes \Omega \mathrm{B}(\operad E \otimes \operad P))
        \ar[d]
        \\
        &\tilde F_n \mathrm{B}(\operad E \otimes \operad P)
        \ar[r]
        & \mathrm{B}(\operad E \otimes \operad P)~,
    \end{tikzcd}
    \]
    where the bottom square is induced by the map
    \[
    \Omega \mathrm{B}(\operad E \otimes \operad P) \longrightarrow \operad P~.
    \]
    The fact that the map $\Delta_{\operad E, \mathrm{B}(\operad E \otimes \operad P)}$ proceeds from the structure of an Hadamard comonoid on the operad $\operad E$ implies that the composite top-down endomorphism of $\mathrm{B}(\operad E \otimes \operad P)$ at the right of the square is the identity. Thus, the composite endomorphism on
    the left of $\tilde F_n \mathrm{B}(\operad E \otimes \operad P)$ is also the identity.
    
    \medskip

    To conclude, the map $F_n^\qp \mathrm{B}(\operad E \otimes \operad P) \longrightarrow \tilde F_n \mathrm{B}(\operad E \otimes \operad P)$ is both degree-wise surjective and injective and therefore it is a natural isomorphism.
\end{proof}


\subsection{The cofibrant resolution} Finally, for any dg operad $\operad P$, there is always a cofibrant replacement of the form $\Omega \C$, where $\C$ is a quasi-planar conilpotent curved cooperad. This coincides a previous result for augmented dg operads, see \cite[Theorem 2, Proposition 17]{BrunoMalte}. A general version of this theorem for any $\mathbb{S}$-projective operad $\mathcal{P}$ was also proven in \cite[Theorem 8.5.4]{Wconstruction}.

\begin{proposition}
    For every operad $\operad P$, $\Omega \mathrm{B}(\operad E \otimes \operad P)$ is cofibrant and the canonical morphism
    $$
    \psi: \Omega \mathrm{B} (\operad E \otimes \operad P)
    \qi \operad P
    $$
    is an arity-wise quasi-isomorphism.
\end{proposition}

\begin{proof}
The fact that $\Omega \mathrm{B} (\operad E \otimes \operad P)$ is cofibrant follows from the fact that $\mathrm{B} (\operad E \otimes \operad P)$ is quasi-planar and that $\Omega$ sends quasi-planar conilpotent curved cooperads to cofibrant dg operads.

\medskip

Let us us show that the map $\psi: \Omega \mathrm{B} (\operad E \otimes \operad P) \longrightarrow \operad P$ is an arity-wise quasi-isomorphism. It suffices to show that the counit morphism $\epsilon: \Omega \mathrm{B} (\operad E \otimes \operad P) \longrightarrow \operad E \otimes \operad P$ is an arity-wise quasi-isomorphism. It admits a canonical section $s$ in the category of dg $\mathbb S$-modules, given by the following composition:
    $$
    \operad E \otimes \operad P \cong  s^{-1}s(\operad E \otimes \operad P) \longrightarrow s^{-1}\mathrm{B} (\operad E \otimes \operad P) \longrightarrow \Omega \mathrm{B} (\operad E \otimes \operad P)~.
    $$
Let us denote $\pi_{\operad P} = s \circ \epsilon$. Since $\epsilon \circ s = \id$, the counit $\epsilon$ is an arity-wise quasi-isomorphism if and only if $\pi_{\operad P}$ is arity-wise quasi-isomorphism. So let us prove that $\pi_{\operad P}$ is an arity-wise quasi-isomorphism.

\medskip

Let us first assume that the unit map $\eta:\operad I \longrightarrow \operad P$ has a left inverse in the category of graded $\mathbb S$-modules. Let us denote $\operad Q$ the graded $\mathbb S$-module defined as follows 
$$
\operad Q \coloneqq \operad E_\pl \otimes \operad P~.
$$
The map $\operad I \longrightarrow \operad Q$ also has a left inverse. We denote $\overline{\operad Q}$ the kernel of this section.
Let us denote $\operad C$ the following planar graded conilpotent cooperad 
$$
\operad C \coloneqq \treemod_\pl(s \operad Q \oplus s^2\operad I)
\cong \treemod_\pl(s \overline{\operad Q} \oplus s\operad I \oplus s^2\operad I)~.
$$
We have a canonical isomorphism of graded operads
$$
\Omega \mathrm{B} (\operad E \otimes \operad P)
\simeq 
\treemod_\pl 
(s^{-1}\overline{\operad C}) \otimes \mathbb S~.
$$

Let $h$ be the degree $1$ endomorphism of $\Omega \mathrm{B} (\operad E \otimes \operad P)$ given as follows.
\begin{enumerate}
    \item On the trivial tree (with no node):
    \begin{align*}
        \operad I &\longrightarrow s^{-1} s^2\operad I \longrightarrow s^{-1} \overline{\operad C}
        \\
        1 &\mapsto s^{-1} s^2 1~.
    \end{align*}
    \item On planar trees with one node, it is zero.
    \item On planar trees $t$ with two nodes or more it is
    $$
    \begin{tikzcd}
        t(\overline{\operad C}) \otimes \mathbb S
        \ar[d, equal]
        \\
        s^{-1}\overline{\operad C}(n_1) \otimes s^{-1}\overline{\operad C}(n_2) \otimes \cdots \otimes \mathbb S
        \ar[d, two heads]
        \\
        s^{-1}F_1^{\mathrm{rad}}(\overline{\operad C})(n_1) \otimes s^{-1}\overline{\operad C}(n_2) \otimes \cdots \otimes \mathbb S
        \ar[d, "\simeq"]
        \\
        s^{-2} F_1^{\mathrm{rad}}(\overline{\operad C})(n_1) \otimes \overline{\operad C}(n_2) \otimes \cdots \otimes \mathbb S
        \ar[d]
        \\
        s^{-1}\overline{\operad C}(n_1) \otimes \overline{\operad C}(n_2) \otimes \cdots \otimes \mathbb S
        \ar[d, "\simeq"]
        \\
        s^{-1}\overline{\operad C}(n_1+n_2-1)  \otimes \cdots \otimes \mathbb S
    \end{tikzcd}
    $$
\end{enumerate}
where the third and the fourth map are given by 
$$
s^{-1}x \otimes s^{-1} y \otimes \cdots \otimes \{\sigma\}
\mapsto -(-1)^{|x|} s^{-2}x \otimes y \otimes \cdots \otimes \{\sigma\}
\mapsto -(-1)^{|x|} s^{-1}x \otimes y \otimes \cdots \otimes \{\sigma\}~.
$$
Let show that $\partial (h) + \pi_{\operad P}$ is an isomorphism. We filter $\Omega \mathrm{B} (\operad E \otimes \operad P)$ using the tree filtration relative to the quasi-planar filtration of $\mathrm{B} (\operad E \otimes \operad P)$, with the difference that we define the first part not to be the trivial tree but $0$. This gives
\begin{align*}
F_0\Omega \mathrm{B} (\operad E \otimes \operad P) &= 0~,
\\
F_1\Omega \mathrm{B} (\operad E \otimes \operad P) &= 
\operad I \oplus s^{-1} (s (\operad E_{\leq 0} \otimes \operad P) \oplus s^2 \operad I)~.
\end{align*}
This filtration is preserved by $h$, $\pi_{\operad P}$ and the derivation. Moreover, the map on the associated graded object $\gr(\partial(h) + \pi_{\operad P})$ is an isomorphism; so is $\partial(h) + \pi_{\operad P}$. This means that $\pi_{\operad P}$ is an arity-wise quasi-isomorphism as it is chain homotopic to an arity-wise quasi-isomorphism.

\medskip

Finally, if the unit $\operad I \longrightarrow \operad P$ does not admit a left inverse in the category of graded $\mathbb S$-modules, then $\operad P = 0$ and the result holds trivially.
\end{proof}


\newpage

\section{Algebras, coalgebras, and Bar-Cobar adjunctions}

\vspace{2pc}

In this section, we recall the various notions of (co)algebras over a (co)operads, as well as the bar-cobar adjunctions that interrelate them. Along the way, we develop some of their main categorical and algebraic properties. Finally, we review the different methods that can give a homotopical meaning to the categories of dg (co)algebras over an dg operad. 

\subsection{From $\mathbb S$-modules to functors}
The category of dg modules is enriched and tensored in dg $\mathbb S$-modules as follows.

\medskip

\begin{enumerate}
    \item For every dg modules $X, Y$, the mapping dg $\mathbb S$-module is $\Mult(X, Y)$
    given by 
    $$
    \Mult(X, Y)(n) \coloneqq [X^{\otimes n},Y].
    $$
    In the case where $Y=X$, we set $\mathrm{End}(X) \coloneqq \Mult(X, X)$.

    \medskip
    
    \item For every dg module  $X$ and every dg $\mathbb S$-module $M$, the tensorisation of $X$ by $M$ is just given by $M \comp X$ as dg module, where $X$ is considered as dg $\mathbb S$-modules concentrated in arity zero. Thus for every additional dg $\mathbb S$-module $N$, the structural map
    $$
    (N \comp M) \comp X \longrightarrow N \comp (M \comp X)
    $$
    is an isomorphism.
\end{enumerate}
\medskip

It is also enriched and cotensored in dg $\mathbb S$-modules as follows.

\medskip

\begin{enumerate}
    \item For every dg modules $X, Y$, the mapping dg $\mathbb S$-module is $\mathrm{coMult}(X, Y)$
    given by 
    $$
    \coMult(X, Y)(n) \coloneqq [X,Y^{\otimes n}].
    $$
    In the case where $Y=X$, we set $\mathrm{coEnd}(X) \coloneqq \coMult(X, X)$.

    \medskip
    
    \item For every dg module  $X$ and every dg $\mathbb S$-module $M$, the cotensorisation of $X$ by $M$ is $X^M$ given by
    $$
    X^M = \prod_{n \geq 0} \left[M(n), X^{\otimes n}\right]^{\mathbb S_n}~.
    $$
    For every additional dg $\mathbb S$-module $N$, the structural map
    $$
    \varphi_{M,N}: \left(X^M\right)^N \longrightarrow X^{N \comp M} 
    $$
    is a degree-wise injection.
\end{enumerate}
\medskip

Similarly, dg modules are on the one hand enriched and tensored over dg $\mathbb N$-module through the bifunctors
$$
 X, Y \mapsto \Mult_\pl(X,Y)  \coloneqq [X^{\otimes -}, Y]; \quad
 M, X \mapsto M \comp_\pl X
$$
and they are on the other hand enriched and cotensored over dg $\mathbb N$-module through the bifunctors
$$
 X, Y \mapsto \coMult_\pl(X,Y)  \coloneqq [X, Y^{\otimes -}]; \quad
X, M \mapsto X^{M} \prod_n [M(n), X^{\otimes n}]~.
$$
There are canonical isomorphisms of dg $\mathbb N$-modules $\Mult_\pl(-,-) \cong U_{\mathbb S}\Mult(-,-)$ and $\coMult_\pl(-,-) \cong U_{\mathbb S}\coMult(-,-)$. We obtain by transposing along the adjunction two canonical natural isomorphisms
$$
(M \otimes \mathbb S) \comp X \simeq M \comp_\pl X~, \quad
X^{M \otimes \mathbb S} \simeq X^M~,
$$
for every dg module  $X$ and every dg $\mathbb N$-module $M$.

\medskip

Let us state some categorical properties that the constructions perfomed so far satisfy. 

\begin{lemma}
    Let $M$ be a dg $\mathbb S$-module. The endofunctor $M \comp -$ commutes with sifted colimits and the endofunctor $(-)^M$ commutes with $\beta$-filtered colimits for a small regular cardinal $\beta$.
\end{lemma}

\begin{proof}
    The assertion about $M \comp -$ can be deduced from the following facts

    \medskip
    
    \begin{enumerate}
        \item the tensor product $X \mapsto X^{\otimes n}$ commutes with sifted colimits;

    \medskip
    
        \item the functor $Y \mapsto M(n)\otimes_{\mathbb S}  Y$
        commutes with colimits for every natural integer $n$;

    \medskip
    
        \item coproducts commute with colimits.
    \end{enumerate}

    \medskip
    
    Now, let $\beta$ be a regular cardinal such that $\beta > \omega$ and such that every $M(n)$, for every $n$ in $\mathbb N$, is $\beta$-small. Then the endofunctor $(-)^M$ commutes with $\beta$-filtered colimits since
    \begin{enumerate}

    \medskip
    
        \item the tensor product $X \mapsto X^{\otimes n}$ commutes with sifted colimits, in particular $\beta$-filtered colimits;

    \medskip
    
        \item the construction $Y \mapsto [M(n), Y]$ commutes with
        $\beta$-filtered colimits;

    \medskip
    
        \item the construction $Z \mapsto Z^{\mathbb S_n}$ commutes with filtered colimits;

    \medskip
    
        \item countable products commute with $\beta$-filtered colimits since $\beta > \omega$.
    \end{enumerate}
\end{proof}

\begin{lemma}
    Let $M$ be a dg $\mathbb S$-module. The endofunctor $(-)^M$ commutes with \textbf{finite} cosifted limits.
\end{lemma}

\begin{proof}
    This can be deduced from the following facts: 

    \medskip
    
    \begin{enumerate}
        \item the tensor product $X \mapsto X^{\otimes n}$ commutes with finite cosifted limits,

    \medskip
    
        \item the functor $Y \mapsto [M(n), Y]^{\mathbb S_n}$ commutes, limits for every natural integer $n$;

    \medskip
        \item products commute with limits.
    \end{enumerate}
\end{proof}

One can perform the same constructions with graded modules or pdg modules instead of dg modules. Moreover, these constructions commute strictly with the forgetful functors $\catdgmod{\kk} \longrightarrow \catpdgmod{\kk} \longrightarrow \catgrmod{\kk}$.

\medskip

Finally, let us notice that when a dg $\mathbb S$-module $M$ is planar as a graded $\mathbb S$-module, one has extra properties on these endofunctors.

\begin{lemma}\label{lemma: quasi-planar functors preserve more stuff}
Let $M$ be a dg $\mathbb S$-module. Let us suppose that the underlying graded $\mathbb S$-module of $M$ has the form $M_\pl \otimes \mathbb S$. Equivalently, $M(n)$ is a quasi-free dg $\kk[\mathbb S_n]$-module for all $n \geq 0$. Then

\medskip

\begin{enumerate}
\item the endofunctor $(-)^M$ commutes with \textbf{finite} sifted colimits;

\medskip

\item the endofunctor $M \comp (-)$ commutes with \textbf{finite} cosifted limits.

\end{enumerate}
\end{lemma}

\begin{proof}\leavevmode
  
    \begin{enumerate}
        \item For the endofunctor $(-)^M$, it is a consequence of the fact that the tensor product, the functors $[M_\pl(n), -]$ and products commute with finite sifted colimits.

    \medskip
    
        \item For the endofunctir $M \circ (-)$, it is a consequence of the fact that the tensor product, the functors $M_\pl(n) \otimes -$ and coproducts commute with finite cosifted limits.
    \end{enumerate}
\end{proof}

\begin{notation}
Let $f: X \longrightarrow Y$ be a map of degree $0$ and $g: X \longrightarrow Y$ be a map of degree $p$ between graded modules $X,Y$. Let us consider 
\[
\shuffle_n(f,g) \coloneqq \sum_{i=0}^n f^{\otimes i-1} \otimes g \otimes f^{\otimes n-i} : X^{\otimes n} \longrightarrow Y^{\otimes n}~,
\]
which is an $\mathbb{S}_n$-equivariant map of degree $p$. Let $M$ be an graded $\mathbb{S}$-module. 

\medskip

\begin{enumerate}
\item It induces a map of degree $p$

\[
\bigoplus_{n \geq 0} \id_{M(n)} \otimes_{\mathbb S_n} \shuffle_n(f,g): \bigoplus_{n \geq 0} M(n) \otimes_{\mathbb S_n} X^{\otimes n} \longrightarrow \bigoplus_{n \geq 0} M(n) \otimes_{\mathbb S_n} Y^{\otimes n}~. 
\]
By a slight abuse of notation, we denote this morphism by $M \comp \shuffle(f,g)$.

\medskip

\item It induces a map of degree $p$

\[
\prod_{n \geq 0} [\id_{M(n)},\shuffle_n(f,g)]^{\mathbb{S}_n}: \prod_{n \geq 0} [M(n),X^{\otimes n}]^{\mathbb{S}_n} \longrightarrow \prod_{n \geq 0} [M(n),Y^{\otimes n}]^{\mathbb{S}_n}
\]
By a slight abuse of notation, this map will be denoted by $\shuffle(f,g)^M$.
\end{enumerate}
\end{notation}

\subsection{Algebras and coalgebras over (co)operads}
Let $\operad P$ be a dg operad and let $\operad C$ be a dg cooperad. We recall the definitions of a dg $\operad P$-algebra and dg $\operad P$-coalgebra, as well as the definitions of a dg $\C$-coalgebra and of a dg $\C$-algebra. See \cite[Section 3]{linearcoalgebras}.

\begin{definition}[dg $\PP$-algebra]
A dg $\operad P$-\textit{algebra} $A$ amounts to the data of a dg module $(A,d_A)$ equipped with a morphism of dg operads $\Gamma_A: \operad P \longrightarrow \mathrm{End}(A)$. 

\medskip

This data is equivalent to a structural map $\gamma_A: \operad P \comp A \longrightarrow A$ that satisfies unitality and associativity conditions. In other words, it is equivalent to the data of an algebra over the monad $\operad P \comp -$.
\end{definition}

\begin{remark}
   More precisely, the data of a dg $\mathcal{P}$-algebra structure on $(A,d_A)$ amounts to a collection of maps 

    \[
    \gamma_A^n: \operad P \otimes_{\mathbb S_n} A^{\otimes n} \longrightarrow A~,
    \]
    \vspace{0.1pc}

    for all $n \geq 0$, which satisfy compatibility conditions. Since we consider on the left-hand side \textit{coinvariants} with respect to the action of $\mathbb S_n$, no divided power operations appear. The definition of a dg $\operad P$-algebra encodes classical types of algebraic structures over a positive characteristic field. For instance, dg Lie-algebras, dg commutative algebras, dg Poisson algebras, etc.
\end{remark}

\begin{definition}[dg $\mathcal{P}$-coalgebra]
A dg $\mathcal{P}$-coalgebra $V$ amounts to the data of a dg module $(V,d_V)$ equipped with a morphism of dg operads $\delta_V: \operad P \longrightarrow \mathrm{coEnd}(V)$. 

\medskip

This data is equivalent to a structural map $\Delta_V: V \longrightarrow \displaystyle V^{\operad P}$ such that the following diagrams commute
$$
\begin{tikzcd}[column sep=4.5pc,row sep=3pc]
V \arrow[r,"\Delta_V"] \arrow[d,"\Delta_V",swap] 
&V^{\operad P} \arrow[r,"(\Delta_V)^{\id}"]
&(V^{\operad P})^{\operad P} \arrow[d,"\varphi_{\operad P, \operad P}"] \\
V^{\operad P} \arrow[rr,"(\mathrm{id})^{\gamma}"]
&
&V^{\operad P ~ \comp ~\operad P}~,
\end{tikzcd}
\quad \quad 
\begin{tikzcd}
    V
    \ar[r, "\Delta_V "] \ar[rd, "\id",swap]
    & V^{\operad P}
    \ar[d,"(\mathrm{id})^{\eta}"]
    \\
    & V~,
\end{tikzcd}
$$
where $\gamma: \operad P \circ \operad P \longrightarrow \operad P$ and $\eta: \operad I \longrightarrow \operad P$ are, respectively, the composition morphism and the unit of the dg operad $\operad P$.
\end{definition}

\begin{remark}
    The data of a dg $\mathcal{P}$-coalgebra structure on $(V,d_V)$ also amounts to a collection of maps 

    \[
    \Delta_V^n: V \longrightarrow [\operad P(n), V^{\otimes n}]^{\mathbb S_n}~,
    \]
    \vspace{0.1pc}

    for all $n \geq 0$, which satisfy compatibility conditions. Since we consider on the right-hand side \textit{invariants} with respect to the action of $\mathbb S_n$, no divided power operations appear. The definition of a dg $\operad P$-coalgebra encodes classical types of coalgebraic structures, \textit{without any conilpotency condition}, over a positive characteristic field. For instance, dg Lie-coalgebras, dg cocommutative coalgebras, dg Poisson coalgebras, etc. 
\end{remark}

Even though $(-)^{\operad{P}}$ fails to be a comonad, the forgetful functor from dg $\operad P$-coalgebras to dg modules is comonadic.

\begin{theorem}[{\cite[Theorem 2.7.11]{anelcofree2014}}]
The forgetful functor from dg $\PP$-coalgebras to dg modules is comonadic. The related comonad $L^{\operad P}(-)$ is given, for a dg module $X$, by the following pullback square
$$
\begin{tikzcd}[column sep=3.5pc,row sep=3.5pc]
    L^{\operad P} X
    \ar[r] \ar[d] \arrow[dr, phantom, "\lrcorner", very near start]
    & (X^{\operad P})^{\operad P}
    \ar[d,"\varphi_{\operad P, \operad P}"]
    \\
    X^{\operad P}
    \ar[r,"(\mathrm{id})^{\operad P}"]
    & X^{\operad P ~\comp ~ \operad P}.
\end{tikzcd}
$$
in the category of dg modules, where $\gamma: \operad P \circ \operad P \longrightarrow \operad P$ is the composition map of $\operad P$.

\medskip

The structure of a comonad on $L^{\operad P}$ is given by the decomposition map $L^{\operad P} X
    \longrightarrow
     (X^{\operad P})^{\operad P}$, which factors through
     $L^{\operad P} L^{\operad P} X$, and by the following counit map $L^{\operad P} X \longrightarrow X^{\operad P} \longrightarrow X$.
\end{theorem}

\begin{definition}[dg $\mathcal{C}$-coalgebra]
A dg $\operad{C}$-coalgebra $W$ amounts to the data of a dg module $(W,d_W)$ equipped with structural map $\Delta_W: W \longrightarrow \operad C \comp W$ such that the following diagrams commute
\[
\begin{tikzcd}[column sep=4.5pc,row sep=3pc]
W \ar[r,"\Delta_W"] \ar[d,"\Delta_W",swap]
& \operad C \comp W
\ar[d, "\Delta~ \comp ~ \id"]
\\
\operad C \comp W
\ar[r, "\id~ \comp ~\Delta_W"]
& \operad C \comp \operad C \comp W
\end{tikzcd}
\quad \quad
\begin{tikzcd}
    W \arrow[r," \Delta_W "] \arrow[rd, "\id",swap]
    & \operad C \comp W
    \arrow[d, "\epsilon ~ \comp ~ \id "]
    \\
    & W.
\end{tikzcd}
\]
where $\Delta: \C \longrightarrow \C \comp \C$ and $\epsilon: \C \longrightarrow \operad I$ are, respectively, the decompositon map and the counit map of the cooperad $\C$. In other words, it is a coalgebra over the comonad $\operad C \comp -$.
\end{definition}

\begin{remark}
    The structural map of a dg $\C$-coalgebra $W$ 
    \[
    \Delta_W: W \longrightarrow \bigoplus_{n \geq 0} \C(n) \otimes_{\mathbb S_n} W^{\otimes n}~,
    \]
    \vspace{0.1pc}

    lands on the direct sum. Therefore any element $w$ in $W$ can only be decomposed into a \textit{finite sum}. This implies that a dg $\C$-coalgebras always satisfies some type \textit{conilpotency} condition. Furthermore, the fact that $\Delta_W$ lands on the \textit{coinvariants} on the right-hand side implies that \textit{divided power operations} will appear for this type of structures. Thus the definition of a dg $\C$-coalgebra encodes divided powers conilpotent types of coalgebraic structures.
\end{remark}

\renewcommand{\L}{\Lambda}

\begin{definition}[dg $\operad C$-algebra]\label{def pdg C algebra}
A dg $\operad{C}$\textit{-algebra} $\L$ amounts to the data of a dg module  $(\Lambda,d_\L)$ equipped with a structural morphism $\gamma_\L: \L^{\operad C} \longrightarrow \L$ such that the following diagrams commute
$$
\begin{tikzcd}[column sep=3.5pc,row sep=3pc]
    (\L^{\operad C})^{\operad C}
    \ar[r, hook, "\varphi_{\C,\C}"] \ar[d,"(\id)^{\gamma_\L}",swap]
    & \L^{\operad C ~\comp ~\operad C}
    \ar[r,"(\id)^{\Delta}"]
    & \L^{{\operad C}}
    \ar[d,"\gamma_\L"]
    \\
    \L^{\operad C}
    \ar[rr,"\gamma_B"]
    && \L
\end{tikzcd}
\quad \quad
\begin{tikzcd}
    \L
    \ar[r,"(\id)^{\epsilon}"] \ar[rd, "\id",swap]
    & \L^{\operad C}
    \ar[d,"\gamma_\L"]
    \\
    &\L
\end{tikzcd}
$$
where $\Delta: \C \longrightarrow \C \comp \C$ and $\epsilon: \C \longrightarrow \operad I$ are, respectively, the decompositon morphism and the counit morphism of the cooperad $\C$. In other words, it is an algebra over the monad $(-)^{\operad C}$. 
\end{definition}

\begin{remark}
    The structural map of a dg $\C$-algebra $\L$

    \[
    \gamma_\L: \prod_{n \geq 0} [\C(n), \L^{\otimes n}]^{\mathbb S_n} \longrightarrow \L~,
    \]
    \vspace{0.1pc}

    comes from an infinite product. Therefore any formal infinite sum of operations admits a well-defined image in $\L$ \textit{by definition}, without presupposing any topology. We refer to these phenomena as \textit{absolute} algebraic structures. They generalize the notion of a \textit{contramodule}, see \cite{positselski2021contramodules} for a precise definition of them. Furthermore, the presence of \textit{invariants} on the left-hand side implies that these algebraic structures are also endowed with \textit{divided powers} operations. Therefore the definition of a dg $\C$-dalgebra encodes divided powers absolute types of algebraic structures.
\end{remark}

\begin{remark}[Planar definitions]
One can also define the dg (co)algebras over a planar dg operad $\operad P$ as the (co)algebras over the dg operad $\operad P \otimes \mathbb S$. And the dg (co)algebras over a dg planar cooperad $\operad C$ as the (co)algebras over the dg cooperad $\operad C \otimes \mathbb S$. In this case, no \textit{divided powers} operations appear in dg $\C$-coalgebras nor in dg $\C$-algebras.
\end{remark}

\medskip

\begin{proposition}
    The category of dg $\operad P$-algebras is $\omega$-presentable.
    The categories of dg $\operad P$-coalgebras, dg $\operad C$-coalgebras and dg $\operad C$-algebras are presentable.
\end{proposition}

\begin{proof}
    The monad $\operad P \comp -$ and the comonad $\operad C \comp -$ are $\omega$-accessible. The monad $(-)^{\operad C}$ and the comonad $L^{\operad P}$ are also accessible but for a larger small cardinal ($\aleph_1$). We conclude by the fact that for a regular small cardinal $\alpha$, the category of algebras over an $\alpha$-accessible monad in an $\alpha$-presentable category is $\alpha$-presentable. See \cite{AdamekRosicky} for more details. Moreover, the category of coalgebras over an accessible comonad in a presentable category is presentable, see \cite{ChingRiehl}.
\end{proof}

\medskip

One can consider the same definitions and perform the same constructions with graded modules or pdg modules instead of dg module. These constructions commute strictly with the forgetful functors $\catdgmod{\kk} \longrightarrow \catpdgmod{\kk} \longrightarrow \catgrmod{\kk}$.

\medskip

\textbf{Curved algebras and coalgebras.} Let $\C$ be a curved cooperad. It is in particular a pdg cooperad, and one can define the categories of pdg $\C$-coalgebras and of pdg $\C$-algebras. Among these, the full subcategories of \textit{curved} objects will be of particular interest. 

\begin{definition}[Curved $\C$-coalgebra]
Let $\operad C$ be a curved cooperad. A pdg coalgebra $(W,\Delta_W,d_W)$ is \textit{curved} if the following diagram commutes

\[
\begin{tikzcd}[column sep=3pc,row sep=3pc]
W  \arrow[r,"\Delta_W"] \arrow[rd,"-d_W^2",swap]
&\operad C \comp W \arrow[d,"\theta ~\comp~ \id"]\\
& W~,
\end{tikzcd}
\]

where $\theta$ denotes the curvature of the cooperad $\C$. 

\medskip

The category $\catcurvcog{\operad C}$ of curved $\operad C$-coalgebras is the full subcategory of pdg $\operad C$-coalgebras spanned by those which are curved.
\end{definition}

\begin{remark}
If the curvature of $\operad C$ is zero, then the category of curved $\C$-coalgebras in pdg modules is isomorphic to the category of dg $\C$-coalgebras.
\end{remark}

\begin{definition}[Curved $\operad C$-algebra]\label{def curved alg over a coop}
Let $\operad C$ be a curved cooperad. A pdg $\operad C$-algebra $(\L,\gamma_\L,d_\L)$ is \textit{curved} if the following diagram commutes: 
\[
\begin{tikzcd}[column sep=4pc,row sep=3pc]
\L
\arrow[r,"\L^{\theta}"] \arrow[rd,"d_\L^2",swap]
&\L^{\operad C}
\arrow[d, "\gamma_\L"]
\\
&\L~.
\end{tikzcd}
\]
The category $\catcurvalg{\operad C}$ of curved $\operad C$-algebras is the full subcategory of pdg $\operad C$-algebras spanned by those which are curved.
\end{definition}

\begin{remark}
If the curvature of $\operad C$ is zero, then the category of curved $\C$-algebras in pdg modules is isomorphic to the category of dg $\C$-algebras.
\end{remark}

\subsection{Categorical properties of the category of (co)algebras over a operad}
Let $\operad P$ be a dg operad. Let us consider the following commutative diagram of forgetful functors
$$
\begin{tikzcd}
    \catdgalg{\operad P}
    \ar[r] \ar[d]
    & \catpdgalg{\operad P}
    \ar[r] \ar[d]
    & \catgralg{\operad P}
    \ar[d]
    \\
    \catdgmod{\kk}
    \ar[r]
    & \catpdgmod{\kk}
    \ar[r]
    & \catgrmod{\kk}.
\end{tikzcd}
$$
All these categories are presentable. Moreover, all the functors are right adjoints because they preserve limits and filtered colimits. They are conservative and they also preserve sifted colimits. Therefore, all these forgetful functors are monadic. 

\medskip

Let us consider the following commutative diagram of forgetful functors 
$$
\begin{tikzcd}
    \catdgcog{\operad P}
    \ar[r] \ar[d]
    & \catpdgcog{\operad P}
    \ar[r] \ar[d]
    & \catgrcog{\operad P}
    \ar[d]
    \\
    \catdgmod{\kk}
    \ar[r]
    & \catpdgmod{\kk}
    \ar[r]
    & \catgrmod{\kk}
\end{tikzcd}
$$
All these categories are presentable. Moreover, all the functors are left adjoints since they preserve colimits. They are also conservative and preserve finite cosifted limits. Therefore, all the functors are comonadic.

\subsection{Categorical properties of coalgebras over a conilpotent curved cooperad}
\label{sectioncommuteslimitscolimitscog}
Let $\operad C$ be a conilpotent curved cooperad whose underlying graded conilpotent cooperad is the image through $- \otimes \mathbb S$ of a graded planar conilpotent cooperad $\operad C_\pl$.

\begin{proposition}
The forgetful functor from curved $\C$-coalgebras to pdg $\C$-coalgebras admits a right adjoint denoted by $\mathrm{Curv}$. This gives an adjunction 
\[
\begin{tikzcd}[column sep=5pc,row sep=3pc]
          \mathsf{curv}~\C\text{-}\mathsf{cog} \arrow[r, shift left=1.1ex, "\mathrm{U}"{name=F}] & \mathsf{pdg}~\C\text{-}\mathsf{cog}~. \arrow[l, shift left=.75ex, "\mathrm{Curv}"{name=U}]
            \arrow[phantom, from=F, to=U, , "\dashv" rotate=-90]
\end{tikzcd}
\]
Hence the category of curved $\operad C$-coalgebras forms a coreflexive full subcategory of the category of pdg $\operad C$-coalgebras. In other words, they are coalgebras over an idempotent comonad in the category of pdg $\operad C$-coalgebras.
\end{proposition}

\begin{proof}
Let $(W,\Delta_W,d_W)$ be a pdg $\operad C$-coalgebra, we consider the following two maps of graded $\kk$-modules
\[
\begin{tikzcd}[column sep=3pc,row sep=3pc]
        W \ar[rr, shift left = 0.7ex, "(\theta~\comp~\id)\Delta_W"]
        \ar[rr, shift right = 0.7ex, "-d_W^2"'] 
        && s^2 W~,
\end{tikzcd}
\]
where $\theta: \C \longrightarrow \operad \I$ is the curvature of the cooperad $\C$. Let $Z_W$ be the cofree pdg $\kk$-module obtained from the graded $\kk$ module $s^2 W$. Since a pdg $\kk$-module is a graded $\kk$-module together with a degree $-1$ endomorphism, $Z_W$ is given by $s^2 W \otimes \kk[\mathrm{u}]$, where $\mathrm{u}$ is a generator in degree $-1$. From the above two maps, one gets a natural coreflexive pair of maps of pdg $\operad C$-coalgebras
\[
\begin{tikzcd}[column sep=3pc,row sep=3pc]
        W \ar[rr, bend left, ]
        \ar[rr, bend right] 
        && (\operad C \comp Z_W) \times W
        \ar[ll]
\end{tikzcd}
\]
where the rightward maps are obtained from the previous diagram of graded $\kk$-modules, using the forgetful-cofree adjunction between graded $\kk$-modules and pdg $\kk$-modules, and where the leftward map is just the projection onto $W$. 

\medskip

This pair has an equaliser $(H,\Delta_H,d_H)$ in pdg $\C$-coalgebras, which can be computed in graded $\kk$-modules. In particular, the morphism of pdg $\operad C$-coalgebras $H \longrightarrow W$ is a degree-wise monomorphism. Furthermore, the following diagram commutes
     $$
     \begin{tikzcd}[column sep=3pc,row sep=3pc]
         H
         \ar[r,rightarrowtail] \ar[d,"(\theta~ \comp~ \id)\Delta_H + d_H^2"']
         & W \ar[d,"(\theta~ \comp ~\id)\Delta_W + d_W^2"]
         \\
         s^2H
         \ar[r,rightarrowtail]
         &s^2W~.
     \end{tikzcd}
     $$
The map from $H$ to $s^2W$ is zero by the universal property of $H$. Since the bottom map is a monomorphism, the left map is also zero. So $H$ is curved. We set $\mathrm{Curv}(W) \coloneqq H$, and it is straightforward to show that it defines a right adjoint to the forgetful functor. 
\end{proof}

\begin{corollary}
The category of curved $\operad C$-coalgebras is the category of coalgebra over an idempotent comonad over pdg $\operad C$-coalgebras which is accessible and preserves coreflexive equalisers. Hence, the category of curved $\operad C$-coalgebras is presentable.
\end{corollary}

\begin{proof}
Let $W$ be a pdg $\C$-coalgebra. Its image by the functor $\mathrm{Curv}$ is also given by the coreflexive equaliser of a the pair of maps of the form
\[
\begin{tikzcd}
        \operad C \comp W \ar[rr, bend left] \ar[rr, bend right]
        && \operad C \comp ((\operad C \comp W) \oplus Z_W)~.
        \ar[ll]
\end{tikzcd}
\]
This coreflexive equaliser is constructed from the one that defines the functor $\mathrm{Curv}$ and from the one that gives $W$ as a  pdg $\C$-coalgebra. The comonad that defines curved $\operad C$-coalgebras preserves finite cosifted limits since its composition with the the forgetful functor to graded $\kk$-modules does. Indeed, we have that:

     \medskip
     
     \begin{enumerate}
        \item the forgetful functor preserves and creates finite cosifted limits;

        \medskip
        
         \item the constructions $X \mapsto \operad C \comp X$ and $W \mapsto \operad C \comp ((\operad C \comp W) \oplus Z_W)$ preserves finite cosifted limits;

         \medskip
         
         \item then the construction $W \mapsto \lim(W \rightrightarrows \operad C \comp W \oplus Z_W) $ preserves finite cosifted limits.
     \end{enumerate}

     \medskip
     
     To prove that this comonad preserves filtered colimits, it suffices to prove that its composition with the forgetful functor towards pdg $\kk$-modules preserves filtered colimits. This follows from:

     \medskip
     
     \begin{enumerate}
         \item the forgetful functor preserves and creates filtered colimits and coreflexive equalisers;

         \medskip
         
         \item the constructions $W \mapsto \operad C \comp W$
         and $W \mapsto \operad C \comp ((\operad C \comp W ) \oplus Z_W)$ preserve filtered colimits;

         \medskip
         
         \item filtered colimits and coreflexive equalisers commute in graded $\kk$-modules.
     \end{enumerate}
\end{proof}

Let us consider the following commutative diagram of forgetful functors 

\[
\begin{tikzcd}
    \catcurvcog{\operad C}
    \ar[r]
    & \catpdgcog{\operad C}
    \ar[r] \ar[d]
    & \catgrcog{\operad C}
    \ar[d]
    \\
    & \catpdgmod{\kk}
    \ar[r]
    & \catgrmod{\kk}~.
\end{tikzcd}
\]

All the categories are presentable and all the functors are left adjoint (since they preserve colimits), conservative and they also preserve coreflexive equalisers. Therefore all the functors are comonadic.

\subsection{Curved algebras over a conilpotent curved cooperad}
\label{sectioncommuteslimitscolimits}

Again, let $\operad C$ be a conilpotent curved cooperad whose underlying graded conilpotent cooperad is the image through $- \otimes \mathbb S$ of a graded planar conilpotent cooperad $\operad C_\pl$.

\begin{proposition}
The forgetful functor from curved $\C$-algebras to pdg $\C$-algebras admits a left adjoint denoted by $\mathrm{Curv}$. This gives an adjunction 
\[
\begin{tikzcd}[column sep=5pc,row sep=3pc]
          \mathsf{curv}~\C\text{-}\mathsf{alg} \arrow[r, shift left=1.1ex, "\mathrm{U}"{name=F}] & \mathsf{pdg}~\C\text{-}\mathsf{alg}~. \arrow[l, shift left=.75ex, "\mathrm{Curv}"{name=U}]
            \arrow[phantom, from=F, to=U, , "\dashv" rotate=90]
\end{tikzcd}
\]
Hence the category of curved $\operad C$-algebras forms a reflexive full subcategory of the category of pdg $\operad C$-algebras. In other words, they are algebras over an idempotent monad in the category of pdg $\operad C$-algebras.
\end{proposition}

\begin{proof}
Let $(\L,\gamma_\L,d_\L)$ be a pdg $\operad C$-algebra. We consider the two maps of graded $\kk$-modules
\[
\begin{tikzcd}
        s^{-2} \L \ar[rr, shift left = 0.7ex, "\gamma_\L~ \L^\theta"]
        \ar[rr, shift right = 0.7ex, "d^2_\L"'] 
        && \L~,
\end{tikzcd}
\]
where $\theta: \C \longrightarrow \operad \I$ is the curvature of the cooperad $\C$. Let $Z_\L$ be the free pdg $\kk$-module obtained from the graded $\kk$ module $s^{-2} \L$. Since a pdg $\kk$-module is a graded $\kk$-module together with a degree $-1$ endomorphism, $Z_\L$ is given by $s^{-2} \L \otimes \kk[\mathrm{u}]$, where $\mathrm{u}$ is a generator in degree $-1$. From the above two maps, one gets a natural reflexive pair of maps of pdg $\operad C$-algebras
     $$
     \begin{tikzcd}[column sep=3pc,row sep=3pc]
        \L \sqcup (Z_\L)^{\operad C} \ar[rr, bend left]
        \ar[rr, bend right] 
        &&\L
        \ar[ll]
    \end{tikzcd}
     $$
     where the leftward maps are obtained from the previous diagram of graded $\kk$-modules, using the free-forgetful adjunction between graded $\kk$-modules and pdg $\kk$-modules, and where the rightward map is just the inclusion of $\L$. 
     
\medskip

     This pair has a coequaliser $(\Xi,\gamma_{\Xi},d_{\Xi})$ in the category of pdg $\C$-algebras, which can be computed in graded $\kk$-modules. In particular, the morphism of pdg $\operad C$-algebras $\L \twoheadrightarrow \Xi$ is a degree-wise epimorphism. Furthermore, the following square commutes:
     $$
     \begin{tikzcd}[column sep=3pc,row sep=3pc]
         s^{-2}\L
         \ar[r,twoheadrightarrow] \ar[d,"\gamma_\L~ \L^\theta - d^2_\L"']
         & s^{-1}\Xi \ar[d,"\gamma_\Xi ~ \Xi^\theta - d^2_{\Xi}"]
         \\
         \L
         \ar[r,twoheadrightarrow]
         &\Xi.
     \end{tikzcd}
     $$
     The map from $s^{-2}\L$ to $\Xi$ is zero by the universal property of $\Xi$. Since the bottom map is an epimorphism, the left map is also zero. So $\Xi$ is curved. We set $\mathrm{Curv}(W) \coloneqq \Xi$, and it is straightforward to show that it defines a left adjoint to the forgetful functor. 
\end{proof}

\begin{corollary}
The category of curved $\operad C$-algebras is the category of algebras over idempotent monad over pdg $\operad C$-algebras which is accessible and preserves reflexive coequalisers. Hence, the category of curved $\operad C$-algebras is presentable.
\end{corollary}

\begin{proof}
Let $\L$ be a pdg $\operad C$-algebra. Its image by the functor $\mathrm{Curv}$ is also given by the reflexive coequaliser of a the pair of maps of the form
    \[
    \begin{tikzcd}
        (Z_\L \oplus \L^{\operad C})^{\operad C} \ar[rr, bend left] \ar[rr, bend right]
        && \L^{\operad C}.
        \ar[ll]
    \end{tikzcd}
    \]
This reflexive coequaliser is constructed from the one that defines the functor $\mathrm{Curv}$ and from the one that gives $\L$ as a  pdg $\C$-algebra. This monad preserves finite sifted colimits since its composition with the forgetful functor towards graded $\kk$-modules does. Indeed, we have that:

     \medskip
     
     \begin{enumerate}
         \item this forgetful functor preserves and creates finite sifted colimits;

         \medskip
         
         \item the constuction $X \mapsto X^{\operad C}$ preserves finite sifted colimits;

         \medskip
         
         \item the construction $\L \mapsto Z_\L$ preserves finite sifted colimits;

         \medskip
         
         \item then the construction $\L \mapsto \varinjlim((Z_\L \oplus \L^{\operad C})^{\operad C} \rightrightarrows \L^{\operad C}) $ preserves finite sifted colimits.
     \end{enumerate}

     \medskip
     
     Let us prove that this monad is accessible. Let $\beta$ be a small regular cardinal so that the monad $(-)^{\operad C}$ is $\beta$-accessible. To prove that the idempotent "curvature" monad  preserves $\beta$-filtered colimits, it suffices to prove that its composition with the forgetful functor towards pdg $\kk$-modules perserves these colimits. This follows from the following facts:

     \medskip
     
     \begin{enumerate}
         \item the forgetful functor preserves and creates reflexive coequalisers and $\beta$-filtered colimits;

        \medskip

         \item the constructions $\L \mapsto \L^{\operad C}$ and $\L \mapsto \operad C \comp ((\operad C \comp \L ) \oplus Z_\L)$ preserve filtered colimits;

         \medskip
         
         \item reflexive coequalisers and $\beta$-filtered colimits commute in pdg $\kk$-modules.
     \end{enumerate}
\end{proof}

Let us consider the following commutative diagram of forgetful functors:

$$
\begin{tikzcd}
    \catcurvalg{\operad C}
    \ar[r]
    & \catpdgalg{\operad C}
    \ar[r] \ar[d]
    & \catgralg{\operad C}
    \ar[d]
    \\
    & \catpdgmod{\kk}
    \ar[r]
    & \catgrmod{\kk}~.
\end{tikzcd}
$$
All the categories are presentable. Moreover all the functors are right adjoint (they are accessible and preserve limits), conservative and they also preserve finite sifted colimits. Therefore they are monadic.

\subsection{Complete algebras over the canonical ladder of a quasi-planar cooperad}
This section is the mirror of \cite[Section 4.4]{linearcoalgebras} in any characteristic for a quasi-planar cooperad ladder equipped with its canonical quasi-planar filtration. For the rest of this section, we fix a quasi-planar curved conilpotent cooperad $\C$.

\medskip

\begin{lemma}
    Let $n$ be a natural integer and let $X$ be a pdg $\kk$ module. Let us consider the inclusion $i_{n}: F_n^{\qp}  \C \rightarrowtail \C$. The natural map
    
    \[
    (\id)^{i_n}: X^{\operad C} \twoheadrightarrow X^{F_n^{\mathrm{qp}}\operad C}
    \]
    \vspace{0.1pc}
    
    is a degree-wise epimorphism. Subsequently, the forgetful functor from pdg $F_n^{\mathrm{qp}} \operad C$-algebras to pdg $\operad C$-algebras is fully faithful.
\end{lemma}

\begin{proof}
    In the category of graded $\kk$-modules, this map may be rewritten as 
    \[
    X^{\operad C_\pl} \twoheadrightarrow X^{F_n^{\mathrm{qp}}\operad C_\pl}
    \]
    Moreover, the map of graded $\mathbb N$-module $F_n^{\mathrm{qp}}\operad C_\pl \rightarrowtail \operad C_\pl$ has a left inverse. Therefore the above map has a section. It implies that it is a degree-wise epimorphism. We thus get a morphism of monads that is object-wise an epimorphism. Therefore the related right adjoint forgetful functor is fully faithful.
\end{proof}

For every natural integer $n$, we denote by 
    $$
    F^n_{\mathrm{qp}}: \catpdgalg{\operad C} \longrightarrow \catpdgalg{F_n^{\mathrm{qp}}\operad C}
    $$
    be the functor that is left adjoint to the forgetful functor. It sends every pdg $\operad C$-algebra $(B,\gamma_B,d_B)$ to the coequaliser of the following pair

\[
\begin{tikzcd}[column sep=4pc,row sep=4pc]
(\L^{\operad C})^{F_n^{\mathrm{qp}}\operad C} \arrow[r,"(\gamma_\L)^{\id}",shift right=1.1ex,swap]  \arrow[r,"(\id)^{\Delta} ~\circ ~ (\id)^{i_n}"{name=SD},shift left=1.1ex ]
&\L^{F_n^{\mathrm{qp}} \operad C}~.
\end{tikzcd}
\]

    This coequalizer is also the pushout of the span
    $$
    \begin{tikzcd}[column sep=2.5pc,row sep=2.5pc]
        \L^{\operad C}
        \ar[r,"(\id)^{i_n}",twoheadrightarrow] \ar[d,"\gamma_\L",swap]
        &\L^{F_n^{\mathrm{qp}}\operad C}
        \\
        \L~.
    \end{tikzcd}
    $$
    Both of these colimits can be computed in the category of pdg $F_n^{\mathrm{qp}}\operad C$-algebras or in the underlying category of graded $\kk$-modules.

\medskip

This gives a natural diagram 
$$
\L \longrightarrow \cdots \longrightarrow F^n_{\mathrm{qp}} \L \longrightarrow F^{n-1}_{\mathrm{qp}} \L \longrightarrow \cdots F^1_{\mathrm{qp}} \L \longrightarrow F^0_{\mathrm{qp}} \L~, 
$$

where we set $F^{-1}_{\mathrm{qp}}(\L)=0$. For every natural integer $n$ in $\mathbb{N}$, the \textit{associated graded} is given by
$$
\gr^n_{\mathrm{qp}} \L \coloneqq \mathrm{Ker}\left(F^n_{\mathrm{qp}}(\L) \longrightarrow F^{n-1}_{\mathrm{qp}}(\L)\right)~.
$$

\begin{definition}[qp-completion functor]
The \textit{qp-completion functor}, denoted by $\widehat{(-)}$, is the endofunctor of pdg $\operad C$-algebras that sends an algebra $\L$ to the limit
    $$
    \widehat \L \coloneqq \lim_{n\in \omega^\op} F^n_{\mathrm{qp}}(\L)~.
    $$
\end{definition}

\begin{definition}[qp-complete pdg $\C$-algebra]
A pdg $\operad C$-algebra $\L$ is called \textit{qp-complete} if the canonical morphism $\L \longrightarrow \widehat \L$ is an isomorphism. 

    \medskip
    
    We denote $\catpdgcompalg{\operad C}$ the full subcategory of that of pdg $\operad C$-algebras spanned by the qp-complete ones.
\end{definition}

\begin{lemma}[{After \cite[Lemma 4.22]{linearcoalgebras}}]\label{lemmaepicomplete0}
Let $f: \Lambda \longrightarrow \Gamma$ be a morphism of pdg $\operad C$-algebras. Let us suppose that for every natural integer $n$, the induced map
$$
\gr^n(f): \gr^n_{\qp} \Lambda \twoheadrightarrow \gr^n_{\qp} \Gamma
$$
is a degree-wise epimorphism. 

\medskip

\begin{enumerate}
    \item For every $n$, the map
    $$
    F^{n+1}_{\mathrm{qp}} \L \twoheadrightarrow F^{n}_{\mathrm{qp}} \L \times_{F^{n}_{\mathrm{qp}} \Gamma} F^{n+1}_{\mathrm{qp}} \Gamma
    $$
    is a degree-wise epimorphism.
    
    \medskip
    
    \item The map $\widehat f : \widehat \L \twoheadrightarrow \widehat \Gamma$ is a degree-wise epimorphism.
\end{enumerate}
\end{lemma}

\begin{proof}
Let us prove the first assertion. Let us consider the square diagram
$$
\begin{tikzcd}[column sep=2.5pc,row sep=2.5pc]
    F^{n+1}_{\mathrm{qp}} \L
    \ar[r, "p_n", twoheadrightarrow] \ar[d, "F^{n+1}_{\mathrm{qp}}(f)"']
    & F^{n}_{\mathrm{qp}} \L
    \ar[d, "F^{n}_{\mathrm{qp}}(f)"]
    \\
    F^{n+1}_{\mathrm{qp}} \Gamma
    \ar[r, "q_n"',twoheadrightarrow]
    & F^{n}_{\mathrm{qp}} \Gamma
\end{tikzcd}
$$
Given an element $(y , z) \in F^{n}_{\mathrm{qp}} \L \times_{F^{n}_{\mathrm{qp}} \Gamma} F^{n+1}_{\mathrm{qp}} \Gamma$
one can find $x'$ so that $p_n(x') = y$. Then $q_n(z - F^{n+1}_{rad}(f)(x'))=0$. In other words
$$
z - F^{n+1}_{\qp}(f)(x') \in \gr^{n+1}_{\qp}\Gamma~.
$$
Let $x''$ be one of its antecedent in $\gr^{n+1}_{\qp}\L$. Then $x = x'+x''$ is an antecedent of $(y , z)$. This proves the first point.

\medskip

Let us prove the second point. Let $z$ be a degree $k$ element of $\widehat \L$. It is a sequence
$(z_n)_{n \in \mathbb N} \in \prod_n  F^{n}_{\mathrm{qp}} \widehat \L_k$, where each $z_n$ is the image of $z_{n+1}$ via the projection $F^{n+1}_{\mathrm{qp}}(\widehat \L) \twoheadrightarrow F^{n}_{\mathrm{qp}}(\widehat \L)$.
Using 
\begin{enumerate}
    \item the fact that the map $\gr^0(f)$ 
is a degree-wise epimorphism;
 \item the first assertion of the lemma;
 \item the axiom of choice;
\end{enumerate}
one can build a sequence $x = (x_n)_{n \in \mathbb N} \in \prod_n F_{\qp}^n \L_k$, where each $x_n$ is the image of $x_{n+1}$ via the projection $F^{n+1}_{\mathrm{qp}}(\L) \twoheadrightarrow F^{n}_{\mathrm{qp}}(\L)$ and such that $F^{n}_{\qp}(f)(x_n) = z_n$. Thus $x$ is an antecedent of $z$ in $\widehat \L$.
\end{proof}

\begin{lemma}[{After \cite[Proposition 4.23]{linearcoalgebras}}]\label{lemmaepicomplete}
The completion endofunctor preserves degree-wise epimorphisms.
\end{lemma}

\begin{proof}
Let $f: \Lambda \longrightarrow \Gamma$ be a degree-wise epimorphism.
In order to show that $\widehat f$ is a degree-wise epimorphism, it suffices using Lemma \ref{lemmaepicomplete0} to prove that $\gr^n_{\qp}(f)$ is a degree-wise epimorphism for every natural integer $n$.

\medskip

Infinite products in graded $\kk$-modules are exact, the contravariant functor $[-, X]$ for every $X$ is also exact, therefore the sequence
$$
0 \longrightarrow \Lambda^{F_n^{\qp}\operad C_\pl/F_{n-1}^{\qp}\operad C_\pl}
\longrightarrow \Lambda^{F_n^{\qp}\operad C_\pl}
\longrightarrow \Lambda^{F_{n-1}^{\qp}\operad C_\pl} \longrightarrow 0
$$
is exact. Thus, the following commutative diagram of graded $\kk$-modules is a pushout
$$
\begin{tikzcd}[column sep=2.5pc,row sep=2.5pc]
    \Lambda^{F_n^{\qp}\operad C/F_{n-1}^{\qp}\operad C}
    \ar[r] \ar[d, hook]
    & 0
    \ar[d]
    \\
    \Lambda^{F_{n}^{\qp}\operad C}
    \ar[r] \ar[d]
    & \Lambda^{F_{n-1}^{\qp}\operad C}
    \ar[d]
    \\
    F^n_{\qp}\L
    \ar[r]
    & F^{n-1}_{\qp}\L.
\end{tikzcd}
$$
Therefore the kernel $\gr^n_{\qp}\L$ is the image in $F^n_{\qp}\L$ of the map $\Lambda^{F_n^{\qp}\operad C/F_{n-1}^{\qp}\operad C} \longrightarrow F^n_{\qp}\L$. Let us consider the following commutative diagram of graded $\kk$-modules 
$$
\begin{tikzcd}[column sep=2.5pc,row sep=2.5pc]
    \Lambda^{F_n^{\qp}\operad C/F_{n-1}^{\qp}\operad C}
    \ar[rr, "f^{F_n^{\qp}\operad C/F_{n-1}^{\qp}\operad C}"]
    \ar[d, two heads]
    && \Gamma^{F_n^{\qp}\operad C/F_{n-1}^{\qp}\operad C}
    \ar[d, two heads]
    \\
    \gr^n_{\qp}\L
    \ar[rr, "\gr^n_{\qp}(f)"']
    && \gr^n_{\qp}\Gamma.
\end{tikzcd}
$$
Since the top map and the right map are degree-wise epimorphisms, so is bottom map, which proves the result.
\end{proof}

\begin{lemma}[{After \cite[Proposition 4.21 and 4.24]{linearcoalgebras}}]\leavevmode 
\begin{enumerate}
\item The $X$ be a pdg module $X$, the free pdg $\operad C$-algebra $X^{\operad C}$ is qp-complete. 

\medskip

\item The natural map
    $$
    \L \twoheadrightarrow \widehat \L 
    $$
    is an degree-wise epimorphism for every dg $\operad C$-algebra $\L$.
    
\end{enumerate}
\end{lemma}

\begin{proof}
For every free pdg $\operad C$-algebra $X^{\operad C}$ on a pdg module $X$, one has
 $F^n_{\mathrm{qp}} X^{\operad C} = X^{F_n^{\mathrm{qp}} \operad C}$ and that it is complete:
    $$
    \widehat{X^{\operad C}} =  \lim_{n \in \omega^\op} X^{F_n^{\mathrm{qp}} \operad C} \cong X^{\operad C}.
    $$
    Now, let $\L$ be any pdg $\operad C$-algebra. We consider the following commutative square
    $$
    \begin{tikzcd}
        \L^{\operad C}
        \ar[r] \ar[d]
        & \widehat{\L^{\operad C}}
        \ar[d]
        \\
        \L
        \ar[r]
        & \widehat \L.
    \end{tikzcd}
    $$
The top horizontal map is an isomorphism and the right vertical map is a degree-wise epimorphism by Lemma \ref{lemmaepicomplete}. The map from $\L^{\operad C}$ to $\L$ is also epimorphism. Therefore so is the map $\L \longrightarrow \widehat \L$.
\end{proof}

\begin{lemma}
    The canonical map 
    \[
    \widehat \L \longrightarrow \widehat{\widehat{(\L)}}
    \]
    is an isomorphism. Therefore the completion functor equipped with the natural map $\L \longrightarrow \widehat \L$ is an idempotent monad in the category of pdg $\C$-algebras.
\end{lemma}

\begin{proof}
    For every natural integer $n$, both maps 
    $\L \longrightarrow \widehat \L \longrightarrow F_n^{\mathrm{qp}} \L$ are degree-wise epimorphisms. Applying the left adjoint functor $F^n_{\mathrm{qp}}$ which preserves degree-wise epimorphisms, one gets a diagram 
    $$
    \begin{tikzcd}
        F_n^{\mathrm{qp}} \L
        \ar[r] \ar[rr, bend left, "\id"]
        & F_n^{\mathrm{qp}} \widehat \L
        \ar[r]
        & F_n^{\mathrm{qp}} \L.
    \end{tikzcd}
    $$
    The first map is both a degree-wise epimorphism and a degree-wise monomorphism. Hence it is an isomorphism and so is the second map. Taking the limit over $n \in \omega^\op$, one gets the required isomorphism. 
\end{proof}

\begin{proposition}
    The category of qp-complete pdg $\operad C$-algebra is presentable.    
\end{proposition}

\begin{proof}
It is a reflexive full subcategory of the category of pdg $\operad C$-algebras which is presentable. The reflector is the completion $\L \mapsto \widehat \L$. One can find a small regular cardinal $\beta$ so that every functor $F^n_{\mathrm{qp}}$ is $\beta$-accessible and so that $\beta$-filtered colimits commute with limits of diagrams indexed by $\omega^{op}$. Then, the completion functor is $\beta$-accessible. Hence the category of qp-complete pdg $\operad C$-algebras is presentable. 
\end{proof}

\begin{remark}
    The category of qp-complete pdg $\operad C$-algebra is the homotopy limit of the (pseudo) diagram of categories 
    \[
    \cdots \longrightarrow \catpdgalg{F_n^{\mathrm{qp}} \operad C} \longrightarrow \catpdgalg{F_{n-1}^{\mathrm{qp}} \operad C} \longrightarrow \cdots \longrightarrow \catpdgalg{F_1^{\mathrm{qp}} \operad C} \longrightarrow \catpdgalg{F_0^{\mathrm{qp}} \operad C}.
    \]
    It is presentable as a small limit of presentable categories and left adjoint functors. 
\end{remark}

\begin{remark}
The analogue definitions and results are valid for dg $\C$-algebras or graded $\C$-algebras as well. 
\end{remark}


\subsection{Categorical properties of qp-complete curved algebras over a quasi-planar cooperad}

\begin{proposition}
    The qp-completion idempotent monad on pdg $\operad C$-algebras preserves curved algebras. Moreover, the restricted monad (to curved algebras) is accessible. Hence, the category of qp-complete curved $\operad C$-algebras is presentable and the forgetful functor towards curved $\operad C$-algebras is accessible.
\end{proposition}

\begin{proof}
For every curved $\operad C$-algebra $\L$, $\widehat \L$ is also curved since the map $\L \longrightarrow \widehat \L$ is a degree-wise epimorphism. Therefore, the qp-completion monad restricts to curved $\C$-algebras. Let $\beta$ be regular small cardinal such that
    \begin{enumerate}
        \item the forgetful functor $\catcurvalg{\C} \longrightarrow \catpdgalg{\C}$ is a $\beta$-accessible functor between $\beta$-accessible categories;
        \item the completion monad on $\catpdgalg{\C}$ is $\beta$-accessible.
    \end{enumerate}
    Then, the restriction of this monad to curved algebras is also $\beta$-accessible.
\end{proof}

\begin{remark}
Let us consider the following
commutative diagram of functors.
$$
\begin{tikzcd}
    \catcurvcompalg{\operad C}
    \ar[r] \ar[d]
    & \catpdgcompalg{\operad C}
    \ar[r] \ar[d]
    & \catgrcompalg{\operad C}
    \ar[d]
    \\
    \catcurvalg{\operad C}
    \ar[r]
    & \catpdgalg{\operad C}
    \ar[r]
    & \catgralg{\operad C}
\end{tikzcd}
$$
The vertical arrows are monadic and accessible but they do not preserve reflexive coequalisers in general; otherwise all $\C$-algebras would be qp-complete.    
\end{remark}


\subsection{Naturality of (co)algebras of a dg operad.}
The definitions of algebras and coalgebras are compatible with morphisms of operads. This will mean that any such morphism induces an adjunction between the appropriate categories. Let $f: \operad P \longrightarrow \operad Q$ be a morphism of dg operads.

\medskip

\textbf{Naturality of algebras.} The morphism $f$ induces a morphism of monads $\operad P \comp - \longrightarrow \operad Q \comp -$, which yields an adjunction 
\[
\begin{tikzcd}[column sep=5pc,row sep=3pc]
          \mathsf{dg}~\operad P\text{-}\mathsf{alg} \arrow[r, shift left=1.1ex, "f_{!}"{name=F}] & \mathsf{dg}~ \mathcal{Q}\text{-}\mathsf{alg}~, \arrow[l, shift left=.75ex, "f^*"{name=U}]
            \arrow[phantom, from=F, to=U, , "\dashv" rotate=-90]
\end{tikzcd}
\]
since the category of dg $\operad Q$-algebras is cocomplete, by the adjoint lifting theorem. See Appendix \ref{Appendix: Adjoint Lifting Theorem} for more details. Since the monad $\operad Q \comp -$ commutes with sifted colimits, the left adjoint functor $f_!$ sends a dg $\operad P$-algebra $A$ to the reflexive coequaliser of the pair

\[
\begin{tikzcd}[column sep=3.5pc,row sep=0pc]
        \operad Q \comp (\operad P \comp A) \ar[rr, shift left = 0.7ex, "\id ~ \circ ~ \gamma_A"]
        \ar[rr, shift right = 0.7ex, "\gamma_{\operad Q}~ (f~  \circ ~ \id)"'] 
        && \operad Q \comp A~,
\end{tikzcd}
\]
\vspace{0.1pc}

between the structural map $\gamma_A$ of $A$ and the map induced by the morphism $f$ composed with the monad structure $\gamma_{\operad Q} \comp -$ of $\operad Q \comp -$. This is a coequalizer both in dg $\operad Q$-algebras and in graded $\kk$-modules.

\medskip

\textbf{Naturality of coalgebras.} The morphism $f$ also induces a morphism of comonads $L^{\operad P} \longrightarrow L^{\operad Q}$, which yields an adjunction
\[
\begin{tikzcd}[column sep=5pc,row sep=3pc]
          \mathsf{dg}~\operad Q\text{-}\mathsf{cog} \arrow[r, shift left=1.1ex, "f^*"{name=F}] & \mathsf{dg}~ \mathcal{P}\text{-}\mathsf{cog}~,
          \arrow[l, shift left=.75ex, "f_!"{name=U}]
            \arrow[phantom, from=F, to=U, , "\dashv" rotate=-90]
\end{tikzcd}
\]
since the category of dg $\operad Q$-coalgebras is complete, again using the adjoint lifting theorem. Since the comonad $L^{\operad Q}$ commutes with \textit{finite} cosifted limits, the functor $f_!$ sends a dg $\operad P$-coalgebra $V$ to the coreflexive equaliser of the pair

\[
\begin{tikzcd}[column sep=3.5pc,row sep=0pc]
        L^{\operad Q} V  \ar[rr, shift left = 0.7ex, "L^{f}(\id)~ \omega_{\operad Q}(V) "]
        \ar[rr, shift right = 0.7ex, "\gamma_{\operad Q}~ (f~  \circ ~ \id)"'] 
        && L^{\operad Q}L^{\operad P}V~,
\end{tikzcd}
\]
\vspace{0.1pc}

between the structure map $\Delta_V$ of $V$ and the map given by the comonad structure $\omega_{\operad Q}$ of $L^{\operad Q}$ composed with the morphism induced by $f$. This is an equalizer both in dg $\operad Q$-coalgebra and in graded $\kk$-modules.

\medskip

\subsection{Naturality of (co)algebras over a cooperad.} 
\label{section : Naturality of (co)algebras over a cooperad.}
The definitions of algebras and coalgebras are compatible with morphisms of of cooperads. This will mean that any such morphism induces an adjunction between the appropriate categories. Let $\operad C, \operad D$ be two conilpotent curved cooperads whose underlying graded conilpotent cooperads are the image through $- \otimes \mathbb S$ of a graded planar conilpotent cooperads respectively $\operad C_\pl, \operad D_\pl$. Let $g: \operad C \longrightarrow \operad D$ be a morphisms of curved cooperads.

\medskip

\textbf{Naturality of coalgebras.} The map $g$ induces a morphism of comonads on the category of pdg $\kk$-modules $\operad C \comp - \longrightarrow \operad D \comp -$, which yields an adjunction 
\[
\begin{tikzcd}[column sep=5pc,row sep=3pc]
          \mathsf{pdg}~\operad D\text{-}\mathsf{cog} \arrow[r, shift left=1.1ex, "g^!"{name=F}] & \mathsf{pdg}~ \mathcal{C}\text{-}\mathsf{cog}~,
          \arrow[l, shift left=.75ex, "g_*"{name=U}]
            \arrow[phantom, from=F, to=U, , "\dashv" rotate=90]
\end{tikzcd}
\]
by the adjoint lifting theorem. Since the comonad $\operad D \comp -$ preserves finite sifted limits, the functor $g^!$ sends a pdg $\operad D$-coalgebra $W$ to the coreflexive equaliser of the pair

\[
\begin{tikzcd}[column sep=3.5pc,row sep=0pc]
        \operad C \comp W \ar[rr, shift left = 0.7ex, "\id ~ \circ ~ \Delta_W "]
        \ar[rr, shift right = 0.7ex, "(g~  \circ ~ \id) ~ \Delta_\C(W)"'] 
        && \operad C \comp \operad D \comp W~,
\end{tikzcd}
\]
between the structure map $\Delta_W$ of $W$ and the map given by the comonad structure $\Delta_\C(-)$ of $\C \comp -$ composed with the morphism induced by $g$. The existence of this adjunction and the fact that this equaliser can be computed in graded $\kk$-modules follows from Lemma \ref{lemma: quasi-planar functors preserve more stuff}.

\begin{proposition}\label{proposition : naturality of curved coalgebras}
The adjunction $g_* \dashv g^!$ restrict to an adjunction 
    \[
\begin{tikzcd}[column sep=5pc,row sep=3pc]
          \mathsf{curv}~\operad D\text{-}\mathsf{cog} \arrow[r, shift left=1.1ex, "g^!"{name=F}] & \mathsf{curv}~ \mathcal{C}\text{-}\mathsf{cog}~,
          \arrow[l, shift left=.75ex, "g_*"{name=U}]
            \arrow[phantom, from=F, to=U, , "\dashv" rotate=90]
\end{tikzcd}
\]
relating curved $\operad D$-coalgebras and curved $\operad D$-coalgebras.
\end{proposition}

\begin{proof}
The fact that $g_*$ sends curved objects to curved objects is clear. Now, given a curved $\operad D$-coalgebra $W$, the squared coderivation $d^2$ on $\operad C \comp W$ is 
   \begin{align*}
d^2 &= d_{\operad C}^2 \comp \id_W + \id_{\operad C} \comp\shuffle(\id_W ,d_W^2)
\\
&= -((\theta_{\operad C} \comp \id_{\operad C} )\Delta_{\operad C}) \comp \id_W 
+
((\id_{\operad C} \comp \shuffle(\epsilon_\C,\theta_\C))\Delta_{\operad C}) \comp \id_W 
-\id_{\operad C} \comp \shuffle(\id_W , (\theta_{\operad D} \comp \id_W)\Delta_W )~,
   \end{align*}
   where $\upsilon_\C$ denotes the counit of the cooperad $\C$. If we denote $j$ the inclusion $g^!W \hookrightarrow \operad C \comp W$, then 
   the pre-composition of $j$ with the sum
   $$
   ((\id_{\operad C} \comp \shuffle(\epsilon_\C,\theta_\C))\Delta_{\operad C}) \comp \id_W 
-\id_{\operad C} \comp\shuffle(\id_W , (\theta_{\operad D} \comp \id_W)\Delta_W )
   $$
   is zero. Thus,
   $$
   j d^2 = d^2 j = -((\theta_{\operad C} \comp \id_{\operad C} )\Delta_{\operad C} \comp \id_W)j
   = - j ((\theta_{\operad C} \comp \id_{g^!_W} )\Delta_{g^!_W})~.
   $$
   Since $j$ is injective, $g^!_W$ is curved.
\end{proof}

\begin{remark}
Actually, when dealing with coalgebras, we do not need $\operad C$ to proceed from a planar cooperad. We will need it in the context of algebras, in order to be able to talk about complete $\C$-algebras.
\end{remark}

\medskip

\textbf{Naturality of algebras.} The morphism $g$ induces a morphism of monads on pdg $\kk$-modules $(-)^{\operad D} \longrightarrow (-)^{\operad C}$, which yields an adjunction 
\[
\begin{tikzcd}[column sep=5pc,row sep=3pc]
          \mathsf{pdg}~\operad C\text{-}\mathsf{alg}
          \arrow[r, shift left=1.1ex, "g_{*}"{name=F}]
          & \mathsf{pdg}~ \mathcal{D}\text{-}\mathsf{alg}~, \arrow[l, shift left=.75ex, "g^!"{name=U}]
            \arrow[phantom, from=F, to=U, , "\dashv" rotate=90]
\end{tikzcd}
\]
by the adjoint lifting theorem. The functor $g^!$ sends a curved $\operad D$-algebra $\L$ to the reflexive coequaliser of the pair

\[
\begin{tikzcd}[column sep=3.5pc,row sep=0pc]
        (\L^{\operad D})^{\operad C}   \ar[rr, shift left = 0.7ex, "(\gamma_\L)^{(\id)} "]
        \ar[rr, shift right = 0.7ex, "(\L)^{\Delta_\C}~ (\L)^{g}"'] 
        && \L^{\operad C}~,
\end{tikzcd}
\]
\vspace{0.1pc}

between the structure map $\gamma_\L$ of $\L$ and the map induced by the morphism $g$ composed with the monad structure $(-)^{\Delta_\C}$ of $(-)^{\C}$. The existence of this adjunction and the fact that this coequaliser can be computed in graded $\kk$-modules follows from Lemma \ref{lemma: quasi-planar functors preserve more stuff}.

\begin{proposition}\label{proposition : naturality of curved algebras}
The adjunction $g^! \dashv g^{*}$ restrict to an adjunction 
\[
\begin{tikzcd}[column sep=5pc,row sep=3pc]
          \mathsf{curv}~\operad C\text{-}\mathsf{alg}
          \arrow[r, shift left=1.1ex, "g^{*}"{name=F}]
          & \mathsf{curv}~ \mathcal{D}\text{-}\mathsf{alg}~, \arrow[l, shift left=.75ex, "g^!"{name=U}]
            \arrow[phantom, from=F, to=U, , "\dashv" rotate=90]
\end{tikzcd}
\]
relating curved $\operad D$-algebras and curved $\operad D$-algebras.
\end{proposition}

\begin{proof}
This follows from dual arguments as those used to prove Proposition \ref{proposition : naturality of curved algebras}. The fact that $g_{*}$ sends curved algebras to curved algebras is clear. Now, given a curved $\operad D$-algebra $\L$, the squared coderivation $d^2$ on $\L^\C \simeq \L^{\C_\pl}$ is
   \begin{align*}
    d^2 &= -\L^{d_{\operad C}^2} +  \shuffle(\id_\L ,d_\L^2)^{\C}~.      
   \end{align*}
   It is thus the sum of the three maps
   \begin{align*}
       \delta^{(1)}:\L^{\C_\pl}& \xrightarrow{\L^{\theta_\C \comp_\pl \id}}
       \L^{\C_\pl \comp_\pl \C_\pl} \longrightarrow \L^{\C_\pl};
       \\
       \delta^{(2)}:\L^{\C_\pl}& \xrightarrow{-\L^{\id_{\operad C} \comp_\pl \shuffle (\upsilon_\C, \theta_\C)}}
       \L^{\C_\pl \comp_\pl \C_\pl} \longrightarrow \L^{\C_\pl};
       \\
       \delta^{(3)}:\L^{\C_\pl}& \xrightarrow{\shuffle(\L^{\upsilon_{\operad D}},\L^{\theta_\C})^{\C}}
       (\L^{\operad D_\pl})^{\C_\pl} \xrightarrow{\gamma_\L^\C} \L^{\C_\pl},
   \end{align*}
where $\upsilon_\C, \upsilon_{\operad D}$ are the counits of the cooperads $\C, \operad D$. Since the curvature on cooperads in zero on arities different from $1$ and since the morphism of cooperads $g$ commutes with the curvatures and the counits, the map
$$
\L^{\id_{\operad C} \comp_\pl \shuffle (\upsilon_\C, \theta_\C)} : \L^{\C_\pl } \longrightarrow \L^{\C_\pl \comp_\pl \C_\pl}
$$
is equal to the composition
$$
\L^{\C_\pl}
\xrightarrow{\shuffle(\L^{\upsilon},\L^{\theta})^{\C}}
(\L^{\operad D_\pl})^{\C_\pl}
\xrightarrow{(\L^{g})^{\C_\pl}}
(\L^{\operad C_\pl})^{\C_\pl}
\hookrightarrow \L^{\C_\pl \comp_\pl \C_\pl}.
$$
As a consequence, if we denote $p$ the projection $p : \L^{\C_\pl} \twoheadrightarrow g^!(\L)$, then
$$
p (\delta^{(2)} + \delta^{(3)}) = 0~.
$$
The map 
$$
\L^{\theta_\C \comp_\pl \id} : \L^{\C_\pl} \longrightarrow \L^{\C_\pl \comp_\pl \C_\pl}
$$
is equal to the composition
$$
\L^{\C_\pl} \xrightarrow{(\L^{\C_\pl})^{\theta_\C}} (\L^{\operad C_\pl})^{\C_\pl}
\hookrightarrow \L^{\C_\pl \comp_\pl \C_\pl}~.
$$
Thus $\delta^{(1)} = \gamma_{\L^{\C_\pl}} (\L^{\C_\pl})^{\theta_\C}$
and therefore
$$
d^2 p = p d^2 = p \delta^{(1)}
= p\gamma_{\L^{\C_\pl}} (\L^{\C_\pl})^{\theta_\C}
= \gamma_{g^!(\L)} (g^!(\L))^{\theta_\C} p ~.
$$
Since $p$ is surjective, this entails that
$$
d^2 = \gamma_{g^!(\L)} (g^!(\L))^{\theta_\C}~,
$$
which proves the result.
\end{proof}

\begin{proposition}\label{proposition : restriction ra complete}
Let $g: \operad C \longrightarrow \operad D$ be a quasi-planar morphism of quasi-planar curved conilpotent cooperads. Then the functor $g_{*}$ sends qp-complete $\operad C$-algebras to qp-complete $\operad D$-algebras.
\end{proposition}

\begin{proof}
    Let us consider the following commutative square diagram of quasi-planar conilpotent curved cooperads 
    $$
    \begin{tikzcd}[column sep=2.5pc,row sep=2.5pc]
        F_n^{\qp}\C
        \ar[r, "F_n^{\qp}g"] \ar[d, "i"]
        & F_n^{\qp}\operad D
        \ar[d, "j"]
        \\
        \C
        \ar[r, "g"']
        & \operad D.
    \end{tikzcd}
    $$
    for every $n \geq 0$. Let $\L$ be a qp-complete pdg $\operad C$-algebra. The
    $\operad D$-algebra $g_{*}(F^n_{\qp}\L)$ rewrites as
    $$
    g_{*}(F^n_{\qp}\L) = g_{*} i^* i_! (\L)
    \cong j^* (F_n^{\qp}g)^* i_! (\L)~.
    $$
    Thus $g_{*}(F^n_{\qp}\L)$ is qp-complete as it is the image of a pdg $F_n^{\qp}\operad D$-algebra. Since $g_{*}$ preserves limits and since $\L$ is complete, the map
    $$
    g_{*}(\L) \longrightarrow \lim_{n \in \omega^\op} g_{*}(F^n_{\qp}\L)
    $$
    is an isomorphism. So $g_{*}(\L)$ is complete as it is the limit of a diagram of complete algebras.
\end{proof}

\begin{proposition}
Let $g: \operad C \longrightarrow \operad D$ be a quasi-planar morphism of curved cooperads. There is an adjunction 
\[
\begin{tikzcd}[column sep=5pc,row sep=3pc]
          \catcurvcompalg{\operad C}
          \arrow[r, shift left=1.1ex, "g^{*}"{name=F}]
          & \catcurvcompalg{\operad D}~.
          \arrow[l, shift left=.75ex, "\widehat{g}^!"{name=U}]
            \arrow[phantom, from=F, to=U, , "\dashv" rotate=90]
\end{tikzcd}
\]
between qp-complete curved $\C$-algebras and qp-complete curved $\D$-algebras, where the left adjoint $\widehat{g}^!$ is given by $g^!$ post-composed with the completion functor of $\C$-algebras.
\end{proposition}

\begin{proof}
Follows from Proposition \ref{proposition : naturality of curved algebras}, using the fact that the completion functor of $\C$-algebras preserves curved objects.
\end{proof}

\medskip

\textbf{Embedding of cooperads.} Let us assume that $g: \operad C \longrightarrow \operad D$ is a arity-wise degree-wise injection. Since the underlying graded $\mathbb S$-module of $\C$ is cofree and thus injective, we get a section of graded $\mathbb S$-modules $p: \operad D \longrightarrow \C$ that is left inverse to $g$.

\medskip

As a consequence, the morphism of comonads on pdg $\kk$-modules $\operad C \comp - \longrightarrow \operad D \comp -$ is object-wise a monomorphism and the morphism of monads on pdg $\kk$-modules $(-)^{\operad D} \longrightarrow (-)^{\operad C}$ is object-wise an epimorphism. This implies that the functors

\begin{align*}
    g_* :& \catpdgcog{\operad C} \longrightarrow \catpdgcog{\operad D}
    \\
    \\
    g_{*} :& \catpdgalg{\operad C} \longrightarrow \catpdgalg{\operad D}
\end{align*}

are fully faithful. Therefore the first $g_*$ is the forgetful functor related to an idempotent comonad and the second $g_{*}$ is the forgetful functor related to an idempotent monad. In this case, their respective adjoint functors admit an easy description in terms of pullbacks and pushouts.

\begin{lemma}
For every pdg $\operad D$-coalgebra $(W,\Delta_W,d_W)$, $g^!(W)$ is given by the following pullback 
$$
\begin{tikzcd}[column sep=3pc,row sep=3pc]
    g^!(W) \ar[r] \ar[d, hook] \arrow[dr, phantom, "\lrcorner", very near start]
    & \operad C \comp W
    \ar[d, "g~ \comp~ \id_W", hook]
    \\
    W \ar[r,"\Delta_W"]
    & \operad D \comp W.
\end{tikzcd}
$$
in the category of pdg $\kk$-modules.
\end{lemma}

\begin{proof}
Follows directly from the dual of Proposition \ref{propadjointliftingepi}.
\end{proof}

\begin{lemma}
For every pdg $\operad D$-algebra $(\L,\gamma_\L,d_\L)$, $g^!(\L)$ is given by the following pushout 
$$
\begin{tikzcd}[column sep=3pc,row sep=3pc]
    \L^{\operad D} \ar[d,"\gamma_\L",swap] \ar[r, "\L^g",twoheadrightarrow] \arrow[dr, phantom, "\ulcorner", very near end]
    & \L^{\operad C}
    \ar[d]
    \\
    \L \ar[r,twoheadrightarrow]
    & g^!(\L).
\end{tikzcd}
$$
in the category of pdg $\kk$-modules.
\end{lemma}

\begin{proof}
Follows directly from Proposition \ref{propadjointliftingepi}.
\end{proof}


\subsection{The Bar-Cobar adjunction}\label{subsection: bar-cobar adjunction}
Koszul duality between dg operads and conilpotent curved cooperads is encoded by the notion of a curved twisting morphism, see \cite{grignou2021algebraic} or \cite{curvedcalculus}. Any curved twisting morphism $\alpha: \C \longrightarrow \operad P$ between a conilpotent curved cooperad and a dg operad induces a bar-cobar adjunction between the categories of dg $\operad P$-algebras and of curved $\C$-coalgebras.

\medskip

\textbf{The bar construction relative to $\alpha$.} Given a dg $\operad P$-algebra $(A,\gamma_A,d_A)$, one can construct a curved $\operad C$-coalgebra $\mathrm{B}_{\alpha} A$ called the bar construction. The underlying graded $\C$-coalgebra of $\mathrm{B}_{\operad C} A$ is given by $\operad C \comp A$. It is equipped with the unique coderivation whose projection onto the generators is the sum of the following maps
\[
\begin{tikzcd}[column sep=3pc,row sep=0pc]
d_1: \operad C \comp A   \arrow[r,twoheadrightarrow]
&A \arrow[r,"d_A"]
&A~,
\end{tikzcd}
\]

\[
\begin{tikzcd}[column sep=3pc,row sep=0pc]
d_2: \operad C \comp A   \arrow[r,"\alpha~\comp~\id"]
&\operad P \comp  A \arrow[r,"\gamma_A"]
&A~.
\end{tikzcd}
\]

One can check that the pdg $\C$-coalgebra $\mathrm{B}_{\alpha} A$ is in fact a curved $\C$-coalgebra.

\medskip

\textbf{The cobar construction relative to $\alpha$.} Given a curved $\operad C$-coalgebra $(W,\Delta_W,d_W)$, one can construct a dg $\operad P$-algebra $\Omega_{\alpha} W$ called de cobar construction. The underlying graded $\operad P$-algebra is given by $\operad P\comp W$. It is equipped with the unique derivation whose restriction to the generators is the sum of the maps
\[
\begin{tikzcd}[column sep=3pc,row sep=0pc]
d_1: W   \arrow[r,"d_W"]
&W \arrow[r,rightarrowtail]
&\operad P \comp W~,
\end{tikzcd}
\]

\[
\begin{tikzcd}[column sep=3pc,row sep=0pc]
d_2: W  \arrow[r,"(-1).(-)"]
&W \arrow[r,"\Delta_W"]
&\operad C \comp W \arrow[r,"\alpha~ \comp~ \id"]
& \operad P \comp W~.
\end{tikzcd}
\]

One can check that the pdg $\operad P$-algebra $\Omega_{\alpha} W$ is in fact a dg $\operad P$-algebra.

\begin{definition}[Twisting morphism]
Let $(W,\Delta_W,d_W)$ be a curved $\C$-coalgebra and let $(A,\gamma_A,d_A)$ be a dg $\operad P$-algebra. A \textit{twisting morphism relative to $\alpha$} is the data of a graded morphism $\lambda: W \longrightarrow A$ which satisfies the following equation
\[
\gamma_A~(\alpha \circ \lambda)~\Delta_W + d_A~\lambda - \lambda~d_W = 0~.
\]

We denote the set of twisting morphisms between $W$ and $A$ by $\mathrm{Tw}^{\alpha}(W,A)$. 
\end{definition}

\begin{proposition}
Let $\alpha: \C \longrightarrow \operad P$ be a curved twisting morphism between a conilpotent curved cooperad $\C$ and a dg operad $\operad P$. There are natural isomorphisms

\[
\mathrm{Hom}_{\mathsf{dg}~ \mathcal{P}\text{-}\mathsf{alg}}(\Omega_\alpha W,A) \cong \mathrm{Tw}^{\alpha}(W,A) \cong \mathrm{Hom}_{\mathsf{curv}~\C\text{-}\mathsf{cog}}(W,\mathrm{B}_{\alpha}A)~,
\]
\vspace{0.1pc}

for any curved $\C$-coalgebra $W$ and any dg $\operad P$-algebra $A$. This gives the bar-cobar adjunction relative to $\alpha$
\[
\begin{tikzcd}[column sep=5pc,row sep=3pc]
          \mathsf{curv}~\C\text{-}\mathsf{cog} \arrow[r, shift left=1.1ex, "\Omega_{\alpha}"{name=F}] & \mathsf{dg}~ \mathcal{P}\text{-}\mathsf{alg}, \arrow[l, shift left=.75ex, "\mathrm{B}_{\alpha}"{name=U}]
            \arrow[phantom, from=F, to=U, , "\dashv" rotate=-90]
\end{tikzcd}
\]

between the categories of dg $\operad P$-algebras and the category of curved $\C$-coalgebras. Furthermore, 
\end{proposition}

\begin{proof}
These constructions can be found in \cite[Section 4.3]{unitalalgebras}. The case where $\C$ is a dg cooperad can be found in \cite[Section 11]{LodayVallette}. 
\end{proof}

\subsection{The complete Bar-Cobar adjunction}
Any curved twisting morphism $\alpha: \C \longrightarrow \operad P$ between a conilpotent curved cooperad and a dg operad also induces a complete bar-cobar adjunction between the categories of dg $\operad P$-coalgebras and of curved $\C$-algebras.

\medskip

\textbf{The complete cobar construction.} Given a dg $\operad P$-coalgebra $(V,\Delta_V,d_V)$, one can construct a curved $\operad C$-algebra $\widehat{\Omega}_{\alpha} V$ called the complete cobar construction. The underlying graded $\C$-algebra is given by $V^{\operad C}$. It is equipped with the unique derivation whose restriction to the generators is the sum of the maps
\[
\begin{tikzcd}[column sep=3pc,row sep=0pc]
d_1: V   \arrow[r,"d_V"]
&V \arrow[r,rightarrowtail]
&V^{\operad C}~,
\end{tikzcd}
\]

\[
\begin{tikzcd}[column sep=3pc,row sep=0pc]
d_2: V  \arrow[r,"(-1).(-)"]
&V \arrow[r,"\Delta_V"]
&V^{\operad P} \arrow[r,"(\id)^{\alpha}"]
&V^{\operad C}~.
\end{tikzcd}
\]

One can check that the pdg $\operad C$-algebra $\widehat{\Omega}_{\alpha} V$ is in fact a qp-complete curved $\operad C$-algebra.

\medskip

\textbf{The complete bar construction.} Given a curved $\operad C$-algebra $(\L,\gamma_\L,d_\L)$, one can construct a dg $\operad P$-coalgebra $\widehat{\mathrm{B}}_\alpha \L$ called the complete bar construction. The underlying graded $\operad P$-coalgebra is given by $L^{\operad P} \L$. It is equipped with the unique coderivation whose projection onto the generators is the sum of the maps
\[
\begin{tikzcd}[column sep=3pc,row sep=0pc]
d_1: L^{\operad P}\L   \arrow[r,twoheadrightarrow]
&\L \arrow[r,"d_\L"]
&\L~,
\end{tikzcd}
\]

\[
\begin{tikzcd}[column sep=3pc,row sep=0pc]
d_2: L^{\operad P} \L  \arrow[r,rightarrowtail]
& \L^{\operad P} \arrow[r,"(\id)^{\alpha}"]
&\L^{\operad C} \arrow[r,"\gamma_B"]
&\L~.
\end{tikzcd}
\]

\begin{remark}
The cofree $\operad P$-coalgebra functor $L^{\operad P}$ behaves well with respect to coderivations, which can be induced simply by a graded map onto the cogenerators. We refer to \cite[Section 6.5]{linearcoalgebras} for these type of results. 

\end{remark}

\begin{definition}[Twisting morphism]
Let $(\L,\gamma_\L,d_\L)$ be a curved $\C$-algebra and let $(C,\Delta_C,d_C)$ be a dg $\operad P$-coalgebra. A \textit{twisting morphism relative to $\alpha$} is the data of a graded morphism $\lambda: C \longrightarrow \L$ which satisfies the following equation
\[
\gamma_\L~(\lambda)^{\alpha}~\Delta_C + d_\L~\lambda - \lambda~d_C = 0~.
\]

We denote the set of twisting morphisms between $C$ and $\L$ by $\mathrm{Tw}^{\alpha}(C,\L)$. 
\end{definition}

\begin{proposition}
Let $\alpha: \C \longrightarrow \operad P$ be a curved twisting morphism between a conilpotent curved cooperad $\C$ and a dg operad $\operad P$. There are natural isomorphisms

\[
\mathrm{Hom}_{\mathsf{dg}~ \mathcal{P}\text{-}\mathsf{coalg}}\left(C,\widehat{\mathrm{B}}_{\alpha}\L \right) \cong \mathrm{Tw}^{\alpha}(C,\L) \cong \mathrm{Hom}_{\mathsf{curv}~\C\text{-}\mathsf{alg}}\left(\widehat{\Omega}_{\alpha}C,\L\right)~,
\]
\vspace{0.1pc}

for any curved $\C$-algebra $\L$ and any dg $\operad P$-coalgebra $C$. This gives the complete bar-cobar adjunction relative to $\alpha$
\[
\begin{tikzcd}[column sep=5pc,row sep=3pc]
          \mathsf{dg}~\operad P\text{-}\mathsf{cog} \arrow[r, shift left=1.1ex, "\widehat{\Omega}_{\alpha}"{name=F}] & \catcurvalg{\operad C} \arrow[l, shift left=.75ex, "\widehat{\mathrm{B}}_{\alpha}"{name=U}]
            \arrow[phantom, from=F, to=U, , "\dashv" rotate=-90]
\end{tikzcd}
\]

between the categories of dg $\operad P$-coalgebras and the category of curved $\C$-algebras.
\end{proposition}

\begin{proof}
These constructions can be found in \cite[Section 8]{linearcoalgebras}.
\end{proof}

\begin{remark}
It is also an adjunction 
\[
\begin{tikzcd}[column sep=5pc,row sep=3pc]
          \mathsf{dg}~\operad P\text{-}\mathsf{cog} \arrow[r, shift left=1.1ex, "\widehat{\Omega}_{\alpha}"{name=F}] & \catcurvcompalg{\operad C}~, \arrow[l, shift left=.75ex, "\widehat{\mathrm{B}}_{\alpha}"{name=U}]
            \arrow[phantom, from=F, to=U, , "\dashv" rotate=-90]
\end{tikzcd}
\]
since $\widehat{\Omega}_{\alpha}$ naturally lands in qp-complete curved $\C$-algebras.
\end{remark}

\subsection{Compatibility with the morphisms of (co)operads.} Let $f: \operad P \longrightarrow \operad Q$ be a morphims of dg operads, let $g: \operad C \longrightarrow \operad D$ be a morphism of conilpotent curved cooperads, and let $\alpha: \C \longrightarrow \operad P$ and $\beta: \operad D \longrightarrow \operad Q$ be two curved twisting morphisms, such that the following diagram commutes:

\[
\begin{tikzcd}[column sep=3pc,row sep=3pc]
\C \arrow[r,"\alpha"] \arrow[d,"g",swap] 
&\mathcal{P} \arrow[d,"f"]\\
\mathcal{D} \arrow[r,"\beta"]
&\mathcal{Q}~.
\end{tikzcd}
\]

Then the bar-cobar adjunctions are compatible with the natural adjunctions induced by the morphisms, meaning that the following square of adjunctions 

\[
\begin{tikzcd}[column sep=5pc,row sep=5pc]
\mathsf{curv}~\mathcal{C}\text{-}\mathsf{cog} \arrow[r,"\Omega_\alpha"{name=B},shift left=1.1ex] \arrow[d,"g^! "{name=SD},shift left=1.1ex ]
&\mathsf{dg}~\mathcal{P}\text{-}\mathsf{alg} \arrow[d,"f^*"{name=LDC},shift left=1.1ex ] \arrow[l,"\mathrm{B}_\alpha"{name=C},,shift left=1.1ex]  \\
\mathsf{curv}~\mathcal{D}\text{-}\mathsf{cog} \arrow[r,"\Omega_\beta "{name=CC},shift left=1.1ex]  \arrow[u,"g_*"{name=LD},shift left=1.1ex ]
&\mathsf{dg}~\mathcal{Q}\text{-}\mathsf{alg} \arrow[l,"\mathrm{B}_\beta"{name=CB},shift left=1.1ex] \arrow[u,"f_!"{name=TD},shift left=1.1ex] \arrow[phantom, from=SD, to=LD, , "\dashv" rotate=0] \arrow[phantom, from=C, to=B, , "\dashv" rotate=-90]\arrow[phantom, from=TD, to=LDC, , "\dashv" rotate=0] \arrow[phantom, from=CC, to=CB, , "\dashv" rotate=-90]
\end{tikzcd}
\] 

is commutative. The complete bar-cobar adjunctions are also compatible with the natural adjunctions induced by the morphisms, meaning that the following square of adjunctions 

\[
\begin{tikzcd}[column sep=5pc,row sep=5pc]
\mathsf{dg}~\mathcal{P}\text{-}\mathsf{cog} \arrow[r,"\widehat{\Omega}_\alpha"{name=B},shift left=1.1ex] \arrow[d,"f_!"{name=SD},shift left=1.1ex ]
&\catcurvalg{\operad C} \arrow[d,"g_{*}"{name=LDC},shift left=1.1ex ] \arrow[l,"\widehat{\mathrm{B}}_\alpha"{name=C},,shift left=1.1ex]  \\
\mathsf{dg}~\mathcal{Q}\text{-}\mathsf{cog} \arrow[r,"\widehat{\Omega}_\beta "{name=CC},shift left=1.1ex]  \arrow[u,"f^*"{name=LD},shift left=1.1ex ]
&\catcurvalg{\operad D} \arrow[l,"\widehat{\mathrm{B}}_\beta"{name=CB},shift left=1.1ex] \arrow[u,"g^!"{name=TD},shift left=1.1ex] \arrow[phantom, from=SD, to=LD, , "\dashv" rotate=0] \arrow[phantom, from=C, to=B, , "\dashv" rotate=-90]\arrow[phantom, from=TD, to=LDC, , "\dashv" rotate=0] \arrow[phantom, from=CC, to=CB, , "\dashv" rotate=-90]
\end{tikzcd}
\] 

is commutative. 

\begin{remark}
Under the extra assumption that the morphism $g: \operad C \longrightarrow \operad D$ is quasi-planar, one can consider qp-complete algebras in the right hand side of the last commutative square of adjunctions, where $g^!$ is replaced by $\widehat{g}^!$.
\end{remark}

\subsection{(Co)admissible operads}
In this subsection, we review various constructions that allow one the give a meaning to the homotopy theory of dg $\operad P$-(co)algebras, where $\operad P$ is a dg operad.

\begin{definition}[Admissible dg operad]
    A dg operad $\operad P$ is called \textit{admissible} if the category of dg $\operad P$-algebras admits a combinatorial model structure on right-transferred along the free-forgetful adjunction
    \[
\begin{tikzcd}[column sep=5pc,row sep=3pc]
          \catdgmod{\kk} \arrow[r, shift left=1.1ex, "\operad P~ \comp ~ -"{name=F}] & \mathsf{dg}~ \mathcal{P}\text{-}\mathsf{alg}~, \arrow[l, shift left=.75ex, "\mathrm{U}"{name=U}]
            \arrow[phantom, from=F, to=U, , "\dashv" rotate=-90]
\end{tikzcd}
\]
    determined by the following sets of maps

    \medskip
    
    \begin{enumerate}
        \item the set of weak-equivalences is given by quasi-isomorphisms;

        \medskip
        
        \item the set of fibrations is given by degree-wise epimorphisms;

        \medskip
        
        \item the generating cofibrations are the maps $\operad P \comp S^n \longrightarrow \operad P \comp D^{n+1}$ for $n \in \mathbb Z$;

        \medskip
        
        \item the generating acyclic cofibrations are the maps $\operad P \comp 0 \longrightarrow \operad P \comp D^{n+1}$ for $n \in \mathbb Z$.
    \end{enumerate}
\end{definition}

\begin{definition}[Coadmissible dg operad]
    A dg operad $\operad P$ is called \textit{coadmissible} if the category of dg $\operad P$-coalgebras admits a combinatorial model structure left-transferred along the forgetful-cofree adjunction
        \[
\begin{tikzcd}[column sep=5pc,row sep=3pc]
          \catdgmod{\kk} \arrow[r, shift left=1.1ex, "L^{\operad P}"{name=F}] & \mathsf{dg}~ \mathcal{P}\text{-}\mathsf{cog}~, \arrow[l, shift left=.75ex, "\mathrm{U}"{name=U}]
            \arrow[phantom, from=F, to=U, , "\dashv" rotate=90]
\end{tikzcd}
\]
    determined by the following sets of maps

    \medskip
    
    \begin{enumerate}
        \item the set of weak-equivalences is given by quasi-isomorphisms,

        \medskip
        
        \item the set of cofibrations is given by degree-wise injections,

        \medskip
        
        \item the generating (acyclic) cofibrations are the maps between small objects so whose image through the forgetful functor is an (acyclic) cofibration.
    \end{enumerate}
\end{definition}

\begin{remark}
    Let $\kk$ be a characteristic zero field. Then every dg operad is admissible. Even in characteristic zero, not every dg operad is coadmissible. For instance, the operad $\mathrm{u}\mathcal{C}\mathrm{om}$ is not coadmissible, see \cite[Section 8.3]{linearcoalgebras}.
\end{remark}

For a dg operad $\operad P$ to be admissible (resp. coadmissible) it suffices to provide a natural path object (resp. a natural cylinder object).

\begin{proposition}[\cite{BergerMoerdijk}]\label{propadmissible}
    A dg operad $\operad P$ that is $\operad E$-split is admissible and coadmissible. In particular, a cofibrant operad is admissible and coadmissible.
\end{proposition}

\begin{proof}
    As described in \cite{BergerFresse}, the cellular model of the interval $I$ has the canonical structure of a $\operad E$-coalgebra. 

    \medskip
    
    For every dg $\operad P$-algebra $A$, the dg module $[I, A]$ inherits the structure of a dg $\operad E \otimes \operad P$-algebra and therefore the structure of a dg $\operad P$-algebra. This provides a natural path object, and proves the existence of the transferred model structure on dg $\operad P$-algebras.

    \medskip

    For every dg $\operad P$-coalgebra $V$, the dg module $V \otimes I$ inherits the structure of a dg $\operad E \otimes \operad P$-coalgebra and therefore the structure of a dg $\operad P$-coalgebra. This provides a natural cylinder object, and proves the existence of the transferred model structure on dg $\operad P$-coalgebras.
\end{proof}

\begin{example}
For any dg operad $\operad P$, the dg operad $\operad E \otimes \operad P$ is $\operad E$-split. Indeed, the Barratt-Eccles operad is a dg Hopf operad: in particular, there exists a morphism of dg operad $\operad E \longrightarrow \operad E \otimes \operad E$, which induces an $\operad E$-splitting on $\operad E \otimes \operad P$. Thus, $\operad E \otimes \operad P$ is always admissible and coadmissible. 
\end{example}

\begin{proposition}
    Let $f: \operad P \longrightarrow \operad Q$ be a morphism of dg operads. 

    \medskip

    \begin{enumerate}
        \item If $\operad P$ and $\operad Q$ are both admissible, the adjunction 

        \[
\begin{tikzcd}[column sep=5pc,row sep=3pc]
          \mathsf{dg}~\operad P\text{-}\mathsf{alg} \arrow[r, shift left=1.1ex, "f_{!}"{name=F}] & \mathsf{dg}~ \mathcal{Q}\text{-}\mathsf{alg} \arrow[l, shift left=.75ex, "f^*"{name=U}]
            \arrow[phantom, from=F, to=U, , "\dashv" rotate=-90]
\end{tikzcd}
\]

        is a Quillen adjunction.

        \medskip

        \item If $\operad P$ and $\operad Q$ are both coadmissible, the adjunction 
        
\[
\begin{tikzcd}[column sep=5pc,row sep=3pc]
          \mathsf{dg}~\operad Q\text{-}\mathsf{cog} \arrow[r, shift left=1.1ex, "f^*"{name=F}] & \mathsf{dg}~ \mathcal{P}\text{-}\mathsf{cog}
          \arrow[l, shift left=.75ex, "f_!"{name=U}]
            \arrow[phantom, from=F, to=U, , "\dashv" rotate=-90]
\end{tikzcd}
\]

        is a Quillen adjunction.
    \end{enumerate}
\end{proposition}

\begin{proof}
This is a consequence of the fact that these model structures are transferred from dg modules. It is immediate to see that $f^*$ preserves quasi-isomorphisms and fibrations of dg $\mathcal{Q}$-algebras, and that $f^*$ preserves quasi-isomorphisms and cofibrations of dg $\mathcal{Q}$-coalgebras.
    \end{proof}

\begin{lemma}[After \cite{linearcoalgebras}]\label{lemmaacycliccofyieldqe}
    Let $\operad P$ be a dg operad and let $K$ be an arity-wise acylic dg $\mathbb N$-module. We denote $\operad Q = \operad P \sqcup \treemod_\pl \left( K \otimes \mathbb S \right)$, $i$ the canonical inclusion of $\operad P$ into $\operad Q$ and $p$ the canonical projection from $\operad Q$ onto $\operad P$. This gives the following diagram
    \[
    \begin{tikzcd}
    \operad P \arrow[r,"i"]
    &\operad Q \arrow[r,"p"]
    &\operad P~,
    \end{tikzcd}  
    \]
    where the inclusion $i$ is an acyclic cofibration and $p$ is an acyclic fibration. If $\operad P$ is $\operad E$-split, so is $\operad Q$. Furthemore, the induced Quillen adjunctions
\[
\begin{tikzcd}[column sep=5pc,row sep=3pc]
          \mathsf{dg}~\operad P\text{-}\mathsf{alg} \arrow[r, shift left=1.1ex, "i_{!}"{name=F}]           
         &\mathsf{dg}~ \mathcal{Q}\text{-}\mathsf{alg} \arrow[l, shift left=.75ex, "i^*"{name=U}] \arrow[r, shift left=1.1ex, "p_{!}"{name=A}]
         &\mathsf{dg}~\operad P\text{-}\mathsf{alg} \arrow[l, shift left=.75ex, "p^*"{name=B}] \arrow[phantom, from=F, to=U, , "\dashv" rotate=-90]\arrow[phantom, from=A, to=B, , "\dashv" rotate=-90]
\end{tikzcd}
\]

\[
\begin{tikzcd}[column sep=5pc,row sep=3pc]
          \mathsf{dg}~\operad P\text{-}\mathsf{cog} \arrow[r, shift left=1.1ex, "i_{!}"{name=F}]           
         &\mathsf{dg}~ \mathcal{Q}\text{-}\mathsf{cog} \arrow[l, shift left=.75ex, "i^*"{name=U}] \arrow[r, shift left=1.1ex, "p_{!}"{name=A}]
         &\mathsf{dg}~\operad P\text{-}\mathsf{cog} \arrow[l, shift left=.75ex, "p^*"{name=B}] \arrow[phantom, from=F, to=U, , "\dashv" rotate=90]\arrow[phantom, from=A, to=B, , "\dashv" rotate=90]
\end{tikzcd}
\]
    are Quillen equivalences.
\end{lemma}

\begin{proof}
    The fact that the inclusion $\operad P \longrightarrow \operad Q$ is an acyclic cofibration and that its left inverse $\operad Q \longrightarrow \operad P$ is a acyclic fibration is clear.
    
    \medskip
    
    Let $e: \operad P \longrightarrow \operad E \otimes \operad P$ be the $\operad E$-splitting of the dg operad $\operad P$. An $\operad E$-splitting of the dg operad $\operad Q$ is given by the morphism 
    \[
    \begin{tikzcd}[column sep=3.5pc,row sep=3pc]
    \operad P
    \ar[r, "e"]
    &\operad E \otimes \operad P
    \ar[r,"\id~\otimes~i"]
    &\operad E \otimes \operad Q
    \end{tikzcd}
    \]
    together with the choice of a section $K \longrightarrow \operad E \otimes K$ of the projection $\operad E \otimes K \longrightarrow K$ in the category of dg $\mathbb N$-module. 
    
    \medskip
    
    Let us prove that the Quillen adjunctions relating dg $\operad Q$-algebras to dg $\operad P$-algebras are Quillen equivalences. This amounts to prove that the two right Quillen functors
    \begin{align*}
        \mathrm{Ho}(p^\ast):\mathrm{Ho}(\catdgalg{\operad P}) \longrightarrow \mathrm{Ho}(\catdgalg{\operad Q})
        \\
        \mathrm{Ho}(i^\ast):\mathrm{Ho}(\catdgalg{\operad Q}) \longrightarrow \mathrm{Ho}(\catdgalg{\operad P})
    \end{align*}
    are equivalences. Let us show that these are inverse equivalences. On the one hand, we have a canonical isomorphism $i^\ast p^\ast \simeq \id$, thus we get an isomorphism $\mathrm{Ho}(i^\ast) \mathrm{Ho}(p^\ast) \simeq \id$. 
    
    \medskip
    
    Notice that the data of a dg $\operad Q$-algebra structure amounts to the data of a dg $\operad P$-algebra structure and of a morphism of dg module $K \comp_\pl A \longrightarrow A$. We define an endofunctor $\tilde{(-)}$ of the category of dg $\operad Q$-algebras as follows. Let $A$ be a dg $\operad Q$-algebra. Its image $\tilde{A}$ is the dg $\operad Q$-algebra whose underlying dg module is the path object $[I, A] = I^\ast \otimes A$ (where $I$ is the cellular model of the interval). Its dg $\operad P$-algebra structure is given by the following maps
    $$
    \begin{tikzcd}
    \operad P(n) \otimes ( I^\ast \otimes A)^{\otimes n}
    \ar[d,"e(n)~\otimes ~\id"]
    \\
    \operad E(n) \otimes \operad P(n) \otimes (I^\ast \otimes A)^{\otimes n}
    \ar[d, "\simeq"]
    \\
    \operad E(n) \otimes (I^\ast)^{\otimes n} \otimes \operad P(n) \otimes A^{\otimes n}
    \ar[d, "\gamma_I^n~\otimes~\gamma_A^n|_{\operad P} "]
    \\
    I^\ast \otimes A,
    \end{tikzcd}    
    $$
    where $\gamma_A^n|_{\operad P}$ is given by the dg $\operad P$-algebra structure on $A$ and where $\gamma_I^n$ is given by the dg $\operad E$-algebra structure on $I$ constructed in \cite{BergerFresse}. The action of $K$ is given as follows
    $$
    \begin{tikzcd}
    K(n) \otimes (I^\ast \otimes A)^{\otimes n}
    \ar[r]
    &K(n) \otimes A^{\otimes n}
    \ar[r,"\gamma_A^n|_K"]
    &A
    \ar[r]
    &I^\ast \otimes A
    \end{tikzcd}
    $$
    where the first map results from the inclusion of the first point into the interval $\kk \longrightarrow I$ which gives a map $I^\ast \longrightarrow \kk$, and where $\gamma_A^n|_K$ is the action of $K$ on $A$ given by its dg $\operad Q$-algebra structure.
    
    \medskip
    
    We have canonical natural transformations
    $$
    p^\ast i^\ast \lqi \tilde{(-)} \qi \id
    $$
    that are object-wise quasi-isomorphisms. Hence we get an isomorphism $\mathrm{Ho}(p^\ast)\mathrm{Ho}(i^\ast) \simeq \id$.
    Therefore, $\mathrm{Ho}(i^\ast)$ and $\mathrm{Ho}(p^\ast)$ are inverse to each other and their related Quillen adjunctions are Quillen equivalences.
    
    \medskip

    Finally, let us prove that the Quillen adjunctions relating dg $\operad Q$-coalgebras to dg $\operad P$-coalgebras are also Quillen equivalences. Again, this amounts to prove that the two right adjoint functors
    \begin{align*}
        \mathrm{Ho}(p^\ast):\mathrm{Ho}(\catdgcog{\operad P}) \longrightarrow \mathrm{Ho}(\catdgcog{\operad Q})
        \\
        \mathrm{Ho}(i^\ast):\mathrm{Ho}(\catdgcog{\operad Q}) \longrightarrow \mathrm{Ho}(\catdgcog{\operad P})
    \end{align*}
    are equivalences. This follows from similar arguments as those used in the algebra case. There is a canonical isomorphism $i^\ast p^\ast \simeq \id$. 
    
    \medskip
    
    Again, a dg $\operad Q$-coalgebra structure is completely determined by a dg $\operad P$-coalgebra and a morphism of dg modules $V \longrightarrow V^{K \otimes \mathbb{S}}$. We define an endofunctor $\tilde{(-)}$ category of dg $\operad Q$-coalgebras as follows. Let $V$ be a dg $\operad Q$-coalgebra. The underlying dg module of $\tilde V$ is the cylinder object $I \otimes V$. Its dg $\operad P$-coalgebra is given by 
    $$
    \begin{tikzcd}
    I \otimes V \otimes \operad P(n)
    \ar[d,"\id ~ \otimes ~e"]
    \\
    I \otimes V \otimes \operad E(n) \otimes \operad P(n)
    \ar[d, "\simeq"]
    \\
    \operad E(n) \otimes I \otimes  V \otimes \operad P(n)  \ar[d,"\Delta_I^n~\otimes~\Delta_V^n|_{\operad P}"]
    \\
    I^{\otimes n} \otimes V^{\otimes n}
    \ar[d, "\simeq"]
    \\
    (I \otimes V)^{\otimes n}~,
    \end{tikzcd}    
    $$
    where $\Delta_V^n|_{\operad P}$ is the given by the dg $\operad P$-coalgebra structure on $A$ and where $\Delta_I^n$ is the dg $\operad E$-algebra on $I$ constructed in \cite{BergerFresse}. The coaction of $K$ is given by 
    $$
    \begin{tikzcd}
    I \otimes V \otimes K(n)
    \ar[r]
    &V \otimes K(n)
    \ar[r,"\Delta_V^n|_K"]
    &V^{\otimes n}
    \ar[r]
    &(I \otimes V)^{\otimes n}~,
    \end{tikzcd}
    $$
    where $\Delta_V^n|_K$ is the coaction of $K$ on $V$ and where the last map proceed from the inclusion of the first point into the interval $\kk \longrightarrow I$.
    We have canonical natural transformations
    $$
    p^\ast i^\ast \lqi \tilde{(-)} \qi \id
    $$
    that are object-wise quasi-isomorphisms and yield an isomorphism $\mathrm{Ho}(p^\ast)\mathrm{Ho}(i^\ast) \simeq \id$.
    Therefore, $\mathrm{Ho}(i^\ast)$ and $\mathrm{Ho}(p^\ast)$ are inverse equivalences and their related Quillen adjunctions are Quillen equivalences.
\end{proof}

\begin{proposition}[\cite{BergerMoerdijk,linearcoalgebras}]\label{propqeqisoperads}
    Let $f: \operad P \qi \operad Q$ be a quasi-isomorphism of cofibrant dg operads. The induced Quillen adjunctions
    
     \[
\begin{tikzcd}[column sep=5pc,row sep=3pc]
          \mathsf{dg}~\operad P\text{-}\mathsf{alg} \arrow[r, shift left=1.1ex, "f_{!}"{name=F}] & \mathsf{dg}~ \mathcal{Q}\text{-}\mathsf{alg} \arrow[l, shift left=.75ex, "f^*"{name=U}]
            \arrow[phantom, from=F, to=U, , "\dashv" rotate=-90]
\end{tikzcd}
\]

\[
\begin{tikzcd}[column sep=5pc,row sep=3pc]
          \mathsf{dg}~\operad Q\text{-}\mathsf{cog} \arrow[r, shift left=1.1ex, "f^*"{name=F}] & \mathsf{dg}~ \mathcal{P}\text{-}\mathsf{cog}
          \arrow[l, shift left=.75ex, "f_!"{name=U}]
            \arrow[phantom, from=F, to=U, , "\dashv" rotate=-90]
\end{tikzcd}
\]
   
are both Quillen equivalences.
\end{proposition}

\begin{proof}
    Let us denote $QE$ the (large) set of morphisms of cofibrant operads $f: \operad P \longrightarrow \operad Q$ whose induced Quillen adjunctions relating their algebras and relating their coalgebras are Quillen equivalence. One can notice that $QE$ is stable through composition and retracts and satisfies the 2-out-of-3 law. By Lemma \ref{lemmaacycliccofyieldqe} is contains morphisms of the form $\operad P \longrightarrow \operad P \sqcup \treemod_\pl (K \otimes \mathbb S)$, where $K$ is an acyclic dg $\mathbb N$-module. Since every acyclic cofibration of cofibrant dg operads is a retract of a map having this from, $QE$ contains all the acyclic cofibration of cofibrant dg operads. Now, let $f: \operad P \longrightarrow \operad Q$ be an weak-equivalence of cofibrant dg operad. We can factorise the map $\operad P \sqcup \operad Q \longrightarrow \operad Q$ by a cofibration followed by an acyclic cofibration
    $$
    \begin{tikzcd}
        \operad P \sqcup \operad Q 
        \ar[r, hook]
        &\operad R
        \ar[r, two heads, "\sim"] 
        &\operad Q.
    \end{tikzcd}
    $$
    The maps $\operad P \longrightarrow \operad R$ and $\operad Q \longrightarrow \operad R$ belong to $QE$ since they are acyclic cofibrations. Then, the map $\operad R \longrightarrow \operad P$ belongs to $QE$ by the 2-out-of-3 law. So $f$ belongs to $QE$ as composition of two maps in $QE$.
\end{proof}

\textbf{Semi-model category structures and reduction the admissible case.} Under weaker assumptions, M. Spitzweck constructed in \cite{Spitzweck} a \textit{semi-model category structure} on the category of dg $\operad P$-algebras. Nevertheless, up to a convenient replacement of the initial dg operad $\operad P$, one can always assume the existence of a model category structure without losing any homotopical information.

\begin{theorem}[{\cite{Spitzweck}}]
Let $\PP$ be a $\mathbb S$-projective dg operad. The category of dg $\PP$-algebras admits a \textit{semi-model category structure}, determined by the following sets of maps

\medskip

\begin{enumerate}
    \item the set of weak-equivalences is given by quasi-isomorphisms,
    
\medskip

    \item the set of fibrations is given by degree-wise epimorphisms,
    
\medskip

    \item the set of cofibrations is given by morphisms with a left-lifting property against acyclic fibrations.
\end{enumerate}
\end{theorem}

These semi-model category structures are stable under quasi-isomorphisms of dg operads.

\begin{theorem}[\cite{Fresse}]\label{thm: semi-model stable by quasi-iso}
Let $f: \PP \qi \QQ$ be a quasi-isomorphism of $\mathbb S$-projective dg operads. The induced adjunction 
  \[
\begin{tikzcd}[column sep=5pc,row sep=3pc]
          \mathsf{dg}~\operad P\text{-}\mathsf{alg} \arrow[r, shift left=1.1ex, "f_{!}"{name=F}] & \mathsf{dg}~ \mathcal{Q}\text{-}\mathsf{alg} \arrow[l, shift left=.75ex, "f^*"{name=U}]
            \arrow[phantom, from=F, to=U, , "\dashv" rotate=-90]
\end{tikzcd}
\]

is an equivalence of semi-model categories.
\end{theorem}

In particular, for any $\mathbb S$-projective dg operad $\operad P$, the canonical quasi-isomorphism $\varphi_{\operad P}: \operad E \otimes \operad P \qi \operad P$ always induces an equivalence of semi-model categories

\[
\begin{tikzcd}[column sep=5pc,row sep=3pc]
          \mathsf{dg}~\operad E \otimes \operad P \text{-}\mathsf{alg} \arrow[r, shift left=1.1ex, "(\varphi_{\operad P})_{!}"{name=F}] & \mathsf{dg}~ \mathcal{P}\text{-}\mathsf{alg}~, \arrow[l, shift left=.75ex, "\varphi_{\operad P}^*"{name=U}]
            \arrow[phantom, from=F, to=U, , "\dashv" rotate=-90]
\end{tikzcd}
\]

where on the left-hand side there is a model category structure on the category of dg $\operad E \otimes \operad P$-algebras. Recall that $\operad E \otimes \operad P$ is admissible since it is $\operad E$-split. Therefore, by replacing $\operad P$ with $\operad E \otimes \operad P$ we can always "rectify" the semi-model structure into a model structure, and both present the same underlying homotopy category.

\subsection{Quasi-planar Bar-Cobar adjunctions}\label{subsection: extended bar-cobar}
In the subsequent sections, we will manly work with dg operads of the form $\Omega \C$, where $\C$ is a quasi-planar conilpotent curved cooperad. Nevertheless, if one starts with a dg operad $\operad P$, there is a canonical quasi-planar conilpotent curved cooperad given by $\mathrm{B}(\operad E \otimes \operad P)$. One can then construct \textit{quasi-planar} bar-cobar adjunctions between their respective (co)algebras categories. It will follow from the results in the subsequent sections that these adjunctions are Quillen adjunctions when $\operad P$ is (co)admissible and induce Quillen equivalences when $\operad P$ is cofibrant. 

\medskip

Let $\operad P$ be a dg operad. Let $\varphi_{\operad P}: \operad E \otimes \operad P \qi \operad P$ be the canonical quasi-isomorphism of dg operads, where $\operad E$ is the Barratt-Eccles operad.

\medskip

\textbf{Quasi-planar bar-cobar adjunction.} The morphism $\varphi_{\operad P}$ induces an adjunction $(\varphi_{\operad P})_! \dashv \varphi_{\operad P}^*$ between dg $\mathcal{P}$-algebras and dg $\operad E \otimes \operad P$-algebras. We define the \textit{quasi-planar bar-cobar adjunction} 

\[
\begin{tikzcd}[column sep=5pc,row sep=3pc]
          \mathsf{dg}~\operad P\text{-}\mathsf{alg} \arrow[r, shift left=1.1ex, "\mathrm{B}^{\mathrm{q.p}}_{\pi}"{name=A}]
         &\mathsf{curv}~\mathrm{B}(\operad E \otimes \operad P) \text{-}\mathsf{cog}~. \arrow[l, shift left=.75ex, "\Omega^{\mathrm{q.p}}_{\pi}"{name=B}] \arrow[phantom, from=A, to=B, , "\dashv" rotate=90]
\end{tikzcd}
\]

as the following composition of adjunctions

\[
\begin{tikzcd}[column sep=5pc,row sep=3pc]
          \mathsf{dg}~\operad P\text{-}\mathsf{alg} \arrow[r, shift left=1.1ex, "\varphi_{\operad P}^*"{name=F}]           
         &\mathsf{dg}~ \operad E \otimes \operad P\text{-}\mathsf{alg} \arrow[l, shift left=.75ex, "(\varphi_{\operad P})_!"{name=U}] \arrow[r, shift left=1.1ex, "\mathrm{B}_{\pi}"{name=A}]
         &\mathsf{curv}~\mathrm{B}(\operad E \otimes \operad P)\text{-}\mathsf{cog}~. \arrow[l, shift left=.75ex, "\Omega_{\pi}"{name=B}] \arrow[phantom, from=F, to=U, , "\dashv" rotate=90]\arrow[phantom, from=A, to=B, , "\dashv" rotate=90]
\end{tikzcd}
\]
\vspace{0.1pc}

\textbf{Quasi-planar complete bar-cobar adjunction.} The morphism $\varphi_{\operad P}$ induces an adjunction $(\varphi_{\operad P})_! \dashv \varphi_{\operad P}^*$ between dg $\mathcal{P}$-coalgebras and dg $\operad E \otimes \operad P$-coalgebras. We define the \textit{quasi-planar complete bar-cobar adjunction} 
\[
\begin{tikzcd}[column sep=5pc,row sep=3pc]
          \mathsf{dg}~\operad P\text{-}\mathsf{cog} \arrow[r, shift left=1.1ex, "\widehat{\Omega}^{\mathrm{q.p}}_{\pi}"{name=A}]
         &\mathsf{curv}~\mathrm{B}(\operad E \otimes \operad P) \text{-}\mathsf{alg}~, \arrow[l, shift left=.75ex, "\widehat{\mathrm{B}}^{\mathrm{q.p}}_{\pi}"{name=B}] \arrow[phantom, from=A, to=B, , "\dashv" rotate=-90]
\end{tikzcd}
\]

as the following composition of adjunctions
\[
\begin{tikzcd}[column sep=5pc,row sep=3pc]
          \mathsf{dg}~\operad P\text{-}\mathsf{cog} \arrow[r, shift left=1.1ex, "\varphi_{\operad P}^*"{name=F}]           
         &\mathsf{dg}~ \operad E \otimes \operad P\text{-}\mathsf{cog} \arrow[l, shift left=.75ex, "(\varphi_{\operad P})_!"{name=U}] \arrow[r, shift left=1.1ex, "\widehat{\Omega}_{\pi}"{name=A}]
         &\mathsf{curv}~\mathrm{B}(\operad E \otimes \operad P)\text{-}\mathsf{alg}~. \arrow[l, shift left=.75ex, "\widehat{\mathrm{B}}_{\pi}"{name=B}] \arrow[phantom, from=F, to=U, , "\dashv" rotate=-90]\arrow[phantom, from=A, to=B, , "\dashv" rotate=-90]
\end{tikzcd}
\]

\newpage


\section{Model structure on coalgebras over a cooperad}

\vspace{2pc}
The goal of this section is to study the homotopical properties of the bar-cobar adjunction between dg $\Omega\C$-algebras and curved $\C$-coalgebras, in the case where $\C$ is a quasi-planar conilpotent curved cooperad. 

\medskip

The dg operad $\Omega \C$ is cofibrant, and therefore admissible by Proposition \ref{propadmissible}. This means that the category of dg $\Omega\C$-algebras admits a model category structure where weak-equivalences are given by quasi-isomorphisms and fibrations by degree-wise epimorphisms. Let us consider the bar-cobar adjunction 

\[
\begin{tikzcd}[column sep=5pc,row sep=3pc]
          \mathsf{dg}~\C\text{-}\mathsf{cog} \arrow[r, shift left=1.1ex, "\Omega_{\C}"{name=F}] & \mathsf{dg}~ \Omega \C\text{-}\mathsf{alg}, \arrow[l, shift left=.75ex, "\mathrm{B}_{\C}"{name=U}]
            \arrow[phantom, from=F, to=U, , "\dashv" rotate=-90]
\end{tikzcd}
\]
\vspace{0.1pc}

relative to $\iota: \C \longrightarrow \Omega \C$, which will be denoted by $\Omega_\C,\mathrm{B}_\C$ from now on, since we will not consider any other curved twisting morphism. Our first goal is going to be to transfer this model category structure along the bar-cobar adjunction to the category of dg $\C$-coalgebras.

\begin{theorem}\label{thm: existence de structure de modèles}
Let $\C$ be a  quasi-planar conilpotent curved cooperad. The exists a combinatorial model category structure on the category of curved $\C$-coalgebras given by the following sets of maps:

\medskip

\begin{enumerate}
\item the set of weak-equivalences is given by morphisms $f$ such that $\Omega_\C(f)$ is a quasi-isomorphism,

\medskip

\item the set of cofibrations is given by morphisms $f$ such that $\Omega_\C(f)$ is a cofibration; they correspond to degree-wise injections,

\medskip 

\item the set of fibrations is given by morphisms with the right lifting property with respect to acyclic cofibrations.
\end{enumerate}
\end{theorem}

\begin{remark}
    Using the standard transfer theorem for model category structures only gives that cofibrations are morphism which are sent by $\Omega_{\C}$ to cofibrations. The theorem contains an additional characterization of cofibrations of curved $\C$-coalgebras as degree-wise injective maps. 
\end{remark}


\subsection{Outline of the transfer of model structures}
Let us first start by defining the sets of morphisms of curved $\C$-coalgebras that we intend to study.

\begin{definition}[Cofibrations]\label{def: cof de C-cogebres}
A morphism $f$ of curved $\C$-coalgebras is a \textit{cofibration} if $\Omega_\C(f)$ is a cofibration of dg $\Omega \C$-algebras.
\end{definition}

\begin{definition}[Weak-equivalences]
A morphism $f$ of curved $\C$-coalgebras is a \textit{weak-equivalence} if $\Omega_\C(f)$ is a quasi-isomorphism of dg $\Omega \C$-algebras.
\end{definition}

\begin{definition}[Fibrations]
A morphism of curved $\C$-coalgebras is a \textit{fibration} if it has the right-lifting property against all acyclic cofibrations. 
\end{definition}

Both categories $\catcurvcog{\operad C}$ and $\catdgalg{\Omega \operad C}$ are presentable, therefore by Appendix \ref{appendixtransfer} it suffices to exhibit a natural cofibrant resolution and a natural cylinder for coalgebras to prove the existence of the transferred model structure. We will show that cofibrations of Definition \ref{def: cof de C-cogebres} are given by degree-wise injective maps in Proposition \ref{propcofibrations} and we will construct a natural cylinder in Proposition \ref{propcylinder}. 

\medskip

For the rest of this section, let us fix a quasi-planar conilpotent curved cooperad $\C$ whose quasi-planar ladder is indexed by some small ordinal $\alpha$.

\subsection{Elementary cofibrations}
Elementary cofibrations are a particularly well-behaved set of cofibrations of curved $\C$-coalgebras. For any such elementary cofibration, its cokernel is a dg $\kk$-module.

\begin{definition}[Elementary cofibrations]
A morphism $f : W' \rightarrowtail W$ of curved $\operad C$-coalgebras is an \textit{elementary cofibration} if it is degree-wise injective and if the map $\overline{\Delta}_{W} : W \longrightarrow \overline{\C} \comp W$ factors through the sub-object $\overline{\C} \comp W' \subseteq \overline{\C} \comp W$, that is, if there exists a dashed map 
\[
\begin{tikzcd}[column sep=3pc,row sep=3pc]
W \ar[r, "\Delta_W"] \arrow[d,dashed]
& \C \comp W
\arrow[d,twoheadrightarrow]
\\
\overline{\C} \comp W'
\arrow[r,"\overline{\C} ~ \comp ~ f"',rightarrowtail]
& \overline{\C} \comp W
\end{tikzcd}
\]
that makes the diagram commute. Here $\Delta_W$ denotes the structural map of $W$. 
\end{definition}

\begin{remark}
The induced map $\overline{\operad C} \comp W' \longrightarrow \overline{\operad C} \comp W$ is also monomorphism since it is has a left inverse in the category of graded $\kk$-modules.
\end{remark}

\begin{lemma}
Let $f: W' \rightarrowtail W$ be an elementary cofibration curved $\operad C$-coalgebras. Then $\mathrm{Coker}(f)$ is a dg module.
\end{lemma}

\begin{proof}
Let us denote $p_f:W \twoheadrightarrow \mathrm{Coker}(f)$ the canonical projection. We have
$$
d^2 p_f = p_f d^2 = -p_f (\theta \comp \id_W) \overline{\Delta}_W
= - (\theta \comp p_f)\overline{\Delta}_W
= - (\theta \comp \id_{\mathrm{Coker}(f)}) (\C \comp p_f) \overline{\Delta}_W.
$$
The definition of an elementary cofibration tells us that $(\C \comp p_f) \overline{\Delta}_W=0$. We conclude by noticing that $p_f$ is an epimorphism.
\end{proof}

\begin{proposition}\label{prop:elemcof}
Let $f: W' \rightarrowtail W$ be an elementary cofibration. Then $f$ is a cofibration.
\end{proposition}

\begin{proof}
Identifying $W'$ with its image in $W$, we can decompose the underlying graded $\kk$-module of $W$ as the a direct $W \cong W' \oplus U$. The pre-differential of $W$ decomposes as the pre-differential $d_{W'}$ on $V$, the differential $d_U$ on $U$ and a degree $-1$ map $\zeta : U \longrightarrow W'$. The map of graded $\kk$-modules
$$
g : U \hookrightarrow W \hookrightarrow \operad P \comp W
$$
induces a morphism of dg modules
$$
 D^0 \otimes U \longrightarrow \Omega_{\operad C} W~,
$$
whose restriction to $S^0 \otimes U$ is $g$ and whose restriction to $S^{-1} \otimes U$ is $- \id \otimes \partial(g)$. One can notice that this restriction to $S^{-1} \otimes U$ factors through the sub-dg module $\Omega_{\operad C} W' \subset \Omega_{\operad C} W$. We thus get a commutative diagram of dg modules
$$
\begin{tikzcd}[column sep=3pc,row sep=3pc]
    S^{-1} \otimes U
    \ar[r] \ar[d]
    & \Omega_{\operad C} W'
    \ar[d, "f"]
    \\
    D^{0} \otimes U
    \ar[r]
    & \Omega_{\operad C} W~,
\end{tikzcd}
$$
which gives a commutative diagram of dg $\Omega \operad C$-algebras
$$
\begin{tikzcd}[column sep=3pc,row sep=3pc]
    \Omega \operad C \circ(S^{-1} \otimes U) \arrow[dr, phantom, "\ulcorner", very near end]
    \ar[r] \ar[d,rightarrowtail]
    & \Omega_{\operad C} W'
    \ar[d, "f"]
    \\
    \Omega \operad C \circ(D^{0} \otimes U)
    \ar[r]
    & \Omega_{\operad C} W~.
\end{tikzcd}
$$
This diagram is a pushout square since the underlying diagram of graded $\Omega \operad C$-algebras is a pushout square. Since the left vertical map is a cofibration, so is the right vertical map $f$.
\end{proof}

\begin{lemma}\label{lemelemcofibequiv}
    Let us consider a commutative diagram of curved $\operad C$-coalgebras
    $$
    \begin{tikzcd}[column sep=3pc,row sep=3pc]
        U
        \ar[r, "\id"] \ar[d, "i"']
        & U
        \ar[d, "j"]
        \\
        W'
        \ar[r, "g"']
        & W
    \end{tikzcd}
    $$
    where
    \begin{enumerate}
    \medskip
        \item $i$ and $j$ are elementary cofibrations,
        
    \medskip
    
        \item the map of dg modules induced by $g$
        $$
        \bar{g} : W'/U \longrightarrow W/U
        $$
        is a quasi-isomorphism.
    \end{enumerate}
    
    \medskip 
    
    Then $g$ is a weak-equivalence.
\end{lemma}

\begin{proof}
Let us decompose the underlying graded $\kk$-module of $W$ as a direct sum $W \cong U \oplus Y$. We decompose the underlying graded $\kk$-module of $W'$ as a direct sum $W' \cong U\oplus X$ in such a way so that the restriction of $g$ to $X$ lands in $Y$. The diagrams built in the proof of Proposition \ref{prop:elemcof} fit in the following commutative cube 
    $$
    \begin{tikzcd}
        \Omega \operad C \comp (S^{-1} \otimes X)
        \ar[rr] \ar[rd] \ar[dd]
        && \Omega \operad C \comp (S^{-1} \otimes Y)
        \ar[rd] \ar[dd]
        \\
        & \Omega_{\operad C} U
        \ar[rr] \ar[dd]
        && \Omega_{\operad C} U
        \ar[dd]
        \\
        \Omega \operad C \comp (D^{0} \otimes X)
        \ar[rr] \ar[rd]
        && \Omega \operad C \comp (D^0 \otimes Y)
       \ar[rd] 
        \\
        & \Omega_{\operad C} W'
        \ar[rr]
        &&  \Omega_{\operad C} W~.
    \end{tikzcd}
    $$
    The left and the right face are pushouts squares which are also homotopy pushouts. The two back horizontal maps and the top front horizontal map are quasi-isomorphisms. Thus the homotopy pushout map $\Omega_{\operad C}g$ is also a quasi-isomorphism, which implies that $g$ is a weak-equivalence.
\end{proof}

\begin{proposition}\label{propelemcofibequiv}
    Let us consider a commutative diagram of curved $\operad C$-coalgebras
    $$
    \begin{tikzcd}[column sep=3pc,row sep=3pc]
        U
        \ar[r, "f"] \ar[d, "i"']
        & W
        \ar[d, "j"]
        \\
        U'
        \ar[r, "g"']
        & W' 
    \end{tikzcd}
    $$
    such that
    \begin{enumerate}
    
    \medskip
    
        \item $f$ is a weak-equivalence,
    
    \medskip
    
        \item $i$ and $j$ are elementary cofibrations,
        
    \medskip
    
        \item the map of dg modules induced by $g$
        $$
        \bar{g} : U'/U \longrightarrow W'/W
        $$
        is a quasi-isomorphism.
    \end{enumerate}
    
    \medskip
    
    Then $g$ is a weak-equivalence.
\end{proposition}

\begin{proof}
 Let us consider the following pushout square in the category of curved $\operad C$-coalgebras
    $$
    \begin{tikzcd}[column sep=3pc,row sep=3pc]
        U
        \ar[r] \ar[d] \arrow[dr, phantom, "\ulcorner", very near end]
        & W
        \ar[d]
        \\
        U'
        \ar[r]
        & Z.
    \end{tikzcd}
    $$
    It yields a pushout square of dg $\Omega \operad C$-algebras
    $$
    \begin{tikzcd}[column sep=3pc,row sep=3pc]
        \Omega_{\operad C} U \arrow[dr, phantom, "\ulcorner", very near end]
        \ar[r] \ar[d]
        & \Omega_{\operad C} W
        \ar[d]
        \\
        \Omega_{\operad C} U'
        \ar[r]
        & \Omega_{\operad C} Z
    \end{tikzcd}
    $$
    which is also an homotopy pushout. Thus, the map
    $\Omega_{\operad C} U'  \qi \Omega_{\operad C} Z$
    is a quasi-isomorphism. To conclude, the map $Z \qi W'$ is a weak-equivalence by Lemma \ref{lemelemcofibequiv}.
\end{proof}


\subsection{Ladders and filtered quasi-isomorphisms}
A sub-set of weak-equivalences of curved $\operad C$-coalgebras is given by filtered quasi-isomorphisms of $\beta$-indexed ladders, for any small ordinal $\beta$. The key example of such ladders is the quasi-planar ladder induced by the canonical quasi-planar filtration of $\C$.

\begin{definition}[$\beta$-ladder of curved $\operad C$-coalgebras]
Let $\beta$ be a small ordinal. A $\beta$-\textit{ladder of curved} $\operad C$\textit{-coalgebras} is a functor 
\[
W: \beta \longrightarrow \catcurvcog{\operad C}~,
\]
that sends every limit ordinal $k\in\beta$ to
$$
W(k) = \colim{i < k}~W(i)~,
$$
with $W(-1)=0$, and such that every map
$$
W(i-1) \rightarrowtail W(i)~, i \in \beta~,
$$
is an \textit{elementary cofibration}.
\end{definition}

\begin{notation}
We denote by
$$
W(\beta) \coloneqq \colim{i \in \beta}~W(i)
$$
the value of the colimit of this $\beta$-ladder.
\end{notation}

\begin{remark}
The first property about limit ordinal is equivalent to the fact that the functor 
\begin{align*}
    (1 + \beta) &\longrightarrow \catcurvcog{\operad C}
    \\
    i &\mapsto W(i+1)
\end{align*}
is cocontinuous.
\end{remark}

\begin{definition}[Associated graded of a ladder]
Given a $\beta$-ladder $W$, we define its \textit{associated graded} as 
\[
\gr_i W \coloneqq W(i)/~\colim{j < i}~W(j)~.
\]
The coderivation squares to zero on this quotient $\gr_i W$, therefore it is a dg module. 
\end{definition}

The following general proposition will allow us to construct $\beta$-ladders of coalgebras in a general setting.

\begin{proposition}\label{prop: cooperad ladder induces ladders}
Let $\beta$ be a small ordinal, and let  
\[
\operad D^{(0)} \longrightarrow \operad D^{(1)} \longrightarrow \cdots \longrightarrow \operad D^{(i)} \longrightarrow \cdots 
\]
be a $\beta$-indexed cooperad ladder, where we denote $\operad D \coloneqq \operad D^{(\beta)}$. For every $i \in \beta$, let $F_i^{\operad D} (-)$ be the idempotent comonad on curved $\operad D$-coalgebras related to coreflexive full subcategory of curved $\operad D^{(i)}$-coalgebras. 

\medskip

   \begin{enumerate}
        \item For every curved $\operad D$-coalgebra $W$ the diagram
        $$
    F_0^{\operad D}W \longrightarrow F_i^{\operad D}W \longrightarrow \cdots \longrightarrow F_i^{\operad D}W \longrightarrow \cdots
        $$
is a $\beta$-ladder of curved $\operad D$-coalgebras.
        
\medskip
   
        \item The canonical map
        \[
        \colim{i \in \beta}~ F_{i}^{\operad D} W \longrightarrow W
        \]
is an isomorphism, therefore the colimit of the ladder is $W$. 
    \end{enumerate}
\end{proposition}

\begin{proof}
    For every $i \in \beta$ and for every pdg $\operad D$-coalgebra $W$, $F_i^{\operad D}W$ is given as the following pullback square 
    $$
    \begin{tikzcd}[column sep=3pc,row sep=3pc]
        F_i^{\operad D}W \arrow[dr, phantom, "\lrcorner", very near start]
        \ar[r] \ar[d]
        & \operad D^{(i)} \comp W
        \ar[d]
        \\
        W
        \ar[r]
        & \operad D \comp W~,
    \end{tikzcd}
    $$
    in the category of graded $\kk$-modules. Combined with the fact that directed colimits commute with pullbacks in graded $\kk$-modules, we get the fact that the map
    $$
    \colim{i<k}~ F_i^{\operad D} W \longrightarrow F_i^{\operad D}W
    $$
     is an isomorphism for every limit ordinal $k \in \beta+1$.
     
     \medskip

     It remains to show that for every $i \in \beta$, the map $F_i^{\operad D}W \longrightarrow F_{i+1}^{\operad D}W$ is an elementary cofibration. It fits in the following pullback diagram of graded $\kk$-modules
     $$
    \begin{tikzcd}[column sep=3pc,row sep=3pc]
        F_i^{\operad D}W \arrow[dr, phantom, "\lrcorner", very near start]
        \ar[r] \ar[d]
        & \operad D^{(i)} \comp F_{i+1}^{\operad D}W
        \ar[d]
        \\
        F_{i+1}^{\operad D}W
        \ar[r]
        & \operad D^{(i+1)}\comp F_{i+1}^{\operad D}W~.
    \end{tikzcd}
    $$
     Since degree-wise injections are preserved by the tensor product and by pullbacks, the map $F_i^{\operad D}W \longrightarrow F_{i+1}^{\operad D}W$ is a degree-wise injection. It remains to show that for every $p,q, j$ (with $p \geq 1$ and $1\leq j \leq p$) the decomposition map
     \[
     \begin{tikzcd}[column sep=3pc,row sep=2pc]
      F_{i+1}^{\operad D}W \arrow[d] \\
      \overline{\operad D}_\pl^{(i+1)}(p) \otimes  (F_{i+1}^{\operad D}W )^{\otimes p} \arrow[d] \\
      \overline{\operad D}_\pl^{(i+1)}(p) \otimes  (F_{i+1}^{\operad D}W )^{\otimes j-1 }
     \otimes \overline{\operad D}_\pl^{(i+1)}(q) \otimes (F_{i+1}^{\operad D}W)^{\otimes q} \otimes (F_{i+1}^{\operad D}W)^{\otimes p-j}
    \end{tikzcd}
     \]
     factors through the sub-object
     $$
     \overline{\operad D}_\pl^{(i+1)}(p) \otimes  (F_{i+1}^{\operad D}W )^{\otimes j-1 }
     \otimes \overline{\operad D}^{(i)}(q) \otimes (F_{i+1}^{\operad D}W)^{\otimes q} \otimes (F_{i+1}^{\operad D}W)^{\otimes p-j}~.
     $$
     Using coassociativity, one can rewrite the map as 
     \[
     \begin{tikzcd}[column sep=3pc,row sep=2pc]
      F_{i+1}^{\operad D}W \arrow[d] \\
      \overline{\operad D}_\pl^{(i+1)}(p+q-1) \otimes  (F_{i+1}^{\operad D}W )^{\otimes p+q-1} \arrow[d] \\
      \left(\overline{\operad D}_\pl^{(i+1)}(p)  \otimes \overline{\operad D}_\pl^{(i+1)}(q)\right) \otimes  (F_{i+1}^{\operad D}W )^{\otimes p+q-1}~.
    \end{tikzcd}
     \]
     Since the sequence $(\operad D^{(i)})_{i \in \beta}$ is a cooperad ladder, the map 
     \[
     \Delta_j: \overline{\operad D}_\pl^{(i+1)}(p+q-1) \longrightarrow \overline{\operad D}_\pl^{(i+1)}(p) \otimes \overline{\operad D}_\pl^{(i+1)}(q)
     \]
     factors through $\overline{\operad D}_\pl^{(i)}(p) \otimes \overline{\operad D}_\pl^{(i)}(q)$, which proves the result.
\end{proof}

\begin{corollary}[Coradical ladder]
Let $\C$ be a quasi-planar conilpotent curved cooperad. For every $n \in \omega$, let $F_n^{\mathrm{rad}} (-)$ be the idempotent comonad of curved $\operad C$-coalgebras that coreflects onto the full subcategory of curved $F_n^{\mathrm{rad}}\operad C$-coalgebras, where $F_n^{\mathrm{rad}}\operad C$ is the $n$-stage of the coradical ladder.

\medskip

Let $W$ be a curved $\C$-coalgebra. The diagram  
\[
F_0^{\mathrm{rad}} W \hookrightarrow \cdots \hookrightarrow F_i^{\mathrm{rad}} W \hookrightarrow \cdots
\]
is an $\omega$-ladder of curved $\C$-coalgebras, called the \textit{coradical ladder}. The colimit of this diagram is again $W$.
\end{corollary}

\begin{proof}
Follows directly from Proposition \ref{prop: cooperad ladder induces ladders}.
\end{proof}

\begin{corollary}[Quasi-planar ladder]\label{cor: quasi-planar ladder}
Let $\C$ be a quasi-planar conilpotent curved cooperad. Recall from Subsection \ref{subsection: quasi-planar canonical filtration} that since $\C$ is quasi-planar, it admits a canonical quasi-planar $\omega$-ladder
$$
F_0^{\mathrm{qp}} \C \longrightarrow F_1^{\mathrm{qp}} \C \longrightarrow \cdots \longrightarrow F_n^{\mathrm{qp}} \C \longrightarrow \cdots
$$
whose colimit is $\C$. For every $i \in \omega$, let $F_i^{\mathrm{qp}} (-)$ be the idempotent comonad of curved $\operad C$-coalgebras that coreflects onto the full subcategory of curved $F_i^{\mathrm{qp}} \C$-coalgebras.

\medskip

Let $W$ be a curved $\C$-coalgebra. The diagram 
$$
F_0^{\mathrm{qp}} \C \longrightarrow F_1^{\mathrm{qp}} \C \longrightarrow \cdots \longrightarrow F_n^{\mathrm{qp}} \C \longrightarrow \cdots
$$
is an $\omega$-ladder of curved $\C$-coalgebras, called the \textit{quasi-planar ladder}. The colimit of this diagram is again $W$. 
\end{corollary}

\begin{proof}
Follows directly from Proposition \ref{prop: cooperad ladder induces ladders}.
\end{proof}

Using ladders, we can define the notion of a filtered quasi-isomorphism of ladders.

\begin{definition}[Filtered quasi-isomorphism of ladders] 
A morphism of $\beta$-ladders of curved $\operad C$-coalgebras $f: W \longrightarrow W'$ is a \textit{filtered quasi-isomorphism} if 

\[
\gr_i (f): \gr_i W \qi \gr_i W'
\]
\vspace{0.1pc}

is a quasi-isomorphism for all $i \in \beta$. 
\end{definition}

\begin{proposition}\label{prop: quasi-iso filtré entre les bar}
    Let $f: A \qi A'$ be a quasi-morphism of dg $\Omega \operad C$-algebras. The morphism of quasi-planar ladders
    
    \[
    F_i^{\mathrm{qp}}\mathrm{B}_{\operad C}(f): F_i^{\mathrm{qp}}\mathrm{B}_{\operad C} A \qi F_i^{\mathrm{qp}}\mathrm{B}_{\operad C} A'
    \]
    \vspace{0.1pc}
    
    is a filtered quasi-isomorphism.
\end{proposition}

\begin{proof}
    For every $i \in \omega$, the map
    $$
    \gr_i^{\mathrm{qp}}(\mathrm{B}_{\operad C} A) \longrightarrow \gr_i^{\mathrm{qp}}(\mathrm{B}_{\operad C} A')
    $$
    can be rewritten as the morphism of dg modules
    $$
    \gr_i^{\mathrm{qp}}(\operad C_\pl) \comp_\pl A \longrightarrow \gr_i^{\mathrm{qp}}(\operad C_\pl) \comp_\pl A'
    $$
    which is a quasi-isomorphism.
\end{proof}


\subsection{Cofibrations}
We characterize cofibrations as degree-wise injections and subsequently prove that any filtered quasi-isomorphism is a weak-equivalence.

\begin{lemma}\label{lemmacofibalgdegree-wiseinj}
A cofibration of dg $\Omega \operad C$-algebras is in particular a degree-wise injection.
\end{lemma}

\begin{proof}
Let us consider the acyclic dg module $D^1$ equipped with
its canonical structure of a unital associative commutative algebra.
Let $f:A \longrightarrow A'$ be a cofibration of dg $\Omega\operad C$-algebras. Its lifting property and the fact that the map $D^1 \otimes A \longrightarrow 0$ is an acyclic fibration imply that the inclusion $A \rightarrowtail D^1 \otimes A$ which is a degree-wise injection factors through $f$. Therefore $f$ is also a degree-wise injection.
\end{proof}

\begin{proposition}\label{propcofibrations}
A morphism of curved $\operad C$-coalgebras is a cofibration if and only if it is a degree-wise injection.
\end{proposition}

\begin{proof}
Let $f: W \longrightarrow W'$ be a morphism of curved $\operad C$-coalgebras. On the one hand, let us suppose that $f$ is a degree-wise injection. Then it can be recovered from the transfinite composition of the sequence
$$
W \longrightarrow W + F_1^{\mathrm{rad}}W' \longrightarrow \cdots \longrightarrow W + F_n^{\mathrm{rad}}W' \longrightarrow \cdots~.
$$
Every morphism in this sequence is an cofibration since it is an elementary cofibration. Cofibrations are stable by transfinite compositions, therefore $f$ is also a cofibration.

\medskip

On the other hand, let us suppose that $f$ is a cofibration. We consider the following diagram of graded $\kk$-modules
$$
\begin{tikzcd}[column sep=3pc,row sep=3pc]
    W
    \ar[r, rightarrowtail] \ar[d,"f",swap]
    & \Omega \operad C \comp W
    \ar[d,"\Omega_\C(f)"]
    \\
    W'
    \ar[r,rightarrowtail]
    & \Omega \operad C \comp W'
\end{tikzcd}
$$
The top horizontal maps are clearly degree-wise injections. The map $\Omega_\C(f)$ is by definition a cofibration of dg $\Omega \C$-algebras, therefore it is in particular a degree-wise injection by Lemma \ref{lemmacofibalgdegree-wiseinj}. This implies that $f$ is also a degree-wise injection.
\end{proof}

\begin{proposition}
Let $f: W \longrightarrow W'$ be a filtered quasi-isomorphism of $\beta$-ladders of curved $\C$-coalgebras. The map $f(\beta): W(\beta) \qi W'(\beta)$ is a weak-equivalence.
\end{proposition}

\begin{proof}
Notice that the following holds.

\medskip

\begin{enumerate}
    \item The map $f(0)$ is a weak-equivalence, since it is the identity of the zero object $0$.
    
\medskip

    \item If $i \in \beta+1$ is a limit ordinal so that $f(j)$ is an equivalence
    for every $j < i$, then the colimits
    $$
    \Omega_{\operad C} W(i) = \colim{j<i}~ \Omega_{\operad C} W(j), \quad \Omega_{\operad C} W'(i) = \colim{j<i}~ \Omega_{\operad C} W'(j)
    $$
    are homotopy colimits, and therefore the map $\Omega_{\operad C}f(i)$ is a quasi-isomorphism. Thus $f(i)$ is a weak equivalence. 
    
\medskip

    \item By Proposition \ref{propelemcofibequiv}, $f(i+1)$ is a weak-equivalence whenever $f(i)$ is a weak-equivalence.
    
\medskip
\end{enumerate}
We conclude by an ordinal induction.
\end{proof}


\subsection{The cylinder object}

Let $W$ be a curved $\operad C$-coalgebra. Let $A$ be a cylinder object of $\Omega_{\operad C} W$ in the category of dg $\Omega \C$-algebras, that is, a dg $\Omega\C$-algebra together with a factorisation 

\[
\begin{tikzcd}[column sep=3pc,row sep=3pc]
\Omega_{\operad C} (W \oplus W)
\ar[r,"i",rightarrowtail]
&A \ar[r,"p",twoheadrightarrow]
&\Omega_{\operad C} W
\end{tikzcd}
\]

where $i$ is a cofibration and $p$ is an acyclic fibration of dg $\Omega \C$-algebras. Let $\mathrm{Cyl}(W)$ be the following pullback in in the category of curved $\operad C$-coalgebras

\[
\mathrm{Cyl}(W) \coloneqq \mathrm{B}_{\operad C}A \times_{\mathrm{B}_{\operad C}\Omega_{\operad C}(W)} W.
\]
\vspace{0.1pc}

Our goal is to show that $\mathrm{Cyl}(W)$ is a natural cylinder object for $W$.

\medskip

Notice that $\mathrm{Cyl}(W)$ fits in the following diagram
$$
\begin{tikzcd}[column sep=3pc,row sep=3pc]
\mathrm{B}_{\operad C}\Omega_{\operad C} (W \oplus W)
\ar[r,"\mathrm{B}_{\operad C}(i)",rightarrowtail]
& \mathrm{B}_{\operad C}A \ar[r,"\mathrm{B}_{\operad C}(p)"]
& \mathrm{B}_{\operad C}\Omega_{\operad C} W 
\\
W \oplus W \ar[r, "i^{\dagger} \times \nabla"] \ar[u,rightarrowtail,"\eta_{W \oplus W}"]
&\mathrm{Cyl}(W) \ar[r,"\mathrm{proj}_{W}"] \ar[u,"j",rightarrowtail]
& W~. \ar[u,"\eta_W",rightarrowtail,swap]
\end{tikzcd}
$$

of curved $\C$-coalgebras. The morphism $i^{\dagger}: W \oplus W \longrightarrow \mathrm{B}_{\operad C}A$ is the transpose of $i$ and the morphism $\nabla: W \oplus W \longrightarrow W$ is the universal codiagonal morphism. The morphism $\eta$ is the unit of the bar-cobar adjunction $\Omega_{\operad C} \dashv \mathrm{B}_{\operad C}$, which is a degree-wise injection, and thus a cofibration by Proposition \ref{propcofibrations}. Notice that $\mathrm{B}_{\operad C}(p)$ is a filtered quasi-isomorphism of quasi-planar ladders by Proposition \ref{prop: quasi-iso filtré entre les bar}.

\medskip

Let us choose a particular summand $W$ in the direct sum $W \oplus W$. We choose one of the two sections $s:W \longrightarrow \mathrm{Cyl}(W)$ of the map $\mathrm{proj}_{W}: \mathrm{Cyl}(W) \longrightarrow W$ and one of the two sections $d: \Omega_{\operad C}W \longrightarrow A$ of the map  $p: A \longrightarrow \Omega_{\operad C}W$, in such a way that the following diagram of curved $\operad C$-coalgebras
$$
\begin{tikzcd}[column sep=3pc,row sep=3pc]
\mathrm{B}_{\operad C}\Omega_{\operad C}W
\ar[r,"\mathrm{B}_{\operad C}(d)"]
& \mathrm{B}_{\operad C}A \ar[r,"\mathrm{B}_{\operad C}(p)"]
& \mathrm{B}_{\operad C}\Omega_{\operad C}W
\\
W \ar[r,"s"] \ar[u,rightarrowtail]
& \mathrm{Cyl}(W) \ar[r,"\mathrm{proj}_{W}"] \ar[u,"j",rightarrowtail]
&W \ar[u,rightarrowtail]
\end{tikzcd}
$$
commutes. Moreover, the dg module $A$ decomposes into $A = \Omega_{\operad C} W \oplus K$ where $K$ is the kernel of the map $p: A \longrightarrow \Omega_{\operad C}W$. The dg module $K$ is acyclic, since $p$ is a quasi-isomorphism. Let us take a contracting homotopy $h$ of $K$, that is, a degree $1$ endomorphism of $K$ such that 
$$
\partial(h) = d_K h + h d_K = \id_K.
$$
We can extend $h$ to $A$ by zero on $\Omega_{\operad C} W$. Then $\partial(h) = \pi_K$. Now let $H$ be the degree $1$
endormophism of the graded complex $\mathrm{B}_{\operad C}A = \operad C_\pl \comp_\pl A$ defined as follows
$$
H \coloneqq \sum_{k=0}^n \id_{\operad C_\pl(n)} \otimes (\pi_{\Omega_{\operad C} W}^{\otimes k} \otimes h \otimes \id_A^{\otimes n-k-1}) 
\quad \text{ on } \quad \operad C_\pl(n) \otimes A^{\otimes n} \quad \text{for} \quad n \geq 1,
$$
and 
\[
H=0 \quad \text{on} \quad \operad C_\pl(0)~.
\]
\vspace{0.1pc}

In particular, for every $0 \leq k \leq n-1$, the restriction of $H$ to $\operad C_\pl(n) \otimes ((\Omega_{\operad C} W)^{\otimes k} \otimes K \otimes A^{\otimes n-k-1}) $ is given by 
$$
\begin{tikzcd}
\operad C_\pl(n) \otimes ((\Omega_{\operad C} W)^{\otimes k} \otimes K \otimes A^{\otimes n-k-1}) 
\ar[dd,"{\id_{\operad C_\pl(n)} \otimes \id_{\Omega_{\operad C} W}^{\otimes k} \otimes\  h \ \otimes \id_A^{\otimes n-k-1}}"]
\\ 
\\
\operad C_\pl(n) \otimes ((\Omega_{\operad C} W)^{\otimes k} \otimes K \otimes A^{\otimes n-k-1}) 
\ar[d, hook]
\\
\operad C_\pl \comp_\pl A~.
\end{tikzcd}
$$
One can extend $H$ to $\operad C_\pl \comp_\pl \mathrm{B}_{\operad C} A = \operad C_\pl \comp_\pl \operad C_\pl \comp_\pl A$ through the same formula

\[
H \coloneqq
 \begin{cases}
 \sum_{k=0}^n \id_{\operad C_\pl(n)} \otimes (\pi_{\mathrm{B}_{\operad C}\Omega_{\operad C} W}^{\otimes k} \otimes H \otimes \id_A^{\otimes n-k-1}) 
\text{ on } \operad C_\pl(n) \otimes (\mathrm{B}_{\operad C} A)^{\otimes n} \quad \text{for} \quad n \geq 1;
\\
0 \quad \text{on} \quad \operad C_\pl(0)~.
 \end{cases}
\]
\vspace{0.1pc}

The same formula mutatis mutandis allows us to extend $H$
to $\operad C_\pl \comp_\pl \operad C_\pl \comp_\pl \mathrm{B}_{\operad C}A = \operad C_\pl \comp_\pl \operad C_\pl \comp_\pl \operad C_\pl \comp_\pl A$. One can notice then that $H$ commutes with the maps
$$
\mathrm{B}_{\operad C}A \longrightarrow
 \operad C_\pl \comp_\pl \mathrm{B}_{\operad C}A
\rightrightarrows \operad C_\pl \comp_\pl \operad C_\pl \comp_\pl \mathrm{B}_{\operad C}A.
$$

\begin{lemma}
The subobject $\mathrm{Cyl}(W) \rightarrowtail \mathrm{B}_{\operad C} A$ is stable by $H$.
\end{lemma}

\begin{proof}
Let $\mathrm{Cyl}(W)'$ be the pullback 
\[
\begin{tikzcd}[column sep=3pc,row sep=3pc]
\mathrm{Cyl}(W)' \arrow[dr, phantom, "\lrcorner", very near start]
\ar[r,rightarrowtail] \ar[d]
& \mathrm{B}_{\operad C}A 
\ar[d]
\\
W
\ar[r,rightarrowtail]
& \mathrm{B}_{\operad C}\Omega_{\operad C} W;
\end{tikzcd}
\]

in the category of pdg $\operad C$-coalgebras. By universal property, one has a morphism of pdg $\operad C$-coalgebras $\mathrm{Cyl}(W) \longrightarrow \mathrm{Cyl}(W)'$. Let us show that $\mathrm{Cyl}(W)'$ is curved and therefore isomorphic to $\mathrm{Cyl}(W)$.

\medskip

Let $X$ be the pullback of the span $\mathrm{B}_{\operad C}A \longrightarrow \mathrm{B}_{\operad C}\Omega_{\operad C} W \longleftarrow W$ in the category of pdg $\kk$-modules and let $Y$ be the pullback of the span $\operad C_\pl \comp_\pl \mathrm{B}_{\operad C}A \longrightarrow \operad C_\pl \comp_\pl \mathrm{B}_{\operad C}\Omega_{\operad C} W \longleftarrow \operad C_\pl \comp_\pl W$ in the category of pdg $\kk$-modules. The pdg $\operad C$-coalgebra $\mathrm{Cyl}(W)'$ can be computed as the following equalizer
$$
\begin{tikzcd}
    \mathrm{Cyl}(W)' \ar[r,dashed]
    &\operad C \comp X
    \ar[rr, bend left] \ar[rr, bend right]
    &&\operad C \comp Y~,
    \ar[ll]
\end{tikzcd}
$$

both in the category pdg $\operad C$-coalgebras and in the category of pdg $\kk$-modules. This follows from the results of Subsection \ref{sectioncommuteslimitscolimitscog}, as the forgetful functor preserves finite cosifted limits. The following square of pdg $\operad C$-coalgebras
\[
\begin{tikzcd}[column sep=3pc,row sep=3pc]
\mathrm{Cyl}(W)' 
\ar[r] \ar[d,rightarrowtail]
& \mathrm{B}_{\operad C}A 
\ar[d]
\\
\C \comp X
\ar[r,rightarrowtail]
& \C \comp \mathrm{B}_{\operad C}A;
\end{tikzcd}
\]
commutes. The bottom horizontal and the left vertical maps are monomorphisms, which implies that the map $\mathrm{Cyl}(W)' \rightarrowtail \mathrm{B}_{\operad C}A$ is also a monomorphism. Therefore $\mathrm{Cyl}(W)'$ is also curved and the canonical morphism $\mathrm{Cyl}(W) \longrightarrow \mathrm{Cyl}(W)'$ is an isomorphism of curved $\C$-coalgebras.

\medskip

It is clear that the sub-objects $\operad C \comp X \subset \operad C_\pl \comp_\pl  \mathrm{B}_{\operad C}A$ and  $\operad C \comp Y \subset \operad C_\pl \comp_\pl \operad C_\pl \comp_\pl  \mathrm{B}_{\operad C}A$ are stable through $H$. Hence, so is the sub-object $\mathrm{Cyl}(W) \subset \mathrm{B}_{\operad C}A$.
\end{proof}

\begin{lemma}
The subobjects $F_{i}^{\mathrm{qp}} \mathrm{B}_{\operad C}A$ and $F_{i}^{\mathrm{qp}} \mathrm{Cyl}(W)$ are stable by $H$ for every $i \in \alpha$.
\end{lemma}

\begin{proof}
It is clear by definition of $H$ that $F_{i}^{\mathrm{qp}} \mathrm{B}_{\operad C} A = F_{i}^{\mathrm{qp}}\operad C \comp A$ is stable by it. Recall that $F_{i}^{\mathrm{qp}} \mathrm{Cyl}(W)$ is given by the following pullback square
    $$
    \begin{tikzcd}[column sep=3pc,row sep=3pc]
        F_{i}^{\mathrm{qp}} \mathrm{Cyl}(W) \arrow[dr, phantom, "\lrcorner", very near start]
        \ar[r,rightarrowtail] \ar[d,rightarrowtail]
        & \mathrm{Cyl}(W)
        \ar[d,rightarrowtail,"j"]
        \\
        F_{i}^{\mathrm{qp}} \mathrm{B}_{\operad C} A
        \ar[r,rightarrowtail]
        & \mathrm{B}_{\operad C} A
    \end{tikzcd}
    $$
    in the category of graded $\kk$-modules. We already know that $\mathrm{Cyl}(W)$ and $F_{i}^{\mathrm{qp}} \mathrm{B}_{\operad C} A$ are stable by $H$. Therefore $F_{i}^{\mathrm{qp}} \mathrm{Cyl}(W)$ is also stable by $H$.
\end{proof}

\begin{lemma}\label{lemma: endo partial(h) vaut id - proj sur le gros truc (cogèbres)}
For every $i +1 \in \omega$, the endomorphism $\partial(H)$ of $\gr_{i+1}^{\mathrm{qp}} \mathrm{B}_{\operad C}A$ is equal to 
the identity minus the projection onto $\gr_{i+1}^{\mathrm{qp}} \mathrm{B}_{\operad C}\Omega_{\operad C} W$.
\end{lemma}

\begin{proof}
    For every natural integer $n$ and every $0 \leq k \leq n-1$, the summands
    
    \begin{align*}
\gr_{i+1}^{\mathrm{qp}} \operad C_\pl(n) \otimes (\Omega_{\operad C} W)^{\otimes n}
\subset \gr_{i+1}^{\mathrm{qp}}\operad C_\pl \comp_\pl A \cong \gr_{i+1}^{\mathrm{qp}} \mathrm{B}_{\operad C} A~,      \\
 \\
\gr_{i+1}^{\mathrm{qp}}\operad C_\pl(n) \otimes (\Omega_{\operad C} W)^{\otimes k}
\otimes K \otimes A^{n-k-1}
\subset \gr_{i+1}^{\mathrm{qp}}\operad C_\pl \comp_\pl A \cong \gr_{i+1}^{\mathrm{qp}} \mathrm{B}_{\operad C} A~,
    \end{align*}
    \vspace{0.1pc}
    
    are both stable by the differential and by $H$. Then, a direct inspection shows that $\partial(H)$ is zero on the first summand and the the identity on the second one, which concludes the proof.
\end{proof}

\begin{proposition}\label{corollarycylinder}
For every $i +1 \in \alpha$, the endomorphism $\partial(H)$ of $\gr_{i+1} \mathrm{Cyl}(W)$ is equal to the identity minus the projection onto $\gr_{i+1}^{\mathrm{qp}} \C$. Therefore the maps

\[
\begin{tikzcd}[column sep=3.5pc,row sep=3pc]
\gr_{i+1}^{\mathrm{qp}} W  \arrow[r,"\gr_{i+1}^{\mathrm{qp}}(s)"]
&\gr_{i+1}^{\mathrm{qp}} \mathrm{Cyl}(W) \arrow[r,"\gr_{i+1}^{\mathrm{qp}}(\mathrm{proj}_W)"]
&\gr_{i+1}^{\mathrm{qp}} W
\end{tikzcd}
\]
\vspace{0.1pc}

are quasi-isomorphisms.
\end{proposition}

\begin{proof}
Let us consider the following commutative diagram of dg modules
$$
\begin{tikzcd}[column sep=3.5pc,row sep=3.5pc]
\gr_{i+1}^{\mathrm{qp}}\mathrm{B}_{\operad C}\Omega_{\operad C}W
\ar[r, "\gr_{i+1}^{\mathrm{qp}}(\mathrm{B}_{\operad C}(d))"]
& \gr_{i+1}^{\mathrm{qp}}\mathrm{B}_{\operad C}A \ar[r, "\gr_{i+1}^{\mathrm{qp}}(\mathrm{B}_{\operad C}(p))"]
& \gr_{i+1}^{\mathrm{qp}} \mathrm{B}_{\operad C}\Omega_{\operad C}W 
\\
\gr_{i+1}^{\mathrm{qp}} W  \ar[r, "\gr_{i+1}^{\mathrm{qp}}(s)"] \ar[u, "\gr_{i+1}^{\mathrm{qp}}(\eta_C)"]
& \gr_{i+1}^{\mathrm{qp}} \mathrm{Cyl}(W) \ar[r, "\gr_{i+1}^{\mathrm{qp}}(\mathrm{proj}_W)"] \ar[u, "\gr_{i+1}^{\mathrm{qp}}(j)"]
& \gr_{i+1}^{\mathrm{qp}} W \ar[u, "\gr_{i+1}^{\mathrm{qp}}(\eta_W)"]
\end{tikzcd}
$$

Let us denote by $\pi_W$ the composition $~\gr_{i+1}^{\mathrm{qp}}(s)~\gr_{i+1}^{\mathrm{qp}}(\mathrm{proj}_W)~$ and by $\pi_{\mathrm{B}_{\operad C}\Omega_{\operad C}W}$ the composition $~\gr_{i+1}^{\mathrm{qp}}(\mathrm{B}_{\operad C}(d))
~\gr_{i+1}^{\mathrm{qp}}(\mathrm{B}_{\operad C}(p))~$.
By Lemma \ref{lemma: endo partial(h) vaut id - proj sur le gros truc (cogèbres)}, we know that:

\[
\partial(h) = (\id - \pi_{\mathrm{B}_{\operad C}\Omega_{\operad C}W})~.
\]
\vspace{0.1pc}

Since $\gr_{i+1}^{\mathrm{qp}}(j)$ commutes with the differential $d$ and $H$, we have that $\gr_{i+1}^{\mathrm{qp}}(j) \partial(h) = \partial(h) \gr_{i+1}^{\mathrm{qp}}(j)$ and we compute that

\[
\gr_{i+1}^{\mathrm{qp}}(j) \partial(h) = \gr_{i+1}^{\mathrm{qp}}(j)(\id - \pi_W)~.
\]
\vspace{0.1pc}

Since $\gr_{i+1}^{\mathrm{qp}}(j)$ is a monomorphism, it implies that $\partial(h) = \id - \pi_W$.
\end{proof}

\begin{proposition}\label{propcylinder}
The factorisation 
\[
\begin{tikzcd}[column sep=3.5pc,row sep=3pc]
W \oplus W \arrow[r,"i^{\dagger} \times \nabla"]
&\mathrm{Cyl}(W) \arrow[r,"\mathrm{proj}_W"]
&W
\end{tikzcd}
\]
\vspace{0.1pc}

makes $\mathrm{Cyl}(W)$ a good cylinder object of $W$, in the sense that

\medskip

\begin{enumerate}
\item the map $i^{\dagger} \times \nabla: W \oplus W \rightarrowtail \mathrm{Cyl}(W)$ is a cofibration,

\medskip

\item the map $\mathrm{proj}_W: \mathrm{Cyl}(W) \qi W$ is a weak-equivalence.

\medskip
\end{enumerate}
\end{proposition}

\begin{proof}
The morphism $i^{\dagger} \times \nabla: W \oplus W \longrightarrow \mathrm{Cyl}(W)$ is a degree-wise monomorphism since both $\eta_W: W \oplus W \longrightarrow \mathrm{B}_{\operad C} \Omega_{\operad C}(W \oplus W)$ and $\mathrm{B}_{\operad C}(i): \mathrm{B}_{\operad C} \Omega_{\operad C}(W \oplus W) \longrightarrow \mathrm{B}_{\operad C} A$ are. Thus, it is a cofibration. To conclude, Proposition \ref{corollarycylinder} tells us that the map $\mathrm{proj}_W: \mathrm{Cyl}(W) \longrightarrow W$ is a filtered quasi-isomorphism. Thus it is a weak-equivalence.
\end{proof}

\begin{remark}
Actually, the map $\mathrm{proj}_W: \mathrm{Cyl}(W) \twoheadrightarrow W$ is also a fibration as the pullback of a fibration.
\end{remark}


\newpage

\section{A Quillen equivalence, $\infty$-morphisms and homotopy transfer theorems for algebras}

\vspace{2pc}

Let $\C$ be a quasi-planar curved conilpotent cooperads. We show that the bar-cobar adjunction relating $\Omega \C$-algebras to $\C$-coalgebras is a Quillen equivalence. This allows us to give another presentation of the homotopy category of dg $\Omega \C$-algebras in terms of curved $\C$-coalgebras together with their transferred model category structure. We introduce $\infty$-morphisms of dg $\Omega \C$-algebras and show that they are (up to homotopy) invertible. We prove a homotopy transfer theorem for dg $\Omega \C$-algebras. Finally, we show how another model categories structures on the category of curved $\C$-coalgebras can be obtained by a left Bousfield localization. 


\subsection{The Quillen equivalence}
Our goal is to prove that the bar-cobar adjunction can be promoted into a Quillen equivalence. Similar theorems in characteristic zero were proven in \cite{Hinich,LefevreHasegawa03,Brunohomotopy,unitalalgebras}. For that, we leverage the quasi-planar context to build a contracting homotopy that works in positive characteristic, as we explain in Subsection \ref{section: The contracting homotopy on trees}.

\begin{theorem}\label{thm: the bar-cobar is a Quillen equivalence}
The adjunction 

\[
\begin{tikzcd}[column sep=5pc,row sep=3pc]
          \mathsf{dg}~\C\text{-}\mathsf{cog} \arrow[r, shift left=1.1ex, "\Omega_{\C}"{name=F}] & \mathsf{dg}~ \Omega \C\text{-}\mathsf{alg}, \arrow[l, shift left=.75ex, "\mathrm{B}_{\C}"{name=U}]
            \arrow[phantom, from=F, to=U, , "\dashv" rotate=-90]
\end{tikzcd}
\]
\vspace{0.1pc}

is a Quillen equivalence, when one considers the transferred model category structure on the category of curved $\C$-coalgebras.
\end{theorem}

\begin{proof}
The theorem follows directly from Lemma \ref{lemmaquillenequivalence}.
\end{proof}

Before going into the proofs, let us spell out a direct consequence of Theorem \ref{thm: the bar-cobar is a Quillen equivalence}.

\begin{corollary}\label{cor: quasi-planar bar-cobar is an equivalence}
Let $\operad P$ be an cofibrant dg operad. The quasi-planar bar-cobar adjunction

\[
\begin{tikzcd}[column sep=5pc,row sep=3pc]
          \mathsf{dg}~\operad P\text{-}\mathsf{alg} \arrow[r, shift left=1.1ex, "\mathrm{B}^{\mathrm{q.p}}_{\pi}"{name=A}]
         &\mathsf{curv}~\mathrm{B}(\operad E \otimes \operad P) \text{-}\mathsf{cog}, \arrow[l, shift left=.75ex, "\Omega^{\mathrm{q.p}}_{\pi}"{name=B}] \arrow[phantom, from=A, to=B, , "\dashv" rotate=90]
\end{tikzcd}
\]

is a Quillen equivalence, when curved $\mathrm{B}(\operad E \otimes \operad P)$-coalgebras are endowed with the transferred structure from dg $\Omega \mathrm{B}(\operad E \otimes \operad P)$-algebras. 
\end{corollary}

\begin{proof}
Let us denote $\psi_{\operad P} = \Omega \mathrm{B}(\varphi_{\operad P}): \Omega \mathrm{B}(\operad E \otimes \operad P) \qi \operad P$ the canonical quasi-isomorphism, and $\epsilon: \Omega \mathrm{B}(\operad E \otimes \operad P) \qi \operad E \otimes \operad P$ the quasi-isomorphism given by the counit map. 

\medskip

First notice that the adjunction $(\psi_{\operad P})_! \dashv \psi_{\operad P}^*$ is a Quillen equivalence since both $\Omega \mathrm{B}(\operad E \otimes \operad P)$ and $\operad P$ are cofibrant dg operads. The adjunctions $(\varphi_{\operad P})_! \dashv \varphi_{\operad P}^*$ and $(\epsilon)_! \dashv \epsilon^*$ are in general just Quillen adjunctions. 

\medskip

Let $A$ be a dg $\operad P$-algebra. Our goal is to show that the counit 

\[
\epsilon_A: \Omega^{\mathrm{q.p}}_{\pi}\mathrm{B}^{\mathrm{q.p}}_{\pi} A \longrightarrow A
\]

is a quasi-isomorphism. A direct computation shows that 

\[
\Omega^{\mathrm{q.p}}_{\pi}\mathrm{B}^{\mathrm{q.p}}_{\pi} \cong (\varphi_{\operad P})_! \Omega_{\pi}\mathrm{B}_{\pi} A \varphi_{\operad P}^* \cong (\psi_{\operad P})_!(\epsilon)_!\Omega_{\pi}\mathrm{B}_{\pi} \epsilon^* \psi_{\operad P}^* \cong (\psi_{\operad P})_!\Omega_{\iota}\mathrm{B}_{\iota}\psi_{\operad P}^*~.
\]
\vspace{0.1pc}

Therefore since the counit of $(\psi_{\operad P})_! \dashv \psi_{\operad P}^*$ and of $\Omega_{\iota}\dashv \mathrm{B}_{\iota}$ are quasi-isomorphisms, so is $\epsilon_A$. Showing that for a curved $\mathrm{B}(\operad E \otimes \operad P)$-coalgebra, the unit of adjunction is a weak-equivalence is completely analogous.
\end{proof}

\begin{remark}
If the dg operad $\operad P$ is just $\mathbb{S}$-projective, the above result can be adapted, and homotopy category of curved $\mathrm{B}(\operad E \otimes \operad P)$-coalgebras given by localizing at transferred weak-equivalences is still equivalent to the homotopy category of dg $\operad P$-algebras given by localizing at quasi-isomorphisms.
\end{remark}

\begin{lemma}\label{lemmaquillenequivalence}
For every dg $\Omega \C$-algebra $A$, the counit map $\epsilon_A: \Omega_{\operad C}\mathrm{B}_{\operad C}A \qi A$ is a quasi-isomorphism.
\end{lemma}

In the following paragraphs, our strategy to prove Lemma \ref{lemmaquillenequivalence} is the following. First, we construct a section $\zeta_A$ of the counit map $\epsilon_A : \Omega_{\operad C}\mathrm{B}_{\operad C}A \longrightarrow A$ in the category of dg modules. The goal is then to show that $\pi_A = \zeta_A \epsilon_A$ is homotopic to an isomorphism on $\Omega_{\operad C}\mathrm{B}_{\operad C}A$. This is done by constructing an explicit degree $1$ endomorphism $H$ and by showing in Proposition \ref{keypropquilleneqcoalg} that $\partial(H) + \pi_A$ is an isomorphism. Therefore $\epsilon_A$ has a left inverse, $\zeta$, and a right inverse, $\zeta_A (\partial(H) + \pi_A)^{-1}$, in the homotopy category of dg modules. This implies that $\epsilon_A$ is a isomorphism in the homotopy category, and therefore a quasi-isomorphism.

\subsection{The contracting homotopy on trees}
\label{section: The contracting homotopy on trees}
The key ingredient to showing Lemma \ref{lemmaquillenequivalence} is to construct a degree $1$ endomorphism of $\Omega \C \comp \C$ that will yield contracting homotopy for algebras. The same construction will also work for coalgebras, and is used in the proof of Lemma \ref{lemma: key lemma of Quillen equiv for coalg and absolute alg}.

\medskip

We abbreviate the dg operad $\Omega\C$ by $\operad P$ . Let us denote:

\medskip

\begin{enumerate}
    \item $\operad P_\pl = \mathbb{T}_\pl s^{-1}\overline{\operad C_\pl}$~,
    \medskip
    \item $\overline{\operad P_\pl} = \overline{\mathbb{T}_\pl} s^{-1}\overline{\operad C_\pl}$.
\end{enumerate}
\medskip

\textbf{The map h.} The graded $\mathbb N$-module $\operad P_\pl \comp_\pl \operad C_\pl = (\treemod_\pl (s^{-1}\overline{\operad C}_\pl)) \comp_\pl \operad C_\pl$ is made of planar trees such that \textit{some} top nodes are labelled by $\C_\pl$ and the other nodes are labelled by $s^{-1} \overline{\C}_\pl$. Let $h$ be its degree $1$ endomorphism that consists in
taking the leftest top node of a labelled planar tree and
\begin{enumerate}
    \item if this node is labelled by an element of $s^{-1}\overline{\operad C_\pl}$, $h$ suspends this element;
    \item otherwise $h$ sends this labelled tree to zero.
\end{enumerate}
Pictorially $h$ it is given as follows
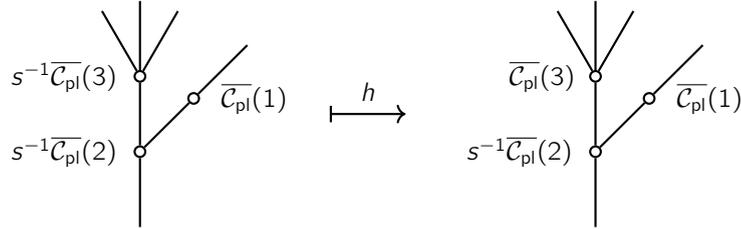
\usetikzlibrary{calc}
    \begin{center}
\begin{tikzpicture}
		\tikzstyle{vertex}=[draw, circle, thick, inner sep=0,
		minimum size=4]
		\tikzstyle{every path}=[thick]
		\node [vertex] (a) at (0,0) {};
		\node [vertex] (c) at (0,1) {};
		\node [vertex] (d) at ($(a)+(45:1cm)$) {};
		\node [left of=c, node distance=1cm] {$s^{-1}\overline{\operad C_\pl}(3)$};
		\node [left of=a, node distance=1cm] {$s^{-1}\overline{\operad C_\pl}(2)$};
		\node [right of=d, node distance=0.8cm] {$\overline{\operad C_\pl}(1)$};
		\node at (3,0.8) {$h$};
		\draw (a) -- (c);
		\draw (a) -- (0,-1);
		\draw (c) -- ($(c)+(60:1cm)$);
		\draw (c) -- ($(c)+(120:1cm)$);
		\draw (c) -- (0,2);
		\draw (a) -- (d);
		\draw (d) -- ($(d)+(45:1cm)$);
		\draw[|->] (2.5,0.5) -- (3.5,0.5);
		\node [vertex] (a2) at (6,0) {};
		\node [vertex] (c2) at (6,1) {};
		\node [vertex] (d2) at ($(a2)+(45:1cm)$) {};
		\node [left of=c2, node distance=0.7cm] {$\overline{\operad C_\pl}(3)$};
		\node [left of=a2, node distance=1cm] {$s^{-1}\overline{\operad C_\pl}(2)$};
		\node [right of=d2, node distance=0.8cm] {$\overline{\operad C_\pl}(1)$};
		\draw (a2) -- (c2);
		\draw (a2) -- (6,-1);
		\draw (c2) -- ($(c2)+(60:1cm)$);
		\draw (c2) -- ($(c2)+(120:1cm)$);
		\draw (c2) -- (6,2);
		\draw (a2) -- (d2);
		\draw (d2) -- ($(d2)+(45:1cm)$);
\end{tikzpicture}
\end{center}
In more precise terms, $h$ is defined as follows by induction on the height of the planar trees that make $\operad P_\pl$:

\medskip

\begin{enumerate}
    \item On $\operad I \comp_\pl \operad C_\pl$, $h$ is zero.
    
\medskip

    \item On $s^{-1}\overline{\operad C_\pl} \comp_\pl \operad C $, $h$ is given by
    \begin{align*}
    s^{-1}\overline{\operad C_\pl} \comp_\pl \operad C 
    \twoheadrightarrow
    s^{-1}\overline{\operad C}_\pl \comp_\pl \operad I
    &\longrightarrow \operad I \comp_\pl \overline{\operad C}_\pl
    \hookrightarrow \operad P_\pl \comp_\pl \operad C_\pl
    \\
    s^{-1}x & \mapsto x~.
    \end{align*}

\medskip

    \item Then on 
    $$
    (\overline{\treemod}_{\pl, \leq n+1}(s^{-1}\overline{\operad C}_\pl) ) \comp_\pl \operad C_\pl
    \simeq 
    s^{-1} \overline{\C}_\pl \comp_\pl ({\treemod}_{\pl, \leq n}(s^{-1}\overline{\operad C}_\pl)) \comp_\pl  \comp_\pl \operad C_\pl
    $$
    $h$ is defined inductively by the sum of the map
    \begin{align*}
    (\overline{\treemod}_{\pl, \leq n+1}(s^{-1}\overline{\operad C}_\pl) ) \comp_\pl \operad C_\pl
    \twoheadrightarrow
    s^{-1}\overline{\operad C}_\pl \comp_\pl \operad I
    &\longrightarrow \operad I \comp_\pl \overline{\operad C}_\pl
    \hookrightarrow \operad P_\pl \comp_\pl \operad C_\pl
    \\
    s^{-1}x & \mapsto x 
    \end{align*}
    with the map
    
    $$
    \id_{s^{-1}\overline{\operad C_\pl}(n)} \otimes \left(
    \sum_{k=0}^{n-1} \pi_{\operad I \comp_\pl \operad I}^{\otimes k}
    \otimes h \otimes 
    \id_{\operad P_\pl \comp_\pl \operad C_\pl}^{\otimes n-k-1}
    \right) .
    $$
    \vspace{0.1pc}
\end{enumerate}

\medskip

\textbf{The qp filtration on trees.} We consider the filtration on $\operad P_\pl \comp_\pl \operad C_\pl = (\treemod_\pl (s^{-1}\overline{\operad C}_\pl)) \comp_\pl \operad C_\pl$ induced by the qp filtration of $\operad C_\pl$. The $n$-th term $F_n^{\qp} (\operad P_\pl \comp_\pl \C_\pl)$ is the subobject of $\operad P_\pl \comp_\pl \C_\pl$ made of planar trees labelled by elements of $\C_\pl$ and $s^{-1}\overline{\C}_\pl$ such that the sum of the qp weights of these elements does not exceed $n$. Remember that the non top nodes are only labelled by $s^{-1}\overline{\C}_\pl$. More precisely
\begin{align*}
    F_0^{\qp} (\operad P_\pl \comp_\pl \C_\pl) &\coloneqq \operad I \comp_\pl \operad I~,
    \\
    F_n^{\qp}  (\operad P_\pl \comp_\pl \C_\pl)
    &\coloneqq \operad I \comp_\pl (F_n^\qp \operad C_\pl)~\oplus
    \\
    &\bigoplus_{k \geq 0} \sum_{j_0 + j_1 + \cdots + j_k=n, j_0>0} (s^{-1}F_{j_0}^{\qp} \overline{\operad C_\pl}(k)) \otimes \left(F_{j_1}^{\qp} (\operad P_\pl \comp_\pl \C_\pl) \otimes \cdots \otimes 
    F_{j_1}^{\qp}  (\operad P_\pl \comp_\pl \C_\pl)\right)~.
\end{align*}
This induces a filtration of the graded $\mathbb S$-module $\operad P \comp \C \cong (\operad P_\pl \comp_\pl \C_\pl) \otimes \mathbb S$:
$$
F_n^\qp \operad P \comp \C = F_n^\qp(\operad P_\pl \comp_\pl \C_\pl) \otimes \mathbb S.
$$

Equipped with the contracting homotopy and the qp-filtration, let us prove Lemma \ref{lemmaquillenequivalence}.

\begin{notation}
Remember that we abbreviate the dg operad $\Omega\C$ by $\operad P$ and that $\operad P_\pl = \mathbb{T}_\pl s^{-1}\overline{\operad C_\pl}$. Let us denote:

\medskip

\begin{enumerate}    

    \item $QA = \operad P_\pl \comp_\pl \operad C_\pl \comp_\pl A$, the underlying graded $\kk$-module of $\Omega_{\operad C} \mathrm{B}_{\operad C} A$.
    
\medskip

    \item $RA = \overline{\operad P_\pl} \comp_\pl \operad C_\pl \comp_\pl A \oplus \overline{\operad C_\pl}  \comp_\pl A \subset QA$, where we have a a canonical isomorphism of graded $\kk$-modules $QA \cong A \oplus RA$.
    
\medskip

    \item $p_A: QA \twoheadrightarrow A$ the projection onto $A$ with respect to $RA$. It results from the projection $\operad P \comp \C \to \operad I \comp \operad I \simeq \operad I$. 
\end{enumerate}

\medskip
\end{notation}

\textbf{The section.} The counit map $\epsilon_A: QA \longrightarrow A$ has a canonical section in the category of dg modules $\zeta_A: A \longrightarrow QA$ whose underlying graded map is the inclusion
$$
A \cong \kk \comp_\pl \kk \comp_\pl A \rightarrowtail \operad P_\pl \comp_\pl \operad C_\pl \comp_\pl A.
$$
There is a canonical isomorphism of dg modules
$$
\Omega_{\operad C}\mathrm{B}_{\operad C}A \cong A \oplus K
$$
where $K$ is the kernel of the counit. Our goal is to prove that $\pi_A = \zeta_A \epsilon_A$ is a quasi-isomorphism.

\medskip

\textbf{The contracting homotopy on algebras.}

Let $H$ be the degree $1$ endomorphism of the underlying graded $\kk$-module $\operad P_\pl \comp_\pl \comp_\pl A$ of $Q A$ defined as
$$
H = h \comp_\pl A.
$$

\medskip

\textbf{The differential of $QA$.}
The differential on the graded $\kk$-module $Q A = \operad P_\pl \comp_\pl \operad C_\pl \comp_\pl A$ may be
decomposed into several terms.

\medskip

\begin{enumerate}
    \item The pre-differential $d_\C$ that comes from the pre-differential of $\operad C$ which is non-zero on elements in $s^{-1}\overline{\operad C}$ and in $\operad C$.
    
\medskip

    \item The differential $d_A$ on $A$.
    
\medskip

    \item The pre-differential $d_{CP}$ induced the operadic curved twisting morphism $\iota: \operad C \longrightarrow \operad P$, which sends elements in $\operad C$ to elements in $\operad P$.
    
\medskip

    \item The pre-differential $d_\Delta$ induced by the decomposition morphism $\Delta$ of the cooperad $\operad C$ which is non-zero on elements in $\operad P$.

\medskip

    \item The pre-differential $d_\theta$ induced by the curvature of $\operad C$ which is non-zero on elements in $\operad P$.
    
\medskip
    \item The pre-differential $d_{CA}$ induced by the dg $\Omega \C$-algebra structure of $A$, which sends elements in $\operad C \comp A$ to elements in $A$.
\end{enumerate}

\medskip

\textbf{The qp filtration on algebras.} First we consider the filtration on $Q A$ induced by the qp filtration on $\operad P_\pl \comp_\pl \operad C_\pl$:
$$
F_n^\qp (QA) = (F_n^\qp (\operad P_\pl \comp_\pl \operad C_\pl)) \comp_\pl A, \quad n \in \mathbb N.
$$
In other words, we have
\begin{align*}
    F_0^{\qp} (QA) &\coloneqq A~,
    \\
    F_n^{\qp}  (Q A)
    &\coloneqq (F_n^\qp \operad C_\pl \comp_\pl A) \oplus 
    \bigoplus_{k \geq 0} \sum_{j_0 + j_1 + \cdots + j_k=n, j_0>0} (s^{-1}F_{j_0}^{\qp} \overline{\operad C_\pl}(k)) \otimes \left(F_{j_1}^{\qp}  Q A \otimes \cdots \otimes 
    F_{j_1}^{\qp}  Q A\right)~,
\end{align*}
where the sum symbol $\sum$ stands for the union within $\operad P_\pl \comp_\pl \operad C_\pl \comp _\pl A$. This filtration is exhaustive, that is,
$$
QA \cong \colim{n \in \omega} ~ F_n^{\qp} QA~. 
$$
Moreover, the qp filtration of $QA$ is stable by the pre-differentials $d_A, d_{CA}, d_C, d_\theta, d_{CP}, d_\Delta$ and by $H$. It is also stable through $\pi_A = \zeta_A \epsilon_A$. Furthermore, one has a canonical isomorphism of graded $\kk$-modules 
$$
\gr_n^\qp (QP) = \gr_n^\qp (\operad P_\pl \comp_\pl \C_\pl) \comp_\pl A
$$

One can notice that $d_\theta$ and $d_{CA}$ vanish on $\gr^{\qp}(QA)$. 

\begin{proposition}\label{keypropquilleneqcoalg}
For every $n \geq 0$, the map $\partial(H) + \pi_A$ on $\gr^{\qp}_n(QA)$ is an isomorphism.
\end{proposition}

\begin{proof}
On $\gr_0^{\qp}(QA)$, $\pi_A$ is the identity while $\partial(H)$ is zero. So $\partial(H) + \pi_A$ is the identity.

\medskip

Let $n>0$. On $\gr_n^{\qp}(QA)$, $\pi_A$ is zero. It remains to show that $\partial (H)$ is an isomorphism. One can notice that on $\gr_n^{\qp}(QA)$, $d_{\C}$ is reduced to its planar part in the sense that it is given by applying to every labels in $\C_\pl$ the planar part of the coderivation of $\C$. As a consequence it commutes with $H$. So, again on  $\gr_n^{\qp}(QA)$, one has
$$
\partial H = d_{CP} H + H d_{CP} + d_\Delta H + H d_\Delta.
$$
The labelled trees that form $\gr_n^{\qp}(\operad P_\pl \comp_\pl \C_\pl)$ have at most $n$ nodes. One can filter it by the opposite of the number of nodes. The induced filtration on 
$\gr_n^{\qp}(QA)$ is preserved by $d_{CP}, d_\Delta$ and $H$. Moreover, on the associated graded object $\gr \gr_n^{\qp}(QA)$, $d_\Delta$ is zero and $\partial H$ is
$$
\partial H = d_{CP} H + H d_{CP} =  \id.
$$
Subsequently, $\partial H$ is an isomorphism of $\gr_n^{\qp}(QA)$.

\medskip

Thus for every natural integer $n$, $\partial(H) + \pi_A$ induces an isomorphism on $\gr_n^{\qp}(QA)$. So it is an automorphism of $QA$, which concludes the proof.
\end{proof}


\subsection{Fibrant objects}
Fibrant objects in the transferred model structure on curved $\C$-coalgebras admit a comparatively much simpler description than cofibrant objects in the model structure of dg $\Omega \C$-algebras. They are given by quasi-cofree $\C$-coalgebras, that is, curved $\C$-coalgebras whose underlying graded $\C$-coalgebra is cofree.

\begin{lemma}\label{lemmaBsiftedcolimits}
    The functor $\mathrm{B}_{\operad C}$ commutes with sifted colimits.
\end{lemma}

\begin{proof}
    This follows from the fact that sifted colimits in dg $\Omega \operad C$-algebras and in curved $\operad C$-coalgebras are computed in graded $\kk$-modules and that the endofunctor of graded $\kk$-modules $\operad C \comp (-)$ preserves sifted colimits.
\end{proof}

\begin{remark}
    The functor $\mathrm{B}_{\operad C}$ is also conservative and thus it is monadic. Moreover, $\Omega_{\operad C}$ preserves finite cosifted limits and is conservative, thus it is comonadic. Hence, the adjunction $\Omega_{\operad C}\dashv \mathrm{B}_{\operad C}$ is bimonadic; see for \cite{augmentedmonadsandbimonadicdjunctions} for more details.
\end{remark}

\begin{proposition}
Let $W$ be a curved $\operad C$-coalgebra. The following assertions are equivalent:

\medskip

\begin{enumerate}
    \item $W$ is fibrant;
    
\medskip

    \item $W$ is in the essential image of the functor $\mathrm{B}_{\operad C}$;
    
\medskip

    \item $W$ is a quasi-cofree $\operad C$-coalgebra, that is, its underlying graded $\C$-coalgebra is cofree.
\end{enumerate}    
\end{proposition}

\begin{proof}
Let $W$ be a quasi-cofree $\operad C$-coalgebra, that is, a curved $\operad C$-coalgebra whose underlying graded $\operad C$-coalgebra is of the form $W \cong \operad C \comp A$. A straightfoward check shows that the degree $0$ map
    
    \[
    s^{-1}\overline{\operad C} \comp A \longrightarrow \overline{\operad C} \comp A
    \hookrightarrow {\operad C} \comp A
    \xrightarrow{d_V} {\operad C} \comp A \twoheadrightarrow A
    \]
    \vspace{0.1pc}

    defines the structure of a dg $\Omega \operad C$-algebra on $A$. Therefore, $W \cong \mathrm{B}_{\operad C} A$. Conversely, every image of a dg $\Omega \operad C$-algebra $A$ through $\mathrm{B}_{\operad C}$ is quasi-cofree by definition.
    
    \medskip

    Now, let $W$ be a fibrant object in curved $\operad C$-coalgebras. Let us prove that it is quasi-cofree. The unit of adjunction $\eta_C: W \longrightarrow \mathrm{B}_{\operad C}\Omega_{\operad C} W$ admits a retraction $r$, since it is an acyclic cofibration and since $W$ is fibrant. Let $A$ be the colimit of the reflexive pair of maps
    $$
    \begin{tikzcd}
    \Omega_{\operad C} \mathrm{B}_{\operad C} \Omega_{\operad C} W
    \ar[rr, bend left, "\Omega_{\operad C}(r)"]
    \ar[rr, bend right, "\epsilon_{\Omega_{\operad C} W}"]
    &&\Omega_{\operad C} W \ar[r] \ar[ll]
    &A    
    \end{tikzcd}
    $$
    in the category of dg $\Omega \operad C$-algebras, where $\epsilon_{\Omega_{\operad C} W}: \Omega_{\operad C} \mathrm{B}_{\operad C} \Omega_{\operad C} W \longrightarrow \Omega_{\operad C} W$ is the counit of adjunction. By Lemma \ref{lemmaBsiftedcolimits}, the colimit of the diagram 
    
    $$
    \begin{tikzcd}
    \mathrm{B}_{\operad C}\Omega_{\operad C} \mathrm{B}_{\operad C} \Omega_{\operad C} W
    \ar[rr, bend left, "\mathrm{B}_{\operad C}\Omega_{\operad C}(r)"]
    \ar[rr, bend right, "\mathrm{B}_{\operad C}(\epsilon_{\Omega_{\operad C} W})"]
    &&\mathrm{B}_{\operad C}\Omega_{\operad C} W
    \end{tikzcd}
    $$
    
    of curved $\C$-coalgebras is $\mathrm{B}_{\operad C} A$, since $\mathrm{B}_{\operad C}$ preserves sifted colimits. It is also clear that this colimit is $W$. So there exists a canonical isomorphism $W \cong \mathrm{B}_{\operad C} A$.
\end{proof}


\subsection{Infinity morphisms}
The notion of $\infty$-morphism extends the usual notion of morphisms of dg $\Omega \C$-algebras. Their main advantage is that $\infty$-quasi-isomorphisms are invertible (up to homotopy), and therefore one can replace a zig-zag of quasi-isomorphisms of dg $\Omega \C$-algebras with two inverse $\infty$-quasi-isomorphisms. This provides a powerful tool to describe the homotopy category of dg $\Omega \C$-algebras.

\medskip

Recall that for every dg $\Omega \operad C$-algebra, the counit map $\epsilon_A: \Omega_{\operad C} \mathrm{B}_{\operad C} A \longrightarrow A$ has a canonical section $\zeta_A: A \longrightarrow \Omega_{\operad C} \mathrm{B}_{\operad C} A$ in the category of dg modules. Thus, if $K$ is the kernel of this counit map, one has a canonical decomposition of the dg module $\Omega_{\operad C} \mathrm{B}_{\operad C} A$ as
$$
\Omega_{\operad C} \mathrm{B}_{\operad C} A \cong K \oplus A.
$$

\begin{definition}[$\infty$-morphisms]
    Let $A, A'$ be two dg $\Omega \operad C$-algebras. An $\infty$-\textit{morphism} from $f:A \rightsquigarrow A'$ amounts to the data of, equivalently:
    
    \medskip
    
    \begin{enumerate}
    \item a morphism of dg $\Omega \operad C$-algebras $f : \Omega_{\operad C} \mathrm{B}_{\operad C} A \longrightarrow A'$,
    
    \medskip
    
    \item a morphism of curved $\operad C$-coalgebras $f^\dagger: \mathrm{B}_{\operad C} A \longrightarrow \mathrm{B}_{\operad C} A'$.
    \end{enumerate}
\end{definition}
    
\textbf{Linear part.} Let $f:A \rightsquigarrow A'$ be an $\infty$-morphism of dg $\Omega \operad C$-algebras. Its \textit{linear part} $f_{\mathrm{dg}}$ is the morphism of dg modules given by the composition

\[
\begin{tikzcd}[column sep=3pc,row sep=3pc]
f_{\mathrm{dg}}: A \arrow[r,"\zeta_A"]
&\Omega_{\operad C} \mathrm{B}_{\operad C} A \arrow[r,"f"]
&A'~.
\end{tikzcd}
\]

Let us denote $\epsilon: \C \longrightarrow \operad I$ and $\mu: \operad I \longrightarrow \C$ the counit and the coaugmentation of the quasi-planar conilpotent curved cooperad $\C$, respectively. The linear part of $f:A \rightsquigarrow A'$ is equivalently given by 

\[
\begin{tikzcd}[column sep=3pc,row sep=3pc]
f_{\mathrm{dg}}: A \arrow[r,"\mu ~\circ ~\id"]
&\mathrm{B}_{\operad C} A \arrow[r,"f^{\dagger}"]
&\mathrm{B}_{\operad C} A' \arrow[r,"\epsilon ~\circ ~\id"]
&A'~.
\end{tikzcd}
\]
    
\textbf{Homotopy part.} Let $f:A \rightsquigarrow A'$ be an $\infty$-morphism of dg $\Omega \operad C$-algebras. Its \textit{homotopy part} $f_{\mathrm{h}}$ is the morphism of dg modules given by the composition

\[
\begin{tikzcd}[column sep=3pc,row sep=3pc]
f_{\mathrm{h}}: K \arrow[r,rightarrowtail]
&\Omega_{\operad C} \mathrm{B}_{\operad C} A \arrow[r,"f"]
&A'~,
\end{tikzcd}
\]
where $K$ is the kernel of the counit map $\Omega_{\operad C} \mathrm{B}_{\operad C} A \longrightarrow A$. Equivalently, it is given by the composition

\[
\begin{tikzcd}[column sep=3pc,row sep=3pc]
f_{\mathrm{h}}: M \arrow[r,rightarrowtail]
&\mathrm{B}_{\operad C} A \arrow[r,"f^{\dagger}"]
&\mathrm{B}_{\operad C} A' \arrow[r,"\epsilon ~\circ ~\id"]
&A'
\end{tikzcd}
\]
where $M$ is the kernel of the map $\epsilon \circ \id : \mathrm{B}_{\operad C} A \twoheadrightarrow A$ induced by the counit of $\C$.

\begin{definition}[$\infty$-quasi-isomorphism]
An $\infty$-morphism $f:A \rightsquigarrow A'$ is called an $\infty$-\textit{quasi-isomorphism} if its linear part $f_{\mathrm{dg}}$ is a quasi-isomorphism of dg modules.
\end{definition}

\begin{definition}[$\infty$-isotopy]
An $\infty$-morphism $f:A \rightsquigarrow A$ is called an $\infty$-\textit{isotopy} if its linear part $f_{\mathrm{dg}}$ is an identity map of $A$.
\end{definition}

\begin{lemma}\label{lemmainftyfibrationavant}
Let $f:A \rightsquigarrow A'$ be an $\infty$-morphism. If $f_{\mathrm{dg}}$ is a fibration, then $f^\dag$ is the composition of an $\infty$-isotopy followed by a strict fibration of curved $\C$-coalgebras, that is, a fibration of the form $\mathrm{B}_{\operad C}(p)$, where $p$ is a fibration of dg $\Omega \C$-algebras.
\end{lemma}

\begin{proof}
Let $s: A' \longrightarrow A$ be a section of $f_{\mathrm{dg}}$ in the category of graded $\kk$-modules that is, a map such that $f_{\mathrm{dg}} s = \id$. Let us consider the degree $0$ map $\tau$ from $\operad C \comp A$ to $A$ defined by $\tau_{\mathrm{dg}} \coloneqq \id_A$ and by 

\[
\begin{tikzcd}[column sep=3pc,row sep=3pc]
\tau_{\mathrm{h}} \coloneqq: \overline{\operad C} \comp A \arrow[r,"f_{\mathrm{h}}"]
&A' \arrow[r,"s"]
&A~.
\end{tikzcd}
\]
It induces a graded endomorphism $t$ of the graded $\operad C$-coalgebra $\operad C \comp A$ given by 
\[
\begin{tikzcd}[column sep=3pc,row sep=3pc]
t: \operad C \comp A \arrow[r,"\Delta ~ \comp ~ \id"]
&\operad C \comp \operad C \comp A \arrow[r,"\id ~ \comp ~ \tau"]
&\operad C \comp A~,
\end{tikzcd}
\]
that is an isomorphism since $\gr^{\mathrm{rad}}(t) = \id$. Let us denote $D$ the coderivation of $\mathrm{B}_{\operad C} A$ and let
$$
\tilde D \coloneqq t D t^{-1}.
$$
This is again a coderivation of the graded $\C$-coalgebras $\operad C \comp A$ such that $\tilde D^2 = (\theta \comp \id)\Delta$. Therefore it defines a dg $\Omega \C$-algebra structure on $A$, and $t: A \rightsquigarrow A$ becomes an $\infty$-isotopy from $(\C \comp A, D)$ to $(\C \comp A, \tilde D)$. Moreover, we have
    $$
    (\operad C \comp f_{\mathrm{dg}}) t = f^\dagger~.
    $$
\end{proof}

\begin{proposition}\label{propinfinitmorphalgebra}
Let $f:A \rightsquigarrow A'$ be an $\infty$-morphism of dg $\Omega \operad C$-algebras. The morphism of curved $\operad C$-coalgebras $f^\dagger: \mathrm{B}_{\operad C} A \longrightarrow \mathrm{B}_{\operad C} A'$ is

\medskip

    \begin{enumerate}
        \item a weak-equivalence if and only if the linear part $f_{\mathrm{dg}}$ is a quasi-isomorphism;
        
\medskip

        \item an isomorphism if and only if the dg part $f_{\mathrm{dg}}$ is an isomorphism;
        
\medskip

        \item a fibration if and only if the dg part $f_{\mathrm{dg}}$ is a degree-wise epimorphisms;
        
\medskip

        \item a cofibration if and only if the dg part $f_{\mathrm{dg}}$ is a degree-wise injection.
    \end{enumerate}
\end{proposition}

\begin{proof}\leavevmode

    \begin{enumerate}
        \item Let us consider the following commutative diagram
        $$
        \begin{tikzcd}[column sep=3pc,row sep=3pc]
            A \ar[d, "f_{\mathrm{dg}}",swap]
            \ar[r,"\simeq"]
            & \Omega_{\operad C} \mathrm{B}_{\operad C} A
            \ar[d, "\Omega_{\operad C}(f^{\dagger})"]
            \\
            A'
            \ar[r,"\simeq"]
            &\Omega_{\operad C} \mathrm{B}_{\operad C} A'.
        \end{tikzcd}
        $$
        in the category of dg modules. The two horizontal maps, $\zeta_A$ and $\zeta_{A'}$, are quasi-isomorphisms. Thus, the left vertical map is a quasi-isomorphism if and only if the right vertical map is a quasi-isomorphism.
        
        \medskip

        \item If $f^{\dagger}$ is an isomorphism, then $f_{\mathrm{dg}} = \gr_0^{\mathrm{rad}}(f^{\dagger})$ is an isomorphism. Conversely, if $f_{\mathrm{dg}}$ is an isomorphism, then 
        $$
        \gr_n^{\mathrm{rad}}(f^{\dagger}) = \id_{\gr_n(\operad C)} \comp f_{\mathrm{dg}}
        $$
        is an isomorphism. Hence, $f^{\dagger}$ is an isomorphism.
        
        \medskip
        
        \item Since $\C$ is coaugmented, one has a functor from dg modules to curved $\operad C$-coalgebras that sends a dg module $X$ to $X$ endowed with the trivial curved $\operad C$-coalgebra structure. This structure is given by the zero structural map $0: X \longrightarrow \overline{\operad C}\comp X$.
        
        \medskip
        
       	This functor sends acyclic cofibrations of dg modules to filtered quasi-isomorphisms that are also degree-wise inclusions. They are in particular acyclic cofibrations of curved $\operad C$-coalgebras. If $f^\dagger$ is a fibration, then it has the right lifting property with respect to all acyclic cofibration, and in particular with respect to such acyclic cofibrations of dg modules. Subsequently, $f_{\mathrm{dg}}$ has the right lifting property with respect to every acyclic cofibration of dg modules. So it is a fibration of dg modules, hence a degree-wise epimorphism. Conversely, if $f_{\mathrm{dg}}$ is a fibration, then $f^\dagger$ is a fibration as a direct consequence of Lemma \ref{lemmainftyfibrationavant}.
       	
       	\medskip
       	
        \item Let us consider the same commutative square diagram shown in point (1). If $f^\dagger$ is a cofibration, then the right vertical map $\Omega_{\operad C}(f^{\dagger})$ and the top horizontal map of this square are degree-wise injection. Thus the left vertical map $f_{\mathrm{dg}}$ is also a degree-wise injection. Conversely, let us suppose that $f_{\mathrm{dg}}$ is a degree-wise injection. Then, it has a left inverse $g$ in the category of graded $\kk$-modules. Let us consider the endomorphism $h$ of graded $\operad C$-coalgebra
        
\[
\begin{tikzcd}[column sep=3pc,row sep=3pc]
\operad C \comp A \arrow[r,"f^\dagger"]
&\operad C \comp A' \arrow[r,"\id ~ \comp ~ g"]
&\operad C \comp A~.
\end{tikzcd}
\]

Its linear part is the identity of $A$. The same arguments as those used to prove point (2) show that $h$ is a graded isomorphism. In particular, it is a degree-wise injection. So $f^\dagger$ is a degree-wise injection, hence a cofibration.
    \end{enumerate}
\end{proof}

\begin{proposition}
Let $A$ and $A'$ be two dg $\Omega \C$-algebras. There exists a zig-zag of quasi-isomorphisms of dg $\Omega \C$-algebras
\[
A \lqi \cdot \qi \cdot \lqi \cdots \qi \cdot \lqi A'~,
\]

if and only if there exist $\infty$-quasi-isomorphisms $A \rightsquigarrow A'$ and $A' \rightsquigarrow A$ which are inverse to each other in homology. 
\end{proposition}

\begin{proof}
Given any zig-zag of quasi-isomorphisms between $A$ and $A'$ we get a zig-zag of weak-equivalences between $\mathrm{B}_{\operad C} A$ and $\mathrm{B}_{\operad C} A'$. All the objects in the zig-zag are fibrant and cofibrant, therefore they admit homotopy inverses, and we have two homotopy inverse arrows between $\mathrm{B}_{\operad C} A$ and $\mathrm{B}_{\operad C} A'$. Conversely, if there is a $\infty$-quasi-isomorphism $f^{\dagger}: \mathrm{B}_{\operad C} A \qi \mathrm{B}_{\operad C} A'$, we get the following zig-zag
\[
\begin{tikzcd}[column sep=3pc,row sep=3pc]
A 
&\Omega_{\operad C} \mathrm{B}_{\operad C} A \arrow[r,"\Omega_{\operad C}(f^\dagger)"] \arrow[l,"\simeq",swap]
&\Omega_{\operad C} \mathrm{B}_{\operad C} A' \arrow[r,"\simeq"]
&A'~.
\end{tikzcd}
\]
\end{proof}


\subsection{Homotopy transfer theorem for algebras}
\label{sectionhomotopytransfertheoremalgebra}
The goal of this subsection is to prove a homotopy transfer theorem for dg algebras over a cofibrant dg operad. First, we show that given a dg operad $\operad P$ and dg $\operad P$-algebra $A$ together with a decomposition $A \cong K \oplus X$, where $K$ is acyclic, we can "perturb" the dg $\Omega \mathrm{B}\operad P$-algebra structure on $A$ so that this new structure "restricts" to $X$. Then, we give an homotopical meaning to this construction in cases of interest.

\subsubsection{Technicalities}
Let $\operad P$ be a dg operad.

\begin{notation}
Let us denote $M$ the graded $\mathbb S$-module $\operad P \oplus s \operad I$:
$$
M = \operad P \oplus s \operad I.
$$
We have a canonical isomorphism of graded conilpotent cooperads $\mathrm{B} \operad P \cong \treemod(sM)$, since $sM$ is precisely the graded $\mathbb{S}$-module of generators of the operadic bar construction of $\operad P$.
\end{notation}

Let us also consider a sum of dg modules
$$
A = X \oplus K
$$
where $K$ is acyclic. Let us choose a contracting homotopy $h$ of $K$; we can extend $h$ to $A$, by $0$ on $X$. Let us denote $\pi_X, \pi_K$ the projections endomorphisms of $A$ onto respectively $X$ and $K$. All these arrows satisfy the following equations:

\[
\left\{
\begin{tikzcd}[column sep=0pc,row sep=0.25pc]
&\pi_X \pi_K = \pi_K \pi_X =0;\\
&\partial(h) = \pi_K = \id_A - \pi_X;\\
&h \pi_K = \pi_K h = h;\\
&h \pi_X = \pi_X h = 0.
\end{tikzcd}
\right.
\]

\begin{example}
Let $X$ be a dg module. Since we are working over a field $\kk$, there exists a splitting $X \cong \mathrm{H}(X) \oplus K$, where $\mathrm{H}(X)$ is the homology of $X$ and $K$ is an acyclic dg module.
\end{example}

Let us suppose that the dg module $(A,d_A)$ is endowed a dg $\operad P$-algebra structure
$$
\gamma_A: \operad P \comp A \longrightarrow A~.
$$
Notice that it induces a dg $\Omega \mathrm{B}\operad P$-algebra structure on $A$ by pulling back along the map $\Omega B \operad P \longrightarrow \operad P$.

\begin{definition}
Let $D_\beta$ be the degree $-1$ coderivation on the graded $\mathrm{B}\operad P$-coalgebra $\mathrm{B}\operad P \comp A$ that proceeds from the structure of a $\Omega \mathrm{B} \operad P$-algebra on the dg module $A$, in the sense that
$$
(\mathrm{B}\operad P \comp A, D_\beta) = \mathrm{B}_{\mathrm{B}\operad P} A.
$$
Moreover, let $\beta$ be the composition of $D_\beta$ with the projection onto $A$. One can notice that the restriction of $\beta$ to $A \cong \operad I \comp A$ is the differential $d_A$ and that its restriction to $\overline{\mathrm{B}\operad P} \comp A$ is
    $$
    \overline{\mathrm{B}\operad P} \comp A \twoheadrightarrow sM \comp A \twoheadrightarrow s\operad P \comp A \xrightarrow{s \gamma} s A \longrightarrow A.
    $$
    where the last map just sends $sa$ to $a$.
\end{definition}

\medskip

The degree $0$ map
    $$
    sM\comp A 
     \xrightarrow{\beta} A  \xrightarrow{-h} A
    $$
    yields a morphism of graded $\mathbb S$-modules $sM \longrightarrow \mathrm{End}(A)$. Therefore there is a morphism of graded operads $
    \treemod(sM) \longrightarrow \mathrm{End}(A)$, 
    which in turn yields a degree $0$ map of graded $\kk$-modules
    $$
    \phi : \mathrm{B}\operad P \comp A = \treemod(sM) \comp A \longrightarrow A.
    $$
    Let $f$ be the related morphism of graded $\mathrm{B}\operad P$-coalgebras
    $$
    f : \mathrm{B}\operad P \comp A \longrightarrow \mathrm{B}\operad P \comp \mathrm{B}\operad P \comp A \xrightarrow{\id \comp \phi} \mathrm{B}\operad P \comp A.
    $$

\begin{definition}
    Let $\chi$ be the degree $-1$ morphism from $\treemod(sM) \comp A$ to $A$ whose restriction to $A$ is zero and whose restriction to $\overline{\treemod} \comp A$ is
    $$
    \overline{\treemod}(sM) \comp A \cong sM \comp \treemod(sM) \comp A
    \xrightarrow{\id \comp \phi} sM \comp A \xrightarrow{\beta} A.
    $$
    Notice that $\phi$ is given by the sum of $-h ~ \chi$ and of the projection onto $A$.
\end{definition}

Recall that the conilpotent curved cooperad $\mathrm{B}\operad P$ is given by the conilpotent graded cooperad $\treemod(sM)$ endowed with the following pre-differentials:

\begin{enumerate}
\item the pre-differential $d_\gamma$ which is induced by the operad structure of $\operad P$;
\item the pre-differential $d_{\operad P}$ which is induced by the differential of $\operad P$;
\item the pre-differential $d_{u}$, which maps $s^2\operad I$ to the unit of $\operad P$.
\end{enumerate}

We will denote the later two pre-differentials by $sd_{M}$, as they are defined on the generators of $\mathrm{B}\operad P$. We refer to Subsection \ref{subsection: operadic bar-cobar} for more details.

\begin{lemma}\label{lemmahttalgebratwoleveltrees}
    The following diagram in graded $\kk$-modules
    $$
    \begin{tikzcd}[column sep=3pc,row sep=3pc]
        \overline{\treemod}_{\leq 2} (sM) \comp A
        \ar[r, "\id"] \ar[d, "d_{\gamma} \comp A"']
        & sM \comp \treemod_{\leq 1} (sM) \comp A
        \ar[r, "{\id{} \comp \shuffle (\phi, \chi)}"]
        & sM \comp A
        \ar[d, "\phi"]
        \\
        \overline{\treemod}_{\leq 2}(sM) \comp A
        \ar[rr, "-\phi"']
        && A
    \end{tikzcd}
    $$
    commutes.
\end{lemma}

\begin{proof}
This follows from a straightforward checking.
\end{proof}

\begin{lemma}\label{lemmahttbfdecomposition}
The map $\beta ~ f: \treemod(sM) \comp A \longrightarrow A$
is equal to $d_A ~\phi + \chi$.
\end{lemma}

\begin{proof}
    This follows from a straightforward computation.
\end{proof}

\begin{definition}
    Let $\mu$ be the degree $-1$ map from $\mathrm{B}\operad P \comp A$ to $A$
    whose restriction to $A \cong \operad I \comp A$ is $d_A$ and whose restriction to 
    $\overline{\mathrm{B}\operad P} \comp A$ is the sum of the two maps
    \begin{align*}
    &\overline{\mathrm{B}\operad P} \comp A
        \xrightarrow{\chi} A \xrightarrow{\pi_X} A
        \\
    &\overline{\mathrm{B}\operad P} \comp A
    \xrightarrow{\theta \comp A} A
        \xrightarrow{-h} A.
    \end{align*}
    Moreover, let $D_\mu$ be the unique degree $-1$ coderivation on 
    the graded $\mathrm{B}\operad P$-coalgebra $\mathrm{B}\operad P \comp A$ whose projection onto $A$ is $\mu$.
\end{definition}

\begin{lemma}\label{lemmahttsumdecomposition}
    The degree $-1$ composite map $\beta f - \mu: \treemod(sM) \comp A \longrightarrow A$
    is equal to the sum of the two maps
    \begin{align*}
    &\overline{\treemod}(sM) \comp A \cong sM \comp \treemod(sM) \comp A \xrightarrow{\id \comp \phi} sM \comp A \xrightarrow{sd_M \comp A} sM \comp A \xrightarrow{\phi} A
    \\
    &\overline{\treemod}(sM) \comp A \cong sM \comp \treemod(sM) \comp A \xrightarrow{\id \comp \phi} sM \comp A \xrightarrow{\id \comp \shuffle(\id, d_A)} sM \comp A \xrightarrow{\phi} A.
\end{align*}
\end{lemma}

\begin{proof}
    Let us denote $g = d_A ~ \phi + \chi - \mu$ and $g'$ the sum of the two maps described in the lemma. We want to show that $g=g'$. Since $g$ and $g'$ are respectively equal to the compositions
    \begin{align*}
    &sM \comp \treemod(sM) \comp A \xrightarrow{\id \comp \phi} sM \comp A \xrightarrow{g} A
    \\
    &sM \comp \treemod(sM) \comp A \xrightarrow{\id \comp \phi} sM \comp A \xrightarrow{g'} A
    \end{align*}
    it suffices to prove the result on $sM \comp A$. We can notice that, on $sM \comp A$ we have
    $$
    g = d_A ~ \phi + \chi - \mu = - d_A h \beta + \beta - \pi_X ~ \beta  + h(\theta \comp \id)
    =  hd_A \beta + h(\theta \comp \id)~.
    $$
    Still on $sM \comp A$, this gives
    $$
    d_A \beta = \beta D_\beta - \beta (d_{sM} \comp A) - \beta (sM \comp \shuffle (\id, d_A)) 
    = - \theta \comp \id_A - \beta (d_{sM} \comp A) - \beta (sM \comp \shuffle (\id, d_A)) .
    $$
    Therefore
    $$
    g =  -h \beta (d_{sM} \comp A) - h \beta (sM \comp \shuffle (\id, d_A)) = g'~.
    $$
\end{proof}

\begin{proposition}
    The following equality between degree $-1$ maps from $B\operad P \comp A$ to $A$
    $$
    \phi D_{\mu} = \beta f
    $$
    holds. Thus $f D_{\mu} = D_{\beta} f$.
\end{proposition}

\begin{proof}
Let us prove the result on the height of the trees that make $\mathrm{B}\operad P = \treemod(sM)$.
First, on $A \cong \operad I \comp A$, one has
$$
\phi D_\mu =  d_A  = \beta f.
$$
Let us assume that the result the restrictions of $\phi D_\mu$ and $\beta f$ are equal on $\treemod_{\leq n} (sM) \comp A$ for some natural integer $n$.
On larger trees $\overline{\treemod}_{\leq n+1} (sM) \comp A \simeq sM \comp \treemod_{\leq n} (sM) \comp A$,
$\phi D_\mu$ is the sum of the maps
\begin{align}
&sM \comp \treemod_{\leq n} (sM) \comp A
\xrightarrow{\mu}
 A~,
\\ 
& sM \comp \treemod_{\leq n} (sM) \comp X
\xrightarrow{sd_M \comp \id}
sM \comp \treemod_{\leq n} (sM) \comp X
\xrightarrow{\psi} A~,
 \\
 &sM \comp \treemod_{\leq n} (sM) \comp X
\xrightarrow{\id \comp \shuffle(\psi, \psi D_\mu)}
sM \comp A 
\xrightarrow{\phi} A~,
\end{align}
together with the contribution (4) to the coderivation of $\mathrm{B}\operad P$ given by the composition of $\operad P$ at the root level
$$
\begin{tikzcd}
    \overline{\treemod}_{\leq n+1} (sM) \comp A
    \ar[d, "\id"]
    \\
    sM \comp (A \oplus sM \comp \treemod_{\leq n-1} (sM) \comp A)
    \ar[d, "\id \comp (\id \oplus sM \comp \phi)"]
    \\
    sM \comp (A \oplus sM \comp A)
    \ar[d, "\id"]
    \\
    \overline{\treemod}_{\leq 2} (sM) \comp A
    \ar[d, "d_{\gamma} \comp \id_A"]
    \\
    \overline{\treemod}_{\leq 2} (sM) \comp A
    \ar[d, "\phi"]
    \\
    A~.
\end{tikzcd}
$$
As a consequence of Lemma \ref{lemmahttalgebratwoleveltrees}, this last map (4) is equal to
$$
sM \comp \treemod_{\leq n} (sM) \comp X
\xrightarrow{\id \comp \shuffle(\psi, -\chi)} sM \comp A \xrightarrow{\phi} A~.
$$
By the induction hypothesis ($\beta f = \phi D_\mu$ on $\treemod_{\leq n} (sM) \comp A$) and by Lemma \ref{lemmahttbfdecomposition},
the sum of the contributions (3) and (4) is
$$
sM \comp \treemod_{\leq n} (sM) \comp A
\xrightarrow{\id \comp \shuffle(\phi, d_A ~ \phi)} sM \comp A \xrightarrow{\phi} A~,
$$
which rewrites as
$$
sM \comp \treemod_{\leq n} (sM) \comp A \xrightarrow{\id \comp \phi}
sM \comp A
\xrightarrow{\id \comp \shuffle(\id, d_A)} sM \comp A \xrightarrow{\phi} A~.
$$
We conclude by Lemma \ref{lemmahttsumdecomposition}.
\end{proof}

\begin{proposition}
The coderivation $D_\mu$ endows $\mathrm{B}\operad P \comp A$  with a curved $B\operad P$-coalgebra structure.
\end{proposition}

\begin{proof}
    Since $f D_{\mu} = D_{\beta} f$ and since $f$ is an isomorphism (by a standard filtration argument):
    $$
    D_\mu = f^{-1}D_{\beta} f.
    $$
    Thus, the coderivation $D_\mu$ makes $\mathrm{B}\operad P \comp A$ a curved $\mathrm{B}\operad P$-coalgebra because so does the coderivation $D_\beta$.
\end{proof}

\begin{proposition}
The sub-graded $B\operad P$-coalgebra $\mathrm{B}\operad P \comp X$ of $\mathrm{B}\operad P \comp A$ is stable through $D_\mu$.
\end{proposition}

\begin{proof}
The sub-graded $\mathrm{B}\operad P$-coalgebra $\mathrm{B}\operad P \comp X$ is actually the quotient/kernel of the idempotent endomorphism $(\id \comp \pi_X)$ of $\mathrm{B}\operad P \comp A$. One can notice that the projection onto $A$ of $(\id \comp \pi_X)~D_\mu ~ (\id \comp \pi_X)$ and $ D_\mu ~ (\id \comp \pi_X)$ are equal:
$$
(\id \comp \pi_X) ~ \mu ~\pi_X = (\id \comp \pi_X) ~ \mu.
$$
Thus 
$$
(\id \comp \pi_X)~D_\mu ~ (\id \comp \pi_X)
= D_\mu ~ (\id \comp \pi_X)
$$
which proves the result.
\end{proof}

To conclude, we have a composition of morphisms of curved $\mathrm{B}\operad P$-coalgebras
$$
(\mathrm{B} \operad P \comp X, D_\mu) \hookrightarrow (\mathrm{B} \operad P \comp A, D_\mu)
\xrightarrow{f} (\mathrm{B} \operad P \comp A, D_\beta).
$$

\subsubsection{The cooperad version of the homotopy transfer theorem for algebras}

\begin{theorem}\label{theoremhttcooperadalgebras}
Let $i : X \longrightarrow A$ be an acyclic cofibration of dg modules and let $\gamma_A : \Omega \operad C \comp A \longrightarrow A$ be a dg $\Omega \operad C$-algebra structure on $A$. There exists another dg $\Omega \operad C$-algebra structure 
\[
\mu_A: \Omega \operad C \comp A \longrightarrow A
\]
which restricts on $X$, together with an $\infty$-isotopy
$$
(A, \mu_A) \rightsquigarrow (A, \gamma_A)
$$

of dg $\Omega \operad C$-algebras.
\end{theorem}

\begin{proof}
Since $i$ is an acylic cofibration of dg modules, it has a left inverse $p$ and one can decompose $A$ as $X \oplus K$ where $K$ is the kernel of $p$. The paragraph just above gives us a diagram of curved $\mathrm{B}\Omega \operad C$-coalgebras
$$
(\mathrm{B}\Omega \operad C \comp X, D_\mu) \hookrightarrow (\mathrm{B}\Omega \operad C \comp A, D_\mu)
\xrightarrow{f} (\mathrm{B}\Omega \operad C, D_\beta)~.
$$
Applying the right adjoint functor from curved $B\Omega\operad C$-coalgebras to curved $\operad C$-coalgebra that results from the unit map $\operad C \longrightarrow \mathrm{B}\Omega\operad C$, we get diagram of curved $\operad C$-coalgebras
$$
(\operad C \comp X, \tilde{D}_\mu) \hookrightarrow (\operad C \comp A, \tilde{D}_\mu)
\xrightarrow{\tilde{f}} (\operad C \comp A, \tilde{D}_\beta) = \mathrm{B}_{\operad C} (A, \gamma_A)~.
$$
In that context, $D_\mu$ is the coderivation on $\operad C \comp A$ that induces the expected dg $\Omega \operad C$-algebra $\mu_A$ structure on $A$ and $\tilde f$ is the expected $\infty$-isotopy. 
\end{proof}

\subsubsection{The homotopy transfer theorem for algebras}

\begin{theorem}
Let $\operad Q$ be a cofibrant dg operad, let $i : X \longrightarrow A$ be an acyclic cofibration of dg modules and let $\gamma_A : \operad Q \comp A \longrightarrow A$ be a dg $\operad Q$-algebra structure on $A$. There exists a dg $\operad Q$-algebra structure $\mu_X$ on $X$, together with a zig-zag of quasi-isomorphisms
\[
(A,\gamma_A) \lqi \cdots \qi (X,\mu_X)
\]
of dg $\operad Q$-algebras. Furthermore, the maps in this zig-zag are homotopic to $i$ in the model category of dg modules.
\end{theorem}

\begin{proof}
    Taking $\operad C$ to be the quasi-planar conilpotent curved cooperad $\mathrm{B}(\operad Q \otimes \operad E)$, Theorem \ref{theoremhttcooperadalgebras} yields the dg $\Omega \operad C$-algebra structure $\mu_X$ on $X$ together with a zig-zag of quasi-isomorphisms of dg $\Omega \operad C$-algebras
    $$
    (X, \mu_X) \xrightarrow{i} (A, \mu_A) \leftarrow \Omega_{\operad C} \mathrm{B}_{\operad C}(A, \mu_A) \xrightarrow{\tilde f} (A, \gamma_A)~.
    $$
Moreover, the acyclic fibration of dg operads $\Omega \operad C \longrightarrow \operad Q$ has a section since $\operad Q$ is cofibrant. Applying the induced right adjoint functor from dg $\Omega \operad C$-algebras to dg $\operad Q$-algebras yields the expected zig-zag in the category of dg $\operad Q$-algebras.
\end{proof}

\begin{remark}
This last result also follows from model-categorical arguments, as developed in \cite{FressePROP}. 
\end{remark}


\subsection{Further localisations and divided powers operations}
Let $\operad Q$ be an admissible dg operad and let $\C$ be a quasi-planar conilpotent curved cooperad. Let us consider a morphism of dg operads $f: \Omega \operad C \longrightarrow \operad Q$. We have two Quillen adjunctions

\[
\begin{tikzcd}[column sep=5pc,row sep=3pc]
          \catcurvcog{\operad C} \arrow[r, shift left=1.1ex, "\Omega_{\C}"{name=F}]           
         &\mathsf{dg}~ \Omega \operad C \text{-}\mathsf{alg} \arrow[l, shift left=.75ex, "\mathrm{B}_{\C}"{name=U}] \arrow[r, shift left=1.1ex, "f_!"{name=A}]
         &\mathsf{dg}~\operad Q\text{-}\mathsf{alg}~. \arrow[l, shift left=.75ex, "f^*"{name=B}] \arrow[phantom, from=F, to=U, , "\dashv" rotate=-90]\arrow[phantom, from=A, to=B, , "\dashv" rotate=-90]
\end{tikzcd}
\]

Let us denote $\Omega_f$ the composite left adjoint and $\mathrm{B}_f$ the composite right adjoint. 

\begin{proposition}\label{propleftbousfield}
There exists a combinatorial model structure on curved $\operad C$-coalgebras called the $f$-model structure, transferred from that of dg $\operad Q$-algebras, determined by the following sets of morphisms

		\medskip
		
    \begin{enumerate}
        \item the set of $f$-cofibrations is given by morphisms $g$ such that $\Omega_f(g)$ is a cofibration,
        
        \medskip
        
        \item the set of $f$-weak-equivalences is given by morphisms $g$ such that $\Omega_f(g)$ is a weak equivalence.
        
        \medskip
        
        \item the set of $f$-fibrations is determined by right-lifting property against all acyclic cofibrations.
    \end{enumerate}
    
    		\medskip
    Moreover, this is a left Bousfield localisation of the canonical model structure transferred from dg $\Omega \C$-algebras. Meaning that the identity functor of curved $\C$-coalgebra, where at the source they are endowed with the canonical model structure, and at the target with the $f$-model structure, is a left Quillen functor. 
\end{proposition}

\begin{proof}
The $f$-cofibrations and the $f$-weak equivalences respectively contain cofibrations and weak-equivalences of the canonical model structure, transfer along the bar-cobar adjunction. Hence, every object is cofibrant and a natural cylinder object is provided by Proposition \ref{propcylinder}. This proves the existence of the transferred $f$-model structure (see Appendix \ref{appendixtransfer}). To prove that this is a left Bousfield localisation of that transferred from dg $\Omega \operad C$-algebras, it suffices to show that $f$-cofibrations are in particular degree-wise injections which results from the same arguments as those used to prove Propostion \ref{propcofibrations}.
\end{proof}

\textbf{Localizing at quasi-isomorphisms.} Let $\C$ be a quasi-planar conilpotent \textit{differential graded} cooperad. The cobar construction $\Omega \C$ is augmented since $\C$ has zero curvature. Let us denote $\nu: \Omega \C \longrightarrow \operad I$ the canonical morphism of dg operads given by the augmentation. We have the following adjunctions 

\[
\begin{tikzcd}[column sep=5pc,row sep=3pc]
          \catdgcog{\operad C} \arrow[r, shift left=1.1ex, "\Omega_{\C}"{name=F}]           
         &\mathsf{dg}~ \Omega \operad C \text{-}\mathsf{alg} \arrow[l, shift left=.75ex, "\mathrm{B}_{\C}"{name=U}] \arrow[r, shift left=1.1ex, "\mathrm{Indec}"{name=A}]
         &\mathsf{dg}~\text{-}\mathsf{mod}~, \arrow[l, shift left=.75ex, "\mathrm{Triv}"{name=B}] \arrow[phantom, from=F, to=U, , "\dashv" rotate=-90]\arrow[phantom, from=A, to=B, , "\dashv" rotate=-90]
\end{tikzcd}
\]

where the adjunction $\nu_! \dashv \nu^*$ is in fact given by the indecomposables functor $\mathrm{Indec}$ (which is $\nu_!$) and by the trivial structure functor $\mathrm{Triv}$ (which is $\nu^*$). Notice that since $\C$ has zero curvature, curved $\C$-coalgebras in pdg modules are precisely given by dg $\C$-coalgebras.

\begin{proposition}\label{prop: triv-indec weak equivalences are quasi-isos}
The set of $\nu$-weak-equivalences is precisely the set of quasi-isomorphisms of dg $\C$-coalgebras.
\end{proposition}

\begin{proof}
The composition $\mathrm{Indec}~\Omega_{\C}$ is isomorphic to the forgetful functor from dg $\C$-coalgebras to dg modules.
\end{proof}

\begin{corollary}
Let $\C$ be a quasi-planar conilpotent dg cooperad. The set of weak-equivalences in the canonical model structure on dg $\C$-coalgebras is contained in the set of quasi-isomorphims.
\end{corollary}

\begin{proof}
It suffices to apply Proposition \ref{propleftbousfield} to the morphism of dg operads $\nu: \Omega \operad C \longrightarrow \operad I$, combining it with Proposition \ref{prop: triv-indec weak equivalences are quasi-isos}.
\end{proof}

\textbf{Divided power operations in the homotopical setting.}
Let $\C$ be a quasi-planar conilpotent dg cooperad. By Proposition \ref{propleftbousfield}, the category of dg $\C$-coalgebras admits a model category structure where 

\medskip
\begin{enumerate}
\item the set of cofibrations is given by degree-wise injections;

\medskip

\item the set of weak-equivalences is given by quasi-isomorphisms;

\medskip

\item the set of fibrations is given by maps with right lifting property with respect to acyclic cofibrations.
\end{enumerate}

Let $(W,\Delta_W,d_W)$ be a dg $\C$-coalgebra. The structural map 

 \[
    \Delta_W: W \longrightarrow \bigoplus_{n \geq 0} \C(n) \otimes_{\mathbb S_n} W^{\otimes n}~,
    \]
    \vspace{0.1pc}
    
lands on the coinvariants on the right-hand side, therefore divided power operations should appear. Nevertheless, since $\C$ is quasi-planar, there is a natural isomorphism

\[
\bigoplus_{n \geq 0} \C(n) \otimes_{\mathbb S_n} W^{\otimes n} \cong \bigoplus_{n \geq 0} \left(\C(n) \otimes W^{\otimes n}\right)^{\mathbb S_n}~,
\]
\vspace{0.1pc}

of dg modules induced by the norm map (Proposition \ref{prop: iso avec la norme}). Therefore no divided power operations appear at the algebraic level.

\medskip

These divided power operations do not disappear at the $\infty$-categorical level. Indeed, $\C(n)$ is a quasi-free dg $\kk[\mathbb{S}_n]$-module which is furthermore \textit{projective} by Proposition \ref{prop: quasi-planaire implique S-projectif dans le cas dg}. Therefore we have equivalences 

\[
\bigoplus_{n \geq 0} \C(n) \otimes_{h\mathbb S_n} W^{\otimes n} \simeq \bigoplus_{n \geq 0} \C(n) \otimes_{\mathbb S_n} W^{\otimes n} \cong \bigoplus_{n \geq 0} \left(\C(n) \otimes W^{\otimes n}\right)^{\mathbb S_n} \not\simeq \bigoplus_{n \geq 0} \left(\C(n) \otimes W^{\otimes n}\right)^{h\mathbb S_n}~,
\]
\vspace{0.1pc}

where on the upmost left-hand side we consider \textit{homotopy coinvariants} and on the upmost right-hand side we consider \textit{homotopy invariants}. This means that the $\infty$-category of dg $\C$-coalgebras localized at quasi-isomorphisms behaves like an $\infty$-category of \textit{divided power conilpotent coalgebras}.


\newpage

\section{Model structure on complete algebras over a cooperad}

\vspace{2pc}

The goal of this section is to study the homotopical properties of the complete bar-cobar adjunction between dg $\Omega\C$-coalgebras and complete curved $\C$-algebras, in the case where $\C$ is a quasi-planar conilpotent curved cooperad. 

\medskip

The dg operad $\Omega \C$ is cofibrant, and therefore also coadmissible by Proposition \ref{propadmissible}. This means that the category of dg $\Omega\C$-coalgebras admits a model category structure where weak-equivalences are given by quasi-isomorphisms and cofibrations by degree-wise injections. Let us
consider the complete bar-cobar adjunction 
\[
\begin{tikzcd}[column sep=5pc,row sep=3pc]
          \catcurvcompalg{\operad C}
          \arrow[r, shift left=1.1ex, "\widehat{\Omega}_{\C}"{name=F}] & \mathsf{dg}~ \Omega \C\text{-}\mathsf{cog}~, \arrow[l, shift left=.75ex, "\widehat{\mathrm{B}}_{\C}"{name=U}]
            \arrow[phantom, from=F, to=U, , "\dashv" rotate=-90]
\end{tikzcd}
\]
\vspace{0.1pc}
relative to $\iota: \C \longrightarrow \Omega \C$, which will be denoted by $\widehat{\Omega}_{\C},\widehat{\mathrm{B}}_{\C}$ from now on. Our first goal is going to be to transfer the model structure on dg $\Omega \C$-coalgebras along this adjunction to the category of qp-complete curved $\C$-algebras.

\begin{theorem}\label{thm: existence model structure for absolute}
Let $\operad C$ be a quasi-planar curved conilpotent cooperad. The category of qp-complete curved $\operad C$-algebras has the structure of a combinatorial model category given by the following sets of maps:

\medskip

\begin{enumerate}
\item the set of weak-equivalences is given by morphisms $f$ such that $\widehat{\Omega}_{\C}$ is a quasi-isomorphism,

\medskip

\item the set of fibrations is given by morphisms $f$ such that $\widehat{\Omega}_{\C}$ is a fibration; these are degree-wise epimorphisms,

\medskip 

\item the set of cofibrations is given by morphisms with the left-lifting property with respect to acyclic fibrations.
\end{enumerate}
\end{theorem}

\begin{remark}
Using the standard transfer theorem for model category structures only gives that fibrations are morphism which are sent by $\widehat{\Omega}_{\C}$ to fibrations. The theorem contains an additional characterization of fibrations of complete curved $\C$-algebras as degree-wise surjective maps. 
\end{remark}


\subsection{Outline of the transfer of model structures}
The proof is somewhat dual to the proof of Theorem \ref{thm: existence de structure de modèles}, except we deal with not all curved $\operad C$-algebras but those that are qp-complete. However, we extend our definitions of fibrations and weak equivalences to morphisms between any pair of curved $\operad C$-algebras.

\begin{definition}[Fibrations]\label{def: fibration de curved absolute}
A morphism $f$ of curved $\C$-algebras is a \textit{fibration} if $\widehat{\Omega}_{\C}(f)$ is a fibration of dg $\Omega \C$-coalgebras.
\end{definition}

\begin{definition}[Weak-equivalences]
A morphism $f$ of curved $\C$-coalgebras is a \textit{weak-equivalence} if $\widehat{\Omega}_{\C}(f)$ is a quasi-isomorphism of dg $\Omega \C$-coalgebras.
\end{definition}

\begin{definition}[Cofibrations]
A morphism of qp-complete curved $\C$-algebras is a \textit{cofibration} if it has the left-lifting property against all acyclic fibrations between qp-complete curved $\C$-algebras. 
\end{definition}

Both categories $\catcurvcompalg{\operad C}$ and $\catdgcog{\operad P}$ are presentable, therefore by Appendix \ref{appendixtransfer} it suffices to exhibit a natural fibrant resolution and a natural path object for qp-complete curved algebras to prove the existence of the transferred model structure. We will show that fibrations of Definition \ref{def: fibration de curved absolute} are given by degree-wise surjective maps in Proposition \ref{propfibration} and we will construct a natural path object in Proposition \ref{proppath}. 

\medskip

For the rest of this section, let us fix a quasi-planar conilpotent curved cooperad $\C$ whose quasi-planar ladder is indexed by some small ordinal $\alpha$.

\subsection{Elementary fibrations}
Elementary fibrations are a particularly well-behaved set of fibrations of (qp-complete) curved $\C$-algebras, such that the kernel of any such elementary fibration is a dg $\kk$-module. 

\begin{definition}[Elementary fibrations]
A morphism $f : \L \twoheadrightarrow \L'$ of curved $\operad C$-algebras is an \textit{elementary fibration} if it is degree-wise surjective and if the map $\overline{\gamma}_\L: \L^{\overline{\C}} \longrightarrow \L$ factors through $(\L')^{\overline{\C}}$, that is, if there exists a dotted arrow
\[
\begin{tikzcd}[column sep=3pc,row sep=3pc]
\L^{\overline{\operad C}} \arrow[d,rightarrowtail] \ar[r,"(f)^{\overline{\C}}", two heads]
&(\L')^{\overline{\operad C}}  \arrow[d,dashed]
\\
\L^{\operad C} \arrow[r,"\gamma_\L",swap]  
&\L
\end{tikzcd}
\]
such that the diagram commutes, where $\gamma_\L$ denotes the structural map of $\L$.
\end{definition}

\begin{remark}
The map $\L^{\overline{\operad C}} \longrightarrow (\L')^{\overline{\operad C}}$ is a degree-wise epimorphism since it is has a retract in the category of graded $\kk$-modules.
\end{remark}

\begin{lemma}
Let $f: \L \twoheadrightarrow \L'$ be an elementary fibration of curved $\operad C$-algebras. Then $\mathrm{Ker}(f)$ is a dg module. 
\end{lemma}

\begin{proof}
Let $i_f: \mathrm{Ker}(f) \longrightarrow \L$ denote the canonical inclusion. We have
\[
i_f d^2 = d^2 i_f = \overline{\gamma}_\L~(\mathrm{id})^{\theta} i_f = \overline{\gamma}_\L~(i_f)^{\id}~(\mathrm{id})^{\theta}~,
\]
and $\overline{\gamma}_\L~(i_f)^{\id} = 0$ since $f$ is an elementary fibration. We conclude using that $i_f$ is a monomorphism.
\end{proof}

\begin{proposition}
Let $f: \L \twoheadrightarrow \L'$ be an elementary fibration of curved $\operad C$-algebras. It is in particular a fibration.
\end{proposition}

\begin{proof}
This amounts to prove that $\widehat{\mathrm{B}}_{\operad C}(f)$ is a fibration. Let $U$ be the kernel of $f$. Let us decompose the underlying graded $\kk$-module of $A$ as $\L \cong \L' \oplus U$. The pre-differential of $\L$ rewrites as the sum of the differential $d_U$ on $U$, the pre-differential $d_{\L'}$ on $\L'$, and a degree $-1$ map $\xi: \L' \longrightarrow U$. Let us consider the morphism of graded $\kk$-modules $\pi_U :\widehat{\mathrm{B}}_{\operad C}\L \longrightarrow \L \longrightarrow U$. It induces the morphism of dg modules 
    \[
    p : D^{0} \otimes \widehat{\mathrm{B}}_{\operad C}\L \longrightarrow U
   	\]
    
    whose restriction to $S^0 \otimes \widehat{\mathrm{B}}_{\operad C}\L$ is $\id \otimes \pi_U$ and whose restriction to $S^{-1} \otimes \widehat{\mathrm{B}}_{\operad C}\L$ is $-\id \otimes \partial(\pi_U)$. The fact that $f$ is an elementary fibration implies that the restriction of $p$ to $S^{-1} \otimes \widehat{\mathrm{B}}_{\operad C}\L$ factors through the projection $S^{-1} \otimes \widehat{\mathrm{B}}_{\operad C}\L \longrightarrow S^{-1} \otimes \widehat{\mathrm{B}}_{\operad C}\L'$. One thus gets a commutative square of dg modules
    $$
\begin{tikzcd}[column sep=3pc,row sep=3pc]
    \widehat{\mathrm{B}}_{\operad C}\L
    \ar[r] \ar[d]
    & {[D^0 , U]}
    \ar[d]
    \\
    \widehat{\mathrm{B}}_{\operad C}\L'
    \ar[r]
    & {[S^{-1},U]}~,
    \end{tikzcd}
    $$
    and thus a commutative square of dg $\Omega \operad C$-coalgebras
    $$
\begin{tikzcd}[column sep=3pc,row sep=3pc]
    \widehat{\mathrm{B}}_{\operad C}\L \arrow[dr, phantom, "\lrcorner", very near start]
    \ar[r] \ar[d]
    & L^P{[D^0 , U]}
    \ar[d]
    \\
    \widehat{\mathrm{B}}_{\operad C}\L'
    \ar[r]
    & L^P{[S^{-1},U]}~.
    \end{tikzcd}
    $$
    This square is a pullback square since the underlying square of graded $\Omega \operad C$-coalgebras
    is a pullback square. Moreover, the map $[D^0 , U] \longrightarrow [S^{-1} , U]$
    is a fibration of dg modules, thus $L^P[D^0 , U] \longrightarrow L^P[S^{-1} , U]$
    is a fibration of dg $\Omega \C$-coalgebras. This implies that the pullback map $\widehat{\mathrm{B}}_{\operad C}\L \longrightarrow \widehat{\mathrm{B}}_{\operad C}\L'$ is a fibration.
\end{proof}

\begin{lemma}\label{lemelemfibequiv}
    Let us consider a commutative diagram of curved $\operad C$-algebras
    \[
    \begin{tikzcd}[column sep=3pc,row sep=3pc]
        \L
        \ar[r, "g"] \ar[d, "p"']
        &\L'
        \ar[d, "q"]
        \\
        Z
        \ar[r, "\id"]
        & Z~,
    \end{tikzcd}
    \]
    where 
    \begin{enumerate}
    
    \medskip
    
        \item $p$ and $q$ are elementary fibrations;
        
    \medskip
    
        \item the map of dg modules induced by $g$
        \[
        \mathrm{Ker}(g): \mathrm{Ker}(p) \longrightarrow \mathrm{Ker}(q)
        \]
        is a quasi-isomorphism.
    \end{enumerate}
    
    \medskip
    
    Then $g$ is a weak-equivalence.
\end{lemma}

\begin{proof}
Let us decompose the underlying graded $\kk$-module of $\L$ as
$\L \cong U \oplus Z$. We can then decompose the underlying graded $\kk$-module of $\L'$ as $\L' \cong U' \oplus Z$ in such a way so that the restriction of $g$ to $U$ targets $U'$. The square diagrams built in the proof of Proposition \ref{prop:elemcof}, fit in the following commutative cube diagram
    $$
    \begin{tikzcd}
        \widehat{\mathrm{B}}_{\operad C}\L
        \ar[rr] \ar[rd] \ar[dd]
        && \widehat{\mathrm{B}}_{\operad C}\L'
        \ar[rd] \ar[dd]
        \\
        & L^P{[D^{0},U]}
        \ar[rr] \ar[dd]
        && L^P{[D^{0},U']}
        \ar[dd]
        \\
        \widehat{\mathrm{B}}_{\operad C}Z
        \ar[rr] \ar[rd]
        && \widehat{\mathrm{B}}_{\operad C}Z
       \ar[rd] 
        \\
        & L^P{[S^{-1},U]}
        \ar[rr]
        &&  L^P{[S^{-1},U']}
    \end{tikzcd}
    $$
    The left and the right faces are pullback (homotopy) squares. The two front horizontal maps and the bottom back horizontal map are quasi-isomorphisms. Thus the homotopy pullback map $\widehat{\mathrm{B}}_{\operad C}g$ is also a quasi-isomorphism.
\end{proof}

\begin{proposition}\label{propelemfibequiv}
    Let us consider a commutative diagram of curved $\operad C$-algebras
    $$
    \begin{tikzcd}[column sep=3pc,row sep=3pc]
        \L
        \ar[r, "g"] \ar[d, "p"']
        &\L'
        \ar[d, "q"]
        \\
        Z
        \ar[r, "f"']
        & Z'
    \end{tikzcd}
    $$
    where 
    
    \medskip
    \begin{enumerate}
        \item $f$ is an equivalence,
    
    \medskip
    
        \item $i$ and $j$ are elementary fibrations,
        
    \medskip
    
        \item the map of dg modules induced by $g$
        $$
        \mathrm{Ker}(g): \mathrm{Ker}(p) \longrightarrow \mathrm{Ker}(q)
        $$
        is a quasi-isomorphism.
    \end{enumerate}
    
   	\medskip
    Then $g$ is a weak-equivalence.
\end{proposition}

\begin{proof}
Let us consider the following pullback square 
    $$
    \begin{tikzcd}[column sep=3pc,row sep=3pc]
        U
        \ar[r] \ar[d] \arrow[dr, phantom, "\lrcorner", very near start]
        & \L'
        \ar[d]
        \\
        Z
        \ar[r]
        & Z'
    \end{tikzcd}
    $$
    in the category of curved $\operad C$-algebras. It yields a pullback square of dg $\Omega \operad C$-coalgebras
    $$
    \begin{tikzcd}[column sep=3pc,row sep=3pc]
        \widehat{\mathrm{B}}_{\operad C}U \arrow[dr, phantom, "\lrcorner", very near start]
        \ar[r] \ar[d]
        & \widehat{\mathrm{B}}_{\operad C} \L'
        \ar[d]
        \\
        \widehat{\mathrm{B}}_{\operad C} Z
        \ar[r]
        & \widehat{\mathrm{B}}_{\operad C} Z'
    \end{tikzcd}
    $$
    which is also an homotopy pullback square. Thus, the map
    $\widehat{\mathrm{B}}_{\operad C} U  \longrightarrow\widehat{\mathrm{B}}_{\operad C} \L'$
    is a quasi-isomorphism. Moreover, the map $\L \longrightarrow U$ is a weak-equivalence by Lemma \ref{lemelemfibequiv}.
\end{proof}


\subsection{Coladders}
We introduce coladders of qp-complete curved $\operad C$-algebras, which lead to the notion of a cofiltered quasi-isomorphism. We will show that these cofiltered quasi-isomorphism are a subset of weak-equivalences of qp-complete curved $\operad C$-algebras. The key example of such coladders is the quasi-planar ladder induced by the canonical quasi-planar filtration of $\C$.

\begin{definition}[$\beta$-coladder]
Let $\beta$ be a small ordinal. A $\beta$-\textit{indexed curved }$\operad C$\textit{-coladder} is a functor 
\[
\Lambda:\beta^\op \longrightarrow \catcurvalg{\operad C}
\]
that sends every limit ordinal $k\in\beta$ to
$$
\Lambda(k) = \lim_{i < k} \Lambda(i)~,
$$
with $\L(-1)=0$, and such that every map
$$
\L(i) \longrightarrow \L(i-1), \quad i \in \beta
$$
is an \textit{elementary fibration}.
\end{definition}

A $\operad C$-coladder $\L$ is called \textit{qp-complete} if $\L(i)$ is qp-complete for all $i < \beta$. We denote by
$$
\Lambda(\beta) \coloneqq \lim_{i \in \beta}~\Lambda(i)~,
$$
the value of the limit of this $\beta$-coladder. It is qp-complete whenever the coladder is.

\begin{remark}
    The property about limit ordinal is equivalent to the fact that
    the functor
\begin{align*}
    (1 + \beta)^\op &\longrightarrow \catcurvalg{\operad C}
    \\
    0 &\mapsto 0
    \\
    i+1 < \omega &\mapsto \L(i)
    \\
    \omega \leq i < \beta & \mapsto \L(i)
\end{align*}
is continuous.
\end{remark}

\begin{definition}[Associated graded of a coladder]
Given a $\beta$-coladder $\Lambda$, we define its \textit{associated graded} as 
\[
\gr^i \Lambda \coloneqq \mathrm{Ker}(\Lambda(i) \longrightarrow \lim{j < i} \Lambda(j))~.
\]
Notice that the derivation squares to zero on this kernel $\gr_i \Lambda$, therefore it is a dg module.    
\end{definition}

\medskip

The following proposition will allow us to construct $\beta$-coladders in a general setting. 

\begin{proposition}\label{prop: celle qui induit des coladders}
Let $\beta$ be a small ordinal, and let  
\[
\operad D^{(0)} \longrightarrow \operad D^{(1)} \longrightarrow \cdots \longrightarrow \operad D^{(i)} \longrightarrow \cdots 
\]
be a $\beta$-indexed cooperad ladder, where we denote $\operad D \coloneqq \operad D^{(\beta)}$. For every $i \in \beta$, let $F^i_{\operad D} (-)$ be the idempotent monad on curved $\operad D$-algebras related to the reflexive full subcategory made up of curved $\operad D^{(i)}$-algebras.

\medskip

Let $\L$ be a curved $\operad D$-algebra which is complete with respect to the cooperad ladder, that is, such 
\[
\L \cong \lim_{i \in \beta} F^i_{\operad D} \L~.
\]
Then the diagram
        $$
        \cdots
        \longrightarrow
        F^i_{\operad D}\L
        \longrightarrow
        \cdots
    \longrightarrow
    F^1_{\operad D}\L
    \longrightarrow  F^0_{\operad D}\L
        $$
is a $\beta$-indexed coladder of curved $\operad D$-algebras, which are all complete with respect to the cooperad ladder.
\end{proposition}

\begin{proof}
Let $\L$ be a curved $\operad D$-algebra that is complete with respect to the cooperad ladder $(\operad D^{(i)})_{i \in \beta}$. Since $\L$ is complete with respect to the cooperad ladder the map
    $$
    F^k_{\operad D}\L \longrightarrow \lim_{i<k} F^i_{\operad D} \L 
    $$
    is an isomorphism for every limit ordinal $k \in \beta+1$.

\medskip

It remains to show that for every $i< i+1 \in \beta$, the transition map $F^{i+1}_{\operad D}\L \longrightarrow F^{i+1}_{\operad D}\L$ is an elementary fibration. Let us denote $K$ its kernel. This map fits in the following pushout diagram of graded $\kk$-modules
     $$
    \begin{tikzcd}[column sep=3pc,row sep=3pc]
        (F^{i+1}_{\operad D}\L)^{\operad D_\pl^{(i+1)}}
        \ar[r] \ar[d] \arrow[dr, phantom, "\ulcorner", very near end]
        & (F^{i+1}_{\operad D}\L)^{\operad D_\pl^{(i)}}
        \ar[d]
        \\
        F^{i+1}_{\operad D}\L
        \ar[r]
        & F^i_{\operad D}\L~.
    \end{tikzcd}
    $$
In particular, it is a degree-wise epimorphism since the map $(F^{i+1}_{\operad D}\L)^{\operad D_\pl^{(i+1)}} \longrightarrow (F^{i+1}_{\operad D}\L)^{\operad D_\pl^{(i)}}$ is a degree-wise epimorphism. 
        Thus proving that it is an elementary fibration amounts to prove that the composite map
        $$
        \shuffle(F^{i+1}_{\operad D}\L, K)^{\overline{\operad D}_\pl^{(i+1)}}
        \longrightarrow
        (F^{i+1}_{\operad D}\L)^{\overline{\operad D}_\pl^{(i+1)}}
        \longrightarrow F^{i+1}_{\operad D}\L
        $$
        is zero, since the kernel of the first map is 
        $(F^{i}_{\operad D}\L)^{\overline{\operad D}_\pl^{(i+1)}}$. Moreover, the kernel $K$ is the image of the map 
        $$
        (F^{i+1}_{\operad D}\L)^{\operad D_\pl^{(i+1)}/\operad D_\pl^{(i)}}
        \hookrightarrow
        (F^{i+1}_{\operad D}\L)^{\operad D_\pl^{(i+1)}}
        \longrightarrow F^{i+1}_{\operad D}\L.
        $$
        Thus, it suffices to
        prove that the composite map
        $$
        \shuffle(F^{i+1}_{\operad D}\L, (F^{i+1}_{\operad D}\L)^{\operad D_\pl^{(i+1)}/\operad D_\pl^{(i)}})^{\overline{\operad D}_\pl^{(i+1)}} \longrightarrow 
        ((F^{i+1}_{\operad D}\L)^{{\operad D}_\pl^{(i+1)}})^{\overline{\operad D}_\pl^{(i+1)}}
        \longrightarrow F^{i+1}_{\operad D}\L
        $$
        is zero. Since degree-wise injections are preserved by the tensor products and by pullbacks, this map is a degree-wise injection. Thus it remains to show that for every $p,q, j$, with $p \geq 1$ and $1\leq j \leq p$, the map
      $$
    \begin{tikzcd}[column sep=3pc,row sep=2pc]
    F_{i+1}^{\operad D}W \arrow[d] \\
    	\overline{\operad D}_\pl^{(i+1)}(p) \otimes  (F_{i+1}^{\operad D}W )^{\otimes p}\arrow[d] \\
    	\overline{\operad D}_\pl^{(i+1)}(p) \otimes  (F_{i+1}^{\operad D}W )^{\otimes j-1 }
     \otimes \overline{\operad D}_\pl^{(i+1)}(q) \otimes (F_{i+1}^{\operad D}W)^{\otimes q} \otimes (F_{i+1}^{\operad D}W)^{\otimes p-j}
    \end{tikzcd}
    $$
     factors through the sub-object
     $$
     \overline{\operad D}_\pl^{(i+1)}(p) \otimes  (F_{i+1}^{\operad D}W )^{\otimes j-1 }
     \otimes \overline{\operad D}^{(i)}(q) \otimes (F_{i+1}^{\operad D}W)^{\otimes q} \otimes (F_{i+1}^{\operad D}W)^{\otimes p-j}~.
     $$
     Using coassociativity, such a map rewrites as follows
     $$
    \begin{tikzcd}[column sep=3pc,row sep=2pc]
    F_{i+1}^{\operad D}W \arrow[d] \\
    	\overline{\operad D}_\pl^{(i+1)}(p+q-1) \otimes  (F_{i+1}^{\operad D}W )^{\otimes p+q-1} \arrow[d] \\
    	(\overline{\operad D}_\pl^{(i+1)}(p)  \otimes \overline{\operad D}_\pl^{(i+1)}(q)) \otimes  (F_{i+1}^{\operad D}W )^{\otimes p+q-1}~.
    \end{tikzcd}
    $$
     Since the sequence $(\operad D^{(i)})_{i \in \beta}$ is a cooperad ladder, the map 
     \[
     \Delta_j: \overline{\operad D}_\pl^{(i+1)}(p+q-1) \longrightarrow \overline{\operad D}_\pl^{(i+1)}(p) \otimes \overline{\operad D}_\pl^{(i+1)}(q)
     \]
     factors through $\overline{\operad D}_\pl^{(i)}(p) \otimes \overline{\operad D}_\pl^{(i)}(q)$, which proves the result.
\end{proof}

\begin{corollary}[Quasi-planar coladder]
Let $\C$ be a quasi-planar conilpotent curved cooperad. Recall from Subsection \ref{subsection: quasi-planar canonical filtration} that since $\C$ is quasi-planar, it admits a canonical quasi-planar $\omega$-ladder
$$
F_0^{\mathrm{qp}} \C \longrightarrow F_1^{\mathrm{qp}} \C \longrightarrow \cdots \longrightarrow F_n^{\mathrm{qp}} \C \longrightarrow \cdots
$$
whose colimit is $\C$. For every $i \in \omega$, let $F_i^{\mathrm{qp}} (-)$ be the idempotent monad of curved $\operad C$-algebras that reflects onto the full subcategory of curved $F_i^{\mathrm{qp}} \C$-algebras.

\medskip

Let $\L$ be a curved $\C$-coalgebra. The diagram 
$$
    \cdots \twoheadrightarrow F^n_{\mathrm{qp}}\L \longrightarrow \cdots \twoheadrightarrow F^n_{\mathrm{qp}} \L 
    $$
is an $\omega$-coladder of curved $\C$-algebras, whose limit is the qp-completion of $\L$, called the \textit{quasi-planar coladder} of $\L$.
\end{corollary}

\begin{proof}
Follows directly from Proposition \ref{prop: celle qui induit des coladders}.
\end{proof}


\subsection{Fibrations}

\begin{lemma}
A fibration of dg $\Omega \C$-coalgebras is in particular a degree-wise epimorphism.
\end{lemma}

\begin{proof}
Let $f:V \longrightarrow V'$ be a fibration of dg $\Omega \C$-coalgebras. We equip $D^0$ with its canonical counital coassociative cocommutative coalgebra structure. The map $D^0 \otimes V' \longrightarrow V'$ which is degree-wise surjective factors through $f$ which is thus also degree-wise surjective.
\end{proof}

\begin{proposition}\label{propfibration}
A morphism of qp-complete curved $\operad C$-algebras is a fibration if and only if it is a degree-wise epimorphism.
\end{proposition}

\begin{proof}
Let $f: \L \longrightarrow \L'$ be a morphism of qp-complete curved $\operad C$-algebras.
On the one hand, let us suppose that $f$ is a degree-wise epimorphism and let $K$ be its kernel. Since $\L$ is qp-complete, the morphism $f$ can be recovered as the transfinite backward composition of the sequence
$$
\cdots \longrightarrow F^n_{\mathrm{qp}} \L \times_{F^n_{\mathrm{qp}} \L'} \L' \longrightarrow \cdots \longrightarrow F^0_{\mathrm{qp}} \L \times_{F^0_{\mathrm{qp}} \L'} \L' \longrightarrow  \L'.
$$
Since every such morphism is an elementary fibration and since fibration are preserved by backward transfinite compositions, the morphism $f$ is therefore a fibration.

\medskip

On the other hand, let us suppose that $f$ is a fibration. Let us consider the following square diagram of graded $\kk$-modules
$$
\begin{tikzcd}[column sep=3pc,row sep=3pc]
    L^{\operad P} \L
    \ar[r] \ar[d]
    &\L
    \ar[d]
    \\
    L^{\operad P} \L'
    \ar[r]
    &\L'
\end{tikzcd}
$$
The bottom horizontal map and the left vertical map are degree-wise epimorphisms. Thus, so is the right vertical map.
\end{proof}


\subsection{Cofiltered quasi-isomorphisms}

\begin{definition}[Cofiltered quasi-isomorphism of coladders]
A morphism of $\beta$-indexed curved $\operad C$-coladders $f:\L \longrightarrow \L'$ is a \textit{cofiltered quasi-isomorphism} if

\[
\gr^i (f): \gr^i \L \qi \gr^i \L'
\]

is a quasi-isomorphism for all $i \in \beta$.
\end{definition}

\begin{proposition}\label{propbarfilteredqis}
    Let $f: V \qi V'$ be a quasi-isomorphism of dg $\Omega \operad C$-coalgebras. The morphism of coladders
    \[
    F^n_{\mathrm{qp}}\widehat{\Omega}_{\operad C} V \longrightarrow F^n_{\mathrm{qp}}\widehat{\Omega}_{\operad C} V'
    \]
    \vspace{0.1pc}
    is a cofiltered quasi-isomorphism.
\end{proposition}

\begin{proof}
    For every $i \in \omega$, the map
    $$
    \gr^i_{\mathrm{qp}}(\widehat{\Omega}_{\operad C} V) \longrightarrow
    \gr^i_{\mathrm{qp}}(\widehat{\Omega}_{\operad C} V')
    $$
    rewrites as the morphism of dg modules
    $$
    V^{\gr_i \operad C_\pl} \qi V'^{\gr_i \operad C_\pl}
    $$
    which is a quasi-isomorphism.
\end{proof}

\begin{proposition}
Let $f: \L \longrightarrow \L'$ be a cofiltered quasi-isomorphism of curved $\operad C$-coladders indexed by an ordinal $\beta$. The map $f(\beta): \L(\beta) \longrightarrow \L'(\beta)$ is a weak-equivalence.
\end{proposition}

\begin{proof}
Notice that the following holds.

\medskip

\begin{enumerate}
    \item The map $f(0)$ is a weak-equivalence, since it is the identity of the zero object $0$.
    
    \medskip
    
    \item If $i \in \beta+1$ is a limit ordinal so that $f(j)$ is an equivalence
    for every $j < i$, then the limits
    $$
    \widehat{\mathrm{B}}_{\operad C} \L(i) \cong \lim_{j<i} \widehat{\mathrm{B}}_{\operad C} \L(j) , \quad \widehat{\mathrm{B}}_{\operad C} \L'(i) \cong \lim_{j<i}  \widehat{\mathrm{B}}_{\operad C} \L'(j)
    $$
    are homotopy limits, and therefore the map $\widehat{\mathrm{B}}_{\operad C} f(i)$ is a quasi-isomorphism. Thus $f(i)$ is a weak-equivalence.
    \medskip
    \item By Proposition \ref{propelemfibequiv}, $f(i+1)$ is a weak-equivalence whenever $f(i)$ is a weak-equivalence.
\end{enumerate}

\medskip
We conclude by an ordinal induction.
\end{proof}


\subsection{Path object}
Let $\L$ be an qp-complete curved $\operad C$-algebra. Let $P$ be a path object of $\widehat{\mathrm{B}}_{\operad C} \L$ in the category of dg $\Omega \C$-coalgebras. This means that there exists a factorization

\[
\begin{tikzcd}
\widehat{\mathrm{B}}_{\operad C} \L \arrow[r,"\iota",rightarrowtail]
&P \arrow[r,"p",twoheadrightarrow]
&\widehat{\mathrm{B}}_{\operad C}(\L \oplus \L)
\end{tikzcd}
\]
\vspace{0.1pc}

where $\iota$ is an acyclic cofibration and $p$ a fibration in the category of dg $\Omega \C$-coalgebras. Let $D$ be the following pushout 

\[
D \coloneqq \widehat{\Omega}_{\operad C}P 
\amalg_{\widehat{\Omega}_{\operad C}\widehat{\mathrm{B}}_{\operad C} (B)} \L.
\]
\vspace{0.1pc}

in the category of curved $\operad C$-algebras. We have that, for every natural integer $n$, the canonical map
\[
F_{\mathrm{qp}}^i D \longrightarrow F_{\mathrm{qp}}^i\widehat{\Omega}_{\operad C}P
\amalg_{F_{\mathrm{qp}}^i\widehat{\Omega}_{\operad C}\widehat{\mathrm{B}}_{\operad C} (B)} F_{\mathrm{qp}}^i\L.
\]
\vspace{0.1pc}
is an isomorphism. Moreover, the qp-completion $\widehat D$ of $D$ is the limit of the diagram 
$$
\cdots \longrightarrow F_{\mathrm{qp}}^i D \longrightarrow \cdots \longrightarrow F_{\mathrm{qp}}^0D~.
$$
It can be computed as the pushout $\widehat{\Omega}_{\operad C}P\amalg_{\widehat{\Omega}_{\operad C}\widehat{\mathrm{B}}_{\operad C} (B)} \L$ in the category of qp-complete curved $\C$-algebras. 

\medskip

Our goal is to show that $\widehat D$ is a natural path object in the category of qp-complete curved $\C$-algebras. These objects fit in the following commutative diagram
$$
\begin{tikzcd}[column sep=3pc,row sep=3pc]
    \widehat{\Omega}_{\operad C}\widehat{\mathrm{B}}_{\operad C} \L
    \ar[r,"\widehat{\Omega}_{\operad C}(\iota)"] \ar[d,"\epsilon_\L",twoheadrightarrow,swap]
    &\widehat{\Omega}_{\operad C} P
    \ar[r,"\widehat{\Omega}_{\operad C}(p)"] \ar[d,"q"] 
    &\widehat{\Omega}_{\operad C} \widehat{\mathrm{B}}_{\operad C}(\L \oplus \L)
    \ar[d,"\epsilon_{\L \oplus \L}",twoheadrightarrow]
    \\
   \L
    \ar[r,"j"]
    &D
    \ar[r,"p^{\dagger} \amalg \Delta"]
    & \L \oplus \L
\end{tikzcd}
$$
of curved $\C$-algebras. Here the morphism $j$ is given by $\L$ in the pushout. The morphism $p^{\dagger} \amalg \Delta$ is induced by the transpose of $p$ and by the diagonal $\Delta: \L \longrightarrow \L \oplus \L$. The morphisms $\epsilon$ are given by the counit of the complete bar-cobar adjunction $\widehat{\Omega}_{\operad C} \dashv \widehat{\mathrm{B}}_{\operad C}$. They are degree-wise epimorphisms and thus fibrations by Proposition \ref{propfibration}. Notice that $\widehat{\Omega}_{\operad C}(i)$ is a filtered quasi-isomorphism of coherent coladders by Proposition \ref{propbarfilteredqis}. The first step is to show that $j$ is a cofiltered quasi-isomorphism.

\medskip

Let us choose a particular summand $\L$ in the product $\L \oplus \L$. This way, we choose one of the two projections $t: D \longrightarrow \L$  and one of the two projections $r: P \longrightarrow \widehat{\mathrm{B}}_{\operad C} \L$
in such a way that the following diagram 
$$
\begin{tikzcd}[column sep=3pc,row sep=3pc]
    \widehat{\Omega}_{\operad C}\widehat{\mathrm{B}}_{\operad C} \L
    \ar[r,"\widehat{\Omega}_{\operad C}(\iota)"] \ar[d,"\epsilon_\L",twoheadrightarrow,swap]
    & \widehat{\Omega}_{\operad C} P
    \ar[d,"q"]
    \ar[r,"\widehat{\Omega}_{\operad C}(r)"] \ar[d] 
    & \widehat{\Omega}_{\operad C} \widehat{\mathrm{B}}_{\operad C}\L
    \ar[d,"\epsilon_{\L \oplus \L}",twoheadrightarrow]
    \\
    \L
    \ar[r, "j"]
    & D
    \ar[r,"t"]
    & \L~.
\end{tikzcd}
$$

commutes. The dg module $P$ decomposes into $r: P \cong \widehat{\mathrm{B}}_{\operad C} \L  \oplus K$ where $K$ is the kernel of the map $P \longrightarrow \widehat{\mathrm{B}}_{\operad C} \L $. The dg module $K$ is acyclic since $r$ is a quasi-isomorphism. Let us take a contracting homotopy $h$ of $K$, that is, a degree $1$ endomorphism of $K$ such that 
$$
\partial(h) = d_K h + h d_K = \id_K.
$$
We can extend $h$ to $P$ by zero on $\widehat{\mathrm{B}}_{\operad C}\L $. Then $\partial(h) = \pi_K$, where $\pi_K$ is the projection onto $K$. Now let $H$ be the degree $1$ endomorphism of the graded module
$$
\widehat{\Omega}_{\operad C} P = P^{\operad C_\pl}
= \prod_{n \geq 0} [\operad C_\pl(n) , P^{\otimes n}]
$$
defined as
$$
H \coloneqq \prod_{n \geq 0} \left[\id_{\operad C_\pl(n)} , \sum_{k=0}^{n-1} \pi_{\widehat{\mathrm{B}}_{\operad C} \L}^{\otimes k} \otimes h \otimes \id_V^{\otimes n-k-1}\right]~.
$$
One can extend $H$ to $\left(P^{\operad C_\pl}\right)^{\operad C_\pl}$ using the same formula
$$
H \coloneqq \prod_{n \geq 0} \left[\id_{\operad C_\pl(n)} , \sum_{k=0}^{n-1} \pi_{(\widehat{\mathrm{B}}_{\operad C} \L)^{\operad C_\pl}}^{\otimes k} \otimes H \otimes \id_{V^{\operad C_\pl}}^{\otimes n-k-1}\right] .
$$
The same formula mutatis mutandis allows us to extend $H$ to $\left(\left(P^{\operad C_\pl}\right)^{\operad C_\pl}\right)^{\operad C_\pl}$. One can notice then that $H$ commutes with the maps
$$
\left(\left(P^{\operad C_\pl}\right)^{\operad C_\pl}\right)^{\operad C_\pl}
\rightrightarrows \left(P^{\operad C_\pl}\right)^{\operad C_\pl} \longrightarrow P^{\operad C_\pl}~.
$$

\begin{lemma}
The degree $1$ graded endomorphism $H$ of $\widehat{\Omega}_{\operad C} P$ projects onto $D$, in the sense that the map
\[
\widehat{\Omega}_{\operad C} P \xrightarrow{H} \widehat{\Omega}_{\operad C} P \xrightarrow{q}  D
\]
factors through the projection $q:\widehat{\Omega}_{\operad C} P \longrightarrow  D$. We also denote by $H$ the resulting unique degree $1$ graded endomorphism of $D$.
\end{lemma}

\begin{proof}
We set

\begin{enumerate}
\item $X$ to be the pushout of the cospan $\L \longleftarrow \widehat{\Omega}_{\operad C}\widehat{\mathrm{B}}_{\operad C}\L \longrightarrow \widehat{\Omega}_{\operad C} P$ in the category of graded $\kk$-modules;

\medskip

\item and $Y$ to be the pushout of the cospan $\L^{\operad C_\pl} \longleftarrow (\widehat{\Omega}_{\operad C}\widehat{\mathrm{B}}_{\operad C}\L)^{\operad C_\pl} \longrightarrow (\widehat{\Omega}_{\operad C} P)^{\operad C_\pl}$ in the category of graded $\kk$-modules.

\end{enumerate}

The diagram of graded $\operad C$-algebras

$$
\begin{tikzcd}[column sep=3pc,row sep=3pc]
    Y^{\operad C_\pl}
    \ar[r, bend left] \ar[r, bend right]
    &X^{\operad C_\pl}
    \ar[l]
    \ar[r,dashed]
    & D
\end{tikzcd}
$$

is colimiting both in the category complete graded $\operad C$-algebras and in the category of graded $\kk$-modules. Indeed, the forgetful functor from graded $\C$-algebras to graded $\kk$-modules commutes with reflexive coequalisers since those are preserved by the comonad $(-)^{\operad C_\pl}$.

\medskip

Notice the following:

\medskip

\begin{enumerate}
\item the degree $1$ endomorphism $H$ of $\left(P^{\operad C_\pl}\right)^{\operad C_\pl}$ projects onto $X^{\operad C_\pl}$;

\medskip

\item the degree $1$ endomorphism $H$ of $\left(\left(P^{\operad C_\pl}\right)^{\operad C_\pl}\right)^{\operad C_\pl}$
projects onto $Y^{\operad C_\pl}$;
\end{enumerate}

\medskip

since their restriction to, respectively, $\widehat{\Omega}_{\operad C}\widehat{\mathrm{B}}_{\operad C}\L$ and $(\widehat{\Omega}_{\operad C}\widehat{\mathrm{B}}_{\operad C}\L)^{\operad C_\pl} $ is zero. Therefore, $H$ projects onto the  of the coequalizer, which is given by $D$.
\end{proof}

\begin{lemma}
The map $H$ projects onto the quotients $F^{i}_{\mathrm{qp}} \widehat{\Omega}_{\operad C} P$ and $F^{i}_{\mathrm{qp}} D$ for every $i \geq 0$.
\end{lemma}

\begin{proof}
It is clear that, by definition, $H$ that it projects onto $F^{i}_{\mathrm{qp}} \widehat{\Omega}_{\operad C}P = P^{F_{i}^{\mathrm{qp}}\operad C}$. Thus $H$ also projects onto $F^{i}_{\mathrm{qp}} D$ since the following square diagram
    $$
    \begin{tikzcd}[column sep=3pc,row sep=3pc]
        \widehat{\Omega}_{\operad C} P
        \ar[r] \ar[d] \arrow[dr, phantom, "\ulcorner", very near end]
        & D
        \ar[d]
        \\
        F^{i}_{\mathrm{qp}} \widehat{\Omega}_{\operad C} P
        \ar[r]
        & F^{i}_{\mathrm{qp}} D
    \end{tikzcd}
    $$
    is a pushout in the category of graded $\kk$-modules and since we already know that $H$ projects onto $D$ and onto $F^{i}_{\mathrm{qp}} \widehat{\Omega}_{\operad C} P$.
\end{proof}

\begin{lemma}
For every $i \geq 0$, the endomorphism $\partial(H)$ of $\gr^{i}_{\mathrm{qp}}\widehat{\Omega}_{\operad C}(P)$ is equal to the identity minus the projection onto $\gr^{i}_{\mathrm{qp}} \widehat{\Omega}_{\operad C}\widehat{\mathrm{B}}_{\operad C} \L$.
\end{lemma}

\begin{proof}
Both summands $\gr^{i}_{\mathrm{qp}}\widehat{\Omega}_{\operad C}\widehat{\mathrm{B}}_{\operad C}B$ and 
    $$
\prod_{n \geq 0}\left[\gr_i^{\mathrm{qp}}\operad C_\pl(n) , \bigoplus_{1\leq k < n} (\widehat{\mathrm{B}}_{\operad C}\L)^{\otimes k}
\otimes K \otimes V^{n-k-1}
\right]
    $$
of $\gr^{i}_{\mathrm{qp}}\widehat{\Omega}_{\operad C}P$ are stable through the differential and $H$. Then, a straightforward check shows that $\partial(H)$ is zero on the first summand and the identity on the second one. 
\end{proof}

\begin{proposition}\label{corollarypath}
For every $i \geq 0$, the endomorphism $\partial(H)$ of $\gr^{i}_{\mathrm{qp}} D$ is equal to the identity minus the projection onto $\gr^{n}_{\mathrm{qp}} \L$. Subsequently, the maps
\[
\begin{tikzcd}[column sep=3.5pc,row sep=3pc]
\gr^{i}_{\mathrm{qp}} \L  \arrow[r,"\gr^{i}_{\mathrm{qp}}(j)"]
&\gr^{i}_{\mathrm{qp}} D \arrow[r,"\gr^{i}_{\mathrm{qp}}(t)"]
&\gr^{i}_{\mathrm{qp}} \L
\end{tikzcd}
\]
\vspace{0.1pc}

are quasi-isomorphisms.
\end{proposition}

\begin{proof}
Let us consider the following commutative diagram of dg modules
$$
\begin{tikzcd}[column sep=3pc,row sep=3pc]
\gr^{i}_{\mathrm{qp}}\widehat{\Omega}_{\operad C}\widehat{\mathrm{B}}_{\operad C} \L
\ar[r, "\widehat{\Omega}_{\C} (\iota)"] \ar[d, "\epsilon_\L",swap]
& \gr^{i}_{\mathrm{qp}}\widehat{\Omega}_{\operad C}P
\ar[r, "\widehat{\Omega}_\C (r)"] \ar[d, "q"]
& \gr^{i}_{\mathrm{qp}}\widehat{\Omega}_{\operad C}\widehat{\mathrm{B}}_{\operad C} \L
\ar[d, "\epsilon_{\L \oplus \L}"]
\\
\gr^{i}_{\mathrm{qp}} \L  \ar[r, "j"] 
& \gr^{i}_{\mathrm{qp}} D \ar[r, "t"]
& \gr^{i}_{\mathrm{qp}}\L~,
\end{tikzcd}
$$
where the maps are denoted by the same letter as before, omitting the functor $\gr^{i}_{\mathrm{qp}}(-)$ applied on them for simplicity. Since $q$ commutes with the differential $d$ and with $H$, and since the squares above are commutative, we have
$$
\partial(h) q = q\partial(h) 
= q(\id - \widehat{\Omega}_\C(\iota r)) = q - q\widehat{\Omega}_\C(\iota r) = q - jtq = (\id - jt)q
= (\id - \pi_{\gr^{n}_{\mathrm{qr}}\widehat{\Omega}_{\operad C}\widehat{\mathrm{B}}_{\operad C} \L})q.
$$
Since $q: \gr^{i}_{\mathrm{qp}}\widehat{\Omega}_{\operad C}P \longrightarrow
\gr^{i}_{\mathrm{qp}}D$ is a degree-wise epimorphism, $\partial(h) = \id - \pi_{\gr^{i}_{\mathrm{qp}}\L}$ on $\gr^{i}_{\mathrm{qp}}D$. Therefore, $\gr^{i}_{\mathrm{qp}}\L$ is a deformation retract of $\gr^{i}_{\mathrm{qp}}D$.
\end{proof}

\begin{proposition}\label{proppath}
Let $\L$ be a qp-complete curved $\C$-algebra. The factorisation 

\[
\begin{tikzcd}[column sep=3.5pc,row sep=0pc]
\L \arrow[r,"j",rightarrowtail]
&\widehat D \arrow[r,"p^{\dagger} \amalg \Delta", twoheadrightarrow]
&\L \oplus \L
\end{tikzcd}
\]
\vspace{0.1pc}

makes $\widehat D$ a good path object of $\L$, in the sense that

\medskip

\begin{enumerate}
\item the map $p^{\dagger} \amalg \Delta: \widehat D \longrightarrow \L \oplus \L$ is a fibration,

\medskip

\item the map $j: \L \longrightarrow \widehat D$ is a weak-equivalence.
\end{enumerate}

\medskip
\end{proposition}

\begin{proof}
The map $\widehat D \longrightarrow \L \oplus \L$ is a degree-wise epimorphism since the maps $\widehat{\Omega}_{\operad C}(p): \widehat{\Omega}_{\operad C}P \twoheadrightarrow \widehat{\Omega}_{\operad C} \widehat{\mathrm{B}}_{\operad C}(\L \oplus \L)$ and $\epsilon_{\L \oplus \L}: \widehat{\Omega}_{\operad C} \widehat{\mathrm{B}}_{\operad C}(\L \oplus \L) \twoheadrightarrow \L \oplus \L$ are degree-wise epimorphisms. Therefore it is a fibration, since both $\widehat D$ and $\L \oplus \L$ are qp-complete. To conclude, Proposition \ref{corollarypath} tells us that the map $j: \L \longrightarrow \widehat D$ is a cofiltered quasi-isomorphism. Thus it is a weak-equivalence.
\end{proof}

\begin{remark}
Actually, the map $j: \L \longrightarrow \widehat D$ is also a cofibration as the pushout of a cofibration.
\end{remark}


\newpage

\section{A Quillen equivalence, $\infty$-morphisms and homotopy transfer theorems for coalgebras}

\vspace{2pc}

We show that the complete bar-cobar adjunction induces a Quillen equivalence. This allows us to give another presentation of the homotopy category of dg $\Omega \C$-coalgebras in terms of complete curved $\C$-algebras. We introduce $\infty$-morphisms and show that they are invertible. This allows us to prove a homotopy transfer theorem for dg $\Omega \C$-coalgebras. Finally, we show how another model categories structures on the category of complete curved $\C$-coalgebras can be obtained by a right Bousfield localization. 

\subsection{The Quillen equivalence}
The goal of this subsection is to show the following theorem.

\begin{theorem}[{After \cite[Section 11]{linearcoalgebras}}]\label{thm: complete bar-cobar is a Quillen equivalence}
Let $\operad C$ be a quasi-planar conilpotent curved cooperad. The complete bar-cobar adjunction 
\[
\begin{tikzcd}[column sep=5pc,row sep=3pc]
          \catcurvcog{\operad C}^{\mathsf{qp}\text{-}\mathsf{comp}} \arrow[r, shift left=1.1ex, "\widehat{\Omega}_{\C}"{name=F}] & \mathsf{dg}~ \Omega \C\text{-}\mathsf{cog}, \arrow[l, shift left=.75ex, "\widehat{\mathrm{B}}_{\C}"{name=U}]
            \arrow[phantom, from=F, to=U, , "\dashv" rotate=-90]
\end{tikzcd}
\]
\vspace{0.1pc}
    
is a Quillen equivalence.
\end{theorem}

\begin{proof}
The theorem follows directly from Lemma \ref{lemma: key lemma of Quillen equiv for coalg and absolute alg}.
\end{proof}

\begin{corollary}
Let $\operad P$ be a cofibrant dg operad. The quasi-planar complete bar-cobar adjunction 
\[
\begin{tikzcd}[column sep=5pc,row sep=3pc]
          \mathsf{dg}~\operad P\text{-}\mathsf{cog} \arrow[r, shift left=1.1ex, "\widehat{\Omega}^{\mathrm{q.p}}_{\pi}"{name=A}]
         &\mathsf{curv}~\mathrm{B}(\operad E \otimes \operad P) \text{-}\mathsf{alg}^{\mathsf{qp}\text{-}\mathsf{comp}}~. \arrow[l, shift left=.75ex, "\widehat{\mathrm{B}}^{\mathrm{q.p}}_{\pi}"{name=B}] \arrow[phantom, from=A, to=B, , "\dashv" rotate=-90]
\end{tikzcd}
\]

is a Quillen equivalence.
\end{corollary}

\begin{proof}
It suffices to notice that this Quillen adjunction factors into two Quillen equivalences
$$
\begin{tikzcd}[column sep=5pc,row sep=3pc]
          \catdgcog{\operad P}
          \arrow[r, shift left=1.1ex, "\psi^*"{name=F}]
          & \catdgcog{\Omega \mathrm{B} (\operad E \otimes \operad P)}
          \arrow[r, shift left=1.1ex, "\widehat{\Omega}_{\mathrm{B} (\operad E \otimes \operad P)}"{name=O}]
          \arrow[l, shift left=.75ex, "\psi_!"{name=R}]
        & \catcurvcompalg{\mathrm{B} (\operad E \otimes \operad P)}~,
          \arrow[l, shift left=.75ex, "\widehat{\mathrm{B}}_{\mathrm{B}(\operad E \otimes \operad P)}"{name=B}]
          \arrow[phantom, from=O, to=B, , "\dashv" rotate=-90] \arrow[phantom, from=F, to=R, , "\dashv" rotate=-90]
\end{tikzcd}
$$
where the first adjunction is induced by the quasi-morphism $\psi: \Omega \mathrm{B} (\operad E \otimes \operad P) \qi \operad P$. 
\end{proof}

\begin{lemma}\label{lemma: key lemma of Quillen equiv for coalg and absolute alg}
For every dg $\Omega \operad C$-coalgebra $V$, the unit map $\eta_V: V \qi \widehat{\mathrm{B}}_{\operad C}\widehat{\Omega}_{\operad C} V$ is a quasi-isomorphism.
\end{lemma}

The result was already proven in \cite[Section 11]{linearcoalgebras}. The proof here is almost the same. The main difference lays in the fact that leveraging the qp-filtration of $\C$ in Lemma \ref{lemmaautor2} makes the cofiltration arguments simpler than in \cite[Lemma 11.28]{linearcoalgebras}.

\medskip

In the following paragraphs, our strategy to show Lemma \ref{lemma: key lemma of Quillen equiv for coalg and absolute alg} is the following. First, we notice that the unit map $\eta_V:V \longrightarrow \widehat{\mathrm{B}}_{\operad C}\widehat{\Omega}_{\operad C}V$ has a canonical left-inverse $\xi_V$ in the category of dg modules. Let $\pi_V = \eta_V\xi_V$. We define a define a degree $1$ map $H$ on $\Omega_{\operad C}\mathrm{B}_{\operad C}B$ and show that $\partial(H) + \pi_V$ is an isomorphism in Proposition \ref{keypropquilleneqcoalg}. Therefore $\pi_V$ is an quasi-isomorphism since it is homotopic to an isomorphism. This implies that $\eta_V$ has both a left-inverse $\xi_V$ and a right-inverse $\xi_V (\partial(H) + \pi_V)^{-1}$ in the homotopy category of dg modules. Therefore $\eta_V$ is a quasi-isomorphism and this concludes the proof.

\begin{notation}
We abbreviate the dg operad $\Omega \C$ by $\operad P$ for the rest of this subsection. We denote:

\medskip

\begin{enumerate}
\item by $z: \operad P \longrightarrow \operad I$ the canonical augmentation of the underlying graded operad. 

\medskip

\item the pre-differenial $d_\C$ on $\operad P$ that is induced by the pre-differential of $\C$, 

\medskip

\item the pre-differential $d_\theta$ that results from the curvature of $\operad C$,

\medskip

\item the pre-differential $d_\Delta$ that results from the cooperad structure on $\operad C$.
\end{enumerate}

\end{notation}

\begin{definition}[Operad-cooperad diagrams]
    The \textit{planar operad-cooperad diagram} is the following one arrow diagram of graded $\mathbb N$-modules
    $$
    \begin{tikzcd}
        \operad P_\pl \comp_\pl \operad C_\pl
        \ar[r]
        & \operad P_\pl \comp_\pl \operad P_\pl \comp_\pl \operad C_\pl.
    \end{tikzcd}
    $$
    The \textit{symmetric operad-cooperad diagram} is the following one arrow diagram of graded $\mathbb S$-modules
    $$
    \begin{tikzcd}
        \operad P \comp \operad C
        \ar[r]
        & \operad P \comp \operad P \comp \operad C.
    \end{tikzcd}
    $$
    This is actually the image through the free $\mathbb S$-module functor $-\otimes \mathbb S$ of the planar operad-cooperad diagram.
\end{definition}

One has the following pre-differentials on the planar operad-cooperad diagram:

\medskip

\begin{enumerate}
    \item the pre-differential $d_\theta$ that results from the eponymous derivation on $\operad P$,

\medskip

    \item the pre-differential $d_{CP}$ that is induced by the canonical twisting morphism $\iota: \C \longrightarrow \Omega \C = \operad P$,
    
\medskip

    \item the pre-differential $d_\Delta$ that result from the eponymous derivation on $\operad P$.
\end{enumerate}

\medskip

These pre-differentials induce pre-differentials the symmetric operad-cooperad diagram. There is an additional pre-differential $d_C$ that again results from the eponymous derivation on $\operad P$, which is a priori non-planar.

\medskip

\textbf{The restriction-extension diagram} The restriction-extension diagram is a tool that allows one to check whether derivations and homotopies can be restricted to cofree coalgebras, which are very hard to describe in general.

\begin{definition}[Restriction-extension diagram]
    Let $V$ be a graded $\kk$-module. The \textit{restriction-extension diagram} $\restrictionextension$ is the following diagram of graded $\kk$-modules
    $$
\begin{tikzcd}
    L^P V^{\operad C_\pl}
    \ar[r] \ar[d]
    & \left(\left(V^{\operad C_\pl}\right)^{\operad P_\pl}\right)^{\operad P_\pl}
    \ar[dd]
    \\
    \left(V^{\operad C_\pl}\right)^{\operad P_\pl}
    \ar[d]
    \\ 
    V^{\operad P_\pl \comp_\pl \operad C_\pl}
    \ar[r]
    &V^{\operad P_\pl \comp_\pl \operad P_\pl \comp_\pl \operad C_\pl}
\end{tikzcd}
\quad
\text{ shortened into }
\quad
\begin{tikzcd}
    W \ar[r] \ar[d]
    & \restrictiontwo
    \ar[dd]
    \\
    \restrictionone
    \ar[d]
    \\ 
    \extensionone
    \ar[r]
    &\extensiontwo .
\end{tikzcd}
$$
\end{definition}

\begin{lemma}[{After \cite[Lemma 11.8]{linearcoalgebras}}]
The restriction-extension diagram $\restrictionextension$ is a pullback square.
\end{lemma}

\begin{proof}
Let us consider the following diagram  of graded $\kk$-modules
    $$
    \begin{tikzcd}
    L^P V^{\operad C_\pl}
    \ar[r] \ar[d]
    & \left(\left(V^{\operad C_\pl}\right)^{\operad P_\pl}\right)^{\operad P_\pl}
    \ar[d]
    \\
    \left(V^{\operad C_\pl}\right)^{\operad P_\pl}
    \ar[r] \ar[d]
    & \left(V^{\operad C_\pl}\right)^{\operad P_\pl \comp_\pl \operad P_\pl}
    \ar[d]
    \\ 
    V^{\operad P_\pl \comp_\pl \operad C_\pl}
    \ar[r]
    &V^{\operad P_\pl \comp_\pl \operad P_\pl \comp_\pl \operad C_\pl}.
\end{tikzcd}
    $$
Both the top square and the bottom square are pullback squares. Thus, the restriction-extension diagram is a pullback square.
\end{proof}

\begin{definition}[Extension diagram]
We define the \textit{extension diagram} $\extension$ as the sub-diagram of the restriction-extension diagram that only contains $\extensionone$ and $\extensiontwo$. 
\end{definition}

\begin{remark}
Notice that the extension diagram is the image of the planar operad-cooperad diagram and of the symmetric operad-cooperad diagram by the functor $V^{(-)}$, for any graded $\kk$-module $V$.
\end{remark}

\medskip

Let $f$ be a degree $k$ endomorphism of the extension diagram. If $f$ extends to the whole restriction-extension diagram, this extension is necessarily unique since $\restrictionone, \restrictiontwo, W$ are sub-graded $\kk$-modules of $\extensiontwo$. Moreover, the existence of such an extension amounts to the fact that the restriction of $f(\extensiontwo)$ to $\restrictionone$, $\restrictiontwo$ and $W$ has its image in respectively $\restrictionone$, $\restrictiontwo$ and $W$.

\begin{definition}[Cocone of the restriction-extension diagram]
 The restriction-extension diagram $\restrictionextension$ has a canonical cocone towards $V$ induced by the unit of $\operad P$ and the coaugmentation of $\operad C$ 
    $$
    V^{\operad P ~\comp ~\operad P ~  \comp ~ \operad C} \longrightarrow V^{\operad I ~ \comp ~ \operad I ~ \comp ~ \operad I} \cong V.
    $$ 
\end{definition}

This cocone restricts to a morphism $\xi_V: L^{\operad P} V^{\operad C} \longrightarrow V$. Notice that the graded map $\xi_V: \widehat{\mathrm{B}}_{\operad C}\widehat{\Omega}_{\operad C} V \longrightarrow V$ is in fact a morphism of dg modules. The morphism $\xi_V$ admits two sections:

\medskip

    \begin{enumerate}
        \item the unit map $\eta_V: V \longrightarrow \widehat{\mathrm{B}}_{\operad C}\widehat{\Omega}_{\operad C} V$, which is also a morphism of dg modules, 
        
\medskip

        \item the map induced by the counit on $\operad C$ and the graded augmentation map on the underlying graded operad of $\operad P$
        $$
        \zeta_V:V \simeq L^{\operad I}V^{\operad I} \longrightarrow L^{\operad P}V^{\operad C}.
        $$
        which is a section of $\xi_V$ in graded $\kk$-modules.
    \end{enumerate}
    
\medskip
By the universal property of the pullback, such a map $\zeta_V$ extends to a cone of the restriction-extension diagram with source $V$.
Let us denote $\pi_V$ the degree $0$ endomorphism of the restriction-extension diagram defined as $\pi_V = \eta_V \circ \xi_V$. We can notice that the restriction of $\pi_V$ onto $\widehat{\mathrm{B}}_{\operad C}\widehat{\Omega}_{\operad C} V$ is a morphism of dg modules. Let us denote $p_V$ the degree $0$ endomorphism of the restriction-extension diagram defined as $\pi_V = \zeta_V \circ \xi_V$.

\medskip

\textbf{The coderivation on $W$} Let $W \coloneqq \widehat{\mathrm{B}}_{\operad C}\widehat{\Omega}_{\operad C} V$. The coderivation $D$ on $W$ may be decomposed as
$$
D = D_{in} + D_{ex}.
$$
On the one hand, $D_{ex}$ is the degree $-1$ endomorphism of the extension diagram given by the formula
$$
D_{ex} = V^{d_{CP}};
$$
it restricts to $W$ but not to $\restrictionone$ nor $\restrictiontwo$. On the other hand, $D_{in}$ is the degree $-1$ endomorphism of the restriction-extension diagram that is the sum of the following pre-differentials:

\medskip

\begin{enumerate}
    \item the pre-differential $D_V$ which results from the differential on $V$ and whose restriction to the extension diagram is $\shuffle(\id, d_V)^{\extension}$,

\medskip

    \item the pre-differential $D_{CV}$ induced by the structural map of $V$ composed with the curved twisting morphism $\iota$, given by $V \longrightarrow V^{\operad P_\pl} \longrightarrow V^{\operad C_\pl}$,
    
\medskip

    \item the pre-differential $D_C$ which results from the pre-differential on $\operad C$. Its restriction to the extension diagram is $V^{d_C}$,
    
\medskip

    \item the pre-differential $D_\theta$ which results from the curvature of $\operad C$. Its restriction to the extension diagram is $V^{d_\theta}$,
    
\medskip

    \item the pre-differential $D_\Delta$ which results from the structure of a cooperad on $\operad C$. Its restriction to the extension diagram is $V^{d_\Delta}$.
\end{enumerate}

\medskip

\textbf{The contracting homotopy $H$.} Let $h$ be the degree $1$ endomorphism of $\operad P_\pl \comp_\pl \operad C_\pl = (\treemod_\pl (s^{-1}\overline{\operad C}_\pl)) \comp_\pl \operad C_\pl$ described in Section \ref{section: The contracting homotopy on trees}. Remember that, as $\operad P_\pl \comp_\pl \operad C_\pl$ consists in planar trees with nodes labelled by $s^{-1} \overline{\C}_\pl$ and some top nodes by ${\C}_\pl$, $h$ consists in taking the leftest most top vertex of the planar tree and
\begin{enumerate}
    \item applying to it the map
\[
s^{-1} x \in s^{-1} \overline{\C}_\pl \mapsto x \in \overline{\C}_\pl
\]
if the label of this node belongs to $s^{-1} \overline{\C}_\pl$;

\medskip

\item sending this labelled tree to $0$ if the label of this node belongs to $\C_\pl$.

\medskip

\end{enumerate}
If the labelled tree is trivial (that is, it has no node), then $h$ sends such a tree to zero.
As already shown in Section \ref{section: The contracting homotopy on trees}, $h$ is pictorially described by the map

    \begin{center}
\begin{tikzpicture}
		\tikzstyle{vertex}=[draw, circle, thick, inner sep=0,
		minimum size=4]
		\tikzstyle{every path}=[thick]
		\node [vertex] (a) at (0,0) {};
		\node [vertex] (c) at (0,1) {};
		\node [vertex] (d) at ($(a)+(45:1cm)$) {};
		\node [left of=c, node distance=1cm] {$s^{-1}\overline{\operad C_\pl}(3)$};
		\node [left of=a, node distance=1cm] {$s^{-1}\overline{\operad C_\pl}(2)$};
		\node [right of=d, node distance=0.8cm] {$\overline{\operad C_\pl}(1)$};
		\node at (3,0.8) {$h$};
		\draw (a) -- (c);
		\draw (a) -- (0,-1);
		\draw (c) -- ($(c)+(60:1cm)$);
		\draw (c) -- ($(c)+(120:1cm)$);
		\draw (c) -- (0,2);
		\draw (a) -- (d);
		\draw (d) -- ($(d)+(45:1cm)$);
		\draw[|->] (2.5,0.5) -- (3.5,0.5);
		\node [vertex] (a2) at (6,0) {};
		\node [vertex] (c2) at (6,1) {};
		\node [vertex] (d2) at ($(a2)+(45:1cm)$) {};
		\node [left of=c2, node distance=0.7cm] {$\overline{\operad C_\pl}(3)$};
		\node [left of=a2, node distance=1cm] {$s^{-1}\overline{\operad C_\pl}(2)$};
		\node [right of=d2, node distance=0.8cm] {$\overline{\operad C_\pl}(1)$};
		\draw (a2) -- (c2);
		\draw (a2) -- (6,-1);
		\draw (c2) -- ($(c2)+(60:1cm)$);
		\draw (c2) -- ($(c2)+(120:1cm)$);
		\draw (c2) -- (6,2);
		\draw (a2) -- (d2);
		\draw (d2) -- ($(d2)+(45:1cm)$);
\end{tikzpicture}
\end{center}

\medskip

One can extend $h$ to $\operad P_\pl \comp_\pl \operad P_\pl \comp_\pl \operad C_\pl$ as follows in a such a way that it commutes with the maps
in the diagram
\[
\begin{tikzcd}
    \operad P_\pl \comp_\pl \operad P_\pl \comp_\pl \operad C_\pl
    \ar[r]
    & \operad P_\pl \comp_\pl \operad C_\pl.
    \ar[l, shift left] \ar[l, shift right]
\end{tikzcd}
\]
The graded $\mathbb N$-module
$\operad P_\pl \comp_\pl \operad P_\pl \comp_\pl \operad C_\pl$
consists in planar trees $t(x_1, x_2, \ldots)$ whose node are labelled by $s^{-1}\overline{\C}_\pl$ and some top nodes are labelled $\C_\pl$
together with a subtree $t'$ that contains the root but no node labelled $\C_\pl$.
In that context, the extension of $h$ to such pairs $(t(x_1, x_2, \ldots), t')$ consists in applying the former map $h$ onto $t(x_1, x_2, \ldots)$ and 
removing from $t'$ the node that is relabelled if it belonged to it.

One can define precisely such an extension by induction on the height of the subtrees.
For convenience, let us denote $M = \operad P_\pl \comp_\pl \operad C_\pl$. Then,

\medskip

\begin{enumerate}
    \item On $\operad I \comp_\pl M$ $h$ is 
    $$
    \operad I \comp_\pl M \xrightarrow{\id \comp h} \operad I \comp_\pl M.
    $$
    
    \medskip
    
    \item On $s^{-1}\overline{\operad C}_\pl \comp_\pl M$, $h$ is the sum of the map
    \begin{align*}
        s^{-1}\overline{\operad C_\pl} \comp_\pl M
    \twoheadrightarrow 
    s^{-1}\overline{\operad C}_\pl \comp_\pl (\operad I  \comp_\pl \operad I)
    &\longrightarrow \operad I  \comp_\pl (\operad I \comp_\pl \operad C_\pl) \hookrightarrow \operad P_\pl \comp_\pl M
    \\
    s^{-1}x & \mapsto x
    \end{align*}
    with the map
    $$
    \id_{s^{-1}\overline{\operad C}_\pl(n)} \otimes \sum_{k=0}^{n-1} \otimes \left(
    \pi_{\operad I \comp \operad I}^{\otimes k}
    \otimes h \otimes 
    \id_M^{\otimes n-k-1}
    \right).
    $$
    
    \medskip
    
    \item Then on
    $$
    (\overline{\treemod}_{\pl, \leq n+1}(s^{-1}\overline{\operad C}_\pl))\comp_\pl M
    \simeq 
    s^{-1}\overline{\operad C}_\pl \comp_\pl 
    ({\treemod}_{\pl, \leq n+1}(s^{-1}\overline{\operad C}_\pl)) \comp_\pl M
    $$
    $h$ is given by
    induction by
    the sum of the map
    \begin{align*}
        (\overline{\treemod}_{\pl, \leq n+1}(s^{-1}\overline{\operad C}_\pl))\comp_\pl M
    \twoheadrightarrow 
    s^{-1}\overline{\operad C}_\pl \comp_\pl (\operad I  \comp_\pl \operad I)
    &\longrightarrow \operad I  \comp_\pl (\operad I \comp_\pl \operad C_\pl) \hookrightarrow \operad P_\pl \comp_\pl M
    \\
    s^{-1}x & \mapsto x
    \end{align*}
    with the map
    $$
    \id_{s^{-1}\overline{\operad C_\pl}(n)} \otimes \left(
    \sum_{k=0}^{n-1} \pi_{\operad I \comp_\pl \operad I \comp_\pl \operad I}^{\otimes k}
    \otimes h \otimes 
    \id_{\operad P_\pl \comp_\pl \operad P_\pl \comp_\pl \operad C_\pl}^{\otimes n-k-1}
    \right).
    $$
\end{enumerate}

\medskip

\begin{definition}[{\cite[Definition 11.16]{linearcoalgebras}}]
    Let $H$ be the degree $1$ endomorphism of the extension diagram defined 
    as 
    $$
    H = -V^h.
    $$
\end{definition}

\begin{lemma}[{\cite[Proposition 11.17]{linearcoalgebras}}]
    The map $H$ extends to the whole restriction extension-diagram.
\end{lemma}

\begin{definition}[The garbage map]
    Let $G$ be the degree $0$ endomorphism of the extension diagram given by the formula
    $$
    G = D_{ex}H + H D_{ex}-\id + p_V.
    $$
\end{definition}

\begin{lemma}[{\cite[Proposition 11.17]{linearcoalgebras}}]
    The maps $\partial_{ex}(H) = D_{ex}H + H D_{ex}$ and $G$ extend to the whole restriction-extension diagram.
\end{lemma}

\begin{definition}[The boundary map]
    Let $B$ be the degree $0$ endomorphism on the restriction-extension diagram $\restrictionextension$
    defined as
    $$
    B = \partial_{ex}(H) + D_{in}H + H D_{in} + \pi_V = G + \id +  D_{in}H + H D_{in} + \pi_V - p_V.
    $$
\end{definition}

\begin{lemma}
The restriction of $B$ to $W = \widehat{\mathrm{B}}_{\operad C}\widehat{\Omega}_{\operad C} V$ is equal to $\partial (H) + \pi_V$.
\end{lemma}

    Let $gB$ be the degree $0$ endomorphism on the restriction-extension diagram $\restrictionextension$
    defined as
    $$
    gB = B + p_V - \pi_V = \id + G + D_{in}H + H D_{in}.
    $$
    Moreover, let $\overline{B}$ be the degree $0$ endomorphism on the restriction-extension diagram $\restrictionextension$
    defined as
    $$
    \overline B = B - \pi_V = G + D_{in}H + H D_{in} - p_V.
    $$
    
\medskip

\textbf{The map $H$ is a contracting homotopy.} Finally, we show that $H$ is a contracting homotopy, which will allow us to conclude the proof of Lemma \ref{lemma: key lemma of Quillen equiv for coalg and absolute alg}.

\begin{proposition}
    The maps $B$ and $gB$ are degree $0$ automorphisms of the restriction-extension diagram $\restrictionextension$.
\end{proposition}

\begin{proof}
    This is a direct consequence of Lemma \ref{lemmaautoeer} and Lemma \ref{lemmaautor2}.
\end{proof}

\begin{lemma}\label{lemmaautoeer}
    The maps $gB$ and $B$ are automorphisms of $\extensionone$, $\extensiontwo$ and $\restrictionone$.
\end{lemma}

\begin{proof}
Let us filter $\operad P$ by the qp-filtration on $\operad C$. This induces a cofiltration on $\extensionone$ and on $\restrictionone$. On the associated graded object of this cofiltration, $\gr H = 0$ and $\gr(G)=0$ and $\gr(\pi_V)=\gr(p_V)$. Thus $\gr(gB)=\gr(B)= \id$. Subsequently, $gB$ and $B$  are automorphisms. For $\extensiontwo$, one can filter $\operad P \comp \operad P$ in a similar fashion and obtain the same result.
\end{proof}

\begin{lemma}\label{lemmaautor2}
    The map $B$ is an automorphism of $R_2$.
\end{lemma}

\begin{proof}
    The proof is similar to that of \cite[Lemma 11.28]{linearcoalgebras}. The main difference
    is the use of the "qp filtration".
    One can filter $\operad P$ from the qp filtration on $\operad C$. This induces a cofiltration
    on $\restrictiontwo  = \restrictionone^{\operad P}$ that is preserved by $D_{in}$, $H$, $G$, and $\pi_V$. 
    Let us prove that on the associated graded object of this cofiltration, the map $\gr(B)$ is an isomorphism.
    The associated graded has the form
    $$
    \gr(\restrictiontwo) = \restrictionone^{\gr \operad P}
    = \restrictionone^{\gr \operad P_\pl}
    =
    \prod_t  \left[t(s^{-1} \gr^{\qp} \overline{\operad C_\pl}), \restrictionone^{\otimes l(t)} \right]~,
    $$
    where the product is taken over planar trees $t$ and where $l(t)$ is the number of leaves of $t$.
    On this graded object
    
    \medskip
    
    \begin{enumerate}
        \item $p_V = \pi_V$, thus $gB = B$;
        
    \medskip
    
        \item $G$ acts independently on each part of the product over planar trees $t$
        as
        $$
        G_t =\left[\sum_{k=0}^{a-1} p_V^{\otimes k} 
        \otimes G(\restrictionone)\otimes
        \id_{\restrictionone}^{\otimes n-1-k}\right]~,
        $$
        where $a$ is the number of the last leaf of the leftest top vertex of $t$ ($1$ if it is the first leaf);
        
  	\medskip
  	
        \item similarly $H$ acts independently on each part of the product over planar trees $t$
        as
        $$
        H_t = \left[\id, \sum_{k=0}^{a-1} p_V^{\otimes k} 
        \otimes H(\restrictionone)\otimes
        \id_{\restrictionone}^{\otimes n-1-k}\right]~,
        $$
        where again $a$ is the number of the last leaf of the leftest top vertex of $t$;
        
   	\medskip
   	
        \item $D_{in}$ is given by the formula
        $$
        \shuffle(\id, D_{in}(\restrictionone))^{\gr \operad P}
        -
        \left[d_C + d_\Delta , \id \right]~.
        $$
    \end{enumerate}
    
    \medskip
    
    For every natural integer $n$,
    $\gr_n \operad P_\pl$ is made up planar trees labelled 
    by $s^{-1}\gr^{\qp}\overline{\operad C_\pl}$ such that the sum
    of the grading degree over the whole tree is $n$; in particular, such a tree
    has at most $n$ nodes. One can furthermore filter $\gr_n \operad P_\pl$
    by the opposite of the number of nodes in trees. On the associated graded object
    $\gr \gr_n \operad P$, the pre-differential $d_\Delta$ disappears.
    
    \medskip

    This filtration of $\gr_n \operad P$ induces a cofiltration of $\gr^n \restrictiontwo$ 
    whose associated graded object is given by
    \begin{align*}
    \gr^{-m} \gr^n \restrictiontwo 
    &= \restrictionone^{\gr_{-m} \gr_n \operad P}    \\
    &\cong \prod_t 
    \bigoplus_{i_1 + \cdots + i_m=n} \left[ s^{-1} \gr^{\mathrm{\qp}}_{i_1}\overline{\operad C}_\pl(n_1(t))\otimes  \cdots \otimes s^{-1} \gr^{\qp}_{i_m}\overline{\operad C}_\pl(n_m(t)) , \restrictionone^{\otimes l(t)}\right]~,
    \end{align*}
where $0 \leq m \leq n$ and where the product is taken over the planar trees $t$ with $m$-nodes whose number of inputs are $n_1(t), \ldots, n_m(t)$. Let us denote by $l(t)$ the number of leaves of $t$.

\medskip
    
    On this a multi-graded object, $D_{in}$, $H$ and $G$ are still given by the same formula on each tree as above;
    but now, $[d_\Delta, \id]$ is zero. Besides, $[d_C, \id]$ acts independently on each tree and thus commutes with $H$.
    For every planar tree with $m$-nodes $t$, let us denote $D_{in,t}$ be the degree $-1$ endomorphism of the $t$ part of the above product
    that is given by the formula
    $$
    D_{in,t} = \left[\id, \sum_{k=0}^{l(t)-1}\id^{k} \otimes D_{in}(\restrictionone) \otimes \id^{\otimes l(t)-k-1} \right]~.
    $$
    Since $[d_{\C}, \id]$ commutes with $H$ on $\gr^n \extensiontwo$, one has on $\gr^{-m} \gr^n \restrictiontwo $ that
    $$
    D_{in} H + H D_{in}=\prod_t D_{in, t} H_t + H_t D_{in,t}~.
    $$
    
    Let us prove that for every planar tree $t$, the map $\id + G_t + D_{in, t} H_t + H_t D_{in,t}$
    is an isomorphism of the $t$ part of $\gr^m \gr^n \restrictiontwo $.
    Again we denote $l(t)$ the number of leaves of $t$ and $a$
    the number of leaves of $t$ that are before the leftest top node or on top of its node; thus $l(t)=a+b$ where $b$ is the
    number of leaves after those of the leftest top node. The case $a=0$ is clear as then $\id + G_t + D_{in, t} H_t + H_t D_{in,t}= \id$.
    Let us tackle the case $a > 0$.
    We can filter $\restrictionone^{\otimes l(t)}$ in a finite way as follows
    \begin{align*}
        F_0 \restrictionone^{\otimes l(t)} &= \overline{\restrictionone} \otimes \restrictionone^{l(t)- 1}
        \\
        F_1 \restrictionone^{\otimes l(t)} &= (F_0 \restrictionone^{\otimes l(t)}) \oplus 
        V \otimes \overline{\restrictionone} \otimes \restrictionone^{l(t)- 2}
        \\
        &\cdots
        \\
        F_a \restrictionone^{\otimes l(t)} &=
        F_{a-1} \restrictionone^{\otimes l(t)}  \oplus V^{\otimes a} \otimes \restrictionone^{l(t)- a}.
    \end{align*}
    The associated graded object is
    \begin{align*}
        \gr_k \restrictionone^{\otimes l(t)} &=V^{\otimes k} \otimes \overline{\restrictionone} \otimes \restrictionone^{l(t)-k -1},
        \quad 0 \leq k \leq a-1~,
        \\
        \gr_a& = V^{\otimes a} \otimes \restrictionone^{l(t)- a}.
    \end{align*}
    Such a filtration induces a filtration on the $t$ part of 
    of $\gr^{-m} \gr^n \restrictiontwo $ that is preserved 
    by $D_{in,t}, G_t, H_t$. On the $k$-th layer of the associated graded object
    $$
    \gr_k(\id + G_t + D_{in, t} H_t + H_t D_{in,t})
    = \left[\id, \id_{\restrictionone}^{\otimes k} \otimes gB(\restrictionone) \otimes \id_{\restrictionone}^{\otimes l(t)-k-1}\right]~,
    $$
    if $k<a$ and 
    $$
    \gr_a(\id + G_t + D_{in, t} H_t + H_t D_{in,t}) = \id~,
    $$
    otherwise. In both case, it is an isomorphism. Thus $\gr(\id + G_t + D_{in, t} H_t + H_t D_{in,t})$ is an isomorphism
    and so is $\id + G_t + D_{in, t} H_t + H_t D_{in,t}$. Subsequently
    $B$ is an isomorphism on $\gr^{-m} \gr^n \restrictiontwo $
    for every $n \in \mathbb N, m \leq n$. Therefore $B$ is an automorphism of $\restrictiontwo$.
\end{proof}


\subsection{Cofibrant objects} Cofibrant complete curved $\C$-algebras admit a simple description: they are exactly those whose underlying graded $\C$-algebra is free.

\begin{lemma}\label{lemmacosiftedlimits}
The functor $\widehat{\Omega}_{\operad C}$ from dg $\Omega \operad C$-coalgebras to curved $\operad C$-algebras commutes with finite cosifted limits.
\end{lemma}

\begin{proof}
This follows from the fact that cosifted limits in dg $\Omega\operad C$-coalgebras and in curved $\operad C$-algebras are computed in graded $\kk$-modules and that the endofunctor of graded $\kk$-modules $(-)^{\operad C}$ preserves cosifted limits.
\end{proof}

\begin{remark}
From the above fact one can deduce that the adjunction $\widehat{\Omega}_{\operad C} \dashv \widehat{\mathrm{B}}_{\operad C}$ is bimonadic. See the upcoming article \cite{augmentedmonadsandbimonadicdjunctions} for more details.
\end{remark}

\begin{proposition}
Let $\L$ be a qp-complete curved $\operad C$-algebra. The following assertions are equivalent.

\medskip

\begin{enumerate}
    \item $\L$ is cofibrant;\label{assertioncofibrant}
    
\medskip

    \item $\L$ is in the essential image of the functor $\widehat{\Omega}_{\operad C}$; \label{assertionimage}
    
\medskip

    \item $\L$ is a quasi-free $\operad C$-algebra that is, its underlying graded algebra is free.\label{assertionqf}
\end{enumerate}    
\end{proposition}

\begin{proof}
For every free graded $\operad C$-algebra $\L \cong X^{\operad C}$, the data of a degree $-1$ derivation is equivalent to the data of a degree $-1$ map $X \longrightarrow X^{\operad C}$ by restricting it to the generators $X$. In turn, any map $X \longrightarrow X^{\operad C}$ is equivalent to a  degree $-1$ map $\operad C \longrightarrow \mathrm{coEnd}(X)$. Therefore the following data are equivalent.

\medskip

\begin{enumerate}
    \item A degree $-1$ derivation $d$ of $\L$.
    
\medskip
      
    \item A morphism of graded operads $\varphi_d: \Omega \C \longrightarrow \mathrm{coEnd}(X)$, which induces a graded $\Omega \C$-coalgebra structure on $X$. 
\end{enumerate}

\medskip

A straightforward computation then shows that the derivation $d$ is \textit{curved}, that is, $d^2 = -\gamma_\L(\L^{\theta})$, if and only if $\varphi_d$ is a morphism of dg operads. This proves the equivalence between assertion \ref{assertionimage} and assertion \ref{assertionqf}.

\medskip

Furthermore, we already know that a qp-complete curved $\operad C$-algebra that is in the essential image of the functor $\widehat{\Omega}_{\operad C}$ is cofibrant.

\medskip

Now, let $\L$ be a cofibrant object in qp-complete curved $\operad C$-algebras. Let us prove that it is quasi-free. Since $\L$ is cofibrant, the counit map $\epsilon_B: \widehat{\Omega}_{\operad C} \widehat{\mathrm{B}}_{\operad C}\L \longrightarrow \L$ which is an acyclic fibration has a section $s$. Let $V$ be the equaliser of the coreflexive pair of maps
    $$
    \begin{tikzcd}
    V \ar[r] 
    &\widehat{\mathrm{B}}_{\operad C}\L
    \ar[rr, bend left, "\widehat{\mathrm{B}}_{\operad C}(s)"]
    \ar[rr, bend right, "\eta_{\widehat{\mathrm{B}}_{\operad C}\L}"]
    &&\widehat{\mathrm{B}}_{\operad C}\widehat{\Omega}_{\operad C}\widehat{\mathrm{B}}_{\operad C}\L
    \ar[ll]
    \end{tikzcd}
    $$
    in the category of dg $\Omega \C$-coalgebras. By Lemma \ref{lemmacosiftedlimits}, the limit of the diagram of curved $\operad C$-algebras
    \[
    \begin{tikzcd}
   	\widehat{\Omega}_{\operad C}\widehat{\mathrm{B}}_{\operad C}\L
    \ar[rr, bend left, "\widehat{\Omega}_{\operad C}\widehat{\mathrm{B}}_{\operad C}(s)"]
    \ar[rr, bend right, "\widehat{\Omega}_{\operad C}(\eta_{\widehat{\mathrm{B}}_{\operad C}\L})"]
    &&\widehat{\Omega}_{\operad C}\widehat{\mathrm{B}}_{\operad C}\widehat{\Omega}_{\operad C}\widehat{\mathrm{B}}_{\operad C}\L
    \end{tikzcd}
    \]
    is $\widehat{\Omega}_{\operad C} V$. It is also clear that this limit is $\L$. So one has a canonical isomorphism $\L \cong \widehat{\Omega}_{\operad C} V$.
\end{proof}


\subsection{Infinity morphisms}
The notion of $\infty$-morphism extends the usual notion of morphisms of dg $\Omega \C$-coalgebras. Their main advantage is that $\infty$-quasi-isomorphisms are invertible, therefore one can replace a zig-zag of quasi-isomorphisms of dg $\Omega \C$-coalgebras with two inverse $\infty$-quasi-isomorphism. This provides a powerful tool to describe the homotopy category of dg $\Omega \C$-coalgebras.

\medskip

Recall that for every dg $\Omega \C$-coalgebra $V$, the unit $\eta_V: V \longrightarrow \widehat{\mathrm{B}}_{\operad C}\widehat{\Omega}_{\operad C} V$ admits a left-inverse $\xi_V: \widehat{\mathrm{B}}_{\operad C}\widehat{\Omega}_{\operad C} V \longrightarrow V$ in the category of dg modules. Let $K$ be the kernel of this map, we get a decomposition of dg modules
$$
\widehat{\mathrm{B}}_{\operad C}\widehat{\Omega}_{\operad C} V = V \oplus K.
$$

\begin{definition}[$\infty$-morphism]
Let $V, V'$ be two dg $\Omega \C$-coalgebras. An $\infty$\textit{-morphism} $f: V \rightsquigarrow V'$ amounts to the data of, equivalently:

\medskip

\begin{enumerate}
\item a morphism $f : V \longrightarrow \widehat{\mathrm{B}}_{\operad C}\widehat{\Omega}_{\operad C} V'$ of dg $\Omega \C$-coalgebras,

\medskip

\item a morphism $f^\dagger: \widehat{\Omega}_{\operad C} V\longrightarrow \widehat{\Omega}_{\operad C} V'$ of qp-complete curved $\C$-algebras.
\end{enumerate}
\end{definition}

\textbf{Linear part.} Let $f: V \rightsquigarrow V'$ be an $\infty$-morphism of dg $\Omega \C$-coalgebras. Its \textit{linear part} $f_{\mathrm{dg}}$ is the morphism of dg modules given by the following composition:

\[
\begin{tikzcd}[column sep=3pc,row sep=3pc]
V \arrow[r,"f"]
&\widehat{\mathrm{B}}_{\operad C}\widehat{\Omega}_{\operad C} V' \arrow[r,"\xi_V"]
&V'
\end{tikzcd}
\]

Let us denote $\epsilon: \C \longrightarrow \operad I$ and $\mu: \operad I \longrightarrow \C$ the counit and the coaugmentation of the quasi-planar conilpotent curved cooperad $\C$, respectively. The linear part of $f_{\mathrm{dg}}:V \rightsquigarrow V'$ is equivalently given by 

\[
\begin{tikzcd}[column sep=3pc,row sep=3pc]
V \arrow[r,"(V)^{\epsilon}"]
&\widehat{\Omega}_{\operad C} V \arrow[r,"f^{\dagger}"]
&\widehat{\Omega}_{\operad C} V' \arrow[r,"(V')^{\mu}"]
&V'~.
\end{tikzcd}
\]

\textbf{Homotopy part.} Let $f: V \rightsquigarrow V'$ be an $\infty$-morphism of dg $\Omega \C$-coalgebras. Its \textit{homotopy part} $f_{\mathrm{h}}$ is the morphism of dg modules given by the composition

\[
\begin{tikzcd}[column sep=3pc,row sep=3pc]
f_{\mathrm{h}}: V \arrow[r,"f"]
&\widehat{\mathrm{B}}_{\operad C}\widehat{\Omega}_{\operad C} V'\arrow[r,twoheadrightarrow]
&K~,
\end{tikzcd}
\]
where $K$ is the kernel of the map $\xi_V: \widehat{\mathrm{B}}_{\operad C}\widehat{\Omega}_{\operad C} V \longrightarrow V$. Equivalently, it is given by the composition

\[
\begin{tikzcd}[column sep=3pc,row sep=3pc]
f_{\mathrm{h}}: V \arrow[r,"(V)^{\epsilon}"]
&\widehat{\Omega}_{\operad C} V \arrow[r,"f^{\dagger}"]
&\widehat{\Omega}_{\operad C} V' \arrow[r,twoheadrightarrow]
&L
\end{tikzcd}
\]
where $L$ is the kernel of the map $(\id)^\mu: \widehat{\Omega}_{\operad C} V \twoheadrightarrow V$ induced by the coaugmentation of $\C$.

\begin{definition}[$\infty$-quasi-isomorphism]
An $\infty$-morphism $f:V \rightsquigarrow V'$ is called an $\infty$-\textit{quasi-isomorphism} if its linear part $f_{\mathrm{dg}}$ is a quasi-isomorphism of dg modules.
\end{definition}

\begin{definition}[$\infty$-isotopy]
An $\infty$-morphism $f:V \rightsquigarrow V$ is called an $\infty$-\textit{isotopy} if its linear part $f_{\mathrm{dg}}$ is an identity map of $V$.
\end{definition}

\begin{lemma}\label{lemmainftycofibration}
Let $f^\dagger : \widehat{\Omega}_{\operad C} V\longrightarrow \widehat{\Omega}_{\operad C} V'$ be an $\infty$-morphism. If $f_{\mathrm{dg}}$ is a cofibration, then $f^\dag$ is the composition of a strict cofibration, that is the image through $\widehat{\Omega}_{\operad C}$ of a cofibration between dg $\Omega \C$-coalgebras, followed by an $\infty$-isotopy.
\end{lemma}

\begin{proof}
Let $p: V' \longrightarrow V$ be a left inverse of $f_{\mathrm{dg}}$ in the category of graded $\kk$-modules that is, a map so that $p f_{\mathrm{dg}}  = \id$. We consider the following degree $0$ map $\tau: V' \longrightarrow (V')^{\operad C}$, defined by $\tau_{\mathrm{dg}} \coloneqq \id_{V'}$ and by 

\[
\begin{tikzcd}[column sep=3pc,row sep=3pc]
\tau_{\mathrm{h}}\coloneqq V' \arrow[r,"p"]
&V \arrow[r,"f_{\mathrm{h}}"]
&(V')^{\overline{\operad C}}~.
\end{tikzcd}
\]
    
It induces a morphism $t$ of the graded $\operad C$-algebras

\[
\begin{tikzcd}[column sep=3pc,row sep=3pc]
t: (V')^{\operad C} \arrow[r,"\tau^{\C}"]
&\left((V')^{\operad C}\right)^{\C} \arrow[r,rightarrowtail]
&(V')^{\operad C \circ \C} \arrow[r,"(V)^{\Delta}"]
&(V')^{\operad C}~,
\end{tikzcd}
\]
where $\Delta: \C \longrightarrow \C \circ \C$ is the decomposition of the cooperad $\C$. Notice that $t$ is an isomorphism since $\gr^{\qp}(t) = \id$. We have the equality of morphism of graded $\operad C$-algebras
    $$
    t (\operad C \comp f_{\mathrm{dg}}) = f^\dagger.
    $$
Let us denote $D$ the derivation of $(V')^{\operad C}$ coming from the dg $\Omega\C$-coalgebra structure of $V'$ and let
    $$
    \tilde D \coloneqq t^{-1}Dt.
    $$
This is a derivation of $(V')^{\operad C}$ that makes it a complete curved $\operad C$-algebra. Thus it defines the structure of a dg $\Omega \operad C$-coalgebra on $V'$ such that $t: V' \rightsquigarrow V'$ becomes a $\infty$-isotopy between $(V',\tilde D)$ to $(V', D)$ and thus $\operad C \comp f_{\mathrm{dg}}$ becomes a strict cofibration from $V$ to $(V',\tilde D)$.
\end{proof}

\begin{proposition}
Let $f: V \rightsquigarrow V'$ be an $\infty$-morphism of dg $\Omega \C$-coalgebras. The morphism of qp-complete curved $\operad C$-algebras $f^\dagger: \widehat{\Omega}_{\operad C} V \longrightarrow \widehat{\Omega}_{\operad C} W$ is

\medskip

    \begin{enumerate}
        \item a weak-equivalence if and only if the dg part $f_{\mathrm{dg}}$ is a quasi-isomorphism;
        
\medskip

        \item an isomorphism if and only if the dg part $f_{\mathrm{dg}}$ is an isomorphism;
        
\medskip

        \item a fibration if and only if the dg part $f_{\mathrm{dg}}$ is a degree-wise epimorphism;
        
\medskip

        \item a cofibration if and only if the dg part $f_{\mathrm{dg}}$ is a degree-wise injection.
    \end{enumerate}
\end{proposition}

\begin{proof}
Let us prove these assertions.
   \begin{enumerate}
        \item Let us consider the following commutative diagram
        $$
        \begin{tikzcd}[column sep=3pc,row sep=3pc]
            \widehat{\mathrm{B}}_{\operad C}\widehat{\Omega}_{\operad C} V
            \ar[r,"\simeq"] \ar[d, "\widehat{\mathrm{B}}_{\operad C}(f^{\dagger})",swap]
            & V \ar[d, "f_{\mathrm{dg}}"] 
            \\
            \widehat{\mathrm{B}}_{\operad C}\widehat{\Omega}_{\operad C} V'
            \ar[r,"\simeq"]
            & V'.
        \end{tikzcd}
        $$
        in the category of dg modules. The two horizontal maps, $\xi_V$ and $\xi_{V'}$, are quasi-isomorphisms. Thus, the left vertical map is a quasi-isomorphism if and only if the right vertical map is a quasi-isomorphism.
        
        \medskip

        \item If $f^{\dagger}$ is an isomorphism, then $f_{\mathrm{dg}} = \gr_0^{\qp}(f^{\dagger})$ is an isomorphism. Conversely, if $f_{\mathrm{dg}}$ is an isomorphism, then 
        $$
        \gr_n^{\qp}(f^{\dagger}) = (f_{\mathrm{dg}})^{\id_{\gr_n^\qp(\operad C)}}
        $$
        is an isomorphism. Hence, $f^{\dagger}$ is an isomorphism.
        
        \medskip
        
        \item Let us suppose that $f_{\mathrm{dg}}$ is a degree-wise epimorphism. Then, it has a section $g$ in the category of graded $\kk$-modules. Let us consider the endomorphism $h$ of graded $\operad C$-algebras
        
\[
\begin{tikzcd}[column sep=3pc,row sep=3pc]
(V')^{\C} \arrow[r,"(g)^{\id}"]
&(V)^{\C}  \arrow[r,"f^\dagger"]
&(V')^{\C}~.
\end{tikzcd}
\]

Its linear part is the identity of $V'$. The same arguments as those used to prove point (2) show that $h$ is a graded isomorphism. In particular, it is a degree-wise epimorphism. So $f^\dagger$ is a degree-wise epimorphism, hence a fibration. 

\medskip

Conversely, suppose that $f^\dagger$ is a fibration. Let us consider the same commutative square diagram shown in point (1). The right vertical map $\widehat{\mathrm{B}}_{\operad C}(f^{\dagger})$ and the bottom horizontal map of this square are degree-wise epimorphisms. Thus the right vertical map $f_{\mathrm{dg}}$ is also a degree-wise epimorphism. 

       	\medskip
       	
        \item One has a functor from dg modules to qp-complete curved $\operad C$-algebras that sends a dg module $X$ to $X$ equipped with the trivial complete curved $\operad C$-algebra structure. This structure is given by the zero structural map $0: X^{\overline{\operad C}} \longrightarrow X$.
        
        \medskip
        
       	This functor sends acyclic fibrations of dg modules to filtered quasi-isomorphisms that are also degree-wise epimorphisms. They are in particular acyclic fibrations of complete curved $\operad C$-algebras. If $f^\dagger$ is a cofibration, then it has the left lifting property with respect to all acyclic fibrations, and in particular with respect to such acyclic fibrations of dg modules.
        As a consequence of the fact that a square of qp-complete curved $\C$-algebras
        $$
        \begin{tikzcd}[column sep=3pc,row sep=3pc]
            \widehat{\Omega}_{\operad C} V
            \ar[r] \ar[d, "f^{\dagger}",swap]
            & X \ar[d, two heads] 
            \\
            \widehat{\Omega}_{\operad C} V'
            \ar[r]
            & Y
        \end{tikzcd}
        $$
        (where $X,Y$ are chain complexes equipped with the zero structure of a qp-complete curved $\C$-algebra)
        amounts to the data of a square chain complexes 
        $$
        \begin{tikzcd}[column sep=3pc,row sep=3pc]
            V
            \ar[r] \ar[d, "f_{\mathrm{dg}}",swap]
            & X \ar[d, two heads] 
            \\
            V'
            \ar[r]
            & Y
        \end{tikzcd}
        $$
        the map $f_{\mathrm{dg}}$ has the left lifting property with respect to every acyclic fibration of dg modules. So it is a cofibration of dg modules that is, a degree-wise injection. Conversely, if $f_{\mathrm{dg}}$ is a cofibration, then $f^\dagger$ is a cofibration as a direct consequence of Lemma \ref{lemmainftycofibration}.      	
\end{enumerate}
\end{proof}

\begin{proposition}
Let $V$ and $V'$ be two dg $\Omega \C$-coalgebras. There exists a zig-zag of quasi-isomorphisms of dg $\Omega \C$-coalgebras
\[
V \lqi \cdot \qi \cdot \lqi \cdots \qi \cdot \lqi V'~,
\]

if and only if there exist $\infty$-quasi-isomorphisms $V \rightsquigarrow V'$ and $V' \rightsquigarrow V$ which are inverse to each other in homology. 
\end{proposition}

\begin{proof}
Given any zig-zag of quasi-isomorphisms between $V$ and $V'$ we get a zig-zag of weak-equivalences between $\widehat{\Omega}_{\operad C}V$ and $\widehat{\Omega}_{\operad C}V'$. All the objects in the zig-zag are fibrant and cofibrant, therefore they admit homotopy inverses, and we have two homotopy inverse arrows between $\widehat{\Omega}_{\operad C} V$ and $\widehat{\Omega}_{\operad C} V'$. Conversely, if there is a $\infty$-quasi-isomorphism $f^{\dagger}: \widehat{\Omega}_{\operad C} V \qi \widehat{\Omega}_{\operad C} V'$, we get the following zig-zag
\[
\begin{tikzcd}[column sep=3pc,row sep=3pc]
V 
&\widehat{\mathrm{B}}_{\operad C} \widehat{\Omega}_{\operad C} V \arrow[r,"\widehat{\mathrm{B}}_{\operad C}(f^\dagger)"] \arrow[l,"\simeq",swap]
&\widehat{\mathrm{B}}_{\operad C} \widehat{\Omega}_{\operad C} V' \arrow[r,"\simeq"]
&V'~.
\end{tikzcd}
\]
\end{proof}


\subsection{Homotopy transfer theorem for coalgebras}
The goal of this subsection is to prove a homotopy transfer theorem for dg coalgebras over a cofibrant dg operad. The idea is the same as in the algebra case. First, we show that given a dg operad $\operad P$ and a dg $\operad P$-coalgebra $V$ together with a decomposition $V \cong K \oplus X$, where $K$ is acyclic, we can "perturb" the dg $\Omega \mathrm{B}\operad P$-coalgebra structure on $V$ so that this new structure "projects" onto $X$. Then, we give an homotopical meaning to this construction in cases of interest.

\subsubsection{Technicalities}
Let $\operad P$ be a dg operad. As in Section \ref{sectionhomotopytransfertheoremalgebra}, we denote $M$ the graded $\mathbb S$-module $\operad P \oplus s \operad I$:
$$
M = \operad P \oplus s \operad I.
$$
Again, we have a canonical isomorphism of conilpotent graded cooperads $\mathrm{B} \operad P = \treemod(sM)$.

\begin{notation}
Let $Y$ and $Z$ be graded modules and let 
\[
a: Y \longrightarrow Z^{{\treemod}(sM)}
\]
be a graded map. For every natural integer $n \geq 1$, we will denote by $a_{\leq n}$ the composition 
\[
a_{\leq n}: Y \xrightarrow{a} Z^{{\treemod}(sM)} \twoheadrightarrow Z^{{\treemod}_{\leq n}(sM)}~.
\]
We will also denote $\overline{a}$ the composition of $a$ with the projection $Z^{{\treemod}(sM)} \twoheadrightarrow Z^{\overline{\treemod}(sM)}$. We can combine these notations. For instance, $\overline{a}_{\leq n}$ refers to the composition of $a$ with the projection $Z^{{\treemod}(sM)} \twoheadrightarrow Z^{\overline{\treemod}_{\leq n}(sM)}$.
\end{notation}

Let us also consider a direct sum of dg modules
$$
V = X \oplus K
$$
where $K$ is acyclic. Let us choose a contracting homotopy $h$ of $K$ that we can extend to all $V$ by $0$ on $X$. Let us denote $\pi_X, \pi_K$ the projections endomorphisms of $V$ onto respectively $X$ and $K$. All these arrows satisfy the following equations:

\[
\left\{
\begin{tikzcd}[column sep=0pc,row sep=0.25pc]
&\pi_X \pi_K = \pi_K \pi_X =0;\\
&\partial(h) = \pi_K = \id_V - \pi_X;\\
&h \pi_K = \pi_K h = h;\\
&h \pi_X = \pi_X h = 0.
\end{tikzcd}
\right.
\]

Let us suppose that the dg module $(V,d_V)$ is endowed with a dg $\operad P$-coalgebra structure 
$$
\Delta_V: V \longrightarrow V^{\operad P}.
$$
Notice that it induces a dg $\Omega \mathrm{B} \operad P$-coalgebra structure on $V$ by pulling back along the map $\Omega \mathrm{B} \operad P \longrightarrow \operad P$.

\begin{definition}
Let $D_\delta$ be the degree $-1$ derivation on the graded $\mathrm{B}\operad P$-algebra $V^{\mathrm{B}\operad P}$ that induced by the $\Omega \mathrm{B} \operad P$-coalgebra structure $\Delta_V$ on the dg module $V$, in the sense that
$$
(V^{\mathrm{B}\operad P}, D_\delta) \cong \widehat{\Omega}_{\mathrm{B}\operad P} V.
$$
Moreover, let $\delta$ be the composition of $D_\delta$ with the inclusion of $V$ inside $V^{\mathrm{B}\operad P}$. One can notice that the projection of $\delta$ onto $V$ is the differential $d_V$ and that its projection onto $V^{\overline{\mathrm{B}\operad P}}$ is given by
    $$
    V \xrightarrow{\Delta} V^{\operad P} \xrightarrow{- V^{s}}
    V^{s\operad P} \hookrightarrow V^{sM} \hookrightarrow V^{\overline{\treemod}(sM)} = V^{\overline{\mathrm{B}\operad P}} ~,
    $$
    where the second map $- V^{s}$ sends a sequence to a sequence
    $$
    (\phi_n)_{n \in \mathbb N} \in \prod_n [\operad P(n), V^{\otimes n}]^{\mathbb S_n}
    \mapsto
    (\psi_n)_{n \in \mathbb N} \in \prod_n [s\operad P(n), V^{\otimes n}]^{\mathbb S_n}~,
    $$
    and where each $\psi_n$ is defined as
    $$
    s\operad P(n) \xrightarrow{sx \mapsto -x} \operad P(n) \xrightarrow{\phi_n}
    V^{\otimes n}.
    $$
\end{definition}

\medskip

The degree $0$ map
    $$
    V \xrightarrow{-h} V 
     \xrightarrow{\overline{\delta}_{\leq 1}} V^{sM}
    $$
yields a morphism of graded $\mathbb S$-modules $sM \longrightarrow \mathrm{coEnd}(V)$. Therefore there is a morphism of graded operads $\treemod(sM) \longrightarrow \mathrm{coEnd}(V)$, which in turn yields a degree $0$ map of graded $\kk$-modules
    $$
    \phi : V \longrightarrow V^{\treemod{sM}} = V^{\mathrm{B}\operad P}.
    $$
    Let $f$ be the related morphism of qp-complete graded $B \operad P$-algebras
    $$
    f : V^{\mathrm{B}\operad P} \xrightarrow{\phi^{\mathrm{B}\operad P}} (V^{\mathrm{B}\operad P})^{\mathrm{B}\operad P} \hookrightarrow V^{\mathrm{B}\operad P \comp \mathrm{B}\operad P} \longrightarrow V^{\mathrm{B}\operad P}.
    $$
    One can notice that the composition of $\phi$ with the projection onto $V$ is the identity and that
    $$
    \overline{\phi} = f \overline{\phi}_{\leq 1}.
    $$

\begin{definition}
Let $\chi$ be the degree $-1$ morphism from $V$ to $V^{\treemod(sM)}$ 
defined as $\chi \coloneqq f \overline{\delta}$. Notice that $\phi = \id_V -\chi h$.
\end{definition}

Recall that the conilpotent curved cooperad $\mathrm{B}\operad P$ is given by the conilpotent graded cooperad $\treemod(sM)$ endowed with the following pre-differentials:

\begin{enumerate}
\item the pre-differential $d_\gamma$ which is induced by the operad structure of $\operad P$;
\item the pre-differential $d_{\operad P}$ which is induced by the differential of $\operad P$;
\item the pre-differential $d_{u}$, which maps $s^2\operad I$ to the unit of $\operad P$.
\end{enumerate}

We will denote the later two pre-differentials by $d_{sM}$, as they are defined on the generators of $\mathrm{B}\operad P$. We refer to Subsection \ref{subsection: operadic bar-cobar} for more details.

\begin{notation}
We denote by $d_{\gamma}^{\mathrm{root}}$ the degree $-1$ endomorphism of $\overline{\treemod} (sM) \hookrightarrow \mathrm{B}\operad P$ given by applying $d_{\gamma}$ only to inner edges of the trees of $\overline{\treemod}(sM)$ that touch the root node. We denote by $d_{\gamma}^{\neq\mathrm{root}}$ the other component of $d_\gamma$, given by applying $d_{\gamma}$ to all the other edges that do not touch the root.

\medskip

Similarly, we denote by $d_{sM}^{\mathrm{root}}$ and $d_{sM}^{\neq\mathrm{root}}$ the degree $-1$ endomorphisms given by applying $d_{sM}$ to the root node of a tree (resp. all the other nodes). 
\end{notation}

\begin{lemma}\label{lemmahttcoalgebratwoleveltrees}
    The following diagram in the category of graded $\kk$-modules
    $$
    \begin{tikzcd}[column sep=3pc,row sep=3pc]
        V 
        \ar[r, "\overline{\phi}_{\leq 1}"] \ar[d,"\overline{\phi}"']
        & V^{sM}
        \ar[rr, "{{\shuffle (\phi, \chi)^{sM}}}"]
        &&  \left(V^{\treemod_{\leq 1} (sM)}\right)^{sM}
        \ar[d, hook]
        \\
        V^{\overline{\treemod}(sM)}
     \ar[rrr, "V^{d_{\gamma}^{\mathrm{root}}}"']
        &&& V^{\overline{\treemod} (sM)}.
    \end{tikzcd}
    $$
    is commutative.   
\end{lemma}

\begin{proof}
A straightforward check shows these two maps from $V$ to $V^{\overline{\treemod}_{\leq 2} (sM)}$ correspond to the same $\mathbb S_n$-equivariant maps from $V \otimes \overline{\treemod}_{\leq 2} (sM)(n)$ to $V^{\otimes n}$.
\end{proof}

\begin{lemma}\label{lemmahttbfdecompositioncoalgebra}
The map $ f{\delta} =  f{\delta}_{\leq 1}: V \longrightarrow V^{{\treemod}(sM)}$ is equal to $(\phi ~ d_V)+ {\chi}$.
\end{lemma}

\begin{proof}
This follows from the equation $ \delta = {\delta}_{\leq 1} = d_V + \overline{\delta}_{\leq 1}$.
\end{proof}

\begin{definition}
    Let $\zeta$ be the degree $-1$ map from $V$ to $V^{B\operad P}$
    whose projection onto to $V \cong V^{\operad I}$ is $d_V$ and whose projection onto 
    $V^{\overline{\mathrm{B}\operad P}} $ is the sum of the two maps
    \begin{align*}
    &V \xrightarrow{\pi_X} V
        \xrightarrow{\chi} V^{\overline{\mathrm{B}\operad P}};
        \\
    & V \xrightarrow{h} V \xrightarrow{V^\theta} V^{sM} \xrightarrow{f} V^{\overline{\treemod}(sM)}.
    \end{align*}
    Moreover, let $D_\zeta$ be the unique degree $-1$ derivation on 
    the graded $\mathrm{B}\operad P$-algebra $V^{{\mathrm{B}\operad P}}$ whose restriction to $V$ is $\zeta$.
\end{definition}

\begin{lemma}\label{lemmahttsumdecompositioncoalgebra}
    The degree $-1$ map $f  \delta - \zeta: V \longrightarrow V^{{\treemod}(sM)}$
    is equal to the sum of the two maps
    \begin{align*}
    &V \xrightarrow{\overline{\phi}_{\leq 1}}V^{sM} \xrightarrow{-V^{d_{sM}}} V^{sM}  \xrightarrow{f} V^{\overline{\treemod}(sM)}
    \\
    & V \xrightarrow{\overline{\phi}_{\leq 1}}V^{sM} \xrightarrow{\shuffle(\id, d_V)^{sM}} V^{sM}
    \xrightarrow{f} V^{\overline{\treemod}(sM)}
\end{align*}
where $d_{sM}$ denotes the restriction of the coderivation of $\mathrm{B}\operad P$ to $sM$.
\end{lemma}

\begin{proof}
    The projection of the two maps onto $V$ are both equal to $0$. Let us prove that their projection onto $V^{\overline{\treemod}(sM)}$ are equal.
    Let us first notice that
    $$
    \overline{f}d_V = \overline{\phi}d_V = \overline{f} \overline{\phi}_{\leq 1}d_V
    = - \overline{f} \overline{\delta}_{\leq 1} h d_V
    $$
    Then, the first map rewrites as
    \begin{align*}
        \overline{f}\delta - \overline{\zeta}
    &= \overline{f}  (d_V  + \overline{\delta}_{\leq 1} - \overline{\delta}_{\leq 1} \pi_X - V^\theta h) 
    \\
    &= \overline{f} \overline{\delta}_{\leq 1} ( - h d_V  + \id - \pi_X ) -\overline{f} V^\theta h
    \\
    &= \overline{f}(\overline{\delta}_{\leq 1} d_Vh - V^\theta h).
    \end{align*}
    Noticing that
    $$
    \overline{\delta}_{\leq 1} d_V = \overline{(D_\delta)}_{\leq 1} \delta
    - \left(\shuffle(\id, d_V)^{sM} - V^{d_{sM}}\right)\overline{\delta}_{\leq 1}
    = V^{\theta} - \shuffle(\id, d_V)^{sM}\overline{\delta}_{\leq 1} + V^{d_{sM}}\overline{\delta}_{\leq 1}
    $$
    we get
    \begin{align*}
        \overline{f}\delta - \overline{\zeta}
    &=\overline{f}( V^{\theta} -  \shuffle(\id, d_V)^{sM} \overline{\delta}_{\leq 1} + V^{d_{sM}} \overline{\delta}_{\leq 1} - V^\theta) h
    \\
    &= \overline{f}(\shuffle(\id, d_V)^{sM} \overline{\phi}_{\leq 1}
    - V^{d_{sM}} \overline{\phi}_{\leq 1}).
    \end{align*}
\end{proof}

\begin{proposition}
    One has an equality of degree $-1$ maps from $V$ to $V^{B\operad P}$
    $$
    D_{\zeta} \phi = f \delta.
    $$
    Thus $D_{\zeta} f = f D_{\delta}$.
\end{proposition}

\begin{proof}
Let us prove the result on the height of the trees that make $\mathrm{B}\operad P = \treemod(sM)$. More precisely, let us prove that for every natural integer $n$
$$
D_{\zeta, \leq n} \phi = f_{\leq n} \delta.
$$

First, for $n=0$
$$
D_{\zeta, \leq 0} \phi = d_V =  f_{\leq 0} \delta
$$

Let us assume that the equality is verifies for some natural integer $n$. The map $\overline{D_{\zeta, \leq n+1}} \phi$ is given on larger trees $V^{\overline{\treemod}_{\leq n+1}(sM)}$ as the following sum of maps

\begin{align*}
    & V \xrightarrow{\phi_{\leq 0} = \id_V} V \xrightarrow{\overline{\zeta}_{\leq n+1}} V^{\overline{\treemod}_{\leq n+1}(sM)}~,
    \\
    &V \xrightarrow{\overline{\phi}} V^{\overline{\treemod}(sM)}
    \xrightarrow{\shuffle(\id_V, \zeta)^{\overline{\treemod}(sM)}} (V^{{\treemod}(sM)})^{\overline{\treemod}(sM)} \to V^{\overline{\treemod}(sM)} \twoheadrightarrow V^{\overline{\treemod}_{\leq n+1}(sM)}~,
    \\
    &V \xrightarrow{\overline{\phi}} V^{\overline{\treemod}(sM)}
    \xrightarrow{-V^{d_{sM}^{\mathrm{root}}}}V^{\overline{\treemod}(sM)} \twoheadrightarrow V^{\overline{\treemod}_{\leq n+1}(sM)}~,
    \\
    &V \xrightarrow{\overline{\phi}} V^{\overline{\treemod}(sM)}
    \xrightarrow{-V^{d_{sM}^{\neq\mathrm{root}}}}V^{\overline{\treemod}(sM)} \twoheadrightarrow V^{\overline{\treemod}_{\leq n+1}(sM)}~,
    \\
    &V \xrightarrow{\overline{\phi}} V^{\overline{\treemod}(sM)}
    \xrightarrow{-V^{d_{\gamma}^{\mathrm{root}}}} V^{\overline{\treemod}(sM)} \twoheadrightarrow V^{\overline{\treemod}_{\leq n+1}(sM)}~,
    \\
    &V \xrightarrow{\overline{\phi}} V^{\overline{\treemod}(sM)}
    \xrightarrow{-V^{d_{\gamma}^{\neq\mathrm{root}}}} V^{\overline{\treemod}(sM)} \twoheadrightarrow V^{\overline{\treemod}_{\leq n+1}(sM)}~,
\end{align*}
where the four last map are actually equal to the composition
$$
V \xrightarrow{\overline{\phi}} V^{\overline{\mathrm{B} \operad P}}
\xrightarrow{-V^{d_{\mathrm{B}\operad P}}}
V^{\overline{\treemod}(sM)} \twoheadrightarrow V^{\overline{\treemod}_{\leq n+1}(sM)}~.
$$

We can the notice that the first map is just $\overline{\zeta}_{\leq n+1}$ and that the sum of the second map, the fourth map and the sixth map is the composition
$$
V \xrightarrow{\overline{\phi}_{\leq 1}} V^{sM} \xrightarrow{\shuffle (\phi_{\leq n}, D_{\zeta, \leq n} \phi)^{sM}} (V^{\treemod_{\leq n}(sM)})^{sM}.
$$
By the induction hypothesis, $D_{\zeta, \leq n} \phi  = f_{\leq n}\delta$ and by Lemma \ref{lemmahttbfdecompositioncoalgebra} $f\delta = \phi d_V + \chi$. So this composition is equal to the composition
$$
V \xrightarrow{\overline{\phi}_{\leq 1}} V^{sM} \xrightarrow{\shuffle (\phi_{\leq n}, \phi_{\leq n}d_V + \chi_{\leq n})^{sM}} (V^{\treemod_{\leq n}(sM)})^{sM}.
$$
using Lemma \ref{lemmahttcoalgebratwoleveltrees}, its sum with the fifth map is
$$
V \xrightarrow{\overline{\phi}_{\leq 1}} V^{sM} \xrightarrow{\shuffle (\phi_{\leq n}, \phi_{\leq n}d_V)^{sM}} (V^{\treemod_{\leq n}(sM)})^{sM}.
$$
To conclude Lemma \ref{lemmahttsumdecompositioncoalgebra} tells that the sum of all the six maps is equal to $\overline{f}_{\leq n+1} \delta$.
\end{proof}

\begin{proposition}
The derivation $D_\zeta$ makes $V^{\mathrm{B}\operad P}$ a qp-complete curved $\mathrm{B}\operad P$-coalgebra.
\end{proposition}

\begin{proof}
    Since $ D_{\zeta} f = fD_{\delta} f$ and since $f$ is an isomorphism (by a standard filtration argument):
    $$
    D_\zeta = fD_{\delta} f^{-1}.
    $$
    Thus, the derivation $D_\zeta$ makes $V^{\mathrm{B}\operad P}$ a curved $\mathrm{B}\operad P$-algebra because so does the derivation $D_\delta$.
\end{proof}

\begin{proposition}
The derivation $D_\zeta$ projects onto the quotient graded $\mathrm{B} \operad P$-algebra $X^{\mathrm{B} \operad P}$ in the sense that there exists a (necessarily unique) derivation on $X^{\mathrm{B} \operad P}$, also denoted $D_{\zeta}$, such that the projection map
$$
V^{\mathrm{B} \operad P} \twoheadrightarrow X^{\mathrm{B} \operad P}
$$
commutes with the derivations. In particular, $(V^{\mathrm{B}\operad P}, D_\zeta)$ is a curved $\mathrm{B}\operad P$-algebra.
\end{proposition}

\begin{proof}
Actually, $X^{\mathrm{B}\operad P}$ is the quotient/kernel in $\mathrm{B}\operad P$-algebras of the idempotent endomorphism $\pi_X^{\mathrm{B}\operad P}$ on $V^{\mathrm{B}\operad P}$. One can notice that the restriction to $V$ of $\pi_X^{\mathrm{B}\operad P}~D_\zeta ~ \pi_X^{\mathrm{B}\operad P}$ and $\pi_X^{\mathrm{B}\operad P} D_\zeta $ are equal:
$$
\pi_X^{\mathrm{B}\operad P}~\zeta ~ \pi_X = \pi_X^{\mathrm{B}\operad P}~\zeta.
$$
Thus 
$$
\pi_X^{\mathrm{B}\operad P}~D_\zeta ~ \pi_X^{\mathrm{B}\operad P}
= \pi_X^{\mathrm{B}\operad P}~D_\zeta 
$$
which proves the result.
\end{proof}

To conclude, we have a composition of morphisms of qp-complete curved $\mathrm{B}\operad P$-algebras
$$
(V^{\mathrm{B} \operad P}, D_\delta)
\xrightarrow{f} (V^{\mathrm{B} \operad P}, D_\zeta)
\twoheadrightarrow
(X^{\mathrm{B} \operad P}, D_\zeta).
$$

\subsubsection{The cooperad version of the homotopy transfer theorem for coalgebras}

\begin{theorem}\label{theoremhttcooperadcoalgebras}
Let $p : V \twoheadrightarrow X$ be an acyclic fibration of dg modules and let $\Delta_V: V \longrightarrow V^{\Omega \operad C}$ be a dg $\Omega \operad C$-coalgebra structure on $V$. There exists another dg $\Omega \operad C$-coalgebra structure
\[
\zeta_V: V \longrightarrow V^{\Omega \operad C}
\]
which projects onto $X$, together with an $\infty$-isotopy
\[
(V, \Delta_V) \rightsquigarrow (V, \zeta_V)
\]
of dg $\Omega \operad C$-coalgebras.
\end{theorem}

\begin{proof}
Since $p$ is an acylic fibration of dg modules, it has a section $i$ and one can decompose $V$ as $X \oplus K$ where $K$ is the kernel of $p$. The paragraph just above gives us a diagram of qp-complete curved $\mathrm{B}\Omega \operad C$-algebras
$$
 (V^{\mathrm{B} \Omega \operad C}, D_\delta)
\xrightarrow{f} (V^{\mathrm{B} \Omega \operad C}, D_\zeta)
\twoheadrightarrow
(X^{\mathrm{B} \Omega \operad C}, D_\zeta).
$$
Applying the left adjoint functor from qp-complete $\mathrm{B}\Omega\operad C$-algebras to qp-complete curved $\operad C$-algebra that results from the unit map $\operad C \to \mathrm{B}\Omega\operad C$, we get diagram of curved $\operad C$-coalgebras
$$
 (V^{\operad C}, D_\delta)
\xrightarrow{\tilde f} (V^{\operad C}, D_\zeta)
\twoheadrightarrow
(X^{\operad C}, D_\zeta).
$$
In that context, $D_\zeta$ is the derivation on $V^{\operad C}$ that induces the expected dg $\Omega \operad C$-coalgebra structure $\zeta_V$ on $V$ and $\tilde f$ is the expected $\infty$-isotopy. 
\end{proof}

\subsubsection{The homotopy transfer theorem for coalgebras}

\begin{theorem}
Let $\operad Q$ be a cofibrant dg operad, let $p : V \twoheadrightarrow X$ be an acyclic fibration of dg modules and let $\Delta_V: V \longrightarrow V^{\operad Q}$ a dg $\operad Q$-coalgebra structure on $V$. There exists a dg $\operad Q$-coalgebra structure $\zeta_X$ on $X$, together with a zig-zag of quasi-isomorphisms
\[
(V, \Delta_V) \lqi \cdots \qi (X,\zeta_X)
\]
of dg $\operad Q$-coalgebras. Furthermore, the maps in this zig-zag are homotopic to $p$ in the model category of dg modules.
\end{theorem}

\begin{proof}
    Taking $\operad C$ to be the quasi-planar conilpotent dg cooperad $\mathrm{B}(\operad Q \otimes \operad E)$, Theorem \ref{theoremhttcooperadcoalgebras} yields a dg $\Omega \operad C$-coalgebra structure on $X$ together with a zig-zag of quasi-isomorphisms of dg $\Omega \operad C$-coalgebras
    $$
    (V, \Delta_V) \qi  \widehat{\mathrm{B}}_{\operad C}\widehat{\Omega}_{\operad C} (X, \zeta_X) \lqi (X, \zeta_X).
    $$
Moreover, the acyclic fibration of dg operads $\Omega \operad C \twoheadrightarrow \operad Q$ has a section since $\operad Q$ is cofibrant. Applying the related left adjoint forgetful functor from dg $\Omega \operad C$-coalgebras to dg $\operad Q$-coalgebras yields the expected zig-zag.
\end{proof}

\begin{remark}
This last result also follows from model-categorical arguments, as developed in \cite{FressePROP}. 
\end{remark}


\subsection{Further localisations and divided powers operations}
Let $\operad Q$ be an coadmissible dg operad and let $\C$ be a quasi-planar conilpotent curved cooperad. Let us consider a morphism of dg operads $f: \Omega \operad C \longrightarrow \operad Q$. We have two Quillen adjunctions

\[
\begin{tikzcd}[column sep=5pc,row sep=3pc]
          \mathsf{curv}~\C\text{-}\mathsf{alg}^{\mathsf{qp}\text{-}\mathsf{comp}} \arrow[r, shift left=1.1ex, "\widehat{\mathrm{B}}_{\C}"{name=F}]           
         &\mathsf{dg}~ \Omega \operad C \text{-}\mathsf{cog} \arrow[l, shift left=.75ex, "\widehat{\Omega}_{\C}"{name=U}] \arrow[r, shift left=1.1ex, "f_!"{name=A}]
         &\mathsf{dg}~\operad Q\text{-}\mathsf{cog}~. \arrow[l, shift left=.75ex, "f^*"{name=B}] \arrow[phantom, from=F, to=U, , "\dashv" rotate=90]\arrow[phantom, from=A, to=B, , "\dashv" rotate=90]
\end{tikzcd}
\]

Let us denote $\widehat{\Omega}_f$ the composite left adjoint and $\widehat{\mathrm{B}}_f$ the composite right adjoint. 

\begin{proposition}\label{proprightbousfield}
There exists a combinatorial model structure on qp-complete curved $\operad C$-algebras, called the $f$-model structure, transferred from that of dg $\operad Q$-coalgebras, determined by the following sets of morphisms

		\medskip
		
    \begin{enumerate}
        \item the set of $f$-fibrations is given by morphisms $g$ such that $\widehat{\mathrm{B}}_f(g)$ is a fibration,
        
        \medskip
        
        \item the set of $f$-weak-equivalences is given by morphisms $g$ such that $\widehat{\mathrm{B}}_f(g)$ is a weak equivalence.
        
        \medskip
        
        \item the set of $f$-cofibrations is determined by left-lifting property against all acyclic fibrations.
    \end{enumerate}
    
    		\medskip
    Moreover, this is a right Bousfield localisation of the canonical model structure transferred from dg $\Omega \C$-coalgebras. Meaning that the identity functor of qp-complete curved $\C$-algebra, where at the source they are endowed with the canonical model structure, and at the target with the $f$-model structure, is a right Quillen functor. 
\end{proposition}

\begin{proof}
Any fibration and weak-equivalence in model structure transferred from dg $\Omega \C$-coalgebras is an $f$-fibration and an $f$-weak-equivalence. Hence, every object is fibrant and a natural path object is given by Proposition \ref{proppath}. This proves the existence of the transferred model structure. To prove that this is a right Bousfield localisation of that transferred from dg $\Omega \operad C$-coalgebras, it suffices to notice that fibrations are in particular degree-wise epimorphisms, in a similar way as shown in Propostion \ref{propfibration}.
\end{proof}

\textbf{Localizing at quasi-isomorphisms.} Let $\C$ be a quasi-planar conilpotent \textit{differential graded} cooperad. The cobar construction $\Omega \C$ is augmented since $\C$ has zero curvature. Let us denote $\nu: \Omega \C \longrightarrow \operad I$ the canonical morphism of dg operads given by the augmentation. We have the following adjunctions 

\[
\begin{tikzcd}[column sep=5pc,row sep=3pc]
          \mathsf{dg}~\C\text{-}\mathsf{alg}^{\mathsf{qp}\text{-}\mathsf{comp}} \arrow[r, shift left=1.1ex, "\widehat{\mathrm{B}}_{\C}"{name=F}]           
         &\mathsf{dg}~ \Omega \operad C \text{-}\mathsf{cog} \arrow[l, shift left=.75ex, "\widehat{\Omega}_{\C}"{name=U}] \arrow[r, shift left=1.1ex, "\mathrm{Prim}"{name=A}]
         &\mathsf{dg}~\text{-}\mathsf{mod}~, \arrow[l, shift left=.75ex, "\mathrm{Triv}"{name=B}] \arrow[phantom, from=F, to=U, , "\dashv" rotate=90]\arrow[phantom, from=A, to=B, , "\dashv" rotate=90]
\end{tikzcd}
\]

where the adjunction $\nu^* \dashv \nu_! $ is in fact given by the primitive elements functor $\mathrm{Prim}$ (which is $\nu_!$) and by the trivial structure functor $\mathrm{Triv}$ (which is $\nu^*$). Notice that since $\C$ has zero curvature, curved $\C$-algebras in pdg modules are precisely given by dg $\C$-algebras.

\begin{proposition}\label{prop: prim-triv weak equivalences are quasi-isos}
The set of $\nu$-weak-equivalences is precisely the set of quasi-isomorphisms of qp-complete dg $\C$-algebras.
\end{proposition}

\begin{proof}
The composition $\mathrm{Prim}~\widehat{\mathrm{B}}_{\C}$ is isomorphic to the forgetful functor from qp-complete dg $\C$-algebras to dg modules.
\end{proof}

\begin{corollary}
Let $\C$ be a quasi-planar conilpotent dg cooperad. The set of weak-equivalences in the canonical model structure on dg $\C$-algebras is contained in the set of quasi-isomorphims.
\end{corollary}

\begin{proof}
It suffices to apply Proposition \ref{proprightbousfield} to the morphism of dg operads $\nu: \Omega \operad C \longrightarrow \operad I$, combining it with Proposition \ref{prop: prim-triv weak equivalences are quasi-isos}.
\end{proof}

\textbf{Divided power operations in the homotopical setting.}
Let $\C$ be a quasi-planar conilpotent dg cooperad. By Proposition \ref{proprightbousfield}, the category of qp-complete dg $\C$-coalgebras admits a model category structure where 

\medskip
\begin{enumerate}
\item the set of fibrations is given by degree-wise epimorphisms;

\medskip

\item the set of weak-equivalences is given by quasi-isomorphisms;

\medskip

\item the set of cofibrations is given by maps with left lifting property with respect to acyclic fibrations.
\end{enumerate}

Let $(\L,\gamma_\L,d_\L)$ be a qp-complete dg $\C$-algebra. The structural map 

\[
    \gamma_\L: \prod_{n \geq 0} [\C(n), \L^{\otimes n}]^{\mathbb S_n} \longrightarrow \L~,
    \]
    \vspace{0.1pc}
    
comes from the invariants on the left-hand side, therefore divided power operations should appear. Nevertheless, since $\C$ is quasi-planar, there is a natural isomorphism

\[
\prod_{n \geq 0} [\C(n), \L^{\otimes n}]_{\mathbb S_n} \cong \prod_{n \geq 0} [\C(n), \L^{\otimes n}]^{\mathbb S_n}~,
\]
\vspace{0.1pc}

of dg modules induced by the norm map (Proposition \ref{prop: iso avec la norme}). Therefore no divided power operations appear at the algebraic level.

\medskip

These divided power operations do not disappear at the $\infty$-categorical level. The reason is that $\C(n)$ is a quasi-free dg $\kk[\mathbb{S}_n]$-module, which is furthermore \textit{projective} by Proposition \ref{prop: quasi-planaire implique S-projectif dans le cas dg}. Therefore we have that 

\[
\prod_{n \geq 0} [\C(n), \L^{\otimes n}]_{h\mathbb S_n} \not\simeq \prod_{n \geq 0} [\C(n), \L^{\otimes n}]_{\mathbb S_n} \cong \prod_{n \geq 0} [\C(n), \L^{\otimes n}]^{\mathbb S_n} \simeq \prod_{n \geq 0} [\C(n), \L^{\otimes n}]^{h\mathbb S_n} ~,
\]
\vspace{0.1pc}

where on the upmost left-hand side we consider \textit{homotopy coinvariants} and on the upmost right-hand side we consider \textit{homotopy invariants}. Indeed, the dg $\kk[\mathbb{S}_n]$-module $[\C(n), \L^{\otimes n}]$ is automatically cofibrant in the injective model structure by Proposition \ref{lemmatensored}, since $\C(n)$ is projective. This means that the $\infty$-category of qp-complete dg $\C$-algebras localized at quasi-isomorphisms behaves like an $\infty$-category of algebraic objects with \textit{divided power operations}. 


\newpage

\section{Linear duality}

\vspace{2pc}

The two bar-cobar Quillen adjunctions constructed so far are intertwined by linear duality adjunctions: they form a commuting square of Quillen adjunctions called the duality square. This allows us to show that for any cofibrant dg operad $\operad P$, the $\infty$-category of dg $\operad P$-algebras with degree-wise finite dimensional bounded below (resp. bounded above) homology and the $\infty$-category of dg $\operad P$-coalgebras with degree-wise finite dimensional bounded above (resp. bounded below) homology are equivalent. This section is based on the results of \cite{absolutealgebras}, which we extend to a positive characteristic setting. They play a key role in the companion paper \cite{deuxiemepapier} about formal moduli problems and will extended into the context of mapping coalgebras in \cite{mappingcoalgebrascharp}.

\medskip

\subsection{Lifting linear duality} Linear duality lifts to (co)algebras over an operad and to (co)algebras over a cooperad, always sending types of coalgebras to types of algebras.

\medskip

The linear duality functor
\[
\begin{tikzcd}[column sep=3.5pc,row sep=0.5pc]
\catgrmod{\kk} \arrow[r,"(-)^*"]
&\catgrmod{\kk}^\op\\
X \arrow[r,mapsto]
&X^\ast \coloneqq [X, \kk]~,
\end{tikzcd}
\]

lifts to pdg $\kk$-modules and to dg modules. 

\medskip

\textbf{The Sweedler dual functor.} The linear duality functor lifts to (co)algebras over an operad and admits an adjoint which generalizes the Sweedler dual functor of \cite{Sweedler}.

\begin{lemma}
Let $\operad P$ be a dg operad. The linear dual functor lifts to a functor 
\[
\begin{tikzcd}[column sep=3.5pc,row sep=0.5pc]
\catdgcog{\operad P} \arrow[r,"(-)^*"]
&\catdgalg{\operad P}^\op
\end{tikzcd}
\]
between dg $\operad P$-coalgebras and dg $\operad P$-algebras.
\end{lemma}

\begin{proof}
Let $C$ be a dg $\operad P$-coalgebra. Any $\mu$ in $\operad P$ gives a decomposition map
\[
\Delta_\mu: C \longrightarrow C^{\otimes n}~.
\]
And any such map induces the following composition map
\[
\gamma_\mu: (C^*)^{\otimes n} \rightarrowtail (C^{\otimes n})^{*} \xrightarrow{(\Delta_\mu)^*} C~.
\]
It can be checked that the collection of $\{\gamma_\mu\}$ induces a dg $\operad P$-algebra structure on $C^*$, and that this defines a functor.
\end{proof}

\begin{proposition}
Let $\operad P$ be a dg operad. There is a contravariant adjunction 
\[
\begin{tikzcd}[column sep=5pc,row sep=3pc]
         \catdgcog{\operad P} \arrow[r, shift left=1.1ex, "(-)^\ast"{name=A}]
         &\catdgalg{\operad P}^\op~, \arrow[l, shift left=.75ex, "(-)^\circ"{name=B}] \arrow[phantom, from=A, to=B, , "\dashv" rotate=-90]
\end{tikzcd}
\]
between dg $\Omega \operad C$-coalgebras and dg $\Omega \operad C$-algebras.
\end{proposition}

\begin{proof}
This follows directly from the adjoint lifting theorems of Appendix \ref{Appendix: Adjoint Lifting Theorem}.
\end{proof}

\textbf{Topological dual functor.} The linear dual functor also lifts to (co)algebras over a cooperad. It admits an adjoint, which can be though as a \textit{topological dual} type of functor.

\begin{lemma}
Let $\C$ be a quasi-planar conilpotent curved cooperad. The linear dual functor lifts to a functor 
\[
\begin{tikzcd}[column sep=3.5pc,row sep=0.5pc]
\catpdgcog{\operad C} \arrow[r,"(-)^*"]
&\left(\catpdgcompalg{\operad C}\right)^\op.
\end{tikzcd}
\]

between pdg $\C$-coalgebras and qp-complete pdg $\C$-algebras.
Moreover, such a functor sends curved pdg $\C$-coalgebras to qp-complete curved $\C$-algebras
\end{lemma} 

\begin{proof}
Let $W$ be a pdg $\C$-coalgebra. The following map
\[
\gamma_{W^*}: \prod_{n \geq 0} [\C(n), (W^*)^{\otimes n}]^{\mathbb S_n} \rightarrowtail \prod_{n \geq 0} [\C(n), (W^{\otimes n})^{*}]^{\mathbb S_n} \xrightarrow{(\Delta_W)^*} W^*
\]
can be shown to induce a pdg $\C$-algebra structure on $W^*$. Let us check that $W^*$ is qp-complete. 

\medskip

It follows from Corollary \ref{cor: quasi-planar ladder} that $W$ can be written as the following colimit
\[
W \cong \colim{i \in \omega}~ F_{i}^{\qp} W 
\]

where $F_{i}^{\qp} W$ is the image of $W$ by the idempotent comonad induced by the inclusion $F_{i}^{\qp} \C \rightarrowtail \C$. Therefore 
\[
W^* \cong \lim_{i \in \omega}~ (F_{i}^{\qp} W)^*~. 
\]
If we show that $(F_{i}^{\qp} W)^*$ is qp-complete for every $i \in \omega$, then it is clear that $W$ is qp-complete. But this follows from the fact that since $F_{i}^{\qp} W$ is a curved $F_{i}^{\qp} \C$-coalgebra, its linear dual $(F_{i}^{\qp} W)^*$ is naturally a curved $F_{i}^{\qp} \C$-algebra. Therefore it is qp-complete. 

Finally, the assertion about curved $\C$-coalgebras follows by direct inspection.
\end{proof}

\begin{proposition}
Let $\C$ be a quasi-planar conilpotent curved cooperad. There is an adjunction 
\[
\begin{tikzcd}[column sep=5pc,row sep=3pc]
         \catcurvcog{\operad C} \arrow[r, shift left=1.1ex, "(-)^\ast"{name=A}]
         &\left(\catcurvcompalg{\operad C}\right)^\op~, \arrow[l, shift left=.75ex, "(-)^\vee"{name=B}] \arrow[phantom, from=A, to=B, , "\dashv" rotate=-90]
\end{tikzcd}
\]

between curved $\C$-coalgebras and qp-complete curved $\C$-algebras.
\end{proposition}

\begin{proof}
by the adjoint lifting theorem (see Appendix \ref{Appendix: Adjoint Lifting Theorem}) the composite functor
$$
\catcurvcog{\C} \to \left(\catcurvcompalg{\C}\right)^\op
\hookrightarrow \catcurvalg{\C}^\op
$$
has a right adjoint. The expected right adjoint is just its restriction to complete algebras.
\end{proof}

\begin{remark}
It can be shown that the adjunction in the above proposition is given by the same functors as the adjunction 
\[
\begin{tikzcd}[column sep=5pc,row sep=3pc]
         \catpdgcog{\operad C} \arrow[r, shift left=1.1ex, "(-)^\ast"{name=A}]
         &\left(\catpdgcompalg{\operad C}\right)^\op~. \arrow[l, shift left=.75ex, "(-)^\vee"{name=B}] \arrow[phantom, from=A, to=B, , "\dashv" rotate=-90]
\end{tikzcd}
\]
Showing this amounts to showing that the topological dual functor $(-)^\vee$ at the level of qp-complete pdg $\C$-algebras sends qp-complete curved $\C$-algebras to curved $\C$-coalgebras. This is the case since the arguments of \cite[Propositions 2.13 and 2.14]{absolutealgebras} generalize easily to positive characteristic, under the assumption that $\C$ is quasi-planar.
\end{remark}

\textbf{Algebraic duality square.} These two linear duality adjunctions intertwine the standard bar-cobar adjunction with the complete bar-cobar adjunction in the following way. 

\begin{proposition}[{After \cite[Theorem 2.16]{absolutealgebras}}]
Let $\alpha: \operad C \longrightarrow \operad P$ be a curved twisting morphism between a quasi-planar conilpotent curved cooperad $\C$ and a dg operad $\operad P$. The following square of adjunctions 
\[
\begin{tikzcd}[column sep=5pc,row sep=5pc]
\catdgalg{\operad P}^{\mathsf{op}} \arrow[r,"\mathrm{B}_{\alpha}^{\mathsf{op}}"{name=B},shift left=1.1ex] \arrow[d,"(-)^\circ "{name=SD},shift left=1.1ex ]
&\catcurvcog{\operad C}^{\mathsf{op}} \arrow[d,"(-)^*"{name=LDC},shift left=1.1ex ] \arrow[l,"\Omega_{\alpha}^{\mathsf{op}}"{name=C},,shift left=1.1ex]  \\
\catdgcog{\operad P}\arrow[r,"\widehat{\Omega}_{\alpha} "{name=CC},shift left=1.1ex]  \arrow[u,"(-)^*"{name=LD},shift left=1.1ex ]
&\catcurvcompalg{\operad C}~, \arrow[l,"\widehat{\mathrm{B}}_{\alpha}"{name=CB},shift left=1.1ex] \arrow[u,"(-)^\vee"{name=TD},shift left=1.1ex] \arrow[phantom, from=SD, to=LD, , "\dashv" rotate=0] \arrow[phantom, from=C, to=B, , "\dashv" rotate=-90]\arrow[phantom, from=TD, to=LDC, , "\dashv" rotate=0] \arrow[phantom, from=CC, to=CB, , "\dashv" rotate=-90]
\end{tikzcd}
\] 

is commutative in the following sense: right adjoints from the top right to the bottom left are canonically naturally isomorphic. This yields a canonical mate isomorphism between the left adjoints going the other way round.
\end{proposition}

\begin{proof}
It follows from the fact that, for any curved $\C$-coalgebra $W$, there is a natural isomorphism
\[
(\Omega_\alpha W)^\circ \cong \widehat{\mathrm{B}}_{\alpha}W^*~,
\]
which comes from the natural isomorphism
\[
(\operad P \circ W)^{\circ} \cong L^{\operad P}(W^*)~,
\]
at the level of graded $\kk$-modules. This later isomorphism is a formal consequence of the construction of $(-)^\circ$. 
\end{proof}

\subsection{Homotopical duality squares} Let us fix a quasi-planar conilpotent curved cooperad $\C$. When we consider the duality square which corresponds to the canonical curved twisting morphism $\iota: \C \longrightarrow \Omega \C$, this square can be promoted in the following way: all the adjunctions in the square are Quillen adjunctions. 

\begin{proposition}[{After \cite[Theorem 2.22]{absolutealgebras}}]
The following square of adjunctions 
\[
\begin{tikzcd}[column sep=5pc,row sep=5pc]
\catdgalg{\Omega \operad C}^{\mathsf{op}} \arrow[r,"\mathrm{B}_{\operad C}^{\mathsf{op}}"{name=B},shift left=1.1ex] \arrow[d,"(-)^\circ "{name=SD},shift left=1.1ex ]
&\catcurvcog{\operad C}^{\mathsf{op}} \arrow[d,"(-)^*"{name=LDC},shift left=1.1ex ] \arrow[l,"\Omega_{\operad C}^{\mathsf{op}}"{name=C},,shift left=1.1ex]  \\
\catdgcog{\Omega \operad C}\arrow[r,"\widehat{\Omega}_{\operad C} "{name=CC},shift left=1.1ex]  \arrow[u,"(-)^*"{name=LD},shift left=1.1ex ]
&\catcurvcompalg{\operad C}~, \arrow[l,"\widehat{\mathrm{B}}_{\operad C}"{name=CB},shift left=1.1ex] \arrow[u,"(-)^\vee"{name=TD},shift left=1.1ex] \arrow[phantom, from=SD, to=LD, , "\dashv" rotate=0] \arrow[phantom, from=C, to=B, , "\dashv" rotate=-90]\arrow[phantom, from=TD, to=LDC, , "\dashv" rotate=0] \arrow[phantom, from=CC, to=CB, , "\dashv" rotate=-90]
\end{tikzcd}
\] 

is made of Quillen adjunctions, when one considers the right/left transferred structure from dg modules on the left hand side, and where one considers the transferred structures along the bar-cobar adjunctions on the right hand side. 
\end{proposition}

\begin{proof}
We already know that the left vertical adjunction and the two horizontal adjunctions are all Quillen adjunctions. The only thing left to check is that the adjunction 
\[
\begin{tikzcd}[column sep=5pc,row sep=3pc]
         \catcurvcog{\operad C} \arrow[r, shift left=1.1ex, "(-)^\ast"{name=A}]
         &\left(\catcurvcompalg{\operad C}\right)^\op~, \arrow[l, shift left=.75ex, "(-)^\vee"{name=B}] \arrow[phantom, from=A, to=B, , "\dashv" rotate=-90]
\end{tikzcd}
\]

is also a Quillen adjunction. This follows from the fact that the model structure on $\catcurvalg{\operad C}^\op$ is transferred from that on $\catdgcog{\Omega \operad C}^\op$, and therefore it suffices to notice that the composite left adjoint functor  
\[
\begin{tikzcd}[column sep=3.5pc,row sep=0.5pc]
\catcurvcog{\operad C} \arrow[r,"\Omega_{\operad C}"]
&\catdgalg{\Omega \operad C} \arrow[r,"(-)^\circ"]
&\catdgcog{\Omega \operad C}^\op
\end{tikzcd}
\]
is left Quillen.
\end{proof}

\begin{remark}
Let $\operad P$ be a cofibrant dg operad. There is an analogue homotopical duality square where the bar-cobar adjunctions are the quasi-planar bar-cobar adjunctions of Subsection \ref{subsection: extended bar-cobar}.
\end{remark}

\begin{proposition}\label{prop: equivalence of infinity cat of finite dimensional objects}
Let $\operad P$ be a cofibrant dg operad. The adjunction induced at the level of $\infty$-categories on the localisations of the model categories
\[
\begin{tikzcd}[column sep=5pc,row sep=3pc]
         \catdgcog{\operad P}~[\mathrm{Q.iso}^{-1}] \arrow[r, shift left=1.1ex, "(-)^\ast"{name=A}]
         &\catdgalg{\operad P}~[\mathrm{Q.iso}^{-1}]^\op~, \arrow[l, shift left=.75ex, "(-)^\circ"{name=B}] \arrow[phantom, from=A, to=B, , "\dashv" rotate=-90]
\end{tikzcd}
\]
restricts to an equivalence 

\[
\begin{tikzcd}[column sep=5pc,row sep=3pc]
         \catdgcog{\operad P}~[\mathrm{Q.iso}^{-1}]_{\mp}^{\mathrm{f.d.}} \arrow[r, shift left=1.1ex, "(-)^\ast"{name=A}]
         &\catdgalg{\operad P}~[\mathrm{Q.iso}^{-1}]_{\pm}^{\mathrm{f.d.}\op}~, \arrow[l, shift left=.75ex, "(-)^\circ"{name=B}] \arrow[phantom, from=A, to=B, , "\dashv" rotate=-90]
\end{tikzcd}
\]
between the full sub-$\infty$-category $\catdgcog{\operad P}[\mathrm{Q.iso}^{-1}]_\mp^{\mathrm{f.d.}}$ of $\catdgcog{\operad P}~[\mathrm{Q.iso}^{-1}]$ spanned by coalgebras with degree-wise finite dimensional and bounded above (resp. bounded below) homology and the full sub-$\infty$-category $\catdgalg{\operad P}~[\mathrm{Q.iso}^{-1}]_{\pm}^{\mathrm{f.d.}}$ of $\catdgalg{\operad P}~[\mathrm{Q.iso}^{-1}]$ spanned by algebras with degree-wise finite dimensional and bounded below (resp. bounded above) homology. 
\end{proposition}

\begin{proof}
Since the weak-equivalence of cofibrant dg operads $\Omega B (\operad E \otimes \operad P) \qi \operad P$ yields Quillen equivalences between their respective the categories of algebras and coalgebras by Theorem \ref{propqeqisoperads}, we can restrict to the case where $\operad P = \Omega \operad C$ for $\operad C$ a quasi-planar conilpotent curved cooperad, without any loss of generality.

\medskip

It suffices then to show that the two derived functors associated to the Quillen adjunction $(-)^\ast \dashv (-)^\circ$ interchange objects with degree-wise finite dimensional and bounded above homology with objects with degree-wise finite dimensional and bounded below homology. One the one hand, it is clear that the (derived) functor $(-)^\ast$ does so. On the other hand, let $A$ be a dg $\operad P$-algebra with degree-wise finite dimensional and bounded below homology (the bounded above case is analogue). 

\medskip

Let us first assume that $A$ is degree-wise finite dimensional and bounded below. Let us denote $\catdgalg{\operad P}_{+}^{\mathrm{f.d.}}$ and $\catdgcog{\operad P}_{-}^{\mathrm{f.d.}}$ respectively the full subcategory of dg $\operad P$-algebras who are degree-wise finite dimensional and bounded below, and the full subcategory of dg $\operad P$-coalgebras spanned by objects who are degree-wise finite dimensional and bounded above. The linear duality functor from dg $\operad P$-coalgebras to dg $\operad P$-algebras restricts to an equivalence of categories
$$
\catdgcog{\operad P}_{-}^{\mathrm{f.d.}} \xrightarrow{(-)^\ast} \catdgalg{\operad P}_{+}^{\mathrm{f.d.}\op}
$$
whose pseudo-inverse is also a lifting $(-)^\ast$ of the linear duality functor of dg modules. The following square diagram is commutative
\[
\begin{tikzcd}[column sep=3pc,row sep=3pc]
\catdgalg{\operad P}_{+}^{\mathrm{f.d.}} \ar[r, "\simeq",swap] \ar[r, "(-)^\ast"]
\ar[d, "B_{\operad C}"']
& \catdgcog{\operad P}_{-}^{\mathrm{f.d.}\op}
\ar[d, "\widehat\Omega_{\operad C}"]
\\
\catcurvcog{\operad C} \ar[r, "(-)^\ast"] 
& \left(\catcurvcompalg{\operad C}\right)^{\op}~.
\end{tikzcd}
\]
This gives a sequence of natural isomorphisms
\[
\begin{tikzcd}[column sep=1.5pc,row sep=0.5pc]
(\Omega_{\operad C}B_{\operad C} A)^\circ \arrow[r,"\cong"]
&\widehat B_{\operad C} ((B_{\operad C} A)^\ast) \arrow[r,"\cong"]
&\widehat B_{\operad C} \widehat\Omega_{\operad C} (A^\ast)~.
\end{tikzcd}
\]
Notice that $\eta_{A^*}: A^\ast \qi \widehat B_{\operad C} \widehat\Omega_{\operad C} A^\ast$ is a quasi-isomorphism of dg $\operad P$-coalgebras. Since $(\Omega_{\operad C}B_{\operad C} A)^\circ$ is weakly-equivalent to the value of the left derived functor of $(-)^\circ$ taken on $A$, then the derived unit of adjunction is a quasi-isomorphism for any $A$ which is degree-wise finite dimensional and bounded below.

\medskip

In the general case where $A$ has degree-wise finite dimensional and bounded below homology, the homotopy transfer theorem endows the dg module $\mathrm{H}(A)$ (with zero differential) with the structure of a dg $\operad P$-algebra and provides a zig-zag of weak-equivalences of dg $\operad P$-algebras relating $A$ to $\mathrm{H}(A)$. Thus the image of $A$ through the derived functor of $(-)^\circ$ is equivalent to that of $\mathrm{H}(A)$.
\end{proof}


\newpage

\appendix
\section{Adjoint lifting theorems, right and left transferred structures}
\subsection{Adjoint lifting theorem}\label{Appendix: Adjoint Lifting Theorem}
The goal of this appendix is to give recollections on the adjoint lifting theorem. We mainly follow the work of P. T. Johnstone in \cite{AdjointLifting}.

\medskip

Let us consider two categories $\categ C, \categ D$. Let $M$ be a monad on $\categ C$ and let $N$ be a monad on $\categ D$. Moreover, let us consider a commutative diagrams of functors
$$
\begin{tikzcd}[column sep=3pc,row sep=3pc]
\categ{Alg}_{\categ C}(M)
\ar[r, "R_m"] \ar[d, "U^M",swap]
& \categ{Alg}_{\categ D}(N)
\ar[d, "U^N"]
\\
\categ C
\ar[r, "R"']
& \categ D~,
\end{tikzcd}
$$
where $U^N$ and $U^M$ are the monadic forgetful functors and where $R$ is a right adjoint functor between the underlying categories. Notice that the commutativity of this diagram induces a natural transformation $\lambda: NR \longrightarrow RM$ that satisfies some coherence conditions.

\medskip

We define the natural transformation $\xi: MLN \longrightarrow ML$ as follows. First we consider the composition
\[
\begin{tikzcd}[column sep=3.5pc,row sep=0pc]
\varphi: L~N \arrow[r,"L~N~\eta_{RL}"]
&L~N~R~L \arrow[r,"L~\lambda~L"]
&L~R~M~L~ \arrow[r,"\epsilon_{LR}~M~L"]
&M~L
\end{tikzcd}
\]
where $\eta_{RL}: \id \longrightarrow R~L$ is the unit of adjunction and $\epsilon_{LR}:L~R \longrightarrow \id $ the counit. Now simply notice that $M \cong U_M~F_M$, where $F_M$ is the free $M$-algebra functor. Therefore by adjunction, we get a natural transformation
\[
\xi: F_M~L~N \longrightarrow F_M~L~,
\]
by taking the transpose of $\varphi$.

\begin{theorem}[Adjoint lifting theorem]\label{theoremadjointlifting}
Let us suppose that $R$ has a left adjoint $L$. Then $R_m$ has a left adjoint $L_m$ if and only if for every $N$-algebra $(A,\gamma_A)$, the reflexive pair

\[
\begin{tikzcd}
        F_M~L~N~(A) \ar[rr, shift left = 0.7ex, "\xi_A"]
        \ar[rr, shift right = 0.7ex, "F_M~L~\gamma_A"'] 
        &&F_M~L~(A)~,
\end{tikzcd}
\]
   
has a coequaliser in $\categ{Alg}_{\categ C}(M)$. If this is the case, then the image of the functor $L_M$ on $A$ is given by the above coequalizer.
\end{theorem}

\begin{proof}
If such coequalisers exists, then the natural construction that sends $A$ to this coequaliser can be checked to be the right adjoint of $R_m$. Conversely, if $R_m$ has a left adjoint $L_m$, then the universal property satisfied by $L_m(A)$ makes it the expected coequaliser.
\end{proof}

Now, let $M, N$ be two monads on a category $\categ C$ and let us consider a morphism of monads $f:N \longrightarrow M$ which leads to a functor $U_f:\catalg{M} \longrightarrow \catalg{N}$ above the identity functor of the category $\categ C$.

\begin{proposition}\label{propadjointliftingepi}
Let us suppose that $\categ C$ has reflexive coequalisers and these coequalisers are preserved by $M$, thus also by the forgetful functor $U_M: \catalg{M}\longrightarrow \categ C$. Let us also suppose that the map $f(X): N(X) \twoheadrightarrow M(X)$ is an epimorphism for every object $X$ in $\categ C$. Then the functor $U^f:\catalg{M} \longrightarrow \catalg{N}$ above $\categ C$ has a left adjoint $T_f$ and the following diagram
$$
\begin{tikzcd}[column sep=3pc,row sep=3pc]
    N(A) \arrow[dr, phantom, "\ulcorner", very near end]
    \ar[r,"f(A)"] \ar[d,"\gamma_A",swap]
    &M(A)
    \ar[d]
    \\
    A
    \ar[r]
    &T_f(A).
\end{tikzcd}
$$
is a pushout diagram in $\categ C$, for every $N$-algebra $(A,\gamma_A)$.
\end{proposition}

\begin{proof}
We can notice that this pushout in $\categ C$ is canonically isomorphic to the image the coequaliser of the reflexive pair of Theorem \ref{theoremadjointlifting} by the forgetful functor $U_M$, which creates and reflects reflexive coequilisers. 
\end{proof}

\subsection{Right and left transfers of model structures}
\label{appendixtransfer}
The goal of this appendix is to review the different results that allow one to transfer a model category structure along an adjunction. For the rest of this appendix, we consider an adjunction between presentable categories
\[
\begin{tikzcd}[column sep=5pc,row sep=3pc]
         \categ C \arrow[r, shift left=1.1ex, "L"{name=A}]
         &\categ D~. \arrow[l, shift left=.75ex, "R"{name=B}] \arrow[phantom, from=A, to=B, , "\dashv" rotate=-90]
\end{tikzcd}
\]


\subsubsection{Right transfer}
Let us suppose that $\categ C$ is endowed with a cofibrantly generated (thus combinatorial) model structure. The sets of maps in $\categ C$ that form its model structure induce a sets of maps in $\categ D$ via the adjunction $L \dashv R$.

\begin{definition}[Sets of maps in $\categ D$]\label{definition : right model structure}
    Let $f: X \longrightarrow Y$ be a morphism in $\categ D$. Let us call it
    \begin{itemize}
    \item a \textit{fibration} if $R(f)$ is a fibration;
    \item a \textit{weak-equivalence} if $R(f)$ is a weak equivalence;
    \item an \textit{acyclic fibration} if $R(f)$ is an acyclic fibration (equivalently if $f$ is both a fibration and a weak-equivalence);
    \item a \textit{cofibration} if it has the left lifting property with respect to acyclic fibrations;
    \item an \textit{acyclic cofibration} if it is both a cofibration and a weak equivalence;
    \item a \textit{generating cofibration} if it is the image through $L$ of a generating cofibration of $\categ C$;
    \item a \textit{generating acyclic cofibration} if it is the image through $L$ of a generating acyclic cofibration of $\categ C$;
    \item a \textit{left fibration-lifting map} is it has the left lifting property with respect to fibrations.
\end{itemize}
We will refer to them as \textit{the sets of maps in} $\categ D$ \textit{induced by the adjunction} $L \dashv R$.
\end{definition}

\begin{lemma}\label{lemma: right model structure}
The sets of maps in $\categ D$ induced by the adjunction $L \dashv R$ satisfy the following properties.
    \begin{enumerate}
        \item The weak-equivalences, the cofibrations and the fibrations are stable through composition and retracts. Furthermore, they contain all isomorphisms.
        \item The weak-equivalences follows the 2-out-of-3 rule and the 2-out-of-6 rule.
        \item Every commutative square in $\categ D$
        $$
        \begin{tikzcd}
            U
            \ar[r] \ar[d, "f"']
            & X \ar[d, "g"]
            \\
            Y \ar[r]
            & Z
        \end{tikzcd}
        $$
        admits a lifting whenever
        \begin{enumerate}
            \item $f$ is a cofibration and $g$ is an acylic fibration,
            \item or $f$ is a left fibration-lifting map and $g$ is a fibration.
        \end{enumerate}
        \item Every map $f: X \longrightarrow Y$ may be factored in a natural way as either
        \begin{enumerate}
            \item the composition of a cofibration followed by an acyclic fibration,
            \item or the composition of a left fibration-lifting map followed by a fibration.
        \end{enumerate}
        \item Cofibrations are retracts of transfinite compositions of pushouts of generating cofibrations.
        \item Left fibration-lifting maps are retracts of transfinite compositions of pushouts of generating acyclic cofibrations.
    \end{enumerate}
\end{lemma}

\begin{proof}
It manly follows from a straightforward check and from the small object argument.
\end{proof}

\begin{theorem}[Right acyclicity condition, \cite{Hovey}]\label{thmrightacyclicitycondition}
The following assertions about the sets of maps in $\categ D$ induced by the adjunction $L \dashv R$ are equivalent.
    \begin{enumerate}
        \item These maps define a combinatorial model category structure on $\categ D$.
        \item The set of left fibration-lifting maps of $\categ D$ is equal to the set of acyclic cofibrations.
        \item The set of left fibration-lifting maps of $\categ D$ is contained in the set of acyclic cofibrations of $\categ D$.
        \item The set of left fibration-lifting maps of $\categ D$ is contained in the set of weak-equivalences of $\categ D$.
    \end{enumerate}
\end{theorem}

\begin{proof}
Clearly, (1) implies (2). Conversely, if (2) is satified, then by Lemma \ref{lemma: right model structure}, the sets of maps in $\categ D$ induced by the adjunction $L \dashv R$ form a model category structure. Therefore (1) and (2) are equivalent.

\medskip

It is clear that (2) implies (3) and that (3) implies (4). Let us assume (4) and prove (2). Let $f: U \to Y$ be an acylic cofibration. Let us decompose $f: X \longrightarrow Y$ as a left fibration-lifting map $a: X \longrightarrow U$ followed by a fibration $b: U \longrightarrow Y$. By the 2-out-of-3 rule and since $f$ and $a$ are weak equivalences, $b$ is an acyclic fibration. Thus the square
        $$
        \begin{tikzcd}[column sep=3pc,row sep=3pc]
            X
            \ar[r, "a"] \ar[d, "f"']
            & U \ar[d, "b"]
            \\
            Y \ar[r, "\id"]
            & Y
        \end{tikzcd}
        $$
    has a lifting $i$, which yields the following retract commutative diagram
        $$
        \begin{tikzcd}[column sep=3pc,row sep=3pc]
            X
            \ar[r,"\id"] \ar[d, "f"']
            & X \ar[d, "a"]
            \ar[r, "\id"]
            & X \ar[d, "f"]
            \\
            Y \ar[r, "i"']
            & U
            \ar[r,"b"']
            & Y~.
        \end{tikzcd}
        $$
      	Since $f$ is a retract of $a$, it is a left fibration-lifting map.
\end{proof}

\begin{proposition}[\cite{Hovey}]\label{propositionacyclicitycondition}
Let us suppose that

\medskip
\begin{enumerate}
        \item every object in $\categ D$ has a natural fibrant replacement functor, that is, there exists an endofunctor $F$ of $\categ D$ together with a natural transformation $e: \id \longrightarrow F$ so that for every object $X$, $F(X)$ is fibrant and the map $X \longrightarrow F(X)$ is a weak-equivalence;
        
\medskip

        \item every fibrant object $X$ has a path object, that is, the diagonal map $X \longrightarrow X \times X$ can be factored by a weak-equivalence followed by a fibration.
\end{enumerate}

\medskip

Then the sets of maps in $\categ D$ induced by the adjunction $L \dashv R$ endow $\categ D$ with a model category structure.
\end{proposition}

\begin{proof}
Let $f: X \longrightarrow Y$ be a left fibration-lifting map. Its lifting property gives us a map $p : Y \longrightarrow F(X)$ such that $pf = e(X)$. Let $P$ be a path object of $Y$. It fits in the following commutative diagram
    $$
    \begin{tikzcd}[column sep=3pc,row sep=3pc]
        X
        \ar[r, "e(X)"] \ar[d, "f"']
        & F(X)
        \ar[r, "F(f)"]
        & F(Y)
        \ar[r, "\simeq"]
        & P
        \ar[d, two heads]
        \\
        Y
        \ar[rrr,"{(F(f) p, e(Y))}"']
        &&& F(Y) \times F(Y)~.
    \end{tikzcd}
    $$
This square has a lifting $g: Y \longrightarrow P$. Since each of the two projections $P \longrightarrow F(Y)$ are weak-equivalences and since the map $e(Y): Y \longrightarrow F(Y)$ is a weak-equivalence, the 2-out-of-3 rule tells us that $g$ and $F(f)p$ are also weak equivalences. Then $pf = e(X)$ is also a weak-equivalence. The 2-out-of-6 rule implies that the three maps $f, p , F(f)$ are also weak-equivalences. We conclude by Theorem \ref{thmrightacyclicitycondition}.
\end{proof}


\subsubsection{Left transfer}
Let us suppose that $\categ D$ is endowed with a cofibrantly generated model structure. The sets of maps in $\categ D$ that form a model structure induce sets of maps in $\categ C$ via the adjunction $L \dashv R$.

\begin{definition}[Sets of maps in $\categ C$]\label{definition : left model structure}
Let $f: X \longrightarrow Y$ be a morphism in $\categ C$. Let us call it
    \begin{itemize}
    \item a \textit{cofibration} if $L(f)$ is a cofibration;
    \item a \textit{weak-equivalence} if $L(f)$ is a weak-equivalence;
    \item an \textit{acyclic cofibration} if $L(f)$ is an acyclic cofibration (equivalently if $f$ is both a cofibration and a weak-equivalence);
    \item a \textit{fibration} if it has the right lifting property with respect to acyclic cofibrations;
    \item an \textit{acyclic fibration} if it is both a cofibration and a weak-equivalence;
    \item a \textit{right cofibration-lifting} map is it has the right lifting property with respect to cofibrations.
\end{itemize}

We will refer to them as \textit{the sets of maps in} $\categ C$ \textit{induced by the adjunction} $L \dashv R$.
\end{definition}

\begin{proposition}[\cite{MakkaiPare}]\label{propmakkaipare}
There exists a small set $\mathcal I$ of cofibrations of $\categ C$ so that cofibrations are retracts of transfinite compositions of pushouts of maps in $\mathcal I$. Similarly, there exists a small set $\mathcal J$ of acyclic cofibrations of $\categ C$ so that acyclic cofibrations are retracts of transfinite compositions of pushouts of maps in $\mathcal J$.

\medskip

We call $\mathcal I$ and $\mathcal J$ respectively the set of generating cofibrations and the set of generating acyclic cofibrations of $\categ C$.
\end{proposition}

\begin{lemma}\label{lemma:left model structure}
The sets of maps in $\categ D$ induced by the adjunction $L \dashv R$ satisfy the following properties.
    \begin{enumerate}
        \item The weak-equivalences, the cofibrations and the fibrations are stable through composition and retracts, and they contain all isomorphisms.
        \item The weak-equivalences follow the 2-out-of-3 rule and the 2-out-of-6 rule.
        \item Every square in $\categ D$
        $$
        \begin{tikzcd}
            U
            \ar[r] \ar[d, "f"']
            & X \ar[d, "g"]
            \\
            Y \ar[r]
            & Z
        \end{tikzcd}
        $$
        has a lifting whenever 
        \begin{enumerate}
            \item $f$ is an acyclic cofibration and $g$ is a fibration,
            \item or $f$ is a cofibration and $g$ is a left right cofibration-lifting map.
        \end{enumerate}
        \item Every map $f: X \longrightarrow Y$ may be factored in a natural way as either
        \begin{enumerate}
            \item the composition of a cofibration followed by a right cofibration-lifting map,
            \item or the composition of an acyclic cofibration map followed by a fibration.
        \end{enumerate}
        \item Cofibrations are retracts of transfinite compositions of pushouts of generating cofibrations.
        \item Left fibration-lifting maps are retracts of transfinite compositions of pushouts of generating acyclic cofibrations.
    \end{enumerate}
\end{lemma}

\begin{proof}
This follows from a straightforward check, the small object argument and Proposition \ref{propmakkaipare}.
\end{proof}

\begin{theorem}[Left acyclicity condition, \cite{womenintopology}]\label{thm:left transfer}
The following assertions about the sets of maps in $\categ C$ induced by the adjunction $L \dashv R$ are equivalent.
    \begin{enumerate}
        \item These maps define a combinatorial model category structure on $\categ C$.
        \item The set of right cofibration-lifting maps of $\categ C$ is equal to the set of acyclic fibrations.
        \item The set of right cofibration-lifting maps of $\categ C$ is contained in the set of acyclic fibrations of $\categ C$.
        \item The set of right cofibration-lifting maps of $\categ C$ is contained in the set of weak-equivalences of $\categ C$.
    \end{enumerate}
\end{theorem}

\begin{proof}
Clearly, (1) implies (2). Conversely, if (2) is satified, then by Lemma \ref{lemma: left model structure}, the sets of maps in $\categ C$ induced by the adjunction $L \dashv R$ are equivalent form a model category structure on $\categ C$. Therefore (1) and (2) are equivalent.

\medskip

It is clear that (2) implies (3) and that (3) implies (4). Let us assume (4) and prove (2). Let $f: U \to Y$ be an acylic fibration. Let us decompose $f: X \longrightarrow Y$ as a cofibration $a: X \longrightarrow U$ followed by a right cofibration-lifting map $b: U \longrightarrow Y$. By the 2-out-of-3 rule and since $b$ and $f$ are weak-equivalences, $a$ is an acyclic cofibration. Thus the commutative square
        $$
        \begin{tikzcd}[column sep=3pc,row sep=3pc]
            X
            \ar[r, "\id"] \ar[d, "a"']
            & X \ar[d, "f"]
            \\
            U \ar[r, "b"']
            & Y
        \end{tikzcd}
        $$
    has a lifting $p$, which yields the following retract commutative diagram
        $$
        \begin{tikzcd}[column sep=3pc,row sep=3pc]
            X
            \ar[r, "a"] \ar[d, "f"']
            & U \ar[d, "b"]
            \ar[r, "p"]
            & X \ar[d, "f"]
            \\
            Y \ar[r, "\id"]
            & Y
            \ar[r,"\id"]
            & U~.
        \end{tikzcd}
        $$
        Since the map $f$ is a retract of $b$, it is a right cofibration-lifting map.
\end{proof}

\begin{proposition}[\cite{womenintopology2}]
    Let us suppose that
    
\medskip

    \begin{enumerate}
        \item every object in $\categ C$ has a natural cofibrant replacement functor, that is, there exists an endofunctor $Q$ of $\categ C$ together with a natural transformation $c: Q \longrightarrow \id$ so that for every object $X$, $Q(X)$ is cofibrant and the map $Q(X) \longrightarrow X$ is a weak-equivalence;
        
\medskip

        \item every cofibrant object $X$ has a cylinder object, that is, the codiagonal map $X \sqcup X \longrightarrow X$ may be factored by a cofibration followed by a weak-equivalence.
    \end{enumerate}
    
\medskip

Then the sets of maps in $\categ C$ induced by the adjunction $L \dashv R$ endow $\categ C$ with a model category structure.
\end{proposition}

\begin{proof}
This follows from dual arguments as those used to prove Proposition \ref{propositionacyclicitycondition}, using Theorem \ref{thm:left transfer}.
\end{proof}

\vspace{4pc}

\bibliographystyle{alpha}
\bibliography{bibax}
\end{document}